\documentclass[12pt]{amsart}
\usepackage{SSdefn}
\usepackage{eucal}
\usepackage{shuffle}

\DeclareFontFamily{U}{medshuffle}{}
\DeclareFontShape{U}{medshuffle}{m}{n}{
  <5-8> s*[1.6] shuffle7
  <8->  s*[1.6] shuffle10
}{}
\DeclareSymbolFont{MedShuffle}{U}{medshuffle}{m}{n}
\DeclareMathSymbol\medshuffle{\mathop}{MedShuffle}{"001}

\DeclareMathOperator{\cone}{cone}
\DeclareMathOperator{\Ext}{Ext}

\DeclareMathOperator{\Inj}{Inj}
\DeclareMathOperator{\IndInj}{IndInj}
\DeclareMathOperator{\Sat}{Sat}
\DeclareMathOperator{\Mod}{\mathbf{Mod}}
\DeclareMathOperator{\Rep}{\mathbf{Rep}}
\DeclareMathOperator{\REP}{\mathbf{REP}}
\DeclareMathOperator{\EXT}{EXT}
\DeclareMathOperator{\uRep}{\ul{\Rep}}
\DeclareMathOperator{\uExt}{\ul{Ext}}

\DeclareMathOperator{\ini}{in}
\DeclareMathOperator{\Ch}{\mathbf{Ch}}
\DeclareMathOperator{\lev}{\mathrm{lev}}
\DeclareMathOperator{\ulev}{\ul{\lev}}

\newcommand{\cor}{\mathrm{cor}}
\newcommand{\rss}{r\text{-}\mathrm{ss}}

\newcommand{\smlim}{\varprojlim\!{}'\,}
\renewcommand{\Vec}{\mathbf{Vec}}
\newcommand{\GrVec}{\mathbf{GrVec}}

\newcommand{\umu}{\ul{\smash{\mu}}}

\newcommand{\FI}{\mathbf{FI}}
\newcommand{\OI}{\mathbf{OI}}

\newcommand{\ldb}{\{\!\{}
\newcommand{\rdb}{\}\!\}}

\newcommand{\fgen}{\mathrm{fg}}
\newcommand{\sm}{\mathrm{sm}}
\newcommand{\lf}{\mathrm{lf}}

\let\whitetri a
\let\blacktri b

\newcommand{\stacks}[1]{\cite[\href{http://stacks.math.columbia.edu/tag/#1}{Tag~#1}]{stacks}}

\title{The representation theory of the increasing monoid}

\author{Sema G\"unt\"urk\"un}
\address{Department of Mathematics, University of Michigan, Ann Arbor, MI   \newline  (current address: Department of Mathematics, University of Connecticut, Storrs, CT)}
\email{\href{mailto:gunturku@umich.edu}{gunturku@umich.edu}}
\urladdr{\url{http://www-personal.umich.edu/~gunturku/}}

\author{Andrew Snowden}
\address{Department of Mathematics, University of Michigan, Ann Arbor, MI}
\email{\href{mailto:asnowden@umich.edu}{asnowden@umich.edu}}
\urladdr{\url{http://www-personal.umich.edu/~asnowden/}}

\thanks{AS was supported by NSF DMS-1303082 and DMS-1453893 and a Sloan Fellowship.}

\date{\today}


\begin{document}

\begin{abstract}
We study the representation theory of the increasing monoid. Our results provide a fairly comprehensive picture of the representation category: for example, we describe the Grothendieck group (including the effective cone), classify injective objects, establish properties of injective and projective resolutions, construct a derived auto-duality, and so on. Our work is motivated by numerous connections of this theory to other areas, such as representation stability, commutative algebra, simplicial theory, and shuffle algebras.
\end{abstract}

\maketitle
\tableofcontents

\newpage
\section{Introduction}

The {\bf increasing monoid} $\cI$ is the monoid of all (tame\footnote{See \S \ref{ss:inc} for the definition of tame.}) order-preserving injections from the set $\{1,2,\ldots\}$ of positive integers to itself. The purpose of this paper is to investigate the representation theory of $\cI$. Our results, outlined in \S \ref{ss:results} below, give a fairly comprehensive picture of the structure of representations.

\subsection{Motivation}

We begin by explaining why the representation theory of $\cI$ is worthy of study, and how it connects to some other topics.

\subsubsection{Symmetric ideals}

Cohen \cite{Cohen,cohen2} proved that ideals in the infinite variable polynomial ring $R=\bk[x_1,x_2,\ldots]$ that are stable under the infinite symmetric group $\fS_{\infty}$ satisfy the ascending chain condition---that is, $R$ is $\fS_{\infty}$-noetherian---and used this theorem to prove a finiteness property of the variety of metabelian groups. This theorem has received intense interest in the last decade (see, e.g., \cite{AschenbrennerHillar, HillarSullivant, KLS, LNNR, NR, NR2}) as researchers from disparate areas have realized its utility in establishing stabilization phenomena; see \cite{Draisma} for a good introduction. In all treatments of $\fS_{\infty}$-ideals, the increasing monoid has played a key role: it respects the natural monomial order on $R$, and thus allows one to carry out equivariant Gr\"obner theory, which is crucial for establishing finiteness results. The theory developed in this paper should therefore shed some new light on the theory of $\fS_{\infty}$-ideals. In fact, we have already used the results of this paper to prove a new result in this direction \cite{inchilb}.

\subsubsection{$\OI$-modules}

A related topic, which has also seen a surge of interest in recent years, is the representation theory of combinatorial categories, especially the categories $\FI$ and $\OI$; see, for example, \cite{fimodule, symc1, catgb}. The category of $\OI$-modules is (essentially) equivalent to the category of graded $\cI$-modules studied in this paper (see \S \ref{ss:oiequiv}). Thus the results of this paper establish a structure theory for $\OI$-modules. In fact, this paper can be viewed as an ``$\OI$ analog'' of the paper \cite{symc1}, which establishes structural results for $\FI$-modules. As a consequence, the results herein should be of use wherever $\OI$-modules are employed. For example, recent work of Putman, Sam, and the second author \cite{fiq} shows that the homology of unipotent groups admits the structure of a finitely generated $\OI$-module (in certain cases). It would be interesting to apply the results of this paper to refine the results of loc.\ cit.; see \S \ref{ques:unipotent} for more details.

\subsubsection{Semi-simplicial vector spaces}

The celebrated Dold--Kan theorem shows that the category of simplicial vector spaces is equivalent to the category of chain complexes of vector spaces supported in non-negative degrees. Chain complexes of vector spaces have a simple structure: there are two indecomposable objects up to shift. Thus, via Dold--Kan, simplicial vector spaces are also rather simple, and one can establish a complete structure theory for them.

Semi-simplicial vector spaces, on the other hand, are far more complicated. However, the category of semi-simplicial vector spaces is (essentially) equivalent to the category of graded $\cI$-modules (see \S \ref{ss:ssequiv}). Thus the results of this paper can be seen as a structure theory for semi-simplicial vector spaces.

\subsubsection{Commutative algebra}

There are two interesting connections between commutative algebra and the theory of $\cI$-modules. First, many definitions, constructions, and results for $\cI$-modules are direct analogues of concepts in commutative algebra. For example, we define versions of Hilbert series, Krull dimension, Tor, Betti tables, the Koszul complex, and local cohomology for $\cI$-modules, and establish familiar relationships such as that between Krull dimension and the degree of the pole of the Hilbert series at $t=1$. The richness of commutative algebra therefore hints that there may be considerable depth to the theory of $\cI$-modules. In this direction, we formulate a version of the Boij--S\"oderberg problem for $\cI$-modules (\S \ref{ques:boij}).

Second, a direct overlap between commutative algebra and $\cI$-modules is found in monomial ideals. Precisely, we introduce an $\cI$-module $\bA^n$, called the $n$th principal module, which is isomorphic, as a vector space, to the $n$-variable polynomial ring. We show that, under this linear isomorphism, monomial submodules of $\bA^n$ correspond bijectively to monomial ideals of the polynomial ring (Proposition~\ref{prop:prinpoly}.). An interesting problem is to understand how properties of monomial ideals relate to the corresponding properties of monomial $\cI$-modules (\S \ref{ques:monomial}).

\subsubsection{Shuffle algebras}

An {\bf $(n,m)$-shuffle} is a bijection $[n] \amalg [m] \to [n+m]$ that is order-preserving on each factor. A {\bf shuffle algebra} is a graded vector space $A = \bigoplus_{n \ge 0} A_n$ equipped with a multiplication map $\ast_{\sigma} \colon A_n \times A_m \to A_{n+m}$ for each $(n,m)$-shuffle $\sigma$, satisfying some natural axioms (see \cite[\S 2]{dotsenko}). Commutative shuffle algebras are the ``ordered version'' of twisted commutative algebras, which have played a prominent role in representation stability. We show that the category of graded $\cI$-modules is (essentially) equivalent to the category of modules over the commutative shuffle algebra $A$ freely generated by one element in degree one (\S \ref{ss:shuffle}). Thus the results of this paper can be viewed as a structure theory for $A$-modules. While shuffle algebras have cropped up in several places \cite{dotsenko,laudone,ronco,sam}, this paper represents the first effort to understand their modules in any detail (as far as we know). As this paper is entirely devoted to the module theory of $A$, which is the simplest shuffle algebra, there is likely much more to discover in this direction.

\subsection{Fundamental definitions} \label{ss:fund}

Before stating our results, we must first introduce a few key concepts. Fix a field $\bk$. By an $\cI$-module, we will mean a representation of $\cI$ over $\bk$, or, equivalently, a left module over the monoid algebra $\bk[\cI]$. As general $\cI$-modules are rather wild, we focus on two more reasonable classes:
\begin{itemize}
\item Let $\cI_{>n}$ be the submonoid of $\cI$ consisting of those elements that fix each of the numbers $1, \ldots, n$. We say that an $\cI$-module $M$ is {\bf smooth} if every vector $x \in M$ is fixed by some $\cI_{>n}$, where $n$ depends on $x$. We let $\Rep(\cI)$ denote the category of smooth representations.
\item A {\bf graded} $\cI$-module is an $\cI$-module $M$ equipped with a grading $M=\bigoplus_{n \ge 0} M_n$ such that $\sigma \in \cI$ maps $M_n$ to $M_{\sigma(n)}$, and acts by the identity on $M_n$ if $\sigma(n)=n$. (We put $\sigma(0)=0$ for all $\sigma \in \cI$.) We let $\uRep(\cI)$ denote the category of graded $\cI$-modules.
\end{itemize}
This paper is concerned with determining the structure of these two categories.

There are three families of $\cI$-modules that will play a central role in our work:
\begin{itemize}
\item For $r \in \bN$, let $\ul{\bB}^r$ be the one-dimensional graded $\cI$-module that is concentrated in degree $r$, and where $\sigma$ acts by the identity if $\sigma(r)=r$ and by~0 otherwise. These are exactly the simple objects in the category $\uRep(\cI)$.
\item For $r \in \bN$, let $\ul{\bA}^r$ be the graded $\cI$-module that has for a basis elements $e_{i_1,\ldots,i_r}$ where $1 \le i_1 < \cdots < i_r$ are integers, and $\sigma \in \cI$ acts in the obvious way, i.e., $\sigma e_{i_1,\ldots,i_r}=e_{\sigma(i_1),\ldots,\sigma(i_r)}$. The element $e_{i_1,\ldots,i_r}$ is assigned degree $i_r$. We call $\ul{\bA}^r$ the $r$th {\bf principal module}. The principal modules are exactly the indecomposable projective objects in $\uRep(\cI)$.
\item A {\bf constraint word} is a word $\lambda=\lambda_1 \ldots \lambda_r$ in the alphabet $\{ \whitetri, \blacktri \}$. We define a graded $\cI$-module $\ul{\bE}^{\lambda}$, called the {\bf standard module} associated to $\lambda$, as follows. First, $\ul{\bE}^{\lambda}$ has for a basis elements of the form $e_{i_1,\ldots,i_r}$ where $1 \le i_1 < \cdots <i_r$, subject to the following constraint: if $\lambda_k=\blacktri$ then $i_k-i_{k-1}=1$ (with the convention $i_0=0$). We define $\sigma e_{i_1,\ldots,i_r}$ to be $e_{\sigma(i_1),\ldots,\sigma(i_r)}$ if this is a basis element of $\ul{\bE}^{\lambda}$, and~0 otherwise. The $\ul{\bE}^{\lambda}$ interpolate between the $\ul{\bB}^r$ and the $\ul{\bA}^r$: indeed, if $\lambda=\blacktri^r$ then $\ul{\bE}^{\lambda}=\ul{\bB}^r$, while if $\lambda=\whitetri^r$ then $\ul{\bE}^{\lambda}=\ul{\bA}^r$.
\end{itemize}
We let $\bB^r$, $\bA^r$, and $\bE^{\lambda}$ be the modules obtained from $\ul{\bB}^r$, $\ul{\bA}^r$, and $\ul{\bE}^{\lambda}$ by forgetting the grading. These are all smooth. It is still true that the $\bB^r$ are the simple objects of $\Rep(\cI)$, however, it is no longer true that the $\bA^r$ are projective; in fact, $\Rep(\cI)$ has no non-zero projective objects.

\subsection{Main results} \label{ss:results}

We now describe some of our main results in the graded case. Many (but not all) of these results have analogs in the smooth case. We note that all results hold over any coefficient field $\bk$.

\subsubsection{The Grothendieck group}
The classes of the standard modules $\ul{\bE}^{\lambda}$ form a basis for the Grothendieck group $\ul{\rK}(\cI)$ of the category $\uRep(\cI)^{\fgen}$ of finitely generated graded $\cI$-modules (Theorem~\ref{thm:groth}). Moreover, $\uRep(\cI)$ admits a (non-symmetric) monoidal operation $\odot$, called the {\bf concatenation product}, that endows $\ul{\rK}(\cI)$ with the structure of a non-commutative ring. As such, it is isomorphic to the non-commutative polynomial $\bZ\{a,b\}$ in two variables (Proposition~\ref{prop:grothring}). The isomorphism $\bZ\{a,b\} \to \ul{\rK}(\cI)$ takes $a$ to $[\ul{\bA}^1]$ and $b$ to $[\ul{\bB}^1]$, and a general word $\lambda$ to $[\ul{\bE}^{\lambda}]$. We also determine the effective cone in $\ul{\rK}(\cI)$ (\S \ref{ss:effective}).

\subsubsection{Generators for the derived category}
The standard modules $\ul{\bE}^{\lambda}$ generate $\rD^b_{\fgen}(\uRep(\cI))$ as a triangulated category (Theorem~\ref{thm:dergen}). This implies that the $[\ul{\bE}^{\lambda}]$ span $\ul{\rK}(\cI)$, but is far stronger (and more useful): for example, it typically allows one to prove finiteness statements about a derived functor by simply computing what the functor does to standard modules.

\subsubsection{Hilbert series}
Let $M$ be a finitely generated graded $\cI$-module. The Hilbert series of $M$ is defined to be
\begin{displaymath}
\ul{\rH}_M(t) = \sum_{n \ge 0} \dim(M_n) t^n.
\end{displaymath}
It was shown in \cite[\S 7.1]{catgb}, using the language of $\OI$-modules, that this is a rational function. We take this result a step further. We first observe that $\dim(M_n)$ can be viewed as the multiplicity of the simple object $\ul{\bB}^n$ in $M$. We extend this by defining an invariant $\umu_{\lambda}(M)$ that is, roughly speaking, the multiplicity of $\ul{\bE}^{\lambda}$ in $M$ (\S \ref{ss:highermult}). We then define the \emph{non-commutative Hilbert series} of $M$ to be
\begin{displaymath}
\ul{\rG}_M = \sum_{\lambda} \umu_{\lambda}(M) \lambda,
\end{displaymath}
where the sum is over all words $\lambda$ in $a$ and $b$. We regard this as a non-commutative power series in the variables $a$ and $b$. We prove that it is rational (Theorem~\ref{thm:noncommhilb}).

\subsubsection{Classification of injective modules}
We explicitly construct the injective envelope $\ul{\bI}^{\lambda}$ of $\ul{\bE}^{\lambda}$, and find that it is finitely generated. We show that every indecomposable injective object of $\uRep(\cI)$ is isomorphic to $\ul{\bI}^{\lambda}$ for some $\lambda$. When $\lambda=\whitetri^r$, we find that $\ul{\bI}^{\lambda}=\ul{\bE}^{\lambda}=\ul{\bA}^r$. Thus the principal objects $\ul{\bA}^r$ are injective, and so all projective objects of $\uRep(\cI)$ are injective. One interesting feature of $\ul{\bI}^{\lambda}$ is that it admits a standard filtration, that is, a finite length filtration where the graded pieces are standard modules. (See \S \ref{s:injectives} for these claims.) 

\subsubsection{Injective resolutions}
Every finitely generated graded $\cI$-module $M$ has finite injective dimension (Theorem~\ref{thm:injres}). In fact, $M$ admits an injective resolution of the form $M \to I^{\bullet}$ where each $I^n$ is a finite sum of indecomposable injectives, and $I^n=0$ for $n \gg 0$. As a consequence, we find that $\ul{\rK}(\cI)$ admits a canonical pairing, namely the Euler characteristic of $\Ext^{\bullet}(-,-)$. We give a recurrence that allows one to effectively compute this pairing (\S \ref{ss:Fseries}).

\subsubsection{Projective resolutions}
Let $M$ be a finitely generated graded $\cI$-module. Just as in commutative algebra, one can speak of the minimal projective resolution of $M$ and its linear strands. We show that only finitely many linear strands are non-zero; in other words, the (Castelnuovo--Mumford) regularity of $M$ is finite. Furthermore, we show that each linear strand (after some simple transformation) canonically admits the structure of a finitely generated graded $\cI$-module. Thus the minimal resolution of $M$ exhibits strong finiteness properties. We note, however, that if $M$ is not projective then it has infinite projective dimension.

We define the graded Betti table $\beta(M)$ of $M$ in \S \ref{ss:betti} just as in commutative algebra: the $i$th row records the number of generators appearing in the $i$th linear strand of various degrees. The finiteness result for regularity ensures that $\beta(M)$ has only finitely many non-zero rows. The finite generation result for each linear strand implies that for each $i$, the function $j \mapsto \beta(M)_{i,j}$ is eventually a polynomial in $j$. (See Theorem~\ref{thm:betti} for these statements.) An interesting open problem is to find an analog of Boij--S\"oderberg theory in this setting to describe the cone of Betti tables; see \S \ref{ques:boij}.

\subsubsection{Koszul duality}
We construct a canonical equivalence $\cD \colon \rD^b_{\fgen}(\uRep(\cI))^{\op} \to \rD^b_{\fgen}(\uRep(\cI))$ that squares to the identity (Corollary~\ref{cor:Dfg}). This can be seen as a version of Koszul duality. The existence of this functor on the full derived category $\rD(\uRep(\cI))$ is fairly formal. The fact that it preserves the bounded finitely generated subcategory is a much deeper result; in fact, it is essentially a reformulation of the properties of projective resolutions mentioned above. The duality functor induces an involution of the ring $\ul{\rK}(\cD) \cong \bZ\{a,b\}$ that takes $a$ to $-b$ and $b$ to $-a$.

\subsubsection{Level categories}
We define a notion of {\bf level} for graded $\cI$-modules. It turns out that the level of a module is equal to its Krull--Gabriel dimension, though that is a non-trivial result (Proposition~\ref{prop:krull}). We let $\uRep(\cI)_{\le r}$ be the full subcategory of $\uRep(\cI)$ spanned by the objects of level $\le r$. These categories filter $\uRep(\cI)$; we refer to this as the {\bf level filtration}. We completely determine the structure of the graded pieces $\uRep(\cI)_r=\uRep(\cI)_{\le r}/\uRep(\cI)_{\le r-1}$. Precisely, we show that $\uRep(\cI)_r$ is equivalent to the category of multi-graded $\cI^{r+1}$-modules that are locally of finite length; equivalently, this can be described as the $(r+1)$st tensor power of the category $\uRep(\cI)^{\lf}$. (See Corollary~\ref{cor:levstruc}.) We also show that $\rD^b_{\fgen}(\uRep(\cI))$ admits a semi-orthogonal decomposition where the pieces are $\{\rD^b_{\fgen}(\uRep(\cI)_r)\}_{r \ge 0}$ (Theorem~\ref{thm:semiorth}).

\subsection{Open problems} \label{ss:problems}

While our work gives a fairly good picture of the structure of $\cI$-modules, it also suggests a number of further avenues of inquiry. We highlight a few potentially interesting problems here.

\subsubsection{The Boij--S\"oderberg problem} \label{ques:boij}
As mentioned, we show that the Betti table $\beta(M)$ of a finitely generated graded $\cI$-module is reasonably well-behaved: it has finitely many non-zero rows and each row is eventually a polynomial function. It therefore seems reasonable to consider the Boij--S\"oderberg problem in this context: can one describe the cone of all Betti tables?

We have one interesting observation to offer regarding this problem. In classical Boij--S\"oderberg theory, the pure resolutions constructed by Eisenbud--Fl\o ystad--Weyman \cite{efw} play an important role (in characteristic~0). We note that these complexes have manifestations in $\uRep(\cI)$: indeed, the EFW complex is built out of Schur functors, and thus can be transformed, via Schur--Weyl duality, to a complex of $\FI$-modules; one can then restrict to $\OI$ and apply the equivalence between $\OI$-modules and $\uRep(\cI)$. This yields a large class of pure resolutions in $\uRep(\cI)$.

\subsubsection{Regularity bounds} \label{ques:reg}
For a graded $\cI$-module $M$, let $t_i(M)$ be the maximal non-zero degree occurring in the $i$th Tor group $\rL_i \cT(M)$, as defined in \S \ref{ss:koszul}. The regularity $\rho(M)$ is defined as the maximum value of $t_i(M)-i$ over $i \ge 0$. Alternatively, $\rho(M)$ is the index of the final non-zero row in the Betti table $\beta(M)$. As mentioned, we show that $\rho(M)$ is finite when $M$ is finitely generated.

Analogous results are known in the setting of $\FI$-modules and $\FI_d$-modules. In fact, even stronger results are known there. Church and Ellenberg \cite{ce} showed that the regularity of an $\FI$-module can be bounded in terms of just $t_0(M)$ and $t_1(M)$, while Sam and the second author \cite{tcareg} proved that the regularity of an $\FI_d$-module can be bounded in terms of $t_0(M), \ldots, t_n(M)$ for $n \approx \tfrac{1}{4} d^2$.

Given these results, it is natural to wonder if such an improvement can exist for graded $\cI$-modules. We suspect that there may be such a result in the following form: given $r \ge 0$, there exists $n(r)$ such that $\rho(M)$ can be bounded in terms of $t_0(M), \ldots, t_{n(r)}(M)$ for all $M \in \uRep(\cI)_{\le r}$. Recall that $\uRep(\cI)_{\le r}$ consists of those modules of level (or, equivalently, Krull--Gabriel dimension) at most $r$. We have not seriously attempted to prove this statement, however.

\subsubsection{Other examples of $\cI$-modules} \label{ques:monomial}
We introduce many invariants of $\cI$-modules, such as local cohomology, (non-commutative) Hilbert series, Betti tables, and so on. However, we have only computed these invariants on standard modules, and the results are typically far too simple to be representative of the general situation. It would therefore be interesting to compute these invariants on more interesting $\cI$-modules. Here are a few potentially interesting sources of examples:
\begin{itemize}
\item As mentioned, we establish a bijection between monomial submodules of $\bA^r$ and monomial ideals in $\bk[x_1, \ldots, x_r]$. Thus each monomial ideal gives an example of an $\cI$-module.
\item Suppose $\bk$ has characteristic~0. Let $L_{\lambda}$ be the $\FI$-module that in degree $n \ge \vert \lambda \vert$ is the irreducible corresponding to $\lambda[n]=(n-\vert \lambda \vert, \lambda_1, \lambda_2, \ldots)$, and that vanishes in lower degrees. We can restrict $L_{\lambda}$ to $\OI$ and then apply the equivalence with $\uRep(\cI)$ to obtain a graded $\cI$-module $L_{\lambda}'$.
\item Let $\Rep(\fS_{\infty})$ be the category of algebraic representations of the infinite symmetric group, as studied in \cite{infrank}. There is a natural restriction functor $\Rep(\fS_{\infty}) \to \Rep(\cI)$, and so every algebraic representation of $\fS_{\infty}$ yields a smooth $\cI$-module.
\end{itemize}
A theorem of Hochster \cite[Theorem~II.4.1]{stanley} relates the local cohomology of a monomial ideal (in the classical sense) to the topology of the associated Stanley--Reisner complex. A theorem of Sam and the second author \cite[Proposition~7.4.3]{symc1} relates the local cohomology of $L_{\lambda}$ (in the sense of $\FI$-modules) to the Borel--Weil--Bott rule. It would be interesting to see how the local cohomology of the corresponding $\cI$-modules (in the sense of this paper) compares in these cases.

\subsubsection{Homology of unipotent groups} \label{ques:unipotent}
Let $U_n(R) \subset \GL_n(R)$ be the group of strictly upper-triangular matrices, where $R$ is a ring whose additive group is finitely generated. Let $M_n=\rH_n(U_n(R), \bk)$ be the group homology of $U_n(R)$ with coefficients in the field $\bk$. By \cite[Theorem~1.3]{fiq}, the collection of groups $\{M_n\}_{n \ge 0}$ admits the structure of a finitely generated $\OI$-module; we can therefore regard it as a finitely generated graded $\cI$-module. Can one describe this $\cI$-module from the perspective of this paper? For example, can one give its class in $\ul{\rK}(\cI)$? When $R=\bZ$ and $\bk=\bQ$, the group $M_n$ is understood \cite[Theorem~1.1]{dwyer}, and so this could be a reasonably feasible problem. (We note that describing the class of $M$ in $\ul{\rK}(\cI)$ is strictly stronger than determining the dimensions of the $M_n$'s, since the Hilbert series map $\ul{\rK}(\cI) \to \bZ\lbb t \rbb$ has a large kernel; in other words, the class in the Grothendieck group records some non-trivial information about the transition maps.)

\subsubsection{Some $\Ext$ calculations} \label{ques:ext}
As mentioned, $\ul{\rK}(\cI)$ admits a canonical pairing, given by the Euler characteristic of $\Ext^{\bullet}$, and we have given recurrences that allows one to effectively compute the pairing (\S \ref{ss:Fseries}). Can one give a closed-form formula for the pairing? Or, even better, can one compute the individual $\Ext$ groups between standard modules?

\subsubsection{$\OI_d$-modules for $d>1$}
The results of this paper can be viewed as a structure theory for $\OI$-modules. A natural problem is to extend these results to $\OI_d$-modules, for $d>1$. In other words, if this paper is the ``$\OI$-analog'' of the paper \cite{symc1} on $\FI$-modules, we are asking for the ``$\OI_d$-analog'' of the paper \cite{symu1} on $\FI_d$-modules. As \cite{symu1} suggests, there should be some geometry present in the general theory of $\OI_d$-modules that is not seen in the $d=1$ case, which means the results and methods of this paper will probably not trivially generalize to the $d>1$ case.

\subsection{Technical highlights}

Since this is a very long paper without a singular goal, it is difficult to summarize the main ideas of proofs in a useful way. Instead, we give a sampling here of some important constructions and intermediate results, which will hopefully provide the reader some direction.

There are three really important technical theorems we wish to highlight:
\begin{itemize}
\item Theorem~\ref{thm:canongr} (the canonical grading): every finite length (or even locally finite length) smooth $\cI$-module admits a canonical grading. In other words, the forgetful functor $\Phi \colon \uRep(\cI)^{\lf} \to \Rep(\cI)^{\lf}$ is an equivalence of categories. This result surprised us when we discovered it, and is still somewhat surprising to us now. This theorem is needed to even define the completion functor $\ul{\Xi}$ (discussed below).
\item Theorem~\ref{thm:multone} (the multiplicity one theorem): the trivial representation has multiplicity one in the principal module $\bA^r$. In other words, if $K$ denotes the kernel of the augmentation map $\bA^r \to \bk$ then no subquotient of $K$ is isomorphic to the trivial representation. This theorem is the key ingredient used to establish the main properties of the truncation functors (discussed below), which have a number of uses.
\item Theorem~\ref{thm:level} (the theorem on level). We define the {\bf rank} of a constraint word $\lambda$, or the associated standard module $\bE^{\lambda}$, to be the number of $a$'s in $\lambda$. For example, $\bA^r$ has rank $r$, while $\bB^r$ has rank~0. We say that a finitely generated smooth $\cI$-module has {\bf level} $\le r$ if it admits a finite length filtration such that each graded piece is isomorphic to a subquotient of some standard module of rank $\le r$. The theorem on level (which holds in both the smooth and graded cases) states that any proper quotient of a rank $r$ standard module has level $<r$. It is an extremely important result on the structure of $\cI$-modules. For example, it immediately implies that the standard modules generate the derived category, and thus span the Grothendieck group.
\end{itemize}

Most of the work in this paper is devoted to studying various constructions (or functors) involving $\cI$-modules. There are four functors in particular that are extremely important:
\begin{itemize}

\item The forgetful functor $\Phi \colon \uRep(\cI) \to \Rep(\cI)$, which simply forgets the grading. It may seem like a simple operation, and it is, but it has some non-trivial properties: for example, it is continuous (Proposition~\ref{prop:Phi-cont}, not too hard) and takes injective objects to injective objects (Corollary~\ref{cor:phiinj}, quite hard). Essentially by definition, we have $\Phi(\ul{\bE}^{\lambda})=\bE^{\lambda}$.

\item The direct limit functor $\Psi \colon \uRep(\cI) \to \Rep(\cI)$. Given a graded $\cI$-module $M$, we can form the following directed system of vector spaces:
\begin{displaymath}
M_1 \stackrel{\alpha_1}{\longrightarrow} M_2 \stackrel{\alpha_2}{\longrightarrow} M_3 \stackrel{\alpha_3}{\longrightarrow} \cdots
\end{displaymath}
(Here $\alpha_i$ is one of the basic generators of $\cI$, see \S \ref{ss:alpha}.) We define $\Psi(M)$ to be the direct limit. It naturally carries the structure of a smooth $\cI$-module. We show that $\Psi$ realizes $\Rep(\cI)$ as the Serre quotient of $\uRep(\cI)$ by a certain category of torsion modules (Proposition~\ref{prop:phigamma}). The effect of $\Psi$ on standard modules is easy to describe (Proposition~\ref{prop:psistd}): if $\lambda$ ends in the letter $a$, that is, it has the form $\mu a$, then $\Psi(\ul{\bE}^{\lambda})=\bE^{\mu}$; otherwise $\Psi(\ul{\bE}^{\lambda})=0$.

\item The invariants functor $\Gamma \colon \Rep(\cI) \to \uRep(\cI)$. Given a smooth $\cI$-module $M$, we let $\Gamma(M)_n$ denote the invariants of $M$ under the monoid $\cI^{\ge n}$ introduced in \S \ref{ss:fund}. It is not hard to see that $\Gamma(M)=\bigoplus_{n \ge 0} \Gamma(M)_n$ admits the structure of a graded $\cI$-module. We show that it is naturally the right adjoint of $\Psi$ (Proposition~\ref{prop:phigamma}). Of the main functors studied in this paper, $\Gamma$ is the only one that is not exact. Using the completion functor (discussed below), we show that the derived functor of $\Gamma$ has amenable finiteness properties (Propositions~\ref{prop:gammafinite} and~\ref{prop:rgammacolim}). The effect of $\Gamma$ on standard modules is quite simple: $\Gamma(\bE^{\lambda})=\ul{\bE}^{\lambda a}$.

\item The completion functor $\ul{\Xi} \colon \Rep(\cI) \to \uRep(\cI)$, which is the most difficult of the four. Given a smooth $\cI$-module $M$, the truncation $\tau^{<n}(M)$ (discussed below) is locally of finite length; therefore, it admits a canonical grading. We define $\ul{\Xi}(M)$ to be the inverse limit of the truncations in the graded category. We show that $\ul{\Xi}$ is exact and is naturally left adjoint to $\Phi$. (This is why $\Phi$ takes injectives to injectives.)  The unit map $M \to \Phi(\ul{\Xi}(M))$ is always injective, and the cokernel is in a sense smaller. This yields an exact sequence that leads to some important inductive arguments; for example, this is how we study the derived functor of $\Gamma$. With great effort, we compute the effect of $\ul{\Xi}$ on standard modules (Proposition~\ref{prop:xistd}): writing $\lambda=\mu a^n$, where $\mu$ does not end in an $a$, we have $\ul{\Xi}(\bE^{\lambda})=\bigoplus_{0 \le i \le n} \ul{\bE}^{\mu a^i}$.

\end{itemize}
There are a few other functors worth mentioning here:
\begin{itemize}
\item The concatenation product $\odot$ (\S \ref{ss:concat}). Given a graded $\cI$-module $M$ and a smooth (or graded) $\cI$-module $N$, we define an interesting $\cI$ action on the tensor product $M \otimes N$. We denote the resulting smooth (or graded) $\cI$-module by $M \odot N$. The effect on standard modules is what one might guess (Proposition~\ref{prop:stdcat}): $\ul{\bE}^{\lambda} \odot \ul{\bE}^{\mu} = \ul{\bE}^{\lambda \odot \mu}$, where $\lambda \odot \mu$ denotes the concatenation of the two words. In particular, we see that we can build a general standard module $\ul{\bE}^{\lambda}$ by successively concatenating $\ul{\bA}^1$ and $\ul{\bB}^1$. This is a useful observation, as it sometimes allows us to prove results for standard modules just by proving them for simple and principal modules (this idea is used, in effect, in the proof of Theorem~\ref{thm:Dfin}).

\item The transpose functor $\dag$ (\S \ref{ss:transpose}). Given a graded $\cI$-module $M$, we define $M^{\dag}$ to be the graded $\cI$-module with the same underlying graded vector space, but where the action of $\cI$ is ``transposed:'' on the $n$th graded piece, the actions of $\alpha_i$ and $\alpha_{n+1-i}$ are interchanged. The effect on standard modules is straightforward (Proposition~\ref{prop:transstd}): $(\ul{\bE}^{\lambda})^{\dag}=\ul{\bE}^{\lambda^{\dag}}$, where $\lambda^{\dag}$ denotes the reversed word. The transpose functor can be handy in constructing more complicated operations: for example, when acting on standard modules, $\Gamma$, $\Psi$, and $\ul{\Xi}$ all affect the rightmost letters; if one wanted the leftmost letters affected instead, simply throw in a transpose functor.

\item The truncation functors $\tau_{\ge r}$ and $\tau^{<r}$ (\S \ref{ss:trunc}). Given a smooth $\cI$-module, we show that there exists a unique exact sequence
\begin{displaymath}
0 \to \tau_{\ge r}(M) \to M \to \tau^{<r}(M) \to 0
\end{displaymath}
such that $\tau_{\ge r}(M)$ has no subquotient isomorphic to $\bB^s$ with $s<r$, and $\tau^{<r}(M)$ has no subquotient isomorphic to $\bB^s$ with $s \ge r$. These functors are exact: this is a non-trivial result (Proposition~\ref{prop:tau-exact}) that relies on the multiplicity one theorem. We have already seen that these functors figure into the very definition of the important completion functor. They also have other uses: for example, we use them to prove the the trivial module is injective (Corollary~\ref{cor:trivinj}).

\item The shift functor $\Sigma$ (\S \ref{ss:shift}) and the coinduction functor $\sC$ (\S \ref{ss:coinduction}). We have an isomorphism of monoids $i \colon \cI \to \cI_{\ge 2}$ given by $\alpha_j \mapsto \alpha_{j+1}$; of course, $\cI_{\ge 2}$ is also a submonoid of $\cI$. The shift functor $\Sigma$ is defined by restricting an $\cI$-module to $\cI_{\ge 2}$, and then transferring this back to an $\cI$-module via $i$. Similary, the coinduction functor $\sC$ is defined by first transferring the $\cI$-module to an $\cI_{\ge 2}$-module via $i$, and then applying coinduction along the inclusion. These functors are both exact, and $(\Sigma, \sC)$ forms an adjoint pair (Proposition~\ref{prop:coind}). From this, we see that coinduction takes injective objects to injective objects. For this reason, it plays an important role in our study of injectives.

\end{itemize}

As one can see, we develop quite an extensive toolkit. This greatly simplifies---and in some cases simply makes possible---the task of proving our main theorems. To see the machinery in action, take a look at the proofs of Theorem~\ref{thm:level} (the theorem on level) or Theorem~\ref{thm:groth} (on the structure of the Grothendieck group), where many of these functors, and their basic properties, are employed in unison.

\subsection{Outline}

We now outline the contents of the paper. In addition to summarizing the contents of each section, we hope that this provides a rough guide to the logical structure of our results (in contrast to the previous section, which paid no heed to logical order).

\S \ref{s:inc}. \textit{The increasing monoid.}
We establish some basic properties of the increasing monoid $\cI$ and its generators $\{\alpha_i\}_{i \ge 1}$.

\S \ref{s:reps}. \textit{Representation categories.}
We introduce the various categories of representations and establish some technical categorical results. We also explain the connections to $\OI$-modules, semi-simplicial vector spaces, and shuffle algebras.

\S \ref{s:monomial}. \textit{Monomial modules.}
We introduce monomial modules generally, and the three monomial modules $\bB^r$, $\bA^r$, and $\bE^{\lambda}$ specifically. We establish some basic properties of these modules (e.g., every smooth representation is a quotient of a sum of $\bA^r$'s).

\S \ref{s:grobner}. \textit{Gr\"obner theory.}
We develop two versions of Gr\"obner theory, and use them to establish some facts about representations. The first version studies submodules of $\bA^r$ by relating them to monomial submodules. The main application of this version is the theorem that finitely generated $\cI$-modules (either smooth or graded) are noetherian. This result was essentially already known, but the smooth case does not seem to appear in the literature, so we have included a proof. The second version of Gr\"obner theory studies inhomogeneous submodules of an arbitrary graded module by relating them to homogeneous submodules. We use this to relate certain properties of graded modules and their underlying smooth module (such as indecomposability).

\S \ref{s:cat}. \textit{Concatenation, shift, and transpose.}
We introduce these three fundamental operations, study how they interact with each other, and compute their effect on standard modules.

\S \ref{s:finlen}. \textit{Finite length modules.}
We show that the $\bB^r$'s and $\ul{\bB}^r$'s account for all simple modules. We then compute the $\Ext^1$ groups between them, by directly analyzing the $2 \times 2$ matrices that define an extension. A consequence of these computations is that the trivial representation is injective and projective in the category $\Rep(\cI)^{\lf}$ of smooth modules that are locally of finite length. A second consequence is the existence of the canonical grading (every smooth $\cI$-module that is locally of finite length admits a canonical grading).

\S \ref{s:mult}. \textit{Multiplicities.}
For $n \in \bN$, we let $\mu_n(M)$ be the multiplicity of the simple object $\bB^n$ in the smooth $\cI$-module $M$. We prove the important multiplicity one theorem, which states $\mu_0(\bA^r)=1$. The multiplicity one theorem has a number of consequences, e.g., we deduce from it that $\mu_n(M)$ is finite for all $n$ whenever $M$ is finitely generated. We also introduce the truncation functors $\tau_{\ge r}$ and $\tau^{<r}$. Using the multiplicity one theorem, we show that these functors are exact. From this, we deduce that injective objects of $\Rep(\cI)^{\lf}$ remain injective in $\Rep(\cI)$. In particular, we find that the trivial representation is injective in $\Rep(\cI)$.

\S \ref{s:psigamma}. \textit{$\Rep(\cI)$ as a Serre quotient of $\uRep(\cI)$.}
We introduce the $\Psi$ and $\Gamma$ functors, and establish their main properties: namely, that they form an adjoint pair and realize $\Rep(\cI)$ as a Serre quotient of $\uRep(\cI)$. We also compute their effect on standard modules. Finally, we show that any smooth module that admits a grading (such as $\bE^{\lambda}$) is $\Gamma$-acyclic. This result will be important later when we study the derived functors of $\Gamma$.

\S \ref{s:xi}. \textit{Completions.}
We define the completion functor $\ul{\Xi}$, and show that it is exact and naturally the left adjoint of $\Phi$. We note that the definition of $\ul{\Xi}$ hinges on the existence of the canonical grading and the truncation functors. We explicitly compute $\ul{\Xi}$ on standard modules. Finally, we give a number of applications of the completion functor, such as: (a) injectivity of principle modules; (b) finiteness properties of $\rR \Gamma$; and (c) finiteness properties of the $\Ext$ functors on $\Rep(\cI)$ and $\uRep(\cI)$.

\S \ref{s:level}. \textit{The theorem on level.}
We introduce the notion of level, and prove a number of simple results concerning how our various functors interact with it. We then study a notion of saturation, and the effect of concatenating with the simple object $\bB^1$. After this, we prove the extremely important theorem on level, which states that any proper quotient of a rank $r$ standard object has level $<r$. We close the section by giving some consequences of the theorem, e.g., that the standard modules generate the derived category (this is the first of the main results listed in \S \ref{ss:results} to be proved).

\S \ref{s:groth}. \textit{Grothendieck groups, Hilbert series, and Krull dimension.}
We prove our main theorem on Grothendieck groups, namely, that the classes of standard modules form a basis. The theorem on level implies that they span; to prove linear independence, we use the plethora of functors available to us at this point to construct sufficiently many linear functionals to distinguish the classes in question. Using this description of the Grothendieck group, we give simple proofs of some facts about Hilbert series. We end with a brief discussion of Krull dimension.

\S \ref{s:coinduction}. \textit{Induction and coinduction.}
We define the induction and coinduction functors, which are the left and right adjoints of the shift functor. The induction functor is not so important, and we give it little attention. The coinduction functor, on the other hand, figures prominently in our study of injective objects, so we analyze it in some detail.

\S \ref{s:injectives}. \textit{Injective modules.}
We prove our main results on injectives: namely, we classify the indecomposable injectives, and prove that finitely generated $\cI$-modules have nice injective resolutions. To do this, we first construct the finite length injective objects, which is easy to do using coinduction. Then, using a combination of coinduction and some of our other functors, we show how to explicitly produce injectives that do not have finite length, and we produce enough of them to complete the classification. From the explicit form of these injectives, it is easy to write down an injective resolution of standard modules, from which it follows that they have finite injective dimension. Since they generate the derived category, it follows that all finitely generated modules have finite injective dimension.

\S \ref{s:loccoh}. \textit{Local cohomology and saturation.}
We review the general theory of saturation and local cohomology introduced in \cite{symu1}, and then specialize it to $\cI$-modules. Using the injective resolutions of standard modules found in the previous section, we compute the derived satuation and local cohomology of standard modules; it turns out that all the higher derived functors vanish in this case. We then introduce a new concept: we say that a module is $r$-semistandard if it admits a filtration where the graded pieces are rank $r$ standard objects. Using our computations of local cohomology, we show that the $r$-semistandard objects form an abelian subcategory, which is rather surprising. Finally, we show that these abelian subcategories yield a semi-orthogonal decomposition of the derived category.

\S \ref{s:levcat}. \textit{Structure of level categories.}
We explicitly identify the structure of the level categories (i.e., the graded pieces of the level filtration). The key input here is the classification of injective modules.

\S \ref{s:koszul}. \textit{Koszul duality.}
In this section, we prove our main results about projective resolutions. We begin by defining the Tor functor on graded $\cI$-modules and introducing the notion of a minimal projective resolution. We then construct the Koszul complex, which is a very natural complex that computes Tor. We then show that the definition of the Koszul complex can be modified to produce a derived auto-equivalence $\cD$ of $\uRep(\cI)$; this is rather formal. We then prove the much more significant theorem that $\cD$ preserves the bounded and finitely generated derived category. The proof of this theorem is quite short given the tools available to us at this point: since the standard modules generate the derived category, it suffices to show that $\cD(\ul{\bE}^{\lambda})$ belongs to the bounded and finitely generated derived category; since $\cD$ is compatible with concatenation produces, the task is further reduced to the case where $\lambda$ is a word of length~1, where it can be handled directly. We then introduce the Betti table of a graded $\cI$-module and deduce finiteness results about it from those of $\cD$.

\S \ref{s:groth2}. \textit{Grothendieck groups revisited.}
In this final section, we prove some deeper results about the Grothendieck group. We begin by defining the $\Ext$-pairings on $\rK(\cI)$ and $\ul{\rK}(\cI)$, which exist thanks to our results on injective resolutions. Next, we define the invariant $\mu_{\lambda}$, which is, in a sense, the multiplicity of $\bE^{\lambda}$ in an $\cI$-module. We use these invariants to define the non-commutative Hilbert series $\rG_M$, which we prove to be rational. We also introduce some variants of $\rG_M$, which can be proved to be rational by the same methods, and yield a finite procedure for computing the pairings on the Grothendieck group. We close by determining the effective cone in the Grothendieck group.

Appendix~\ref{s:catbg}. \textit{Categorical background.}
Here we collect a variety of abstract results on abelian categories.

\subsection{Notation}

We collect here some of the most important notation introduced in the body of the paper. This list is not meant to be exhaustive.
\begin{itemize}
\item $\cA^{\rf}$: the category of finite length objects in the abelian category $\cA$.
\item $\cA^{\lf}$: the category of locally finite objects in the abelian category $\cA$.
\item $\cA^{\fgen}$: the category of finitely generated objects in the abelian category $\cA$.
\item $\rD(\cA)$: the derived category of the abelian category $\cA$.
\item $\rD^b_{\fgen}(\cA)$: the bounded and finitely generated derived category of $\cA$.
\item $\bN$: the set of non-negative integers.
\item $\bN_+$: the set of positive integers.
\item $\fS_n$: the symmetric group on $n$ letters.
\item $\cI$: the increasing monoid (\S \ref{ss:inc}).
\item $\alpha_i$, for $i \ge 1$: one of the basic generators of $\cI$ (\S \ref{ss:alpha}).
\item $\cI_{\ge n}$: the submonoid of $\cI$ generated by the $\alpha_i$ with $i \ge n$ (\S \ref{ss:smooth}).
\item $\bk$: the coefficient field (fixed throughout).
\item $\REP(\cI)$: the category of all $\cI$-modules (\S \ref{ss:bigrep}).
\item $\Rep(\cI)$: the category of smooth $\cI$-modules (\S \ref{ss:rep}).
\item $\uRep(\cI)$: the category of graded $\cI$-modules (\S \ref{ss:urep}).
\item $\fa_n$: the ideal of $\bk[\cI]$ generated by the $\alpha_i-1$ with $i \ge n$ (\S \ref{ss:inv}).
\item $\Phi$: the forgetful functor $\uRep(\cI) \to \Rep(\cI)$ (\S \ref{ss:urep}).
\item $\bA^r$: the $r$th principal module (\S \ref{ss:principal}).
\item $\bB^n$: the $n$th simple module (\S \ref{ss:simple}).
\item $\bE^{\lambda}$: the standard module associated to $\lambda$ (\S \ref{ss:gap}).
\item $\ini(-)$: initial term or module, in the sense of Gr\"obner theory (\S \ref{s:grobner}).
\item $\odot$: the concatenation product (\S \ref{ss:concat}).
\item $\Sigma$: the shift functor (\S \ref{ss:shift}).
\item $M^{\dag}$: the transpose of the graded $\cI$-module $M$ (\S \ref{ss:transpose}).
\item $\mu_n(M)$: the multiplicity of $\bB^n$ in $M$ (\S \ref{ss:mult}).
\item $\tau^{<n}$, $\tau_{\ge n}$: the truncation functors (\S \ref{ss:trunc}).
\item $\Psi$: the direct limit functor $\uRep(\cI) \to \Rep(\cI)$ (\S \ref{ss:psigamma}).
\item $\Gamma$: the right adjoint of $\Psi$, given by $M \mapsto \bigoplus_{n \ge 0} M^{\cI_{\ge n}}$ (\S \ref{ss:psigamma}).
\item $\ul{\Xi}$: the completion functor, left adjoint of $\Phi$ (\S \ref{ss:xi}).
\item $\Rep(\cI)_{\le r}$: the level $\le r$ category (\S \ref{ss:level}).
\item $\rK(\cI)$: the Grothendieck group of $\Rep(\cI)^{\fgen}$ (\S \ref{ss:groth}).
\item $\rH_M$: the Hilbert series of $M$ (\S \ref{ss:hilb}).
\item $\sC$: the co-induction functor (\S \ref{ss:coinduction}).
\item $\sI$: the induction functor (\S \ref{ss:induction}).
\item $\bJ^n$: the injective envelope of the simple $\bB^n$ (\S \ref{ss:fininj}).
\item $\bI^{\lambda}$: the injective module associated to $\lambda$ (\S \ref{ss:inj}).
\item $\cH_{\le r}$: the local cohomology functor with respect to $\Rep(\cI)_{\le r}$ (\S \ref{ss:loccoh}).
\item $\cS_{>r}$: the saturation functor with respect to $\Rep(\cI)_{\le r}$ (\S \ref{ss:loccoh}).
\item $\rL_i \cT$: the $i$th Tor functor (\S \ref{ss:tor}).
\item $\cK(M)$: the Koszul complex on $M$ (\S \ref{ss:koszul}).
\item $\cD$: the duality functor (\S \ref{ss:duality}).
\item $\langle,\rangle$: the pairing on $\ul{\rK}(\cI)$ (\S \ref{ss:pairing}).
\item $\mu_{\lambda}$: the higher multiplicity invariant (\S \ref{ss:highermult}).
\item $\rG_M$: the non-commutative Hilbert seires of $M$ (\S \ref{ss:noncommhilb}).
\end{itemize}
We typically use an underline to denote something in the graded setting, and no underline to denote something in the smooth setting. For example, $\ul{\bE}^{\lambda}$ is the standard module in the graded category and $\bE^{\lambda}$ is the standard module in the smooth category; similarly, $\ul{\Sigma}$ is the shift functor on graded modules and $\Sigma$ is the shift functor on smooth modules.

\subsection*{Acknowledgements}

We thank Steven Sam for helpful discussions.

\section{The increasing monoid} \label{s:inc}

\subsection{The increasing monoid} \label{ss:inc}

The {\bf big increasing monoid}, denoted $\cI^{\rm big}$, is the monoid of all order-preserving injective maps $\bN_+ \to \bN_+$, where $\bN_+$ denotes the set of positive integers. We say that $\sigma \in \cI^{\rm big}$ is {\bf tame} if there exists an integer $\ell \ge 0$, called the {\bf length} of $\sigma$, such that $\sigma(n)=n+\ell$ for all $n \gg 0$. The {\bf small increasing monoid}, denoted $\cI$, is the submonoid of $\cI^{\rm big}$ consisting of all tame elements. One easily verifies that it is indeed a submonoid, and that the function $\ell \colon \cI \to \bN$ assigning to each element its length is a monoid homomorphism. For the questions studied in this paper, it turns out not to matter which version of the increasing monoid one uses, as we explain in \S \ref{ss:smooth}. We will therefore always work with the small version $\cI$.

\begin{remark}
The two versions of the increasing monoid are analogous to the two versions of the infinite symmetric group: the big version $\fS_{\infty}^{\rm big}$ is the group of all bijections of $\bN_+$, while the small version $\fS_{\infty}$ is the union of the $\fS_n$. The length homomorphism on $\cI$ is analogous to the sign homomorphism on $\fS_{\infty}$.
\end{remark}

\begin{remark}
Given $\sigma \in \cI^{\rm big}$, we extend it to a function $\bN \to \bN$ by declaring $\sigma(0)=0$.
\end{remark}

\subsection{The generators} \label{ss:alpha}

For $k \ge 1$, let $\alpha_k$ denote the element of $\cI$ defined by
\begin{displaymath}
\alpha_k(n) = \begin{cases}
n+1 & \text{if $n \ge k$} \\
n & \text{if $n<k$.} \end{cases}
\end{displaymath}
Clearly, $\ell(\alpha_k)=1$. One can show that the $\alpha_k$ are the only elements of length~1; in fact, this follows from Proposition~\ref{prop:alphagen} below.

The following proposition gives the basic relations satisfied by the $\alpha$'s. We will use it constantly throughout the paper:

\begin{proposition}[Fundamental relation] \label{prop:alpharel}
For $n>m\ge 1$ we have $\alpha_n \alpha_m=\alpha_m \alpha_{n-1}$.
\end{proposition}

\begin{proof}
This is a simple verification that is left to the reader.
\end{proof}

\begin{proposition} \label{prop:alphagen}
Let $\sigma \in \cI$.
\begin{enumerate}
\item There exists a unique strictly increasing sequence of integers $1 \le n_1<n_2<\cdots<n_r$ such that $\sigma=\alpha_{n_r} \cdots \alpha_{n_1}$.
\item There exists a unique weakly increasing sequence of integers $1 \le m_1 \le m_2 \le \cdots \le m_s$ such that $\sigma=\alpha_{m_1} \cdots \alpha_{m_s}$.
\end{enumerate}
Moreover, $r=s$ is the length of $\sigma$.
\end{proposition}

\begin{proof}
(a) Suppose that we have a decomposition $\sigma=\alpha_{n_r} \cdots \alpha_{n_1}$ with $n_r>n_{r-1}>\cdots>n_1 \ge 1$. We then have:
\begin{displaymath}
\sigma(k) = \begin{cases}
k & \text{if $1 \le k < n_1$} \\
k+1 & \text{if $n_1 \le k < n_2-1$} \\
k+2 & \text{if $n_2-1 \le k < n_3-2$} \\
\vdots & \vdots \\
k+r & \text{if $n_r-r+1 \le k$}
\end{cases}
\end{displaymath}
We can thus recover the $n$'s from $\sigma$, and so we have uniqueness.

Conversely, we can start with $\sigma$ and define the $n$'s so that the above holds: for example, $n_1$ is the minimal value of $k$ for which $\sigma(k)>k$. With these $n$'s, we then have $\sigma=\alpha_{n_r} \cdots \alpha_{n_1}$, and so we have existence.

(b) The proof is similar to (a), and the details are left to the reader.
\end{proof}

\begin{corollary} \label{cor:alphagen}
The $\alpha$'s generate $\cI$, and the relations in Proposition~\ref{prop:alpharel} generate all relations among the $\alpha$'s.
\end{corollary}

\begin{proof}
The proposition obviously implies that the $\alpha$'s generate $\cI$. Since the fundamental relations are sufficient to transform any word in the $\alpha$'s into a word of the type appearing in Proposition~\ref{prop:alphagen}(a) (or Proposition~\ref{prop:alphagen}(b)), it follows that they generate all relations.
\end{proof}

\subsection{Smooth actions} \label{ss:smooth}

Let $\cI^{\rm big}_{>n}$ be the submonoid of $\cI^{\rm big}$ consisting of those elements $\sigma$ that fix each of the numbers $1, \ldots, n$. Let $\cI_{>n}=\cI \cap \cI^{\rm big}_{>n}$. Clearly, $\cI_{>n}$ is the submonoid of $\cI$ generated by the $\alpha_k$ with $k>n$.

\begin{definition}
Let $X$ be an $\cI$-set. We say that an element $x \in X$ is {\bf smooth} if $x$ is fixed by $\cI_{>r}$ for some $r$. We say that $X$ is {\bf smooth} if every element is.
\end{definition}

We make a similar definition for $\cI^{\rm big}$. We almost exclusively study smooth representations in this paper. The following proposition thus explains why we could work with either version of the increasing monoid:

\begin{proposition}
The restriction functor
\begin{displaymath}
\sF \colon \{ \text{smooth $\cI^{\rm big}$-sets} \} \to \{ \text{smooth $\cI$-sets} \}
\end{displaymath}
is an isomorphism of categories.
\end{proposition}

\begin{proof}
Let $X$ be a smooth $\cI$-set. We claim that there is a unique smooth action of $\cI^{\rm big}$ on $X$ extending that of $\cI$. To see this, suppose that $\sigma \in \cI^{\rm big}$ and $x \in X$. Let $r$ be such that $x$ is fixed by $\cI_{>r}$. We can then factor $\sigma$ as $\sigma_1 \sigma_2$ where $\sigma_1 \in \cI$ and $\sigma_2 \in \cI^{\rm big}_{>r}$: indeed, we take $\sigma_1(n)=\sigma(n)$ for $1 \le n \le r$, and then put $\sigma_1(n)=n+\ell$ for $n \ge r$, where $\ell=\sigma(r)-r$; we then take $\sigma_2(n)=1$ for $1 \le n \le r$ and $\sigma_2(n)=\sigma(n)-\ell$ for $n \ge r+1$. We define $\sigma x$ to be $\sigma_1 x$. We leave to the reader the routine verification that this does indeed define a smooth action of $\cI^{\rm big}$ on $X$. To verify uniqueness, suppose that $\bullet$ is a second smooth action of $\cI^{\rm big}$ on $X$ extending the action of $\cI$. Let $\sigma$ and $x$ be as above, let $r$ be such that $x$ is fixed by $\cI^{\rm big}_{>r}$ under the action $\bullet$ (and thus by $\cI_{>r}$), and factor $\sigma$ as $\sigma_1 \sigma_2$ as above. Then $\sigma \bullet x = \sigma_1 \bullet x = \sigma_1 x = \sigma x$, which shows that $\bullet$ agrees with the previously defined action.

The above construction defines a functor
\begin{displaymath}
\sG \colon \{ \text{smooth $\cI$-sets} \} \to \{ \text{smooth $\cI^{\rm big}$-sets} \}.
\end{displaymath}
It is clear that $\sF \circ \sG$ and $\sG \circ \sF$ are both equal to the identity functor, and so $\sF$ and $\sG$ are mutually inverse isomorphisms of categories.
\end{proof}

The following proposition and corollary give a convenient criterion for an element of an $\cI$-set to be smooth.

\begin{proposition} \label{prop:fixed}
Let $X$ be an $\cI$-set and let $x \in X$. Suppose $\alpha_n x = x$. Then $x$ is fixed by all of $\cI_{\ge n}$.
\end{proposition}

\begin{proof}
Suppose $\alpha_i x = x$. We have the identity $\alpha_{i+1} \alpha_i = \alpha_i^2$ in $\cI$ by Proposition~\ref{prop:alpharel}. Applying both sides to $x$, we find $\alpha_{i+1} x = x$. We thus see that $x$ is fixed by $\alpha_i$ for all $i \ge n$, and these generate $\cI_{\ge n}$.
\end{proof}

\begin{corollary}
An element of an $\cI$-set is smooth if and only if it is fixed by some $\alpha_n$.
\end{corollary}

\section{Representation categories} \label{s:reps}

\subsection{The category of all $\cI$-modules} \label{ss:bigrep}

Fix, for the remainder of the paper, a field $\bk$. We let $\REP(\cI)$ be the category of all representations of $\cI$ over $\bk$. Thus $\REP(\cI)$ is simply the category of left modules over the monoid algebra $\bk[\cI]$. As such, it is a Grothendieck abelian category. This category will only play a tangential role in this paper.

\subsection{The category of smooth $\cI$-modules} \label{ss:rep}

Let $\Rep(\cI)$ be the full subcategory of $\REP(\cI)$ spanned by smooth modules. As any subquotient of a smooth module is again smooth, we see that $\Rep(\cI)$ is an abelian subcategory of $\REP(\cI)$.

Let $\{M_i\}_{i \in I}$ be a family of smooth $\cI$-modules. It is clear that the coproduct $\bigoplus_{i \in I} M_i$ (taken in the category of all $\cI$-modules) is again smooth, and is thus the coproduct in $\Rep(\cI)$. It follows that arbitrary colimits in $\Rep(\cI)$ can be computed in $\REP(\cI)$. We thus see that $\Rep(\cI)$ is cocomplete and that filtered colimits in $\Rep(\cI)$ are exact. It is clear that $\Rep(\cI)$ has a generator: simply take the direct sum of all smooth cyclic $\cI$-modules. (Or, more accurately, one from each isomorphism class. It is clear that the collection of all isomorphism classes forms a set.) We thus see that $\Rep(\cI)$ is a Grothendieck abelian category.

For an $\cI$-module $M$, we let $M^{\sm}$ be the set of all smooth elements in $M$; it is easily verified to be an $\cI$-submodule of $M$, and is clearly the maximal smooth submodule. Explicitly,
\begin{displaymath}
M^{\sm} = \bigcup_{r \ge 1} M^{\cI_{\ge r}} = \varinjlim_{r \to \infty} M^{\cI_{\ge r}}.
\end{displaymath}
Here $M^{\cI_{\ge r}}$ denotes the subspace of $M$ consisting of elements invariant under $\cI_{\ge r}$. One easily sees that $M \mapsto M^{\sm}$ is right adjoint to the inclusion functor $\Rep(\cI) \to \REP(\cI)$.

Let $\{M_i\}_{i \in I}$ be a diagram of smooth $\cI$-modules. The limit $\varprojlim M_i$, as computed in $\REP(\cI)$, need not be smooth. We define the {\bf smooth limit}, denoted $\varprojlim' M_i$, to be the maximal smooth submodule of $\varprojlim M_i$. Explicitly,
\begin{displaymath}
\varprojlim_{i \in I}\!{}' M_i
= \varinjlim_{r \to \infty} \left [\varprojlim_{i \in I} M_i \right]^{\cI_{\ge r}}
= \varinjlim_{r \to \infty} \left[\varprojlim_{i \in I} M_i^{\cI_{\ge r}} \right]
\end{displaymath}
In the second step, we have used the fact that formation of invariants commutes with limits. One readily verifies that $\varprojlim' M_i$ is the limit of the diagram $\{M_i\}$ in the category $\Rep(\cI)$. When the category $I$ is discrete, so that limits over $I$ are just products, we write $\prod'_{i \in I} M_i$ for the smooth limit, and refer to it as the {\bf smooth product}. One has the explicit description
\begin{displaymath}
\prod_{i \in I}\!{}'\, M_i = \varinjlim_{r \to \infty} \left[ \prod_{i \in I} M_i^{\cI_{\ge r}} \right]
\end{displaymath}
Of course, the smooth product is the product in the category $\Rep(\cI)$.

\begin{proposition}
Products in $\Rep(\cI)$ are not exact, that is, Grothendieck's (AB4*) axiom does not hold.
\end{proposition}

\begin{proof}
For an integer $n \ge 1$, let $M_n=\bigoplus_{k \ge n} \bk e_k$, with $\cI$ acting by $\sigma e_k=e_{\sigma(k)}$. This is a smooth $\cI$-module. Let $\epsilon_n \colon M_n \to \bk$ be the augmentation map (i.e., the linear map sending each $e_i$ to 1), which is a surjection of $\cI$-modules (giving $\bk$ the trivial action). We claim that $\prod'_{n \ge 1} \epsilon_n$ is not surjective. Indeed, suppose $x=(x_n)_{n \ge 1}$ is an element of $\prod'_{n \ge 1} M_n$. Then, by definition, $x$ belongs to $\prod_{n \ge 1} M_n^{\cI_{\ge r}}$ for some $r$. Note that $M_n^{\cI_{\ge r}}=0$ for $n \ge r$. Thus $x_n=0$ for $n \ge r$, and so $\epsilon_n(x_n)=0$ for $n \ge r$ as well. We thus see that the image of $\prod'_{n \ge 1} \epsilon_n$ is contained in $\bigoplus_{n \ge 1} \bk$, which is a proper submodule of $\prod'_{n \ge 1} \bk=\prod_{n \ge 1} \bk$.
\end{proof}

Due to the proposition, we must take care when dealing with the derived functors of $\lim'$: for example, we are not guaranteed that $\rR^i \lim'$ vanishes for $i>1$, or that Mittag--Leffler systems are $\lim'$-acyclic.

\subsection{The category of graded $\cI$-modules} \label{ss:urep}

A {\bf graded $\cI$-module} is an $\cI$-module $M$ equipped with a decomposition $M=\bigoplus_{n \ge 0} M_n$ such that two conditions hold:
\begin{enumerate}
\item Given $x \in M_n$ and $\sigma \in \cI$, the element $\sigma x$ belongs to $M_{\sigma(n)}$.
\item Given $x \in M_n$ and $\sigma \in \cI$ such that $\sigma(n)=n$, we have $\sigma x = x$.
\end{enumerate}
By convention, every element of $\cI$ fixes~0; thus every degree~0 element of a graded $\cI$-module is $\cI$-invariant. We let $\uRep(\cI)$ denote the category of graded $\cI$-modules. We use the notation $\uHom_{\cI}$ for the $\Hom$ sets in $\uRep(\cI)$; that is, for $M,N \in \uRep(\cI)$, we let $\uHom_{\cI}(M,N)$ denote the space of all homogeneous $\cI$-linear maps $M \to N$. If $M$ is a graded $\cI$-module then the action of $\cI$ on $M$ is smooth by condition~(b). We thus have a forgetful functor
\begin{displaymath}
\Phi \colon \uRep(\cI) \to \Rep(\cI).
\end{displaymath}

\begin{proposition} \label{prop:Phi-cont}
The forgetful functor $\Phi$ is continuous and cocontinuous.
\end{proposition}

\begin{proof}
Since the forgetful functor is exact, it suffices to show that it commutes with products and co-products. Co-products is clear: for both graded and smooth modules, corproducts are computed using the coproduct of the underlying vector space. We now treat products. Let $\{M_i\}_{i \in I}$ be a family of graded $\cI$-modules. We have
\begin{displaymath}
\left( \prod_{i \in I} \left[ \bigoplus_{n \ge 0} M_{i,n} \right] \right)^{\cI_{\ge r}}
= \prod_{i \in I} \left[ \bigoplus_{0 \le n < r} M_{i,n} \right]
= \bigoplus_{0 \le n < r} \left[ \prod_{i \in I} M_{i,n} \right].
\end{displaymath}
In the first equality, we used the fact that taking invariants commutes with products and coproducts, as well as the fact that $M_{i,n}^{\cI_{\ge r}}$ vanishes for $n \ge r$, and equals $M_{i,n}$ for $n<r$. In the second step, we used the fact that products commute with finite coproducts. Taking the union over $r$, we now find
\begin{displaymath}
\prod_{i \in I}{}'\! \left[ \bigoplus_{n \ge 0} M_{i,n} \right] = \bigoplus_{n \ge 0} \left[ \prod_{i \in I} M_{i,n} \right].
\end{displaymath}
The left side computes the result of first applying the forgetful functor and then taking products, while the right side computes the result of first taking products and then forgetting. The result follows.
\end{proof}

Let $M$ be a graded $\cI$-module. We let $M_{\ge n}$ be the subspace that is $M_k$ in degree $k \ge n$ and~0 in degree $k<n$. This is clearly a homogeneous $\cI$-submodule of $M$. We let $M^{<n}=M/M_{\ge n}$. We also put $M_+=M_{\ge 1}$. We say that $M$ is {\bf pure} if $M_0=0$, and denote the category of pure modules by $\uRep(\cI)_+$.

\begin{remark}
Since elements of $\cI$ fix~0, the degree~0 part of a graded $\cI$-module cannot interact with the positive degree part. In other words, the category $\uRep(\cI)$ decomposes as $\Vec \oplus \uRep(\cI)_+$. It may therefore seem somewhat odd to include the degree~0 piece in the definition. However, we will see several instances in which having the degree~0 part is more natural; see, for example, Remarks~\ref{rmk:catzero} and~\ref{rmk:urepbest}.
\end{remark}

\subsection{Coinvariants} \label{ss:inv}

Let $M$ be a an $\cI$-module. We write $M_{\cI_{\ge n}}$ for the coinvariant space under $\cI_{\ge n}$. This is the maximal quotient of $M$ on which $\cI_{\ge n}$ acts trivially; explicitly, it is the quotient of $M$ by the subspace spanned by elements of the form $\sigma x -x$ with $\sigma \in \cI_{\ge n}$ and $x \in M$. Let $\fa_n$ be the right ideal of $\bk[\cI]$ generated by the elements $\sigma-1$ with $\sigma \in \cI_{\ge n}$. Thus $M_{\cI_{\ge n}}=M/\fa_n M$.

\begin{proposition} \label{prop:ideal}
The right ideal $\fa_n$ is a two-sided ideal of $\bk[\cI]$. It is generated, as a left or right ideal, by the elements $\alpha_i-1$ with $i \ge n$. 
\end{proposition}

\begin{proof}
We first show that $\fa_n$ is generated as a right ideal by the $\alpha_i-1$ with $i \ge n$. Let $\fa'_n$ be the right ideal generated by the $\alpha_i-1$ with $i \ge n$. We obviously have $\fa'_n \subset \fa_n$; we prove the reverse inclusion. Let $\sigma \in \cI_{\ge n}$. It suffices to show that $\sigma-1 \in \fa_n'$. Since the $\alpha_i$ with $i \ge n$ generated $\cI_{\ge n}$, we can write $\sigma=\tau \alpha_i$ where $i \ge n$ and $\tau \in \cI_{\ge n}$ has shorter length than $\sigma$. We have the identity
\begin{displaymath}
\sigma-1=(\tau-1) \alpha_i+(\alpha_i-1).
\end{displaymath}
We can assume, by induction on length, that $\tau-1$ belongs to $\fa'_n$. Since $\alpha_i-1$ belongs to $\fa'_n$ by definition, we thus find that $\sigma-1 \in \fa'_n$, as required.

We now show that $\fa_n$ is a two-sided ideal. It suffices to show that $\alpha_j (\alpha_i-1) \in \fa_n$ for all $i \ge n$ and all $j \ge 1$. First suppose that $j \le i$. Then
\begin{displaymath}
\alpha_j (\alpha_i-1) = (\alpha_{i+1}-1) \alpha_j,
\end{displaymath}
and thus belongs to $\fa_n$, as $i+1 \ge n$. Now suppose $j>i$. Then
\begin{displaymath}
\alpha_j (\alpha_i-1) = \alpha_i \alpha_{j-1} - \alpha_j
=(\alpha_i-1) \alpha_{j-1}+(\alpha_{j-1}-1)-(\alpha_j-1),
\end{displaymath}
and thus belongs to $\fa_n$, as $i$, $j$, and $j-1$ are all $\ge n$.

Finally, let $\fa_n''$ be the left ideal generated by the $\alpha_i-1$ for $i \ge n$. Since $\fa_n$ is a left ideal containing these elements, we have $\fa_n'' \subset \fa_n$. An argument as in the previous paragraph shows that $\fa_n''$ is a two-sided ideal. Since $\fa_n$ is the smallest right ideal containing the elements $\alpha_i-1$ for $i \ge n$, it is also the smallest two-sided ideal containing these elements, and so $\fa_n''=\fa_n$.
\end{proof}

\begin{corollary}
Let $M$ be an $\cI$-module. Then $\fa_n M$ is an $\cI$-submodule of $M$. Thus $M_{\cI_{\ge n}}$ is canonically an $\cI$-module and the canonical map $M \to M_{\cI_{\ge n}}$ is $\cI$-linear.
\end{corollary}

\subsection{$\OI$-modules} \label{ss:oiequiv}

Recall that $\OI$ is the category whose objects are finite totally ordered sets and whose morphisms are order-preserving injections. For $n \in \bN$, let $[n]$ be the set $\{1,\ldots,n\}$ equipped with its usual total order. Every object of $\OI$ is isomorphic to a unique $[n]$. For $1 \le i \le n+1$, let $\beta_{n,i} \colon [n] \to [n+1]$ be the map defined by
\begin{displaymath}
\beta_{n,i}(j) = \begin{cases}
j & \text{if $j<i$} \\
j+1 & \text{if $j \ge i$}
\end{cases}
\end{displaymath}
These are all the morphisms from $[n]$ to $[n+1]$ in the category $\OI$. Moreoever, every morphism in $\OI$ can be realized as a composition of these morphisms. We have the relations
\begin{equation} \label{eq:betarel}
\beta_{n+1,i} \circ \beta_{n,j} = \beta_{n+1,j} \circ \beta_{n,i-1}
\end{equation}
for $1 \le j < i \le n+2$, and these relations generate all relations between the $\beta$'s. The above claims can be proved similarly to Corollary~\ref{cor:alphagen}; they are also similar to standard facts about the simplex category, see \cite[\S 8.1]{weibel}.

\begin{proposition} \label{prop:oiequiv}
We have an equivalence of categories $\uRep(\cI)_+ \cong \Mod_{\OI}$.
\end{proposition}

\begin{proof}
Let $M$ be an $\OI$-module. We define a graded $\cI$-module $N$ as follows. As a graded vector space, $N_0=0$ and $N_n=M([n-1])$ for $n \ge 1$. For $1 \le i \le n$, we define $\alpha_i \colon N_n \to N_{n+1}$ to be the map $\beta_{n-1,i} \colon M([n-1]) \to M([n])$. The fundamental relations correspond exactly to the identities \eqref{eq:betarel}. This construction is reversible: starting with $N$ we can define $M$. It thus furnishes the claimed equivalence.
\end{proof}

\subsection{Semi-simplicial vector spaces} \label{ss:ssequiv}

Let $\Delta$ be the simplex category: its objects are finite non-empty totally ordered sets and its morphisms are order-preserving functions. Let $\Delta_{\rm inj}$ be the subcategory of $\Delta$ with the same objects but where the morphisms are injective. Recall that a {\bf simplicial object} in a category $\cC$ is a functor $\Delta^{\op} \to \cC$, and that a {\bf semi-simplicial object} in $\cC$ is a functor $\Delta_{\rm inj}^{\op} \to \cC$. One also defines {\bf co-simplicial} and {\bf co-semi-simplicial} objects in $\cC$ as functors $\Delta \to \cC$ and $\Delta_{\rm inj} \to \cC$. For general background on simplicial objects, see \cite[\S 8]{weibel}.

The Dold--Kan theorem \cite[Theorem~8.4.1]{weibel} states that the category of simplicial objects in an abelian category $\cA$ is equivalent to the category of chain complexes in $\cA$ supported in non-negative homological degrees. Since chain complexes of vector spaces are the same as graded modules over the ring $\bk[\epsilon]/(\epsilon^2)$, one can give very precise results on their structure. Thus, by the Dold--Kan theorem, simplicial vector spaces have a fairly simple structure. Semi-simplicial vector spaces, however, are much more complicated. The following proposition shows that the results of this paper can be used to understand them:

\begin{proposition}
The category of co-semi-simplicial vector spaces is equivalent to the full subcategory of $\uRep(\cI)$ spanned by graded $\cI$-modules $M$ with $M_0=M_1=0$.
\end{proposition}

\begin{proof}
The category $\Delta_{\rm inj}$ is simply the full subcategory $\OI'$ of $\OI$ spanned by non-empty sets. An $\OI'$-module $M$ can be extended to an $\OI$-module by simply putting $M(\emptyset)=0$, and this construction is an equivalence between the category of $\OI'$-modules and the category of $\OI$-modules $M$ with $M(\emptyset)=0$. The result now follows from Proposition~\ref{prop:oiequiv}.
\end{proof}

\begin{remark}
Semi-simplicial and co-semi-simplicial vector spaces are not so different: if $M$ is a co-semi-simplicial vector space then $S \mapsto M(S)^*$ is a semi-simplicial vector space, and this construction is an equivalence on the categories of (co-)semi-simplicial objects taking values in finite dimensional vector spaces.
\end{remark}

\subsection{Shuffle algebras} \label{ss:shuffle}

Given $n,m \in \bN$, an {\bf $(n,m)$-shuffle} is a bijection $[n] \amalg [m] \to [n+m]$ that is order-preserving on each summand. Here, $[n]$ denotes the set $\{1,\ldots,n\}$. Let $\cS(n,m)$ denote the set of $(n,m)$ shuffles. Given $\bN$-graded vector spaces $V$ and $W$, define a new graded vector space $V \medshuffle W$ by
\begin{displaymath}
(V \medshuffle W)_n = \bigoplus_{i+j=n} \bk[\cS(i,j)] \otimes V_i \otimes W_j.
\end{displaymath}
The construction $\medshuffle$ endows the category $\GrVec$ of graded vector spaces with a symmetric monoidal structure. The associativity constraint comes from the natural bijection
\begin{displaymath}
\cS(n,m) \times \cS(n+m, \ell) \cong \cS(n,m+\ell) \times \cS(m,\ell),
\end{displaymath}
while the symmetry comes from the natural bijection
\begin{displaymath}
\cS(n,m) \cong \cS(m,n).
\end{displaymath}
A {\bf (commutative) shuffle algebra} is a (commutative) associative unital algbera in the symmetric monoidal category $(\GrVec, \medshuffle)$. More explicitly, a shuffle algebra is a graded vector space $A=\bigoplus_{n \ge 0} A_n$ equipped with, for each shuffle $\sigma \in \cS(n,m)$, a multiplication $\ast_{\sigma} \colon A_n \otimes A_m \to A_{n+m}$, satisfying a number of axioms.

The monoidal operation $\medshuffle$ on $\GrVec$ and shuffle algebras are discussed in \cite[\S 2]{dotsenko}, and also \cite[\S 2]{ronco}, to which we refer for further details. Shuffle algebras also appear in the works of Sam \cite{sam} and Laudone \cite{laudone} that study syzygies of certain families of varieties. Our interest in shuffle algberas comes from the following observation:

\begin{proposition} \label{prop:shuffle}
Let $\bk\langle 1 \rangle$ denote the graded vector space that is $\bk$ in degree~1, and~0 in other degrees. Let $A$ be the shuffle algebra $\Sym_{\medshuffle}(\bk\langle 1 \rangle)$. Then the category $\Mod_A$ of left $A$-modules is equivalent to the category $\Mod_{\OI}$.
\end{proposition}

\begin{proof}
Since this result is only used as motivation, we simply sketch the proof. Given a shuffle $\sigma \in \cS(n,m)$, let $f(\sigma) \colon [m] \to [n+m]$ be the restriction of $\sigma$ to $[m]$, regarded as a morphism in $\OI$. Observe that
\begin{displaymath}
f \colon \cS(n,m) \to \Hom_{\OI}([m], [n+m])
\end{displaymath}
is a bijeciton.

Now, let $M$ be a graded vector space. To give an $A$-module structure on $M$ requires giving a multiplication map $\ast_{\sigma} \colon A_n \otimes M_m \to M_{n+m}$ for each $\sigma \in \cS(n,m)$. However, $A_n$ is one-dimensional, so this is the same as giving a map $M_m \to M_{n+m}$ for each element of $\cS(n,m)$. By the above bijection, this is the same as giving a map $M_m \to M_{n+m}$ for each element of $\Hom_{\OI}([m], [n+m])$, which is exactly the data required to give $M$ the structure of an $\OI$-module. One must now verify that the conditions on the data on either side match, which we leave to the reader.
\end{proof}

\begin{corollary} \label{cor:modRrepI}
We have an equivalence of categories $\Mod_A \cong \uRep(\cI)_+$.
\end{corollary}

\begin{proof}
Simply combine the proposition with the equivalence $\uRep(\cI)_+ \cong \Mod_{\OI}$ (Proposition~\ref{prop:oiequiv}).
\end{proof}

\begin{remark} \label{rem:roid}
One can show that the category of left modules over the shuffle algebra $\Sym_{\medshuffle}(\bk\langle 1 \rangle^{\oplus d})$ is equivalent to the category $\Mod_{\OI_d}$ of $\OI_d$-modules, where $\OI_d$ is the generalization of $\OI$ defined in \cite[\S 7.1]{catgb}.
\end{remark}

\begin{remark}
The shuffle algebra perspective on $\cI$-modules brings great clarity to the work in \S \ref{s:koszul} on Koszul duality. See \S \ref{ss:conceptual} for details.
\end{remark}

\section{Monomial modules} \label{s:monomial}

\subsection{Monomial modules}

A {\bf pointed $\cI$-set} is a set $\cU$ equipped with an action of $\cI$ and a distinguished point $\ast$ (called zero) that is fixed by $\cI$. Let $\cU$ be a pointed set. We let $\cU_+=\cU \setminus \{\ast\}$. Wet $\bk[\cU]$ be the vector space with basis $\cU_+$, and for $i \in \cU_+$, we let $e_i \in \bk[\cU]$ be the corresponding basis vector. We also put $e_{\ast}=0$. The space $\bk[\cU]$ carries an action of $\cI$ via $\sigma e_i=e_{\sigma(i)}$. If the action of $\cI$ on $\cU$ is smooth then the action of $\cI$ on $\bk[\cU]$ is also smooth. We call modules of this form {\bf monomial modules}, and we refer to the elements $e_i$ as {\bf monomials}.

Let $\cU$ be a pointed $\cI$-set. A {\bf grading} on $\cU$ is a function $\delta \colon \cU \to \bN$ that is $\cI$-equivariant and satisfies the following condition: if $i \in \cU$ and $\sigma \in \cI$ is such that $\sigma$ stabilizes $\delta(i)$ then $\sigma$ stabilizes $i$. Given a grading on $\cU$, we give $\bk[\cU]$ the structure of a graded $\cI$-module by declaring $e_i$ to have degree $\delta(i)$.

\subsection{Principal modules} \label{ss:principal}

Let $\cA^r_+$ be the set consisting all strictly increasing tuples in $\bN_+^r$, and let $\cA^r=\cA^r_+ \cup \{\ast\}$. This set is naturally a smooth pointed $\cI$-set. We let $\bA^r$ be the associated monomial module. We call these {\bf principal modules}. For $1 \le i_1<\cdots<i_r$, we let $e_{i_1,\ldots,i_r}$ denote the corresponding basis vector of $\bA^r$. The set $\cA^r$ admits a grading, defined by $\delta(i_1,\ldots,i_r)=i_r$. We let $\ul{\bA}^r$ denote the corresponding graded $\cI$-module. The basis vector $e_{i_1,\ldots,i_r}$ of $\ul{\bA}^r$ has degree $i_r$.

\begin{proposition}
Let $M$ be an arbitrary $\cI$-module. Then $\Hom_{\cI}(\bA^r, M)$ is canonically isomorphic to $M^{\cI_{>r}}$. Precisely, if $x \in M^{\cI_{>r}}$ then there is a unique $\cI$-linear map $f \colon \bA^r \to M$ such that $f(e_{1,\ldots,r})=x$.
\end{proposition}

\begin{proof}
Let $x \in M^{\cI_{>r}}$ be given. Suppose that $\sigma,\tau \in \cI$ satisfy $\sigma(e_{1,\ldots,r}) =\tau(e_{1,\ldots,r})$. Thus $\sigma(i)=\tau(i)$ for $1 \le i \le r$. Thus if $\sigma=\alpha_1^{n_1} \alpha_2^{n_2} \cdots$ and $\tau=\alpha_1^{m_1} \alpha_2^{m_2} \cdots$ then $n_i=m_i$ for $1 \le i \le r$; indeed, $\sigma(1)=n_1$, $\sigma(2)=n_1+n_2$, and so on. It follows that $\sigma x = \tau x$. Indeed, since $x$ is fixed by the $\alpha_k$ with $k>r$, we have $\sigma x = \alpha_1^{n_1} \cdots \alpha_r^{n_r} x = \tau x$. It follows that we have a well-defined $\cI$-equivariant map $\cA^r \setminus \{\ast\} \to M$ by $\sigma (1,\ldots,r) \mapsto \sigma x$ for $\sigma \in \cI$. This extends linearly to an $\cI$-linear map $f \colon \bA^r \to M$ satisfying $f(e_{1,\ldots,r})=x$. Since $e_{1,\ldots,r}$ generates $\bA^r$, it follows that $f$ is the unique map taking $e_{1,\ldots,r}$ to $x$.
\end{proof}

\begin{corollary} \label{cor:pringen}
Let $M$ be a smooth $\cI$-module. Then $M$ is a quotient of a direct sum of principal modules. If $M$ is finitely generated, then it is a quotient of a finite direct sum of principle modules.
\end{corollary}

\begin{proof}
Let $M$ be a smooth $\cI$-module and let $\{x_i\}_{i \in I}$ be a set of generators for $M$. Since $M$ is smooth, each $x_i$ is fixed by some $\cI_{>r_i}$. We thus have a unique map $f_i \colon \bA^{r_i} \to M$ that sends $e_{1,\ldots,r_i}$ to $x_i$. Since the $x_i$ generate $M$, the map $\bigoplus_{i \in I} f_i \colon \bigoplus_{i \in I} \bA^{r_i} \to M$ is surjective.
\end{proof}

\begin{proposition}
Let $M$ be a graded $\cI$-module. Then $\uHom_{\cI}(\ul{\bA}^r, M)$ is canonically isomorphic to $M_r$. Precisely, if $x \in M_r$ then there is a unique homogeneous $\cI$-linear map $f \colon \ul{\bA}^r \to M$ such that $f(e_{1,\ldots,r})=x$.
\end{proposition}

\begin{proof}
Let $x \in M_r$ be given. Since $x$ is $\cI_{>r}$ invariant, the previous proposition yields a unique $\cI$-linear map $f \colon \bA^r \to M$ satisfying $f(e_{1,\ldots,r})=x$. Thus $f(\sigma e_{1,\ldots,r})=\sigma x$. Since both $\sigma e_{1,\ldots,r}$ and $\sigma x$ have degree $\sigma(r)$, and the $\sigma e_{1,\ldots,r}$ span $\bA^r$, it follows that $f$ is homogeneous. This proves the proposition.
\end{proof}

\begin{corollary}
The graded principal modules are projective. Furthermore, every (finitely generated) graded $\cI$-module is a quotient of a (finite) direct sum of graded principal modules.
\end{corollary}

\begin{proof}
The proposition shows that the functor $\uHom_{\cI}(\ul{\bA}^r, -)$ is isomorphic to the functor $M \mapsto M_r$, and is thus exact; therefore $\ul{\bA}^r$ is projective. The second statement is clear.
\end{proof}

A {\bf monomial submodule} of a monomial $\cI$-module is one that is spanned (over $\bk$) by the monomials it contains. The following proposition classifies the monomial submodules of $\bA^r$, and establishes an important link to commutative algebra.

\begin{proposition} \label{prop:prinpoly}
Let $r \in \bN$. Consider the linear map $f \colon \bA^r \to \bk[x_1, \ldots, x_r]$ defined by
\begin{displaymath}
f(e_{i_1,\ldots,i_r}) = x_1^{i_1-1} x_2^{i_2-i_1-1} \cdots x_r^{i_r-i_{r-1}-1}.
\end{displaymath}
Then $f$ is an isomorphism of vector spaces. Moreover, $f$ induces a bijection
\begin{displaymath}
\{ \text{monomial submodules of $\bA^r$} \} \to
\{ \text{monomial ideals of $\bk[x_1,\ldots,x_r]$} \}.
\end{displaymath}
\end{proposition}

\begin{proof}
It is clear that $f$ induces a bijection between the set of basis vectors in $\bA^r$ and the set of monomials in $\bk[x_1,\ldots,x_r]$, and is therefore an isomorphism of vector spaces. To prove the second statement, it suffices to show that if $V \subset \bA^r$ is a subspace spanned by some collection of basis vectors then $V$ is an $\cI$-submodule of $\bA^r$ if and only if $f(V)$ is an ideal of $\bk[x_1, \ldots, x_r]$. Observe that we have the identity:
\begin{displaymath}
x_k f(e_{i_1,\ldots,i_r}) = f(e_{i_1,\ldots,i_{k-1},i_k+1,\ldots,i_r+1}).
\end{displaymath}
From this, we see that the following two conditions are equivalent:
\begin{enumerate}
\item $f(V)$ is an ideal in $\bk[x_1,\ldots,x_r]$.
\item If $e_{i_1,\ldots,i_r}$ belongs to $V$ then so does $e_{i_1,\ldots,i_{k-1},i_k+1,\ldots,i_r+1}$, for any $1 \le k \le r$.
\end{enumerate}
It is clear that (b) is equivalent to $\alpha_j V \subset V$ for all $j$, and thus to $V$ being an $\cI$-submodule. This completes the proof.
\end{proof}

\subsection{Simple modules} \label{ss:simple}

For $n \in \bN$, let $\cB^n$ be the set $\{\ast, n\}$. We give $\cB^n$ the structure of a pointed $\cI$-set by:
\begin{displaymath}
\alpha_i \cdot n = \begin{cases}
\ast & \text{if $i \le n$} \\
n & \text{if $i>n$}
\end{cases}
\end{displaymath}
As such, it is smooth. The set $\cB^n$ admits a grading, in which $n$ is given degree $n$.

We let $\bB^n$ be the monomial module associated to $\cB^n$, and $\ul{\bB}^n$ the graded version. These are one-dimensional, and thus simple. We note that $\bB^0$ is the 1-dimensional trivial representation, and thus isomorphic to $\bA^0$. The module $\ul{\bB}^n$ is pure for $n>0$.

\subsection{Standard modules} \label{ss:gap}

A {\bf constraint word} a finite word $\lambda=\lambda_1 \cdots \lambda_k$ in the 2-letter alphabet $\{\whitetri, \blacktri \}$. Fix such $\lambda$. Let $\cE^\lambda_+$ be the set of strictly-increasing tuples $(n_1, \ldots, n_k)$ with $n_i \in \bN_+$ such that $n_i-n_{i-1}=1$ if $\lambda_i=\blacktri$; here we use the convention $n_0=0$. If $\lambda$ is the empty word (i.e., $k=0$) then $\cE^\lambda_+$ is the singleton set consisting of the empty tuple. We let $\cE^\lambda=\cE^\lambda_+ \cup \{\ast\}$. We let $\cI$-act on this set by:
\begin{displaymath}
\sigma (n_1, \ldots, n_k) = \begin{cases}
(\sigma n_1, \ldots, \sigma n_k) & \text{if this belongs to $\cE^\lambda_+$} \\
\ast & \text{otherwise.}
\end{cases}
\end{displaymath}
Of course, we also put $\sigma \ast=\ast$ for all $\sigma$. In this way, $\cE^\lambda$ has the structure of a smooth pointed $\cI$-set. It is also admits a natural grading, by declaring $(n_1, \ldots, n_k)$ to have degree $n_k$. (For $k=0$, the empty singleton is given degree~0.)

We let $\bE^\lambda$ be the monomial module associated to $\cE^\lambda$, and let $\ul{\bE}^\lambda$ be the graded version. We call these {\bf standard modules}. A few remarks:
\begin{itemize}
\item If $\lambda=\whitetri^k$ is the length $k$ word with $\lambda_i=\whitetri$ for all $i$ then $\bE^{\lambda} \cong \bA^k$ and $\ul{\bE}^{\lambda} \cong \ul{\bA}^k$.
\item If $\lambda=\blacktri^k$ is the length $k$ word with $\lambda_i=\blacktri$ for all $i$ then $\bE^{\lambda} \cong \bB^k$ and $\ul{\bE}^{\lambda} \cong \ul{\bB}^k$.
\item If $\lambda$ is the empty word then $\bE^\lambda \cong \bB^0 \cong \bA^0$ is the trivial representation. Otherwise, $\ul{\bE}^{\lambda}$ is pure.
\end{itemize}
We define the {\bf rank} of $\lambda$ to be the number of times $\whitetri$ occurs in it. We define the {\bf rank} of $\bE^{\lambda}$ and $\ul{\bE}^{\lambda}$ to be that of $\lambda$. Thus, for example, $\bA^r$ has rank~$r$, while $\bB^n$ has rank~0.

\section{Gr\"obner theory} \label{s:grobner}

\subsection{Principal modules}

Recall that the principal module $\bA^r$ has a basis indexed by the set $\cA^r_+$ of increasing sequences $(i_1,\cdots,i_r)$ of positive integers. Let $<$ be the lexicographic order on $\cA^r_+$, where the larger numbers are compared first; to be precise, $(i_1,\ldots,i_r)<(j_1,\ldots,j_r)$ if $i_k<j_k$, where $k$ is the maximal index such that $i_k \ne j_k$. The order $<$ is a well-order on $\cA^r_+$, and is compatible with the $\cI$-action in the sense that $(i_1,\ldots,i_r)<(j_1,\ldots,j_r)$ implies $\sigma (i_1,\ldots,i_r) < \sigma (j_1, \ldots, j_r)$ for any $\sigma \in \cI$.

Given a non-zero element $x=\sum_{(i_1,\ldots,i_r) \in \cA^r_+} c_{i_1,\ldots,i_r} e_{i_1,\ldots,i_r}$ of $\bA^r$, we define the {\bf initial term} of $x$, denoted $\ini(x)$, to be $e_{i_1,\ldots,i_r}$, where $(i_1,\ldots,i_r)$ is the maximal element of $\cA^r_+$ with $c_{i_1,\ldots,i_r}$ non-zero. Since $\cI$ is compatible with the order, we see that $\ini(\sigma x)=\sigma \ini(x)$ for any $\sigma \in \cI$.

Now let $M$ be an $\cI$-submodule of $\bA^r$. We define the {\bf initial submodule}, denoted $\ini(M)$, to be the the subspace of $\bA^r$ spanned by the elements $\ini(x)$ for $x \in M$. This is a monomial submodule of $\bA^r$: indeed, it is a submodule since $\sigma \ini(x)=\ini(\sigma x)$, and it is monomial since it is spanned by a collection of basis vectors. We have the usual Gr\"obner lemma:

\begin{proposition} \label{prop:groblem}
Suppose that $M \subset N$ are $\cI$-submodules of $\bA^r$ such that $\ini(M)=\ini(N)$. Then $M=N$.
\end{proposition}

\begin{proof}
Suppose that $M \ne N$, and choose $x \in N \setminus M$ with $\ini(x)$ minimal; this is possible since $<$ is a well-order on $\cA^r_+$. Since $\ini(N)=\ini(M)$, we have $\ini(x)=\ini(y)$ for some $y \in M$. There is thus a non-zero scalar $c$ such that $\ini(x-cy)<\ini(x)$. By the minimality of $x$, we see that $x-cy \in M$. Since $cy \in M$, this implies $x \in M$, a contradiction. Thus $M=N$.
\end{proof}

The main application of the above material is the following theorem:

\begin{theorem} \label{thm:noeth}
The categories $\Rep(\cI)$ and $\uRep(\cI)$ are locally noetherian.
\end{theorem}

\begin{proof}
We first claim that $\bA^r$ is a noetherian object of $\Rep(\cI)$, for any $r$. Indeed, suppose that $M_1 \subset M_2 \subset \cdots \subset \bA^r$ is an ascending chain of $\cI$-submodules. We then have an ascending chain $\ini(M_1) \subset \ini(M_2) \subset \cdots$ of monomial submodules of $\bA^r$. By Proposition~\ref{prop:prinpoly}, this corresponds to an ascending chain of monomial ideals in the polynomial ring $\bk[x_1,\ldots,x_r]$, and thus stabilizes since this ring is noetherian. (It is also easy to directly prove that this chain stabilizes.) By Proposition~\ref{prop:groblem}, it follows that the original chain $M_{\bullet}$ stabilizes, and so $\bA^r$ is noetherian.

It now follows that $\ul{\bA}^r$ is a noetherian object of $\uRep(\cI)$: indeed, an ascending chain of homogeneous submodules of $\ul{\bA}^r$ is an ascending chain of submodules of $\bA^r$ that happens to also be homogeneous, and so stabilizes.

Any subquotient of a noetherian object is noetherian; also, any finite direct sum of noetherian objects is noetherian. Since every finitely generated object in $\Rep(\cI)$ is a quotient of a finite sum of principal modules (Corollary~\ref{cor:pringen}), it follows that all finitely generated objects of $\Rep(\cI)$ are noetherian. Thus $\Rep(\cI)$ is locally noetherian. The same reasoning applies to $\uRep(\cI)$.
\end{proof}

\begin{remark}
The local noetherianity of $\uRep(\cI)$ was previously known: in \cite[Theorem~7.1.1]{catgb}, it is shown that the category $\Mod_{\OI}$ of $\OI$-modules is locally noetherian (also using a Gr\"obner-theoretic argument), which implies local noetherianity of $\uRep(\cI)$ via the equivalence in Proposition~\ref{prop:oiequiv}. The noetherian result for $\Rep(\cI)$ has certainly been known to experts in the field for some time, but we do not know if it previously appeared in the literature.
\end{remark}

\subsection{Graded modules} \label{ss:grgrob}

We now develop a version of Gr\"obner theory that applies to inhomogeneous submodules of graded modules. We use this theory here to deduce a few properties of the forgetful functor $\Phi$. We will also use it later in our study of saturation (see Proposition~\ref{prop:bdsat}).

Let $M$ be a graded $\cI$-module. We will regard $M$ as an object of $\Rep(\cI)$ in what follows. Given a non-zero element $x \in M$, let $x=\sum_{i \in \bN} x_i$ be its decomposition into homogeneous pieces. We define the {\bf initial term} of $x$, denoted $\ini(x)$, to be $x_i$ where $i$ is maximal so that $x_i$ is non-zero; we also define the {\bf degree} of $x$, denoted $\deg(x)$, to be this value of $i$. We define the {\bf initial submodule} of a (possibly inhomogeneous) $\cI$-submodule $N$ of $M$, denoted $\ini(N)$, to be the subspace of $M$ spanned by the elements $\ini(x)$ as $x$ varies over $N$; it is a homogeneous $\cI$-submodule of $M$. We say that $x_1, \ldots, x_n \in N$ are a {\bf Gr\"obner basis} if $\ini(x_1), \ldots, \ini(x_n)$ generate $\ini(N)$.

We shall require the following results on Gr\"obner bases. They are entirely standard, at least in the usual context of polynomial rings, but we include proofs to be complete.

\begin{proposition}
Let $M$ be a finitely generated graded $\cI$-module, and let $N$ be a (possibly inhomogeneous) $\cI$-submodule. Then $N$ admits a Gr\"obner basis.
\end{proposition}

\begin{proof}
Suppose not. Inductively define a sequence $x_1, x_2, \ldots$ in $N$ by taking $x_{n+1}$ to be any element of $N$ such that $\ini(x_{n+1})$ does not belong to the submodule of $M$ generated by $\ini(x_1), \ldots, \ini(x_n)$; note that such $x_{n+1}$ exists since $x_1, \ldots, x_n$ are not a Gr\"obner basis. We thus see that the submodule of $M$ generated by $\{\ini(x_n)\}_{n \ge 1}$ is not finitely generated, contradicting the noetherianity of $M$ (Theorem~\ref{thm:noeth}).
\end{proof}

\begin{proposition} \label{prop:grobexp}
Let $M$ be a graded $\cI$-module, let $N$ be a (possibly inhomogeneous) $\cI$-submodule, and let $x_1, \ldots, x_n$ be a Gr\"obner basis of $N$. Let $y \in N$. Then we have an expression
\begin{displaymath}
y = \sum_{i=1}^m c_i \sigma_i x_{a_i}
\end{displaymath}
where $c_i \in \bk$, $\sigma_i \in \cI$, and $\deg(\sigma_i x_{a_i}) \le \deg(y)$ for all $j$.
\end{proposition}

\begin{proof}
If $y=0$ then the statement is obvious. Thus suppose $y \ne 0$, and let us proceed by induction on $d=\deg(y)$. Since $\ini(y)$ belongs to $\ini(N)$, it can be generated by the homogeneous elements $\ini(x_1), \ldots, \ini(x_n)$. We therefore have an expression $\ini(y) = \sum_{i=1}^m c_i \sigma_i \ini(x_{a_i})$ for some scalars $c_i$ and indices $a_i$. As the terms in this sum are homogeneous and the result has degree $d$, all terms not of degree $d$ cancel, and we can therefore exclude them. We may thus assume $\deg(\sigma_i \ini(x_{a_i}))=d$ for all $i$. Note that $\deg(\sigma_i x_{a_i})=d$ as well. Let $y'=y-\sum_{i=1}^m c_i \sigma_i x_{a_i}$. Then $y'$ has degree $<d$. By the inductive hypothesis, we have an expression $y'=\sum_{j=1}^{\ell} c'_j \sigma'_j x_{a'_j}$ where $\deg(\sigma'_j x_{a'_j})<d$ for all $j$. We thus find $y=\sum_{j=1}^{\ell} c'_j \sigma'_j x_{a'_j}-\sum_{i=1}^m c_i \sigma_j x_{a_i}$, as required.
\end{proof}

\begin{proposition} \label{prop:inhom-decomp}
Let $M$ be a graded $\cI$-module. Suppose that $A$ and $B$ are possibly inhomogeneous submodules of $M$ such that $M=A \oplus B$. Then $M=\ini(A) \oplus \ini(B)$.
\end{proposition}

\begin{proof}
We first claim that $M=\ini(A)+\ini(B)$. Thus let $x \in M_k$ be given. Write $x=a+b$ with $a \in A$ and $b \in B$. Since $x$ belongs to $M_k$, it is invariant under $\cI_{>k}$, and so $a$ and $b$ are each invariant under $\cI_{>k}$. Thus $a,b \in M^{\cI_{>k}}=\sum_{i=0}^k M_i$. Let $a=\sum_{i=0}^k a_i$ and $b=\sum_{i=0}^k b_i$ be the decompositions of $a$ and $b$ into homogeneous pieces. Thus $x=a_k+b_k$. As $a_k$ is either~0 or $\ini(a)$, and similarly for $b_k$, the claim is established.

We now show that $\ini(A) \cap \ini(B)=0$. Suppose not, and say $\ini(a)=\ini(b)$ is non-zero of degree $k$, where $a\in A$ and $b \in B$ have degree $k$. Then $a-b$ has degree $<k$, and so, by the previous paragraph, can be written in the form $\sum_{i=0}^{k-1} \left[ \ini(a_i) - \ini(b_i) \right]$, where the $a_i$ belong to $A$ and the $b_i$ belong to $B$, and $a_i$ and $b_i$ have degree $\le i$. Let $a'=a-\sum_{i=0}^{k-1} a_i$ and $b'=b-\sum_{i=0}^{k-1} b_i$. Then $a'=b'$. Since this belongs to $A \cap B=0$, we find $a'=0$ and $b'=0$. But this is a contradiction, as $\ini(a')=\ini(a)$ is non-zero. We conclude that $\ini(A) \cap \ini(B)=0$.
\end{proof}

\begin{proposition} \label{prop:phi-indecomp}
Let $M$ be a graded $\cI$-module. Then $M$ is indecomposable if and only if $\Phi(M)$ is indecomposable.
\end{proposition}

\begin{proof}
It is clear that if $\Phi(M)$ is indecomposable then $M$ is indecomposable. Suppose now that $\Phi(M)$ is decomposable, so that $M=A \oplus B$ for (possibly inhomogeneous) non-zero submodules $A$ and $B$ of $M$. By the proposition, we have $M=\ini(A) \oplus \ini(B)$, and so $M$ is decomposable.
\end{proof}

\begin{proposition} \label{prop:phi-split}
An exact sequence of graded $\cI$-modules is split if and only if it is split after applying $\Phi$.
\end{proposition}

\begin{proof}
Consider an exact sequence of graded $\cI$-modules:
\begin{displaymath}
0 \to M_1 \to M_2 \to M_3 \to 0.
\end{displaymath}
Obviously, if it is split then it remains split after applying $\Phi$. Suppose that it splits after applying $\Phi$. We can thus find a decomposition $M_2=M_1 \oplus K$ where $K$ is a possibly inhomogeneous submodule of $M_2$. By Proposition~\ref{prop:inhom-decomp}, we have $M_2=M_1 \oplus \ini(K)$; note that since $M_1$ is homogeneous, we have $\ini(M_1)=M_1$. Since $\ini(K)$ is a homogenous submodule of $M_2$, we see that the sequence splits in the graded category.
\end{proof}

\section{Concatenation, shift, and transpose} \label{s:cat}

\subsection{The concatenation product} \label{ss:concat}

Let $M$ be a graded $\cI$-module and let $N$ be an arbitrary $\cI$-module. We define a new $\cI$-module, called the {\bf concatenation} of $M$ and $N$, and denoted $M \odot N$, as follows. As a vector space, $M \odot N$ is just the usual tensor product $M \otimes N$, where here $M$ denotes the underlying ungraded vector space of $M$. For $x \in M$ and $y \in N$, we write $x \odot y$ in place of $x \otimes y$ when we regard it as an element of $M \odot N$; this will help us keep track of the $\cI$-actions better. The action of $\cI$ on $M \odot N$ is defined as follows: for $x \in M_n$ and $y \in N$, we put
\begin{displaymath}
\alpha_k \cdot (x \odot y) = \begin{cases}
(\alpha_k x) \odot y  & \text{if $1 \le k \le n$}   \\
x \odot (\alpha_{k-n} y) & \text{if $n<k$}
\end{cases}
\end{displaymath}
We leave to the reader the simple verification that this action is well-defined (i.e., the fundamental relations are respected). It is clear that $M \odot N$ is smooth if $N$ is. If $N$ is graded then $M \odot N$ is in fact a graded $\cI$-module, under the usual grading (i.e., if $x \in M_i$ and $y \in N_j$ then $x \odot y \in (M \odot N)_{i+j}$). We note that, essentially by definition, we have $\Phi(M \odot N) = M \odot \Phi(N)$. One readily verifies that $\odot$ defines a (non-symmetric) monoidal structure on $\uRep(\cI)$, with unit object $\ul{\bB}^0=\ul{\bA}^0$, and also defines on $\Rep(\cI)$ the structure of a module category over the monoidal category $\uRep(\cI)$.

For constraint words $\lambda=\lambda_1 \cdots \lambda_r$ and and $\mu=\mu_1 \cdots \mu_s$, let $\lambda \odot \mu$ be the concatenated word $\lambda_1 \cdots \lambda_r \cdot \mu_1 \cdots \mu_s$.

\begin{proposition} \label{prop:stdcat}
We have a natural isomorphism $\ul{\bE}^{\lambda} \odot \ul{\bE}^{\mu} \cong \ul{\bE}^{\lambda \odot \mu}$ of graded $\cI$-modules. Similarly, we have an isomorphism $\ul{\bE}^{\lambda} \odot \bE^{\mu} \cong \bE^{\lambda \odot \mu}$.
\end{proposition}

\begin{proof}
Let $\lambda=\lambda_1 \cdots \lambda_r$ and $\mu=\mu_1 \cdots \mu_s$. Define a map $\ul{\bE}^{\lambda} \odot \ul{\bE}^{\mu} \to \ul{\bE}^{\lambda \odot \mu}$ on basis vectors by
\begin{displaymath}
e_{i_1,\ldots,i_r} \odot e_{j_1,\ldots,j_s} \mapsto e_{i_1,\ldots,i_r,i_r+j_1,\ldots,i_r+j_s}.
\end{displaymath}
It is clear that this is a bijection on basis vectors and compatible with the action of $\cI$. It thus defines the requisite isomorphism.
\end{proof}

\begin{corollary} \label{cor:gapprod}
Let $\lambda=\lambda_1 \cdots \lambda_r$ be a constraint word. Then we have
\begin{displaymath}
\ul{\bE}^{\lambda} \cong \ul{\bE}^{\lambda_1} \odot \cdots \odot \ul{\bE}^{\lambda_r}.
\end{displaymath}
The $i$th factor on the right is $\ul{\bB}^1$ if $\lambda_i=\blacktri$, or $\ul{\bA}^1$ if $\lambda_i=\whitetri$.
\end{corollary}

\begin{remark} \label{rmk:catzero}
The nature of the concatenation product is one piece of evidence that having graded $\cI$-modules start in degree~0 is the best choice: the unit object for $\odot$ is concentrated in degree~0. Thus, for example, $\uRep(\cI)_+$ is not a monoidal category under $\odot$.
\end{remark}

\subsection{The shift functor} \label{ss:shift}

Let $M$ be an $\cI$-module. We define $\Sigma(M)$ to be the pullback of $M$ under the homomorphism $\cI \to \cI$ defined by $\alpha_i \mapsto \alpha_{i+1}$. The vector space underlying $\Sigma(M)$ is just $M$; however, for clarity, for $x \in M$ we denote the corresponding element of $\Sigma(M)$ by $x^{\flat}$. The action on $\Sigma(M)$ can thus be described by the formula
\begin{displaymath}
\alpha_i \cdot x^{\flat} = (\alpha_{i+1} \cdot x)^{\flat}.
\end{displaymath}
It is clear that if $M$ is smooth then so is $\Sigma(M)$, and so $\Sigma$ defines an endofunctor of $\Rep(\cI)$. It is obvious that $\Sigma$ is exact and cocontinuous. It is also continuous: indeed, an element $x$ of an $\cI$-module $M$ is smooth if and only if $x^{\flat} \in \Sigma(M)$ is smooth, and so $\Sigma$ commutes with the operation $(-)^{\sm}$, and thus with smooth products since it obviously commutes with products.

We also define a version of the shift functor in the graded case. Let $M$ be a graded $\cI$-module. We define $\ul{\Sigma}(M)$ to be the following graded $\cI$-module. As a graded vector space, $\ul{\Sigma}(M)_n=M_{n+1}$, and as before, for $x \in M_{n+1}$ we denote by $x^{\flat}$ the corresponding element of $\Sigma(M)_n$. The $\cI$-module structure is defined just as before, i.e., $\alpha_i \cdot x^{\flat}=(\alpha_{i+1} \cdot x)^{\flat}$. The one thing to verify is that $\cI_{>n}$ acts trivially on $\ul{\Sigma}(M)_n$. Thus suppose $i>n$, and let $x^{\flat}$ be an element of $\ul{\Sigma}(M)_n$, with $x \in M_{n+1}$. Then $\alpha_i \cdot x^{\flat}=(\alpha_{i+1} \cdot x)^{\flat}=x^{\flat}$, since $\alpha_{i+1}$ acts trivially on $M_{n+1}$. We thus see that $\ul{\Sigma}$ is a well-defined endofunctor of $\uRep(\cI)$. It is clearly exact, continuous, and cocontinuous. The graded and ungraded shifts are compatible on pure modules, that is, if $M$ is a pure graded module then $\Phi(\ul{\Sigma}(M))=\Sigma(\Phi(M))$. If $M$ is a graded module concentrated in degree~0 then $\ul{\Sigma}(M)=0$, while $\Sigma(\Phi(M))=\Phi(M)$.

\begin{proposition} \label{prop:shiftinv}
Suppose $M$ is an $\cI$-module, and let $n \ge 1$. Then $M^{\cI_{\ge n}}=(\Sigma^{n-1} M)^{\cI}$ and $M_{\cI_{\ge n}}=(\Sigma^{n-1} M)_{\cI}$.
\end{proposition}

\begin{proof}
This follows immediately from the definitions.
\end{proof}

\begin{proposition} \label{prop:shiftmap}
Let $M$ be an $\cI$-module. Then the function $f \colon M \to \Sigma(M)$ defined by $f(x)=(\alpha_1 \cdot x)^{\flat}$ is $\cI$-linear. Moreover, if $M$ is graded then $f$ is homogeneous.
\end{proposition}

\begin{proof}
We have
\begin{displaymath}
f(\alpha_i x)=(\alpha_1\alpha_ix)^{\flat}=(\alpha_{i+1} \alpha_1 x)^{\flat}=\alpha_i(\alpha_1 x)^{\flat}=\alpha_i f(x),
\end{displaymath}
which proves the result.
\end{proof}

\begin{proposition} \label{prop:shiftcat}
Let $M$ be a graded $\cI$-module and let $N$ be an arbitrary $\cI$-module. Then we have a natural isomorphism
\begin{displaymath}
\Sigma(M \odot N) \cong \left[ \ul{\Sigma}(M_+) \odot N \right] \oplus \left[ M_0 \odot \Sigma(N) \right].
\end{displaymath}
If $N$ is graded, then this isomorphism is homogeneous.
\end{proposition}

\begin{proof}
Define
\begin{displaymath}
f \colon \Sigma(M_+) \odot N \to \Sigma(M \odot N), \qquad
f(x^{\flat} \odot y)=(x \odot y)^{\flat}
\end{displaymath}
and
\begin{displaymath}
g \colon M_0 \odot \Sigma(N) \to \Sigma(M \odot N), \qquad
g(x \odot y^{\flat})=(x \odot y)^{\flat}.
\end{displaymath}
Then $f \oplus g$ is an isomorphism of vector spaces: the source are target are both identified with $M \otimes N$, and under this identification $f \oplus g$ is simply the identity map. Moreover, if $N$ is graded then $f$ and $g$ are homogeneous. We verify that $f$ and $g$ are $\cI$-linear.

Suppose $x \in M_{i+1}$ and $y \in N$. Then
\begin{displaymath}
\alpha_k \cdot (x^{\flat} \odot y) = \begin{cases}
(\alpha_{k+1} x)^{\flat} \odot y & \text{if $1 \le k \le i$} \\
x^{\flat} \odot (\alpha_{k-i} y) & \text{if $i+1 \le k$} \end{cases}
\end{displaymath}
On the other hand,
\begin{displaymath}
\alpha_k \cdot (x \odot y)^{\flat}
=(\alpha_{k+1} \cdot (x \odot y))^{\flat}
=\begin{cases}
(\alpha_{k+1} x \odot y)^{\flat} & \text{if $1 \le k+1 \le i+1$} \\
(x \odot \alpha_{k-i} y)^{\flat} & \text{if $i+2 \le k+1$}
\end{cases}
\end{displaymath}
This proves that $f$ is $\cI$-linear. The proof for $g$ is easier, and left to the reader.
\end{proof}

\begin{proposition} \label{prop:shiftstd}
Let $\lambda=\lambda_1 \cdots \lambda_k$ be a constraint word with $k \ge 1$, and let $\mu=\lambda_2 \cdots \lambda_k$. Then
\begin{displaymath}
\ul{\Sigma}(\ul{\bE}^{\lambda}) = \begin{cases}
\ul{\bE}^{\lambda} \oplus \ul{\bE}^{\mu} & \text{if $\lambda_1=\whitetri$} \\
\ul{\bE}^{\mu} & \text{if $\lambda_1=\blacktri$}
\end{cases}
\end{displaymath}
The analogous result holds in the ungraded case as well.
\end{proposition}

\begin{proof}
It is clear that $\ul{\Sigma}(\bB^1)=\bB^0$ and that $\ul{\Sigma}(\bA^1)=\bB^0 \oplus \bA^1$. By Proposition~\ref{prop:stdcat}, we have
\begin{displaymath}
\ul{\bE}^{\lambda} = \begin{cases}
\ul{\bA}^1 \odot \ul{\bE}^{\mu} & \text{if $\lambda_1=\whitetri$} \\
\ul{\bB}^1 \odot \ul{\bE}^{\mu} & \text{if $\lambda_1=\blacktri$}
\end{cases}
\end{displaymath}
The computation of $\ul{\Sigma}(\ul{\bE}^{\lambda})$ thus follows from the previous proposition. For the ungraded case, simply note that $\Phi(\ul{\Sigma}(\ul{\bE}^{\lambda}))=\Sigma(\Phi(\ul{\bE}^{\lambda}))$ as $\ul{\bE}^{\lambda}$ is pure.
\end{proof}

\begin{example} \label{ex:shift}
The following examples follow immediately from the above propositions:
\begin{enumerate}
\item We have $\ul{\Sigma}(\ul{\bA}^n) = \ul{\bA}^n \oplus \ul{\bA}^{n-1}$ for $n \ge 1$.
\item We have $\ul{\Sigma}(\ul{\bB}^n)=\ul{\bB}^{n-1}$ for $n \ge 1$.
\item We have $\ul{\Sigma}(\bE^{ab})=\bE^{ab} \oplus \bE^b$. \qedhere.
\end{enumerate}
\end{example}

\subsection{The transpose functor} \label{ss:transpose}

Suppose $M$ is a graded $\cI$-module. We define a new graded $\cI$-module, called the {\bf transpose} of $M$, and denoted $M^{\dag}$, as follows. The underlying graded vector space of $M^{\dag}$ is simply that of $M$; for $x \in M$, we write $x^{\dag}$ when we regard it as an element of $M^{\dag}$. For $x \in M_n$, we define
\begin{displaymath}
\alpha_i \cdot x^{\dag} = \begin{cases}
(\alpha_{n+1-i} x)^{\dag} & \text{if $1 \le i \le n$} \\
x^{\dag} & \text{if $i>n$} \end{cases}
\end{displaymath}
We note that if $1 \le i \le n$ then $\alpha_i x^{\dag}$ has degree $n+1$, as required. We now verify that the above definition respects the fundamental relations, and thus defines an action of $\cI$. To this end, suppose that $x \in M_n$ and $i>j$. We must verify
\begin{displaymath}
\alpha_i \alpha_j x^{\dag} = \alpha_j \alpha_{i-1} x^{\dag}.
\end{displaymath}
If $j>n$ then $\alpha_i \alpha_j x^{\dag} = x^{\dag}$ and $\alpha_j \alpha_{i-1} x^{\dag} = x^{\dag}$; note that $i-1 \ge j > n$. Similarly, if $i>n+1$ then $\alpha_{i-1} x^{\dag} = x^{\dag}$ and $\alpha_i \alpha_j x^{\dag} = \alpha_j x^{\dag}$, since $\alpha_j x^{\dag}$ belongs to $M_n$ or $M_{n+1}$, and so the identity holds. Thus suppose that $1 \le j \le n$ and $j < i \le n+1$. Then $\alpha_i \alpha_j x^{\dag} = (\alpha_{n+2-i} \alpha_{n+1-j} x)^{\dag}$. On the other hand, $\alpha_j \alpha_{i-1} x^{\dag} = (\alpha_{n+2-j} \alpha_{n+2-i} x)^{\dag}$. Since $n+2-j>n+2-i$, we have $\alpha_{n+2-j} \alpha_{n+2-i} = \alpha_{n+2-i} \alpha_{n+1-j}$, and so the identity holds.

It is clear that transpose defines a covariant involution of the category $\uRep(\cI)$. We now examine how it interacts with the concatenation product.

\begin{proposition} \label{prop:revcat}
Let $M$ and $N$ be graded $\cI$-modules. Then we have a natural isomorphism
\begin{displaymath}
(M \odot N)^{\dag} \cong N^{\dag} \odot M^{\dag}.
\end{displaymath}
\end{proposition}

\begin{proof}
Let $\iota \colon (M \odot N)^{\dag} \to N^{\dag} \odot M^{\dag}$ be the canonical isomorphism of graded vector spaces: explicitly, $\iota((x \odot y)^{\dag})=y^{\dag} \otimes x^{\dag}$. We verify that $\iota$ is $\cI$-equivariant. Thus let $x \in M_i$ and $y \in N_j$ with $i+j=n$ be given, and let $k \in \bN$ be given. We must show that
\begin{displaymath}
\iota(\alpha_k \cdot (x \odot y)^{\dag}) = \alpha_k \cdot \iota((x \odot y)^{\dag})
\end{displaymath}
This clearly holds if $k>n$, as then $\alpha_k$ acts by the identity on each side. Thus suppose $k \le n$.

First suppose that $k \le j$. Since $j=n-i$, we have $n-k \ge i$. Thus
\begin{displaymath}
\alpha_k \cdot (x \odot y)^{\dag}
=(\alpha_{n+1-k} \cdot (x \odot y))^{\dag}
=(x \odot (\alpha_{n+1-k-i} y))^{\dag}
=(x \odot (\alpha_{j+1-k} y))^{\dag},
\end{displaymath}
and so
\begin{displaymath}
\iota(\alpha_k \cdot (x \odot y)^{\dag})=(\alpha_{j+1-k} y)^{\dag} \odot x^{\dag}.
\end{displaymath}
On the other hand,
\begin{displaymath}
\alpha_k \cdot \iota((x \odot y)^{\dag})
= \alpha_k \cdot (y^{\dag} \odot x^{\dag})
= (\alpha_k \cdot y^{\dag}) \odot x^{\dag}
=(\alpha_{j+1-k} y)^{\dag} \odot x^{\dag}.
\end{displaymath}
Thus the identity holds in this case.

Now suppose $k>j$, so that $n-k<i$. We thus have
\begin{displaymath}
\alpha_k \cdot (x \odot y)^{\dag}
= (\alpha_{n+1-k} \cdot (x \odot y))^{\dag}
= ((\alpha_{n+1-k} x) \odot y)^{\dag},
\end{displaymath}
and so
\begin{displaymath}
\iota(\alpha_k \cdot (x \odot y)^{\dag}) = y^{\dag} \odot (\alpha_{n+1-k} x)^{\dag}
\end{displaymath}
On the other hand,
\begin{align*}
\alpha_k \cdot \iota((x \odot y)^{\dag})
&=\alpha_k \cdot (y^{\dag} \odot x^{\dag})
=y^{\dag} \odot (\alpha_{k-j} \cdot x^{\dag}) \\
&=y^{\dag} \odot (\alpha_{i+1-k+j} x)^{\dag}
=y^{\dag} \odot (\alpha_{n+1-k} x)^{\dag}.
\end{align*}
Thus the identity holds in this case as well.
\end{proof}

For a constraint word $\lambda=\lambda_1 \cdots \lambda_r$, let $\lambda^{\dag}$ be the reversed word $\lambda_r \cdots \lambda_1$.

\begin{proposition} \label{prop:transstd}
For a constraint word $\lambda$, we have $(\ul{\bE}^{\lambda})^{\dag} \cong \ul{\bE}^{\lambda^{\dag}}$.
\end{proposition}

\begin{proof}
This follows from Proposition~\ref{prop:revcat} and Corollary~\ref{cor:gapprod}, together with the easily verified isomorphisms $(\ul{\bA}^1)^{\dag} \cong \ul{\bA}^1$ and $(\ul{\bB}^1)^{\dag} \cong \ul{\bB}^1$.
\end{proof}

\section{Finite length modules} \label{s:finlen}

\subsection{Classification of simples}

In \S \ref{ss:simple}, we introduced the simple modules $\bB^n$. We now prove that these exhaust the simple smooth $\cI$-modules.

\begin{proposition} \label{prop:simple}
Any simple smooth $\cI$-module is isomorphic to $\bB^n$ for some $n \in \bN$.
\end{proposition}

\begin{proof}
Let $M$ be a simple $\cI$-module. Let $\rho \colon \cI \to \End_{\bk}(M)$ be the homomorphism giving the action of $\cI$ on $M$. Since $M$ is smooth and non-zero, we cannot have $\rho(\alpha_r)=0$ for all $r$. Let $n \ge 1$ be minimal such that $\rho(\alpha_n) \ne 0$. Since $\alpha_n \cI_{\ge n} = \cI_{\ge n} \alpha_n$, we see that $\ker(\rho(\alpha_n))$ and $\im(\rho(\alpha_n))$ are $\cI_{\ge n}$-stable, and thus $\cI$-stable since the $\alpha_r$ with $r<n$ act by~0. Since $\rho(\alpha_n)$ is non-zero, we must have $\ker(\rho(\alpha_n))=0$ and $\im(\rho(\alpha_n))=M$ by simplicity. Thus $\alpha_n$ is a linear isomorphism.

We now claim that $\alpha_r=\alpha_n$ for all $r \ge n$. Suppose that we have proved this for $r$, and let us prove it for $r+1$. We have the relation $\alpha_{r+1} \alpha_r=\alpha_r^2$, and so $\rho(\alpha_{r+1}) \rho(\alpha_r) = \rho(\alpha_r)^2$. Since $\rho(\alpha_r)=\rho(\alpha_n)$ is invertible on $M$, we can cancel it, and so we find $\rho(\alpha_{r+1})=\rho(\alpha_r)=\rho(\alpha_n)$, as claimed.

It now follows that $\rho(\alpha_n)=\id$. Indeed, suppose $x \in M$ is non-zero. By smoothness, there exists $r$ such that $\rho(\alpha_r) x = x$. Since $\rho(\alpha_r)=0$ for $r<n$, we must have $r \ge n$. But then $\rho(\alpha_r) = \rho(\alpha_n)$, and so $\rho(\alpha_n) x = x$. Thus $\rho(\alpha_n)$ is the identity.

We have thus shown that
\begin{displaymath}
\rho(\alpha_r) = \begin{cases}
0 & \text{if $r<n$} \\
\id & \text{if $r \ge n$} \end{cases}
\end{displaymath}
Thus every linear subspace of $M$ is a submodule, and so by simplicity, $M$ must be one-dimensional. It is evidently isomorphic to $\bB^{n-1}$.
\end{proof}

We have an analogous result in the graded case:

\begin{proposition}
Any simple object of $\uRep(\cI)$ is isomorphic to $\ul{\bB}^n$ for a unique $n$.
\end{proposition}

\begin{proof}
Let $M$ be a simple graded $\cI$-module. Let $n$ be such that $M_n \ne 0$. By simplicity, we must have $M_{\ge n}=M$ and $M_{\ge n+1}=0$, since $M_{\ge n+1}$ is a proper submodule. Thus $M$ is concentrated in degree~$n$. Thus $\alpha_k$ with $1 \le k \le n$ act by~0 on $M$, while $\alpha_k$ with $k>n$ acts by the identity. It follows that every subspace of $M_n$ is an $\cI$-submodule of $M$, and so $M_n$ must be one-dimensional. Thus $M_n \cong \ul{\bB}^n$.
\end{proof}

\begin{remark} \label{rmk:simple}
Suppose that $\bk$ is algebraically closed. For $n \in \bN \cup \{\infty\}$ and $a \in \bk^{\times}$, let $L_{n,a}$ be the one-dimensional $\cI$-module where $\alpha_r$ acts by~0 for $1 \le r \le n$ and by $a$ for $r>n$. Arguing as in the proof of Propsoition~\ref{prop:simple}, one can show that every simple object of $\REP(\cI)$ is isomorphic to some $L_{n,a}$.
\end{remark}

\subsection{Some Ext computations}

We now compute the $\Ext^1$ groups between simple objects in our various representation categories. We use the notations $\EXT_{\cI}$, $\Ext_{\cI}$, and $\uExt_{\cI}$ for the $\Ext$ functor on the categories $\REP(\cI)$, $\Rep(\cI)$, and $\uRep(\cI)$ respectively.

\begin{proposition} \label{prop:ext}
For $n,m \in \bN$, we have
\begin{displaymath}
\dim \EXT^1_{\cI}(\bB^n, \bB^m) = 
\begin{cases}
1 & \text{if $m=n$} \\
n & \text{if $m=n+1$} \\
0 & \text{otherwise}
\end{cases}
\end{displaymath}
\end{proposition}

\begin{proof}
Consider an extension
\begin{displaymath}
0 \to \bB^m \to E \to \bB^n \to 0.
\end{displaymath}
Let $v, w$ be a basis for $E$ with $v \in \bB^m$, and let $A_r$ be the matrix for $\alpha_r$ on $E$ with respect to this basis. Let $\delta_{r>s}$ be~1 if $r>s$ and 0 otherwise. We have
\begin{displaymath}
A_r=\begin{pmatrix} \delta_{r>m} & c_r \\ 0 & \delta_{r>n} \end{pmatrix},
\end{displaymath}
where $c_r$ is some scalar. For any $r,s \in \bN$, we have
\begin{displaymath}
A_s A_r = \begin{pmatrix}
\delta_{r>m} \delta_{s>m} &
c_r \delta_{s>m} + c_s \delta_{r>n} \\
0 & \delta_{s>n} \delta_{r>n}
\end{pmatrix}.
\end{displaymath}
Suppose now that $r<s$, so that we have the identity
\begin{displaymath}
A_s A_r=A_rA_{s-1}.
\end{displaymath}
Equating the upper right coefficients gives the equation
\begin{equation} \label{eq:cr}
c_r \delta_{s>m}+c_s \delta_{r>n} = c_{s-1} \delta_{r>m} + c_r \delta_{s-1>n}.
\end{equation}
This holds for all $1 \le r <s$. The other entries provide no information. We now proceed by cases.

\vskip.5\baselineskip
\textit{\textbf{Case 1:} $m<n$.}
Suppose $r \le m$. Taking $s=m+1$ in the identity \eqref{eq:cr} gives $c_r=0$. Now suppose $m<r\le n$. Then \eqref{eq:cr} gives $c_r=c_{s-1}+c_r \delta_{s-1>n}$ for $r<s$; equivalently, we have $c_r=c_s+c_r \delta_{s>n}$ for $r \le s$. For $s \le n$ this yields $c_r=c_s$, while for $s>n$, we obtain $c_s=0$. We have thus shown:
\begin{displaymath}
c_r = \begin{cases}
p & \text{if $m < r \le n$} \\
0 & \text{if $r\le m$ or $r>n$}
\end{cases}
\end{displaymath}
for some $p \in \bk$. Making a change of basis (replace $w$ with $w-p v$), we can assume $p=0$, and thus $c_r=0$ for all $r$. Thus the extension is split.

\vskip.5\baselineskip
\textit{\textbf{Case 2:} $m>n+1$.}
Suppose $r \le n$. Taking $s=n+2$ in \eqref{eq:cr} gives $c_r=0$. Now suppose $n < r \le m$. Then \eqref{eq:cr} gives $c_r \delta_{s>m}+c_s=c_r$ for $r<s$. For $s \le m$ we find $c_r=c_s$, while for $s>m$, we find $c_s=0$. We have thus shown:
\begin{displaymath}
c_r = \begin{cases}
p & \text{if $n < r \le m$} \\
0 & \text{if $r \le n$ or $r>m$}
\end{cases}
\end{displaymath}
for some $p \in \bk$. Once again, making a change of basis (replace $w$ with $w-p v$) we can assume $p=0$, and thus $c_r=0$ for all $r$. Thus the extension is split.

\vskip.5\baselineskip
\textit{\textbf{Case 3:} $m=n$.}
Suppose $r\le n$. Taking $s=n+1$ in \eqref{eq:cr} gives $c_r=0$. Now suppose $n<r$. Then \eqref{eq:cr} gives $c_r+c_s = c_{s-1}+ c_r$ for $r<s$; in other words, $c_s=c_{s-1}$ holds for all $s>n+1$. We thus have
\begin{displaymath}
c_r = \begin{cases}
p & \text{if $r>n$} \\
0 & \text{if $r \le n$}
\end{cases}
\end{displaymath}
for some $p \in \bk$. In this case, no change of basis allows us to modify $p$. Furthermore, for any $p$ taking $c_r$ as above gives a well-defined extension. We thus see that the $\EXT$ group is 1-dimensional, as claimed.

\vskip.5\baselineskip
\textit{\textbf{Case 4:} $m=n+1$.}
Suppose $r=n+1$. Then \eqref{eq:cr} gives $c_n+c_s=c_n$, and thus $c_s=0$, for all $n+1<s$. It turns out this is all the information \eqref{eq:cr} yields. Thus $c_1, \ldots, c_{n+1}$ are arbitrary, and $c_r=0$ for $r \ge n+2$. We can make a unique change of basis to ensure $c_{n+1}=0$. The $n$ remaining parameters show that the $\EXT$ group has dimension $n$.
\end{proof}

\begin{corollary} \label{cor:ext1}
For $n,m \in \bN$, we have
\begin{displaymath}
\dim \Ext^1_{\cI}(\bB^n, \bB^m) = 
\begin{cases}
n & \text{if $m=n+1$} \\
0 & \text{otherwise}
\end{cases}
\end{displaymath}
\end{corollary}

\begin{proof}
This follows from the proposition and two simple observations: first, the unique non-trivial self-extension of $\bB^n$ (constructed in Case 3 ) is not smooth; and second, every extension of $\bB^{n+1}$ by $\bB^n$ (as described in Case 4) is smooth. (Note that, the extension $E$ considered in the proof of the proposition is smooth if and only if $c_r=0$ for $r \gg 0$.)
\end{proof}

\begin{remark}
In Corollary~\ref{cor:Lext}, we compute all $\Ext$ groups between simple objects.
\end{remark}

\begin{remark}
It is not difficult to directly compute the $\uExt^1$ groups between simple objects of $\uRep(\cI)$. However, we will not do this, since these groups can also be obtained by combining the Corollary~\ref{cor:ext1} with Theorem~\ref{thm:canongr} below.
\end{remark}

\begin{remark} \label{rmk:trivext}
The unique non-trivial self-extension of $\bB^0$ constructed in Case~3 of Proposition~\ref{prop:ext} corresponds to the homomorphism $\rho \colon \cI \to \rM_2(\bk)$ given by
\begin{displaymath}
\rho(\sigma) = \begin{pmatrix} 1 & \ell(\sigma) \\ 0 & 1 \end{pmatrix},
\end{displaymath}
where $\ell(\sigma)$ is the length of $\sigma$ (\S \ref{ss:inc}).
\end{remark}

\begin{remark}
Using the same analysis as in the proof of Proposition~\ref{prop:ext}, one can compute $\EXT^1_{\cI}(L_{n,a}, L_{m,b})$ for all $n$, $m$, $a$, and $b$. (See Remark~\ref{rmk:simple} for the definition of $L_{n,a}$.)
\end{remark}

\subsection{Invariants and coinvariants}

For an abelian category $\cA$, we let $\cA^{\rf}$ (resp.\ $\cA^{\lf}$) denote the full subcategory on finite length (resp.\ locally finite length) objects of $\cA$.

\begin{proposition}
The trivial representation $\bB^0$ is both projective and injective in $\Rep(\cI)^{\lf}$.
\end{proposition}

\begin{proof}
The groups $\Ext^1_{\cI}(\bB^n, \bB^0)$ and $\Ext^1_{\cI}(\bB^0, \bB^n)$ vanish for all $n \ge 0$ by Corollary~\ref{cor:ext1}. Since the $\bB^n$ account for all simple object of $\Rep(\cI)$ (Proposition~\ref{prop:simple}), a standard d\'evissage shows that $\Ext^1_{\cI}(M, \bB^0)$ and $\Ext^1_{\cI}(\bB^0, M)$ vanish for all finite length objects $M$ of $\Rep(\cI)$. Thus $\bB^0$ is projective and injective in $\Rep(\cI)^{\rf}$. The proposition now follows from general considerations. Injectivity in $\Rep(\cI)^{\lf}$ follows from a version of Baer's criterion (see Proposition~\ref{prop:baer}). For projectivity, if $M \to \bB^0$ is a surjection with $M$ locally finite then we can choose a finite length submodule $N$ of $M$ that surjects onto $\bB^0$, and a splitting of $N \to \bB^0$ gives a splitting for $M \to \bB^0$.
\end{proof}

\begin{remark}
We will show (Corollary~\ref{cor:trivinj}) that the trivial representation is in fact injective in the larger category $\Rep(\cI)$. However, it is clearly not projective in $\Rep(\cI)$, as the augmentation map $\bA^1 \to \bB^0$ is a non-split surjection.
\end{remark}

\begin{proposition} \label{prop:inv}
The functors $M \mapsto M^{\cI_{\ge n}}$ and $M \mapsto M_{\cI_{\ge n}}$ are exact on $\Rep(\cI)^{\lf}$. Moreover, the composition $M^{\cI_{\ge n}} \to M \to M_{\cI_{\ge n}}$ is an isomorphism for $M \in \Rep(\cI)^{\lf}$.
\end{proposition}

\begin{proof}
We have $M^{\cI}=\Hom_{\cI}(\bB^0, M)$, which is an exact functor of $M$ since $\bB^0$ is projective. Similarly, $(M_{\cI})^* = \Hom_{\cI}(M, \bB^0)$ is an exact functor of $M$ since $\bB^0$ is injective, which implies that $M_{\cI}$ is an exact functor of $M$. Since $M^{\cI} \to M$ is injective, so is the map on coinvariants. As $(M^{\cI})_{\cI}=M^{\cI}$, we thus see that $M^{\cI} \to M_{\cI}$ is injective. We similarly see that this map is surjective. We have thus proved the proposition for $n=1$. The statements for general $n$ follow from the identities  $M^{\cI_{\ge n}}=\Sigma^{n-1}(M)^{\cI}$ and $M_{\cI_{\ge n}}=\Sigma^{n-1}(M)_{\cI}$ (Proposition~\ref{prop:shiftinv}).
\end{proof}

\subsection{The canonical grading}

The most important consequence of our $\Ext$ calculations is perhaps the following result:

\begin{theorem} \label{thm:canongr}
Every locally finite smooth $\cI$-module admits a canonical grading. That is, the forgetful functor $\Phi \colon \uRep(\cI)^{\lf} \to \Rep(\cI)^{\lf}$ is an isomorphism of categories.
\end{theorem}

\begin{proof}
Let $M$ be a locally finite smooth $\cI$-module. For $n \ge 0$, let $I_n(M)=M^{\cI_{>n}}$ and let $K_n(M)$ be the kernel of the map $M \to M_{\cI_{>n}}$, i.e., $\fa_{n+1} M$. By Proposition~\ref{prop:inv}, we have $M=I_n(M) \oplus K_n(M)$. Put $G_n(M)=I_n(M) \cap K_{n-1}(M)$, using the convention $K_{-1}(M)=M$.

We claim that the canonical map $I_{n-1}(M) \oplus G_n(M) \to I_n(M)$ is an isomorphism, for all $n \ge 1$. Since $G_n(M) \subset K_{n-1}(M)$, we clearly have $G_n(M) \cap I_{n-1}(M)=0$, and so the map is injective. To prove surjectivity, suppose that $x \in I_n(M)$ is given. Since the map $I_{n-1}(M) \to M/K_{n-1}(M)$ is an isomorphism, there exists a unique element $y \in I_{n-1}(M)$ having the same image in $M/K_{n-1}(M)$ as $x$. Thus $z=x-y$ belongs to $K_{n-1}(M) \cap I_n(M)=G_n(M)$. As $x=y+z$, this proves surjectivity.

It now follows by induction that the map $f_n \colon \bigoplus_{k=0}^n G_k(M) \to I_n(M)$ is an isomorphism for any $n$. We claim that the map $f \colon \bigoplus_{k \ge 0} G_k(M) \to M$ is an isomorphism. Since the restriction of $f$ to $\bigoplus_{k=0}^n G_k(M)$ is injective for all $n$, it follows that $f$ is injective. To show surjectivity, suppose $x \in M$. Then, by smoothness, $x \in I_n(M)$ for some $n$. Thus $x \in \im(f_n) \subset \im(f)$, as required.

The isomorphism $\bigoplus_{k \ge 0} G_k(M) \to M$ endows the vector space $M$ with a grading. We claim that this grading turns $M$ into a graded $\cI$-module. Thus suppose $x \in G_n(M)$. Then $x \in I_n(M)$, so if $i>n$ then $\alpha_i x = x$, as required. Now suppose $1 \le i \le n$. We must show that $\alpha_i x \in G_{n+1}(M)=I_{n+1}(M) \cap K_n(M)$. If $j>n+1$ then, applying the fundamental relation, we find $\alpha_j \alpha_i x = \alpha_i \alpha_{j-1} x=\alpha_i x$, as $\alpha_{j-1} x = x$; thus $\alpha_i x \in I_{n+1}(M)$. Since $x \in K_{n-1}(M)$, we can write $x=\sum_{j \ge n} (\alpha_j-1) y_j$ for elements $y_j \in M$, all but finitely many of which vanish; this follows directly from the definition of coinvariants. Applying the fundamental relation, we find $\alpha_i x=\sum_{j \ge n} (\alpha_{j+1}-1) \alpha_i y_j$, which belongs to $K_n(M)$.

We thus see that $M$, equipped with the grading provided by the $G_n$'s, is a graded $\cI$-module. Denote this graded $\cI$-module by $G(M)$. If $f \colon M \to N$ is a morphism in $\Rep(\cI)^{\lf}$ then $f$ clearly carries $G_k(M)$ into $G_k(N)$, and thus induces a map $G(M) \to G(N)$ in $\uRep(\cI)^{\lf}$. We thus see that $G$ defines a functor
\begin{displaymath}
G \colon \Rep(\cI)^{\lf} \to \uRep(\cI)^{\lf}.
\end{displaymath}
Moreover, it is clear that $\Phi \circ G=\id$ (actual equality), since $G(M)$ has $M$ as its underlying vector space.

Suppose now that $N$ is a locally finite graded $\cI$-module, and put $M=\Phi(N)$. It follows from the definition of graded $\cI$-module that $I_n(M)=\sum_{0 \le k \le n} N_k$. We claim that $K_n(M)=\sum_{k>n} N_k$. We check each containment separately.

To prove $K_n(M) \subset \sum_{k>n} N_k$, it suffices to show that $(1-\alpha_i) x \in \sum_{k>n} N_k$ for any $x \in M$ and $i>n$, since $K_n(M)$ is spanned by such elements. Thus let $x$ and $i$ be given. Let $x=\sum_{j \ge 0} x_j$ be the decomposition of $x$ into its homogeneous pieces. Since $x_j \in N_j$, we have $\alpha_i x_j=x_j$ for $j \le n$. Thus $(1-\alpha_i) x = \sum_{j>n} (1-\alpha_i) x_j$, which belongs to $\sum_{k>n} N_k$.

We now prove $K_n(M) \supset \sum_{k>n} N_k$. Since $N$ is locally finite, it suffices to show that for every finite length submodule $N'$ of $N$ we have $N'_k \subset K_n(M)$ for $k>n$. We proceed by descending induction on $k$. For $k \gg 0$, the statement is clear since $N'_k=0$. Suppose now that $k>n$ and $N'_{k+1} \subset K_n(M)$, and let us show that $N'_k \subset K_n(M)$. Thus let $x \in N'_k$ be given. Then $x=(1-\alpha_k)x + \alpha_k x$. We have $(1-\alpha_k) x \in K_n(M)$ by definition, while $\alpha_k x \in N'_{k+1} \subset K_n(M)$ by the inductive hypothesis. We have thus verified that $K_n(M)=\sum_{k>n} N_k$.

Combining our descriptions of $I_n(M)$ and $K_n(M)$, we see that $G_n(M)=N_n$. We thus see that $G \circ \Phi=\id$ (again, actual equality). Thus $\Phi$ and $G$ provide mutually inverse isomorphisms between the categories $\Rep(\cI)^{\lf}$ and $\uRep(\cI)^{\lf}$.
\end{proof}

\begin{remark}
The forgetful functor does \emph{not} induce an equivalence between $\uRep(\cI)$ and $\Rep(\cI)$. Indeed, it is not even fully faithful, as there are no non-zero maps $\ul{\bA}^n \to \ul{\bA}^0$, but there is a non-zero map $\bA^n \to \bA^0$, namely the augmentation map.
\end{remark}

\section{Multiplicities} \label{s:mult}

\subsection{Definitions and simple results} \label{ss:mult}

For a smooth $\cI$-module $M$ and $n \in \bN$, we let $\mu_n(M) \in \bN \cup \{\infty\}$ be the multiplicity of the simple module $\bB^n$ in $M$, as defined in \S \ref{ss:catmult}. The main properties of this are summarized in the following proposition:

\begin{proposition} \label{prop:multprop}
We have the following (for smooth $\cI$-modules):
\begin{enumerate}
\item Given a short exact sequence
\begin{displaymath}
0 \to M_1 \to M_2 \to M_3 \to 0,
\end{displaymath}
we have $\mu_n(M_2)=\mu_n(M_1)+\mu_n(M_3)$.
\item If $N$ is a subquotient of $M$ then $\mu_n(N) \le \mu_n(M)$.
\item If $M=\bigcup_{i \in I} M_i$ (directed union) then $\mu_n(M)=\sup_{i \in I} \mu_n(M_i)$.
\item If $\mu_n(M)=0$ for all $n \in \bN$ then $M=0$.
\end{enumerate}
\end{proposition}

\begin{proof}
See \S \ref{ss:catmult}.
\end{proof}

For a graded $\cI$-module $M$, we let $\umu_n(M)$ be the multiplicity of $\ul{\bB}^n$ in $M$. We note that this is simply the dimension of $M_n$. We write $\mu_n(M)$ in place of $\mu_n(\Phi(M))$.

\begin{proposition} \label{prop:mu-umu}
Suppose that $M$ is a locally finite graded $\cI$-module. Then $\umu_n(M)=\mu_n(M)$ for all $n \ge 0$.
\end{proposition}

\begin{proof}
Using the properties of multiplicities, we can first reduce to the case where $M$ is finite length, and then to the case where it is simple, where the statement is obvious.
\end{proof}

\begin{remark}
The hypothesis in Proposition~\ref{prop:mu-umu} that $M$ be locally finite is necessary. For example, if $M=\ul{\bA}^1$ then $\mu_0(M)=1$ (Theorem~\ref{thm:multone} below) while $\umu_0(M)=\dim(M_0)=0$.
\end{remark}

The following simple bound will be useful in developing the theory of truncation functors in \S \ref{ss:trunc}. It will be greatly improved in Proposition~\ref{prop:multshift}.

\begin{proposition} \label{prop:mu-ineq}
Let $M$ be a smooth $\cI$-module and let $s \le r$. Then $\mu_s(M) \le \mu_0(\Sigma^r M)$.
\end{proposition}

\begin{proof}
Suppose that $\mu_s(M) \ge n$. Then, by definition, we can find a chain
\begin{displaymath}
F^1 \subset G^1 \subset \cdots \subset F^n \subset G^n
\end{displaymath}
of $\cI$-submodules of $M$ such that $G^i/F^i \cong \bB^s$ for each $1 \le i \le n$. Since $\Sigma$ is exact and $\Sigma^r(\bB^s)=\bB^0$ (Proposition~\ref{prop:shiftstd}, or direct observation), the chain
\begin{displaymath}
\Sigma^r(F^1) \subset \Sigma^r(G^1) \subset \cdots \subset \Sigma^r(F^n) \subset \Sigma^r(G^n)
\end{displaymath}
shows that $\mu_0(\Sigma^r M) \ge n$, and so the result follows.
\end{proof}

\subsection{The multiplicity one theorem} \label{ss:multone}

The purpose of this section is to prove the following important theorem:

\begin{theorem} \label{thm:multone}
The trivial representation has multiplicity one in any principal module; that is, we have $\mu_0(\bA^r)=1$ for all $r \in \bN$.
\end{theorem}

\begin{remark}
The augmentation map $\epsilon \colon \bA^r \to \bA^0=\bk$ shows that $\mu_0(\bA^r) \ge 1$. The content of the theorem is that $\ker(\epsilon)$ has no trivial subquotient.
\end{remark}

The rest of \S \ref{ss:multone} is devoted to the proof. We begin by giving a criterion for vanishing of $\mu_0$. Let $\epsilon \colon \bk[\cI] \to \bk$ be the augmentation map. We say that an element $x$ of an $\cI$-module is {\bf anti-trivial} if there exists $a \in \bk[\cI]$ with $\epsilon(a) \ne 0$ such that $ax=0$.

\begin{lemma} \label{lem:antitriv}
Let $M$ be an $\cI$-module. Then $\mu_0(M)=0$ if and only if every element of $M$ is anti-trivial.
\end{lemma}

\begin{proof}
Suppose every element of $M$ is anti-trivial. Also suppose, by way of contradiction, that $M$ has a trivial subquotient, that is, there is an $\cI$-submodule $N \subset M$ and a surjection $\lambda \colon N \to \bA^0$ of $\cI$-modules. Let $x \in N$ be such that $\lambda(x) \ne 0$. Since $x$ is anti-trivial, there exists $a \in \bk[\cI]$ with $\epsilon(a) \ne 0$ such that $ax=0$. But then $0=\lambda(ax)=\epsilon(a) \lambda(x) \ne 0$, a contradiction. Thus $M$ has no trivial subquotient, and so $\mu_0(M)=0$.

Conversely, suppose that $\mu_0(M)=0$ and let $x \in M$. Let $I \subset \bk[\cI]$ be the kernel of the augmentation map $\epsilon$ and let $N$ be the submodule of $M$ generated by $x$. Then $\cI$ acts trivially on $N/I N$, and so $N/IN=0$ since $M$ has no trivial subquotient. We thus find $N \subset IN$, and so $x \in Ix$. Writing $x=bx$ for some $b \in I$, we have $ax=0$ with $a=1-b$. As $\epsilon(a)=1$, this shows that $x$ is anti-trivial.
\end{proof}

\begin{lemma} \label{lem:mu0cat}
Let $K$ be a smooth $\cI$-module with $\mu_0(K)=0$, let $N$ be a pure graded $\cI$-module, and let $M=N \odot K$. Then $\mu_0(M)=0$.
\end{lemma}

\begin{proof}
Let $x$ be an element of $M$. We can write $x=\sum_{i=1}^n y_i \odot z_i$ where $y_i \in N$ is homogeneous of degree $i$ and $z_i \in K$. We define the {\bf degree} of $x$ to be the maximal $i$ with $z_i \ne 0$, or $-\infty$ if $x=0$. We claim that, if $x \ne 0$, there exists $a \in \bk[\cI]$ with $\epsilon(a) \ne 0$ and $\deg(ax)<\deg(x)$. Granted this, it follows that every element of $M$ is anti-trivial (repeatedly apply the claim), and so $\mu_0(M)=0$.

Let $x \ne 0$ be given of degree $n$, expressed as above. Since $\mu_0(K)=0$, every element of $K$ is anti-trivial, and so there exists $b \in \bk[\cI]$ with $\epsilon(b) \ne 0$ such that $bz_n=0$. Let $f_r \colon \bk[\cI] \to \bk[\cI_{>r}]$ be the ring isomorphism mapping $\alpha_i$ to $\alpha_{i+r}$, and put $a=f_n(b)$. Note that $\epsilon(a) \ne 0$. By definition of the $\cI$ action on the concatenation product, we have
\begin{displaymath}
ax = \sum_{i=1}^n y_i \odot f_{n-i}(b) z_i
\end{displaymath}
The $i=n$ term vanishses, as $f_0(b)=b$, and so $\deg(ax)<\deg(x)$. The claim is verified.
\end{proof}

\begin{lemma} \label{lem:multone-2}
Let $K$ be the kernel of the augmentation map $\epsilon \colon \bA^1 \to \bk$. Then $\mu_0(K)=0$.
\end{lemma}

\begin{proof}
For $x=\sum_{n \ge 1} c_n e_n$ in $\bA^1$, define the {\bf degree} of $x$ to be the maximal $n$ such that $c_n \ne 0$, or $-\infty$ if $x=0$. We claim that if $x \in K$ is non-zero then there exists $a \in \bk[\cI]$ with $\epsilon(a) \ne 0$ such that $\deg(ax)<\deg(x)$. Applying the claim repeatedly, we see that any element of $K$ is anti-trivial, and so $\mu_0(K)=0$ by Lemma~\ref{lem:antitriv}.

Let $x \in K$ non-zero be given of degree $n$. We may as well assume the coefficient of $e_n$ in $x$ is~1. Since $\epsilon(x-e_n)=-1$, it follows that $x-e_n$ is non-zero; let its degree be $n-k$, with $k>0$. Thus $x=e_n+ce_{n-k}+y$ for some non-zero scalar $c$ and some element $y \in \bA^1$ of degree $<n-k$. Let $\sigma=\alpha_{n+k-1} \cdots \alpha_n$ and $\tau=\alpha_{n-1} \cdots \alpha_{n-k}$. Then
\begin{displaymath}
\sigma x = e_{n+k}+ce_{n-k}+y, \qquad
\tau x = e_{n+k}+ce_n+y
\end{displaymath}
and so
\begin{displaymath}
(c+\sigma-\tau)x = c(c+1)e_{n-k}+cy.
\end{displaymath}
Thus, taking $a=c+\sigma-\tau$, we have $\epsilon(a)=c \ne 0$ and $\deg(ax) \le n-k < \deg(x)$. This verifies the claim, and completes the proof.
\end{proof}

\begin{proof}[Proof of Theorem~\ref{thm:multone}]
Taking the exact sequence
\begin{displaymath}
0 \to K \to \bA^1 \stackrel{\epsilon}{\to} \bA^0 \to 0
\end{displaymath}
and applying $\ul{\bA}^{r-1} \odot -$ yields the exact sequence
\begin{displaymath}
0 \to \ul{\bA}^{r-1} \odot K \to \bA^r \to \bA^{r-1} \to 0.
\end{displaymath}
By Lemmas~\ref{lem:mu0cat} and~\ref{lem:multone-2}, we have $\mu_0(\ul{\bA}^{r-1} \odot K)=0$. Thus $\mu_0(\bA^r)=\mu_0(\bA^{r-1})$. Since $\mu_0(\bA^0)=1$ (obvious), we find $\mu_0(\bA^r)=1$ for all $r$.
\end{proof}

\subsection{Truncations} \label{ss:trunc}

Let $M$ be a smooth $\cI$-module. For $r \in \bZ$, we define $\tau_{\ge r}(M)$ to be $\fa_r M$ if $r \ge 1$ and $M$ if $r \le 0$. Here $\fa_r$ is the ideal of $\bk[\cI]$ introduced in \S \ref{ss:inv}. By Proposition~\ref{prop:ideal}, we see that $\tau_{\ge r}(M)$ is an $\cI$-submodule of $M$. We also define $\tau^{<r}(M)$ to be $M/\tau_{\ge r}(M)$. Thus we have a short exact sequence
\begin{displaymath}
0 \to \tau_{\ge r}(M) \to M \to \tau^{<r}(M) \to 0
\end{displaymath}
of smooth $\cI$-modules. For $r \ge 1$, we have $\tau^{<r}(M)=M_{\cI_{\ge r}}$, essentially by definition of $\fa_r$, while for $r \le 0$ we have $\tau^{<r}(M)=0$. We call $\tau_{\ge r}$ and $\tau^{<r}$ the {\bf truncation functors}. The following proposition explains the name and notation of the truncation functors:

\begin{proposition} \label{prop:mutau}
Let $M$ be a smooth $\cI$-module. Then
\begin{displaymath}
\mu_s(\tau_{\ge r}(M)) = \begin{cases}
\mu_s(M) & \text{if $s \ge r$} \\
0 & \text{if $s<r$} \end{cases}
\qquad
\mu_s(\tau^{<r}(M)) = \begin{cases}
0 & \text{if $s \ge r$} \\
\mu_s(M) & \text{if $s<r$} \end{cases}
\end{displaymath}
\end{proposition}

\begin{proof}
We proceed in four steps.

\textit{\textbf{Step 1:} $\mu_0(\tau_{\ge 1}(M))=0$.}
We have an exact sequence
\begin{displaymath}
0 \to \tau_{\ge 1}(\bA^r) \to \bA^r \to \tau^{<1}(\bA^r) \to 0,
\end{displaymath}
and $\tau^{<1}(\bA^r) \cong \bB^0$ via the augmentation map. By Theorem~\ref{thm:multone}, we have $\mu_0(\bA^r)=1$. Since $\mu_0(\bB^0)=1$ as well, it follows from the additivity of $\mu_0$ that $\mu_0(\tau_{\ge 1}(\bA^r))=0$. Now, let $F \to M$ be a surjection with $F$ a sum of principal modules. It follows directly from the definition of $\tau_{\ge 1}$ that the map $\tau_{\ge 1}(F) \to \tau_{\ge 1}(M)$ is surjective. Since $\mu_0(\tau_{\ge 1}(F))=0$, it follows that $\mu_0(\tau_{\ge 1}(M))=0$ as well.

\textit{\textbf{Step 2:} $\mu_s(\tau_{\ge r}(M))=0$ for $s<r$.}
It follows directly from the definitions that $\tau_{\ge 1}(\Sigma^{r-1}(M))=\Sigma^{r-1}(\tau_{\ge r}(M))$. Thus
\begin{displaymath}
\mu_s(\tau_{\ge r}(M)) \le \mu_0(\Sigma^{r-1}(\tau_{\ge r}(M)) = \mu_0(\tau_{\ge 1}(\Sigma^{r-1}(M)) = 0,
\end{displaymath}
where the inequality comes from Proposition~\ref{prop:mu-ineq}, and the vanishing comes from Step~1.

\textit{\textbf{Step 3:} $\mu_s(\tau^{<r}(M))=0$ for $s \ge r$.}
This is clear: $\tau^{<r}(M)=M_{\cI_{\ge r-1}}$, and so each $\alpha_i$ with $i \ge r-1$ acts trivially on $\tau^{<r}(M)$. Thus $\tau^{<r}(M)$ cannot contain and $\bB^s$ with $s \ge r$ as a subquotient, since $\alpha_{r-1}$ acts by zero on $\bB^s$.

\textit{\textbf{Step 4:} the remaining identities.}
Since $\mu_s(M)=\mu_s(\tau_{\le r}(M))+\mu_s(\tau^{>r}(M))$, the remaining two identities follow from the two already established.
\end{proof}

\begin{proposition} \label{prop:taufin}
Let $M$ be a smooth $\cI$-module. If $M$ is finitely generated, then $\tau^{<r}(M)$ has finite length. In general, $\tau^{<r}(M)$ is locally of finite length.
\end{proposition}

\begin{proof}
Suppose $M$ is finitely generated. Write $M$ as a quotient of a finite sum $F$ of principal modules. We have $\tau^{<1}(\bA^r)=(\bA^r)_{\cI}=\bB^0$ for all $r$, which has finite length. Thus $\tau^{<1}(F)$ has finite length. Since $\tau^{<1}$ is right-exact, we see that $\tau^{<1}(M)$ is a quotient of $\tau^{<1}(F)$, and thus of finite length. Thus the result holds for $r=1$. In general, note that $\Sigma^{r-1}(\tau^{<r}(M))=\tau^{<1}(\Sigma^{r-1}(M))$, and so $\Sigma^{r-1}(\tau^{<r}(M))$ has finite length. This implies that $\tau^{<r}(M)$ has finite length, since finite length is equivalent to the underlying vector space being finite dimensional, and $\Sigma$ does not change the underlying vector space.

Now suppose $M$ is a general smooth $\cI$-module. Then $M$ is the filtered colimit of its finitely generated submodules $\{M_i\}$. Since $\tau^{<r}$ is cocontinuous, we see that $\tau^{<r}(M)$ is the colimit of the $\tau^{<r}(M_i)$, each of which is finite length by the previous paragraph. Thus $\tau^{<r}(M)$ is locally of finite length.
\end{proof}

\begin{proposition} \label{prop:mufin}
Let $M$ be a finitely generated smooth $\cI$-module. Then $\mu_r(M)$ is finite for all $r \in \bN$.
\end{proposition}

\begin{proof}
By Proposition~\ref{prop:mutau}, we have $\mu_r(M)=\mu_r(\tau^{\le r}(M))$. Since $\tau^{\le r}(M)$ has finite length by Proposition~\ref{prop:taufin}, the result follows.
\end{proof}

\begin{proposition} \label{prop:tau-exact}
The functors $\tau_{\ge r}$ and $\tau^{<r}$ are exact and cocontinuous, for all $r \in \bZ$.
\end{proposition}

\begin{proof}
The description of $\tau^{<r}$ in terms of coinvariants shows that it is cocontinuous. We now show it is exact. Consider an exact sequence
\begin{displaymath}
0 \to M_1 \to M_2 \to M_3 \to 0
\end{displaymath}
of finitely generated smooth $\cI$-modules. Applying $\tau^{<r}$, we obtain an exact sequence
\begin{displaymath}
0 \to K \to \tau^{<r}(M_1) \to \tau^{<r}(M_2) \to \tau^{<r}(M_3) \to 0
\end{displaymath}
for some finitely generated smooth $\cI$-module $K$. Since $\mu_s(\tau^{<r}(M_1))=0$ for $s \ge r$ (Proposition~\ref{prop:mutau}), we have $\mu_s(K)=0$ for $s \ge r$ as well. Suppose now that $s<r$. Then $\mu_s(\tau^{<r}(M_i))=\mu_s(M_i)$ (Proposition~\ref{prop:mutau}). Combining this with the first exact sequence, we find
\begin{displaymath}
\mu_s(\tau^{<r}(M_2))=\mu_s(\tau^{<r}(M_1))+\mu_s(\tau^{<r}(M_3)).
\end{displaymath}
From the second exact sequence, we find
\begin{displaymath}
\mu_s(\tau^{<r}(M_2))+\mu_s(K)=\mu_s(\tau^{<r}(M_1))+\mu_s(\tau^{<r}(M_3)).
\end{displaymath}
Since all the quantities appearing in these equations are finite by Proposition~\ref{prop:mufin}, it follows that $\mu_s(K)=0$. Thus $K=0$ by Proposition~\ref{prop:multprop}(d), and so $\tau^{<r}(M_1) \to \tau^{<r}(M_2)$ is injective.

Suppose now that $M \to N$ is an injection of smooth $\cI$-modules. We must show that $\tau^{<r}(M) \to \tau^{<r}(N)$ is injective. By the previous paragraph, this holds if $M$ and $N$ are finitely generated. Since $\tau^{<r}$ is cocontinuous, we conclude that it holds in general. Precisely, we can realize $M \to N$ as a filtered colimits of injections $M_i \to N_i$ of finitely generated smooth $\cI$-modules. Since each $\tau^{<r}(M_i) \to \tau^{<r}(N_i)$ is injective, so is the colimit, as filtered colimits are exact.

We have thus shown that $\tau^{<r}$ is exact and cocontinuous. It now follows from Proposition~\ref{prop:funseq} that $\tau_{\ge r}$ is exact and cocontinuous.
\end{proof}

\begin{proposition} \label{prop:sep}
Let $M$ be a smooth $\cI$-module. Then $\bigcap_{r \ge 0} \tau_{\ge r}(M)=0$. In other words, the filtration $\{\fa_r M\}_{r \ge 1}$ of $M$ is separated.
\end{proposition}

\begin{proof}
Let $N=\bigcap_{r \ge 0} \tau_{\ge r}(M)$. Let $r \in \bN$ be given. As $N \subset \tau_{>r}(M)$, we have $\mu_r(M)=0$ by Proposition~\ref{prop:mutau}. As this holds for all $r$, we have $N=0$ by Proposition~\ref{prop:multprop}(d).
\end{proof}

\begin{proposition}
Suppose $M$ is a smooth $\cI$-module. Then $\tau_{\ge r}(\tau_{\ge s}(M))=\tau_{\ge \max(r,s)}(M)$. Similarly, $\tau^{<r}(\tau^{<s}(M))=\tau^{<\min(r,s)}(M)$.
\end{proposition}

\begin{proof}
We have $\tau_{\ge r}(\tau_{\ge s}(M)) \subset \tau_{\ge \max(r,s)}(M)$. By Proposition~\ref{prop:mutau}, the two modules have the same $\mu_i$ for all $i \ge 0$. Thus they are equal (the quotient has $\mu_i=0$ for all $i$, and thus vanishes). The proof of the second statement is similar.
\end{proof}

\begin{remark}
As $\tau_{\ge r}(M)=\fa_r M$ for $r \ge 1$, the above proposition might suggest that $\fa_r \fa_s=\fa_{\max(r,s)}$. In fact, this is not the case: indeed, if $N$ is the non-trivial self-extension of the trivial representation (see Remark~\ref{rmk:trivext}) then $\fa_1^2 N=0$ but $\fa_1 N \ne 0$. Thus $\fa_1^2 \ne \fa_1$. The above proposition implies, however, that $\fa_1^2 M=\fa_1 M$ holds for all \emph{smooth} $\cI$-modules $M$.
\end{remark}

\subsection{Some applications}

We now give some applications of the truncation functors.

\begin{proposition} \label{prop:lfinj}
Any injective object of $\Rep(\cI)^{\lf}$ remains injective in $\Rep(\cI)$.
\end{proposition}

\begin{proof}
Let $I$ be an injective object of $\Rep(\cI)^{\lf}$. Let $M \subset N$ be finitely generated smooth $\cI$-modules, and let $f \colon M \to I$ be an $\cI$-linear map. The image of $f$ has finite length. Thus we can find $r \in \bN$ such that $\mu_s(\im(f))=0$ for all $s \ge r$. Since $\mu_s(\tau_{\ge r}(M))=0$ for all $s<r$ (Proposition~\ref{prop:mutau}), we see that $\tau_{\ge r}(M) \subset \ker(f)$. Thus $f$ factors through $\tau^{<r}(M)$. Since $\tau^{<r}$ is exact (Proposition~\ref{prop:tau-exact}), the map $\tau^{<r}(M) \to \tau^{<r}(N)$ is injective. As both have finite length (Proposition~\ref{prop:taufin}), we see that the map $\tau^{<r}(M) \to I$ induced by $f$ extends to a map $\tau^{<r}(N) \to I$. The composition $N \to \tau^{<r}(N) \to I$ thus extends $f$ to $N$. Baer's criterion (Proposition~\ref{prop:baer}) now shows that $I$ is injective.
\end{proof}

\begin{corollary} \label{cor:trivinj}
The trivial representation $\bB^0$ is injective in $\Rep(\cI)$.
\end{corollary}

We now establish some results that will be needed later.

\begin{proposition} \label{prop:multshift}
Let $M$ be a smooth $\cI$-module. Then
\begin{displaymath}
\mu_r(\Sigma^n M) = \begin{cases}
\mu_0(M) + \cdots + \mu_n(M) & \text{if $r=0$} \\
\mu_{r+n}(M) & \text{if $r>0$} \end{cases}
\end{displaymath}
\end{proposition}

\begin{proof}
The result is clear if $M$ is simple. It thus follows for $M$ locally finite by the formal properties of multiplicities. We now treat the general case. Consider the short exact sequence
\begin{displaymath}
0 \to \Sigma^n(\tau_{>r+n}(M)) \to \Sigma^n(M) \to \Sigma^n(\tau^{\le r+n}(M)) \to 0.
\end{displaymath}
We have
\begin{displaymath}
\tau^{\le r}(\Sigma^n(\tau_{>r+n}(M)))=\Sigma^n(\tau^{\le r+n}(\tau_{>r+n}(M)))=0,
\end{displaymath}
and so
\begin{displaymath}
\mu_r(\Sigma^n(M)) = \mu_r(\Sigma^n(\tau^{\le r+n}(M))).
\end{displaymath}
Since $\tau^{\le r+n}(M)$ is locally finite, and $\mu_s(\tau^{\le r+n}(M))=\mu_s(M)$ for $s \le r+n$, the result follows.
\end{proof}

\begin{proposition} \label{prop:mu0cat}
Let $M$ be a smooth $\cI$-module and let $N$ be a graded $\cI$-module. Then we have a natural isomorphism $(N \odot M)_{\cI} \cong N_{\cI} \otimes M_{\cI}$. In particular, we have $\mu_0(N \odot M)=\mu_0(N) \mu_0(M)$.
\end{proposition}

\begin{proof}
Let $x \in N_n$ and let $y \in M$. We have
\begin{displaymath}
(\alpha_i-1) (x \odot y) = \begin{cases}
((\alpha_i-1)x) \odot y & \text{if $i \le n$} \\
x \odot ((\alpha_{i-n}-1) y) & \text{if $i>n$} \end{cases}
\end{displaymath}
We thus find
\begin{displaymath}
\fa_1 (N \odot M) = (\fa_1 N) \odot M + N \odot (\fa_1 M),
\end{displaymath}
and so the result follows.
\end{proof}

\begin{proposition} \label{prop:multA1cat}
Let $M$ be a smooth $\cI$-module. Then
\begin{displaymath}
\mu_r(\ul{\bA}^1 \odot M) = \begin{cases}
\mu_0(M) & \text{if $r=0$} \\
\mu_0(M) + \cdots + \mu_{r-1}(M) & \text{if $r>0$} \end{cases}
\end{displaymath}
\end{proposition}

\begin{proof}
For $r=0$ this follows from Propsosition~\ref{prop:mu0cat}. Now suppose $r>0$. We have $\ul{\Sigma}(\ul{\bA}^1)=\ul{\bA}^1 \oplus \ul{\bB}^0$. Thus, by Proposition~\ref{prop:shiftcat}, we have
\begin{displaymath}
\Sigma(\ul{\bA}^1 \odot M)=(\ul{\bA}^1 \odot M) \oplus M
\end{displaymath}
It follows that we have
\begin{displaymath}
\Sigma^r(\ul{\bA}^1 \odot M)=(\ul{\bA}^1 \odot M) \oplus M \oplus \cdots \oplus \Sigma^{r-1}(M).
\end{displaymath}
We thus have
\begin{align*}
\mu_r(\ul{\bA}^1 \odot M)
&= \mu_0(\Sigma^r(\ul{\bA}^1 \odot M))-\mu_0(\Sigma^{r-1}(\ul{\bA}^1 \odot M)) \\
&= \mu_0(\Sigma^{r-1}(M)) \\
&= \mu_0(M) + \cdots + \mu_{r-1}(M),
\end{align*}
where in the first and last steps we used Proposition~\ref{prop:multshift}. This completes the proof.
\end{proof}

\section{$\Rep(\cI)$ as a Serre quotient of $\uRep(\cI)$} \label{s:psigamma}

\subsection{The functors $\Psi$ and $\Gamma$} \label{ss:psigamma}

For a graded $\cI$-module $M$, define $\Psi(M)$ to be the direct limit of the directed system
\begin{displaymath}
\xymatrix{
M_1 \ar[r]^{\alpha_1} & M_2 \ar[r]^{\alpha_2} & M_3 \ar[r]^{\alpha_3} \ar[r] & \cdots }
\end{displaymath}
Let $i \ge 1$ be given. Then we have a commutative diagram
\begin{displaymath}
\xymatrix{
M_i \ar[r]^{\alpha_i} \ar[d]^{\alpha_i} & M_{i+1} \ar[r]^{\alpha_{i+1}} \ar[d]^{\alpha_i} & M_{i+2} \ar[r]^{\alpha_{i+2}} \ar[d]^{\alpha_i} & \cdots \\
M_{i+1} \ar[r]^{\alpha_{i+1}} & M_{i+2} \ar[r]^{\alpha_{i+2}} \ar[r] & M_{i+3} \ar[r]^{\alpha_{i+3}} & \cdots }
\end{displaymath}
Commutativity of the squares follows from the fundamental relation (Proposition~\ref{prop:alpharel}). Both rows have direct limit equal to $\Psi(M)$. We thus get an endomorphism of $\Psi(M)$ that we denote by $\alpha_i$. Concretely, an element $x$ of $\Psi(M)$ is represented by an element $\wt{x} \in M_n$ for some $n \ge i$, and $\alpha_i x$ is represented by $\alpha_i \wt{x} \in M_{n+1}$. From this description of $\alpha_i$, it is clear that $\alpha_j \alpha_i = \alpha_i \alpha_{j-1}$ holds on $\Psi(M)$ for $j>i$. Thus $\Psi(M)$ is naturally an $\cI$-module. Moreover, it is obviously smooth: indeed, if $x \in \Psi(M)$ is represented by $\wt{x} \in M_n$ then $\alpha_n x$ is represented by $\alpha_n \wt{x} \in M_{n+1}$, but this also represents $x$, and so $\alpha_n x = x$. We have thus defined a functor
\begin{displaymath}
\Psi \colon \uRep(\cI) \to \Rep(\cI).
\end{displaymath}
Basic properties of direct limits imply that $\Psi$ is exact and cocontinuous.

We now define a functor $\Gamma$ in the opposite direction. Let $M$ be a smooth $\cI$-module. Define a graded vector space $\Gamma(M)$ by $\Gamma(M)_0=0$ and $\Gamma(M)_n = M^{\cI_{\ge n}}$ for $n \ge 1$. For $1 \le i \le n$, we define $\alpha_i \colon \Gamma(M)_n \to \Gamma(M)_{n+1}$ to be function induced by the restriction of $\alpha_i$ to $\Gamma(M)_n$; we note that it does indeed take values in $\Gamma(M)_{n+1}$, and that $\alpha_n$ is the canonical inclusion. For $i>n$, we define $\alpha_i$ on $\Gamma(M)_n$ to be the identity. We leave to the reader the easy verification that these maps satisfy the appropriate relations, and thus give $\Gamma(M)$ the structure of a graded $\cI$-module. The definition of $\Gamma$ on morphisms is clear, and so we do indeed have a functor
\begin{displaymath}
\Gamma \colon \Rep(\cI) \to \uRep(\cI).
\end{displaymath}
It is clear that $\Gamma$ is left exact; in fact, it is not difficult to verify that it is continuous.

We now explain the relation between the functors $\Psi$ and $\Gamma$. For this, we introduce a piece of terminology. Let $M$ be a graded $\cI$-module. We say that $x \in M_n$, with $n>0$, is {\bf torsion} if $\alpha_n^k x = \alpha_{n+k} \cdots \alpha_n x$ vanishes for some $k \ge 0$; we also declare all degree~0 elements to be torsion. We say that $M$ is {\bf torsion} if all of its homogeneous elements are. We let $\uRep(\cI)^{\tors}$ be the category of torsion modules; it is a localizing subcategory of $\uRep(\cI)$.

\begin{proposition} \label{prop:phigamma}
We have the following:
\begin{enumerate}
\item The functors $(\Psi, \Gamma)$ are naturally an adjoint pair.
\item The co-unit $\Psi \Gamma \to \id$ is an isomorphism.
\item The kernel of $\Psi$ is the category $\uRep(\cI)^{\tors}$ of torsion modules.
\item $\Psi$ induces an equivalence $\uRep(\cI)/\uRep(\cI)^{\tors} \to \Rep(\cI)$.
\end{enumerate}
\end{proposition}

\begin{proof}
(a) Let $M$ be a graded $\cI$-module. There is a natural linear map $M_n \to \Psi(M)$ which lands in the $\cI_{\ge n}$-invariants of $\Psi(M)$: indeed, if $x$ is the image of $\wt{x} \in M_n$ then $\alpha_n x$ is represented by $\alpha_n \wt{x}$, which represents $x$. Thus we have a map $M_n \to \Psi(M)^{\cI_{\ge n}}=\Gamma(\Psi(M))_n$, and so a map of graded vector spaces $M \to \Gamma(\Psi(M))$. One easily verifies that it is $\cI$-equivariant. This is the unit of the adjunction.

Now suppose that $M$ is a smooth $\cI$-module. Then we have an inclusion $\Gamma(M)_n=M^{\cI_{\ge n}} \to M$. Moreover, the diagram
\begin{displaymath}
\xymatrix{ \Gamma(M)_n \ar[r]^{\alpha_n} \ar[rd] & \Gamma(M)_{n+1} \ar[d] \\ & M }
\end{displaymath}
commutes, as $\alpha_n \colon \Gamma(M)_n \to \Gamma(M)_{n+1}$ is the canonical inclusion. Taking direct limits, we get a map $\Psi(\Gamma(M)) \to M$, which is the co-unit of the adjunction. We leave to the reader the verification of the compaitiblities of the unit and co-unit.

(b) Let $M$ be a smooth $\cI$-module. Since each map $\Gamma(M)_n \to M$ is injective and every element of $M$ belongs to the image of one of these maps, the map on the direct limit $\Psi(\Gamma(M)) \to M$ is an isomorphism.

(c) Suppose that $M$ is a graded $\cI$-module and $\Psi(M)=0$. This means that every element $x \in M_n$, for any $n$, represents~0 in the direct limit; by definition, this means that $x$ is killed after applying some number of the transition maps, which exactly says that $x$ is torsion.

(d) This follows from (a)--(c) and \cite[Prop.~III.5]{gabriel}.
\end{proof}

\begin{corollary} \label{cor:gaminj}
The functor $\Gamma$ takes injective objects of $\Rep(\cI)$ to injective objects of $\uRep(\cI)$.
\end{corollary}

\subsection{Computing $\Psi$ and $\Gamma$}

We now prove some general results about $\Psi$ and $\Gamma$, and use them to compute what these functors do to standard modules.

\begin{proposition} \label{prop:psicat}
Let $M$ and $N$ be graded $\cI$-modules. We then have a canonical isomorphism
\begin{displaymath}
\Psi(M \odot N) \cong \left[ M \odot \Psi(N_+) \right] \oplus \left[ \Psi(M) \otimes N_0 \right].
\end{displaymath}
\end{proposition}

\begin{proof}
It suffices to treat separately the case where $N$ is pure and that where it is concentrated in degree~0. The latter case is trivial, so we now treat the first. Thus suppose $N$ is pure. By definition, $\Psi(M \odot N)$ is the direct limit of the system
\begin{displaymath}
\xymatrix{
(M \odot N)_1 \ar[r]^{\alpha_1} & (M \odot N)_2 \ar[r]^{\alpha_2} \ar[r] & \cdots }
\end{displaymath}
Consider a pure tensor $x \odot y$ in $(M \odot N)_n$, where $x \in M_i$ and $y \in N_j$. Since $N_0=0$, this element vanishes if $j=0$. Thus suppose $j>0$. Then $i<n$, and so $\alpha_n \cdot (x \odot y) = x \odot (\alpha_{n-i} y)$. Thus the transition maps in the above system are the identity on the $M$ piece. More precisely, if $M_i \odot N$ denotes the evident subspace of $M \odot N$, then $(M_i \odot N)_n$ is mapped into $(M_i \odot N)_{n+1}$ under the transition map. Furthermore, the $M_i \odot N$ piece of the system is
\begin{displaymath}
\xymatrix@C=2cm{
(M_i \odot N)_{n+1} \ar[r]^{\alpha_{n+1}} \ar@{=}[d] & (M_i \odot N)_{n+2} \ar[r]^-{\alpha_{n+2}} \ar[r] \ar@{=}[d] & \cdots \\
M_i \otimes N_{n+1-i} \ar[r]^{\id \otimes \alpha_{n+1-i}} & M_i \otimes N_{n+2-i} \ar[r]^-{\id \otimes \alpha_n+2+i} \ar[r] & \cdots }
\end{displaymath}
The direct limit of the bottom system is simply $M_i \otimes \Psi(N)$. Since $M$ is the sum of the $M_i$'s, we see that $\Psi(M \odot N)$ is the sum of the spaces $M_i \otimes \Psi(N)$, which is the space underlying $M \odot \Psi(N)$. We thus have a natural linear isomorphism $\Psi(M \odot N) \cong M \odot \Psi(N)$. It is easily seen to be compatible with the $\cI$-actions.
\end{proof}

\begin{proposition} \label{prop:psistd}
Let $\lambda=\lambda_1 \cdots \lambda_r$ be a constraint word with $r \ge 1$, and let $\mu=\lambda_1 \cdots \lambda_{r-1}$. Then
\begin{displaymath}
\Psi(\ul{\bE}^{\lambda}) \cong \begin{cases}
\bE^{\mu} & \text{if $\lambda_r=\whitetri$} \\
0 & \text{if $\lambda_r=\blacktri$}
\end{cases}
\end{displaymath}
\end{proposition}

\begin{proof}
By Proposition~\ref{prop:stdcat}, we have $\ul{\bE}^{\lambda}=\ul{\bE}^{\mu} \odot \ul{\bE}^{\lambda_r}$. Proposition~\ref{prop:psicat} thus yields $\Psi(\ul{\bE}^{\lambda}) \cong \ul{\bE}^{\mu} \odot \Psi(\ul{\bE}^{\lambda_r})$. If $\lambda_r=\blacktri$ then $\ul{\bE}^{\lambda_r}=\ul{\bB}^1$, and it is clear that $\Psi(\ul{\bB}^1)=0$; thus $\Psi(\ul{\bE}^{\lambda})=0$. If $\lambda_r=\whitetri$ then $\ul{\bE}^{\lambda_r}=\ul{\bA}^1$, and it is clear that $\Psi(\ul{\bA}^1)=\bA^0$; thus $\Psi(\ul{\bE}^{\lambda}) \cong \bE^{\mu}$.
\end{proof}

\begin{remark} \label{rmk:psistd}
In the case $\lambda_r=\whitetri$ above, the isomorphism $\Psi(\ul{\bE}^{\lambda}) \cong \bE^{\mu}$ admits a simple description: it maps the element of $\Psi(\ul{\bE}^{\lambda})$ represented by $e_{i_1,\ldots,i_r}$ to $e_{i_1,\ldots,i_{r-1}}$.
\end{remark}

\begin{remark}
Recall (Proposition~\ref{prop:shiftmap}) that for a graded $\cI$-module $M$ there is a canonical morphism $M \to \ul{\Sigma}(M)$. Define the {\bf reduced shift} of $M$, denoted $\ul{\Sigma}^{\red}(M)$, to be the cokernel of this map. Using Proposition~\ref{prop:shiftstd}, one can show that $\Sigma^{\red}(\ul{\bA}^r) \cong \ul{\bA}^{r-1}$. We thus have an isomorphism $\Psi(M) \cong \Phi(\ul{\Sigma}^{\red}(M))$ whenever $M$ is a principal module. In fact, the functors $\Psi$ and $\Phi \circ \ul{\Sigma}^{\red}$ are \emph{not} isomorphic: indeed, $\Psi$ kills $\ul{\bB}^n$, while $\Phi \circ \ul{\Sigma}^{\red}$ does not (for $n \ge 1$). There is a precise relationship between $\Psi$, $\Phi$, and the shift functor (or, more precisely, its left adjoint), however: see \S \ref{ss:phiright}.
\end{remark}

\begin{proposition} \label{prop:gammagraded}
Let $M$ be a graded $\cI$-module. Then we have a natural isomorphism $\Gamma(\Phi(M))=M \odot \ul{\bA}^1$.
\end{proposition}

\begin{proof}
We have
\begin{displaymath}
\Gamma(\Phi(M))_n = \Phi(M)^{\cI_{\ge n}} = \bigoplus_{k=0}^{n-1} M_k,
\end{displaymath}
where the second equality follows directly from the definition of graded $\cI$-module. On the other hand, we have
\begin{displaymath}
(M \odot \bA^1)_n = \bigoplus_{k=0}^{n-1} M_k \odot e_{n-k}.
\end{displaymath}
We thus have an isomorphism $\Gamma(\Phi(M))_n \cong (M \odot \bA^1)_n$ taking $x \in M_k$ to $x \odot e_{n-k}$. This isomorphism is easily verified to be compatible with the $\cI$ actions.
\end{proof}

\begin{proposition} \label{prop:gammastd}
Let $\lambda=\lambda_1 \cdots \lambda_r$ be a constraint word, and let $\mu=\lambda_1 \cdots \lambda_r \cdot \whitetri $. Then we have a natural isomorphism $\Gamma(\bE^{\lambda}) \cong \ul{\bE}^{\mu}$.
\end{proposition}

\begin{proof}
We have
\begin{displaymath}
\Gamma(\bE^{\lambda}) \cong \Gamma(\Phi(\ul{\bE}^{\lambda})) \cong \ul{\bE}^{\lambda} \odot \ul{\bA}^1 \cong \ul{\bA}^{\mu},
\end{displaymath}
where in the second step we used Proposition~\ref{prop:gammagraded} and in the third Proposition~\ref{prop:stdcat}.
\end{proof}

\begin{corollary} \label{prop:gammaprin}
We have $\Gamma(\bA^r)=\ul{\bA}^{r+1}$ for all $r \in \bN$.
\end{corollary}

\begin{remark} \label{rmk:gammastd}
The isomorphism in Proposition~\ref{prop:gammastd} admits a simple direct description: it takes $e_{i_1,\ldots,i_r} \in \Gamma(\bE^{\lambda})_n=(\bE^{\lambda}_n)^{\cI_{\ge n}}$ to $e_{i_1,\ldots,i_r,n} \in \ul{\bE}^{\mu}_n$.
\end{remark}

%

\subsection{Some $\Gamma$-acyclic modules}

The following result gives an important source of $\Gamma$-acyclic objects.

\begin{proposition} \label{prop:Rgamma}
Let $M$ be a smooth $\cI$-module that admits a grading (i.e., $M$ belongs to the essential image of $\Phi$). Then $M$ is $\Gamma$-acyclic, that is, $\rR^i \Gamma(M)=0$ for $i>0$.
\end{proposition}

\begin{proof}
First suppose that $M$ is locally finite; any such module admits a grading, namely the canonical grading (Theorem~\ref{thm:canongr}). Let $M \to J^{\bullet}$ be an injective resolution of $M$ in $\Rep(\cI)^{\lf}$. Each $J^i$ is injective in $\Rep(\cI)$ (Proposition~\ref{prop:lfinj}), and so $\Gamma(J^{\bullet})$ computes $\rR \Gamma(M)$. But $\Gamma$ is exact on $\Rep(\cI)^{\lf}$, since formation of $\cI_{\ge n}$ invariants is exact here (Proposition~\ref{prop:inv}), and so $\Gamma(J^{\bullet})$ has no higher cohomology.

We now treat the general case. Thus suppose that $M=\Phi(N)$ for some $N \in \uRep(\cI)$. Recall that $N_{>n}$ is the homogeneous submodule $\bigoplus_{k>n} N_k$ of $N$, and $N_{\le n}=N/N_{>n}$ is the quotient. Since inverse limits in $\uRep(\cI)$ are computed degreewise, we have $N=\varprojlim N_{\le n}$. Since $\Phi$ is continuous (Proposition~\ref{prop:Phi-cont}), we have $M=\smlim M_n$, where $M_n=\Phi(N_{\le n})$. Clearly, $M_n$ is locally finite.

Now, by general theory (Proposition~\ref{prop:dercont}), $\rR \Gamma$ commutes with derived limits. We thus find
\begin{equation} \label{eq:Rgamma}
\rR \Gamma(\rR \smlim M_n) = \rR \varprojlim \rR \Gamma(M_n).
\end{equation}
By the first paragraph, $\rR \Gamma(M_n)=\Gamma(M_n)$, since each $M_n$ is locally finite. Moreover, since $M_{n+1} \to M_n$ is surjective and $\Gamma$ is exact on locally finite modules, we see that $\Gamma(M_{n+1}) \to \Gamma(M_n)$ is surjective. Thus for each degree $d$, the inverse system $\{\Gamma(M_n)_d\}_{n \ge 0}$ of vector spaces satisfies the Mittag--Leffler condition, and so its higher derived limits vanish. We therefore see that the right side of \eqref{eq:Rgamma} is simply $\varprojlim \Gamma(M_n)$.

Now, the higher derived functors of $\rR \Gamma$ are torsion, and thus killed by $\Psi$. Therefore, $\Psi \circ \rR \Gamma$ is simply the identity functor. Applying $\Psi$ to \eqref{eq:Rgamma}, we thus find
\begin{displaymath}
\rR \smlim M_n = \Psi(\varprojlim \Gamma(M_n)).
\end{displaymath}
Thus the higher derived smooth limits of the inverse system $\{M_n\}$ vanish, and so $M=\rR \smlim M_n$. We therefore have
\begin{displaymath}
\rR \Gamma(M)=\rR \Gamma(\rR \smlim M_n)=\varprojlim \Gamma(M_n),
\end{displaymath}
and so $\rR^i \Gamma(M)=0$ for $i>0$.
\end{proof}

\begin{corollary} \label{prop:Rgammastd}
We have $\rR^i \Gamma(\bE^{\lambda})=0$ for $i>0$, for any $\lambda$.
\end{corollary}

\begin{remark}
Let $\{M_i\}_{i \in \cU}$ be a cofiltered inverse system of smooth $\cI$-modules. The argument in the proof of Proposition~\ref{prop:Rgamma} yields an isomorphism
\begin{displaymath}
\rR \smlim M_i = \Psi(\rR \varprojlim \rR \Gamma(M_i)).
\end{displaymath}
One can use this isomorphism to study $\rR \smlim$. For example, if each $M_i$ has level $\le r$ (as defined in \S \ref{ss:level}) then $\rR^n \Gamma(M_i)=0$ for $n \ge r+2$ (Proposition~\ref{prop:gammalev}), and so $\rR^n \smlim M_i=0$ for $n \ge r+3$.
\end{remark}

\section{Completions} \label{s:xi}

\subsection{The completion functor} \label{ss:xi}

Let $M$ be a smooth $\cI$-module. The truncations $\tau^{<n}(M)$ form an inverse system of $\cI$-modules. We define the {\bf ungraded completion} of $M$, denoted $\Xi(M)$, to be the smooth inverse limit of this system. The $\cI$-module $\tau^{<n}(M)$ is locally finite (Proposition~\ref{prop:taufin}), and therefore admits a canonical grading (Theorem~\ref{thm:canongr}). We define the {\bf graded completion} of $M$, denoted $\ul{\Xi}(M)$, to be the inverse limit of this system in the category of graded $\cI$-modules. Since the forgetful functor $\Phi$ is continuous (Proposition~\ref{prop:Phi-cont}), we have $\Xi(M)=\Phi(\ul{\Xi}(M))$.

\begin{proposition} \label{prop:xi}
We have the following:
\begin{enumerate}
\item The functors $(\ul{\Xi}, \Phi)$ are naturally an adjoint pair.
\item The functors $\ul{\Xi}$ and $\Xi$ are exact.
\item The unit map $\epsilon_M \colon M \to \Xi(M)$ is injective, for any smooth $\cI$-module $M$. If $M$ admits a grading, then $\epsilon_M$ is split.
\item The counit map $\eta_N \colon \ul{\Xi}(\Phi(N)) \to N$ is a split surjection, for any graded $\cI$-module $N$.
\end{enumerate}
\end{proposition}

\begin{proof}
(a) Let $M$ be a smooth $\cI$-module. Since $M$ naturally maps to each $\tau^{<n}(M)$, it naturally maps to the smooth inverse limit $\Xi(M)$. This is the unit of the adjunction $(\ul{\Xi}, \Phi)$.

Suppose now that $N$ is a graded $\cI$-module. Recall that $N_{\ge n}=\sum_{k \ge n} N_k$ and $N_{<n}=N/N_{\ge n}$. The Jordan--H\"older constituents of $N_{<n}$ are the $\ul{\bB}^k$ with $k<n$, and similarly for $\Phi(N_{<n})$. Thus $\tau^{<n}(\Phi(N_{<n}))=\Phi(N_{<n})$, and so the natural map $\Phi(N) \to \Phi(N_{<n})$ induces a map $\tau^{<n}(\Phi(N)) \to \Phi(N_{<n})$. These are locally finite modules, and so this map is homogeneous with respect to the canonical gradings; note that the canonical grading on $\Phi(N_{<n})$ is simply $N_{<n}$. Taking inverse limits in the graded category, and using that $N=\varprojlim N_{<n}$, we obtain a map $\ul{\Xi}(\Phi(N)) \to N$. This is the counit of the adjoint $(\ul{\Xi}, \Phi)$.

We leave to the reader the verification that the unit and co-unit thus defined satisfy the requisite relations to give an adjunction. The adjunction $(\ul{\Xi}, \Phi)$ is very similar, and also left to the reader.

(b) Consider an exact sequence of smooth $\cI$-modules:
\begin{displaymath}
0 \to M_1 \to M_2 \to M_3 \to 0.
\end{displaymath}
Since the truncation functor $\tau^{<n}$ is exact (Proposition~\ref{prop:tau-exact}), we get an exact sequence of inverse systems
\begin{displaymath}
0 \to \{\tau^{<n}(M_1)\} \to \{\tau^{<n}(M_2)\} \to \{\tau^{<n}(M_3)\} \to 0.
\end{displaymath}
Since the transition maps in each of these systems is surjective, the inverse limit (in the graded category) remains exact. Thus $\ul{\Xi}$ is exact. Since $\Xi=\Phi \circ \ul{\Xi}$ is a composition of exact functors, it too is exact.

(c) The kernel of the map $M \to \tau^{<n}(M)$ is $\tau_{\ge n}(M)$. Thus the kernel of the map $M \to \Xi(M)$ is $\bigcap_{n \ge 1} \tau_{\ge n}(M)$, which vanishes (Proposition~\ref{prop:sep}). The claim that $\epsilon_M$ is split when $M$ admits a grading is proven below.

(d) By the definition of adjunction, the composition
\begin{displaymath}
\xymatrix{
\Phi(N) \ar[r]^-{\epsilon_{\Phi(N)}} & \Phi(\ul{\Xi}(\Phi(N)) \ar[r]^-{\Phi(\eta_N)} & \Phi(N) }
\end{displaymath}
is the identity. We thus see that $\epsilon_{\Phi(N)}$ is a split injection, proving the remaining claim from (c). Furthermore, we see that $\Phi(\eta_N)$ is a split surjection, which implies that $\eta_N$ is a split surjection by Proposition~\ref{prop:phi-split}.
\end{proof}

\begin{corollary} \label{cor:phiinj}
The forgetful functor $\Phi$ takes injective objects of $\uRep(\cI)$ to injective objects of $\Rep(\cI)$.
\end{corollary}

We now prove a few simple facts about $\Xi$.

\begin{proposition} \label{prop:Xistab}
Let $M$ be a smooth $\cI$-module and let $n \ge r$. Then the natural map $\ul{\Xi}(M)_r \to \tau^{\le n}(M)_r$ on degree $r$ pieces is an isomorphism.
\end{proposition}

\begin{proof}
Since limits are computed degreewise for graded modules, $\ul{\Xi}(M)_r$ is the inverse limit of the system $\{ \tau^{\le n}(M)_r \}_{n \ge 0}$. Consider the exact sequence
\begin{displaymath}
0 \to \tau_{>r}(\tau^{\le n}(M)) \to \tau^{\le n}(M) \to \tau^{\le r}(M) \to 0.
\end{displaymath}
By Proposition~\ref{prop:mutau}, the left term has $\mu_r=0$, and so by Proposition~\ref{prop:mu-umu} it has $\umu_r=0$. Thus the right map is an isomorphism on degree $r$ pieces, and so the inverse system defining $\ul{\Xi}(M)_r$ is constant for $n \ge r$. This proves the proposition.
\end{proof}

\begin{proposition} \label{prop:Ximult}
Let $M$ be a smooth $\cI$-module. Then $\umu_r(\ul{\Xi}(M))=\mu_r(M)$ for all $r \ge 0$.
\end{proposition}

\begin{proof}
We have
\begin{displaymath}
\umu_r(\ul{\Xi}(M))=\umu_r(\tau^{\le r}(M))=\mu_r(\tau^{\le r}(M))=\mu_r(M),
\end{displaymath}
where in the first step we used Proposition~\ref{prop:Xistab}, in the second Proposition~\ref{prop:mu-umu}, and in the third Proposition~\ref{prop:mutau}.
\end{proof}

\subsection{Computing $\ul{\Xi}$} \label{ss:computexi}

We now compute $\ul{\Xi}$ on standard modules. This will take a fair amount of work.

\begin{proposition} \label{prop:A1xi}
Let $M$ be a smooth $\cI$-module. We have a natural isomorphism
\begin{displaymath}
\ul{\Xi}(\ul{\bA}^1 \odot M) \cong \left[ \ul{\bA}^1 \odot \ul{\Xi}(M) \right] \oplus \left[ \ul{\bB}^0 \otimes M_{\cI} \right]
\end{displaymath}
Moreover, if $\iota(M) \colon M \to \Xi(M)$ denotes the canonical morphism, then the diagram
\begin{displaymath}
\xymatrix{
& \ul{\bA}^1 \odot M \ar[rd]^{\iota(\ul{\bA}^1 \odot M)} \ar[ld]_{[\id \odot \iota(M)] \oplus a} \\
\left[ \ul{\bA}^1 \odot \ul{\Xi}(M) \right] \oplus \left[ \ul{\bB}^0 \otimes M_{\cI} \right] \ar@{=}[rr] &&
\Xi(\ul{\bA}^1 \odot M) }
\end{displaymath}
commutes, where the bottom equality is the natural isomorphism and $a$ is the map taking $e_i \odot x$ to the image of $x$ in $M_{\cI}$.
\end{proposition}

\begin{proof}
The proof is rather lengthy, and divided into six steps:
\begin{itemize}
\item Step 1: study of a filtration on $\ul{\bA}^1 \odot M$.
\item Step 2: definition of the map $f \colon \ul{\bA}^1 \odot \ul{\Xi}(M) \to \ul{\Xi}(\ul{\bA}^1 \odot M)$.
\item Step 3: definition of the map $g \colon \ul{\bB}^0 \otimes M_{\cI} \to \ul{\Xi}(\ul{\bA}^1 \odot M)$.
\item Step 4: proof that $f \oplus g$ is surjective.
\item Step 5: proof that $f \oplus g$ is an isomorphism.
\item Step 6: proof that the diagram in the statement of the proposition commutes.
\end{itemize}
We use the following notation in this proof: for an element $x$ of an $\cI$-module $K$, we let $[x]_n$ denote the image of $x$ in the truncation $\tau^{<n}(M)$. Recall that if $n>0$ then $\tau^{<n}(M)$ is the coinvariant space $M_{\cI_{\ge n}}$, while if $n \le 0$ then $\tau^{<n}(M)=0$. We also put $N=\ul{\bA}^1 \odot M$. 

\textbf{\textit{Step 1.}} 
For $n \ge 1$, let $F^nN$ be the subspace of $N$ given by $\bigoplus_{i \ge 1} \left( e_i \odot \tau_{\ge n-i}(M) \right)$. We claim that $F^n N$ is an $\cI$-submodule of $N$. Indeed, suppose $x \in \tau_{\ge n-i} (M)$. If $j \le i$ then $\alpha_j (e_i \odot x)=e_{i+1} \odot x$, and this belongs to $e_{i+1} \odot \fa_{n-i-1} M$ since $\tau_{\ge n-i}(M) \subset \tau_{\ge n-i-1}(M)$. If $j>i$ then $\alpha_j (e_i \odot x)=e_i \odot (\alpha_{j-i} x)$, and this belongs to $e_i \odot \tau_{\ge n-i}(M)$ since $\tau_{\ge n-i}(M)$ is an $\cI$-submodule of $M$. This establishes the claim.

We now examine the inverse limit of the system $N/F^n N$. We have
\begin{displaymath}
N/F^n N = \bigoplus_{i=1}^{n-1} e_i \odot \tau^{<n-i}(M).
\end{displaymath}
This module is locally finite, and thus admits a canonical grading. We have a natural map
\begin{equation} \label{eq:A1xi-1}
\ul{\bA}^1 \odot \ul{\Xi}(M) \to \big( \ul{\bA}^1/\sum_{i \ge n} \bk e_i \big) \odot M_{\cI_{\ge n}} \to N/F^n N
\end{equation}
by using the natural maps $M_{\cI_{\ge n}} \to M_{\cI_{\ge n-i}}$. One readily verifies that this is $\cI$-linear. Moreover, it respects the gradings, as the middle group is locally finite and thus canonically graded. It is clear that the above maps are compatible with the maps $N/F^{n+1} N \to N/F^n N$. We thus get a natural map
\begin{displaymath}
\ul{\bA}^1 \odot \ul{\Xi}(M) \to \varprojlim N/F^n N
\end{displaymath}
of graded $\cI$-modules. We claim this map is an isomorphism. It suffices to show that it is an isomorphism in each degree; thus fix $d \ge 0$, and let $n \ge d$. One easily verifies that the degree $d$ piece of $N/F^n N$ is $\bigoplus_{i=1}^d e_i \odot (M_{\cI_{\ge n-i}})_{d-i}$. Similarly, the degree $d$ piece of $\ul{\bA}^1 \odot \Xi(M)$ is $\bigoplus_{i=1}^d e_i \odot \Xi(M)_{d-i}$. The map $\Xi(M)_{d-i} \to (M_{\cI_{\ge n-i}})_{d-i}$ is an isomorphism by Proposition~\ref{prop:Xistab}, as $n-i \ge d-i$. We thus see that the degree $d$ piece of the map \eqref{eq:A1xi-1} is an isomorphism for $n \ge d$, which yields the claim.

\textbf{\textit{Step 2.}} 
Consider the map
\begin{displaymath}
\wt{f}_n \colon N \to N_{\cI_{\ge n}}, \qquad
\wt{f}_n(e_i \odot x) = [(e_i-e_n) \odot x]_n.
\end{displaymath}
We note that $[e_n \odot x]_n$ is $\cI$-invariant. Indeed it is invariant under $\alpha_i$ for $i \ge n$ by definition, while if $i<n$ then
\begin{displaymath}
\alpha_i [e_n \odot x]_n=[e_{n+1} \odot x]_n=\alpha_n [e_n \odot x]_n=[e_n \odot x]_n
\end{displaymath}
In particular, we see that $[e_m \odot x]_n=[e_n \odot x]_n$ for $m \ge n$, as the left side is obtained from the right by applying $\alpha_n^{m-n}$.

We claim that $\wt{f}_n$ is $\cI$-linear. It is clear that $e_i \odot x \mapsto [e_i \odot x]_n$ is $\cI$-linear, so it suffices to show that $h \colon e_i \odot x \mapsto [e_n \odot x]_n$ is $\cI$-linear. Thus fix $i$ and $x$, and let us show that $h(\alpha_j (e_i \odot x))=h(e_i \odot x)$ for all $j \in \bN$. First suppose that $j \le i$. Then
\begin{displaymath}
h(\alpha_j (e_i \odot x))=h(e_{i+1} \odot x)=[e_n \odot x]_n=h(e_i \odot x).
\end{displaymath}
Now suppose that $j>i$. Then
\begin{displaymath}
h(\alpha_j (e_i \odot x))=h(e_i \odot \alpha_{j-i}x)=[e_n \odot \alpha_{j-i}x]_n=\alpha_{n+j-i} [e_n \odot x]_n=h(e_i \odot x).
\end{displaymath}
This verifies the claim that $\wt{f}_n$ is $\cI$-linear.

We now claim that $\wt{f}_n$ kills $F^n N$. Thus we must show that $\wt{f}_n(e_i \odot x)=0$ for all $i \ge 1$ and $x \in \tau_{\ge n-i}(M)$. If $i \ge n$ then
\begin{displaymath}
\wt{f}_n(e_i \odot x)=[(e_i-e_n) \odot x]_n=[e_i \odot x]_n - [e_n \odot x]_n=0,
\end{displaymath}
as explained above. Suppose $i<n$. It suffices to treat the case where $x=(\alpha_j-1) y$ for some $j \ge n-i$ and $y \in M$. We have
\begin{displaymath}
[e_i \odot (\alpha_j-1)y]_n=[e_i \odot \alpha_j y]_n-[e_i \odot y]_n=\alpha_{i+j} [e_i \odot y]_n - [e_i \odot y]_n=0,
\end{displaymath}
as $i+j \ge n$, and so $\alpha_{i+j}$ acts trivially on $M_{\cI_{\ge n}}$. The same reasoning shows that $[e_n \odot (\alpha_j-1)y]_n=0$, and so $\wt{f}_n(e_i \odot (\alpha_j-1)y)=0$. The claim follows.

We thus see that $\wt{f}_n$ induces an $\cI$-linear map
\begin{displaymath}
f_n \colon N/F^n N \to N_{\cI_{\ge n}}.
\end{displaymath}
We note that the domain and target are both locally finite, and thus carry a canonical grading, which is repsected by $f_n$. We claim that the diagram
\begin{displaymath}
\xymatrix{
N/F^{n+1} N \ar[d] \ar[r]^{f_{n+1}} & N_{\cI_{\ge n+1}} \ar[d] \\
N/F^n N \ar[r]^{f_n} & N_{\cI_{\ge n}} }
\end{displaymath}
commutes. Indeed, let $e_i \odot x \in N$ be given. Taking its class in the top left group and mapping to the left gives $[(e_i-e_{n+1}) \odot x]_n$; mapping down gives $[(e_i-e_{n+1}) \odot x]_n=[(e_i-e_n) \odot x]_n$. This is exactly what we get by first mapping down and then to the right. We thus see that the $f_n$'s form a map of inverse systems. Taking inverse limits, we thus get a map
\begin{displaymath}
f \colon \ul{\bA}^1 \odot \ul{\Xi}(M) \to \ul{\Xi}(\ul{\bA}^1 \odot M)
\end{displaymath}
of graded $\cI$-modules.

\textbf{\textit{Step 3.}}
Define a map
\begin{displaymath}
\wt{g}_n \colon M \to N_{\cI_{\ge n}}, \qquad \wt{g}_n(x)=[e_n \odot x]_n.
\end{displaymath}
For $i \ge 1$, we have
\begin{displaymath}
\wt{g}_n((\alpha_i-1)x)=[e_n \odot ((\alpha_i-1)x]_n=(\alpha_{n+i-1}-1) [e_n \odot x]_n=0.
\end{displaymath}
We thus see that $\wt{g}_n$ induces a map
\begin{displaymath}
g_n \colon M_{\cI} \to N_{\cI_{\ge n}}.
\end{displaymath}
As $\cI$ acts trivially on the source and the image lands in the $\cI$-invariants of the target (by the computations from Step~1), we see that $g_n$ is $\cI$-linear. As both the source and target of $g_n$ are locally finite, they carry the canonical grading, which $g_n$ respects; in fact, the source is concentrated in degree~0 since $\cI$ acts trivially on it. As $[e_{n+1} \odot x]_n=[e_n \odot x]_n$, we see that the diagram
\begin{displaymath}
\xymatrix{ && N_{\cI_{\ge n+1}} \ar[d] \\
M_{\cI} \ar[urr]^{g_{n+1}} \ar[rr]^{g_n} && N_{\cI_{\ge n}} }
\end{displaymath}
commutes. Thus the $g_n$ induce a map
\begin{displaymath}
g \colon M_{\cI} \to \Xi(\ul{\bA}^1 \odot M)
\end{displaymath}
of graded $\cI$-modules, where $M_{\cI}$ is concentrated in degree~0. We could denote the domain of $g$ by $\ul{\bB}^0 \otimes M_{\cI}$ to more clearly express its structure.

\textbf{\textit{Step 4.}}
Combining Steps~1 and~2, we have a map
\begin{displaymath}
f \oplus g \colon (\ul{\bA}^1 \odot \ul{\Xi}(M)) \oplus (\ul{\bB}^0 \otimes M_{\cI}) \to \ul{\Xi}(N).
\end{displaymath}
We claim that this map is surjective. Let $d \ge 0$ be given, and let us show that it is surjective in degree $d$. Let $n \ge d$. Since the map $\ul{\Xi}(N)_d \to (N_{\cI_{\ge n}})_d$ is an isomorphism (Proposition~\ref{prop:Xistab}), it suffices to show that the composition of the above map with this one is surjective in degree $d$. Thus, it suffices to show that the map
\begin{displaymath}
f_n \oplus g_n \colon N/F^n N \oplus M_{\cI} \to N_{\cI_{\ge n}}
\end{displaymath}
is surjective. For $x \in M$, we have
\begin{displaymath}
[e_i \odot x]_n = [(e_i-e_n) \odot x]_n + [e_n \odot x]_n.
\end{displaymath}
As the first term on the right belongs to the image of $f_n$, and the second to the image of $g_n$, the claim follows.

\textbf{\textit{Step 5.}}
We now aim to show that $f \oplus g$ is an isomorphism. Since the completion and coinvariant functors commute with filtered colimits, it suffices to treat the case where $M$ is finitely generated. As we have already shown that $f \oplus g$ is surjective, it thus suffices to show that the graded pieces of the domain and target have equal dimensions. For $r \ge 0$, we have
\begin{displaymath}
\umu_r(\ul{\bA}^1 \odot \ul{\Xi}(M))=\sum_{i=0}^{r-1} \umu_i(\ul{\Xi}(M))=\sum_{i=0}^{r-1} \mu_i(M);
\end{displaymath}
the first equality follows from the definition of the concatenation product, while the second follows from Proposition~\ref{prop:Ximult}. Obviously, $\umu_r(\ul{\bB}^0 \otimes M_{\cI})$ is $\dim{M_{\cI}}=\mu_0(M)$ for $r=0$, and vanishes otherwise. On the other hand, we have
\begin{displaymath}
\umu_r(\ul{\Xi}(\ul{\bA}^1 \odot M))=\mu_r(\ul{\bA}^1 \odot M)=
\begin{cases}
\mu_0(M) & \text{if $r=0$} \\
\mu_0(M) + \cdots + \mu_{r-1}(M) & \text{if $r>0$} \end{cases}
\end{displaymath}
by Propositions~\ref{prop:Ximult} and~\ref{prop:multA1cat}. We thus see that the value of $\umu_r$ agrees on the domain and target, as required.

\textbf{\textit{Step 6.}}
Finally, we show that the diagram in the statement of the proposition commutes. For this, it suffices to show that the diagram
\begin{displaymath}
\xymatrix{ & \ul{\bA}^1 \odot M \ar[ld]_{p \oplus a} \ar[rd]^q \\
N/F^n N \oplus M_{\cI} \ar[rr]^{f_n \oplus g_n} && N_{\cI_{\ge n}} }
\end{displaymath}
commutes, where $p$ and $q$ denote the canonical maps, and $a$ is as in the statement of the proposition. Thus let $e_i \odot x$ be a given element of $\ul{\bA}^1 \odot M$. As $f_n \circ p$ is simply $\wt{f}_n$, we have $f_n(p(e_i \odot x))=[(e_i-e_n) \odot x]_n$. Similarly, as $g_n \circ a$ is simply $\wt{g}_n$, we have $g_n(a(e_i \odot x))=[e_n \odot x]_n$. We thus see that by mapping $e_i \odot x$ under $p \oplus a$ follows by $f_n \oplus g_n$, we get $[(e_i-e_n) \odot x]_n + [e_n \odot x]_n = [e_i \odot x]_n$, which is exactly the image of $e_i \odot x$ under $q$. This completes the proof.
\end{proof}

\begin{proposition} \label{prop:L1xi}
Let $M$ be a smooth $\cI$-module. Then we have a natural isomorphism
\begin{displaymath}
\ul{\Xi}(\ul{\bB}^1 \odot M) \cong \ul{\bB}^1 \odot \ul{\Xi}(M).
\end{displaymath}
Moreover, the diagram
\begin{displaymath}
\xymatrix{
& \ul{\bB}^1 \odot M \ar[ld]_{i(\ul{\bB}^1 \odot M)} \ar[rd]^{\id \odot i(M)} \\
\Xi(\ul{\bB}^1 \odot M) \ar@{=}[rr] && \ul{\bB}^1 \odot M }
\end{displaymath}
commutes, with notation as in Proposition~\ref{prop:A1xi}.
\end{proposition}

\begin{proof}
Let $\xi$ be denote a basis vector for $\ul{\bB}^1$. For $i \ge 2$ and $x \in M$, we have
\begin{displaymath}
(\alpha_i-1) (\xi \odot x) = \xi \odot (\alpha_{i-1}-1) x.
\end{displaymath}
We thus see that, for $n \ge 2$, we have
\begin{displaymath}
\tau_{\ge n}(\ul{\bB}^1 \odot M) = \ul{\bB}^1 \odot \tau_{\ge n-1}(M)
\end{displaymath}
as subspaces of $\ul{\bB}^1 \odot M$. We thus have a natural isomorphism
\begin{displaymath}
\tau^{<n}(\ul{\bB}^1 \odot M) \cong \ul{\bB}^1 \odot \tau^{<n-1}(M)
\end{displaymath}
of graded $\cI$-modules. Taking inverse limits yields the claimed isomorphism. (We note that the inverse limit commutes with the operation $\ul{\bB}^1 \odot -$ since $\ul{\bB}^1$ is finite dimensional.)

To prove that the diagram in the statement of the proposition commutes, it suffices to show that the diagram
\begin{displaymath}
\xymatrix{ & \ul{\bB}^1 \odot M \ar[rd] \ar[ld] \\
\tau^{<n}(\ul{\bB}^1 \odot M) \ar@{=}[rr] && \ul{\bB}^1 \odot \tau^{<n-1}(M) }
\end{displaymath}
commutes for all $n \ge 2$. This is clear.
\end{proof}

\begin{proposition} \label{prop:xistd}
Let $\lambda=\mu \cdot a^n$ be a constraint word, where either $\mu$ is the empty word or ends in $b$. Then
\begin{displaymath}
\ul{\Xi}(\bE^{\lambda}) \cong \bigoplus_{i=0}^n \ul{\bE}^{\mu \cdot a^i}.
\end{displaymath}
Moreoever, the canonical injection $\bE^{\lambda} \to \Xi(\bE^{\lambda})$ is split, as is the canonical surjection $\ul{\Xi}(\bE^{\lambda}) \to \ul{\bE}^{\lambda}$.
\end{proposition}

\begin{proof}
The computation of $\ul{\Xi}(\bE^{\lambda})$ follows from repeatedly applying Propositions~\ref{prop:A1xi} and~\ref{prop:L1xi}, together with the fact that if $\nu$ is a constraint word containing any $b$ then $(\bE^{\nu})_{\cI}=0$. The claims about splittings follow from Proposition~\ref{prop:xi}(c,d).
\end{proof}

\begin{corollary} \label{cor:xiprin}
We have $\ul{\Xi}(\bA^n) \cong \bigoplus_{0 \le k \le n} \ul{\bA}^k$.
\end{corollary}

\begin{corollary}
Let $M$ be a finitely generated smooth $\cI$-module. Then $\ul{\Xi}(M)$ and $\Xi(M)$ are finitely generated.
\end{corollary}

\begin{proof}
Choose a surjection $F \to M$ with $F$ a finite sum of principal modules. Since $\ul{\Xi}$ is exact, the map $\ul{\Xi}(F) \to \ul{\Xi}(M)$ is surjective. Since $\ul{\Xi}(F)$ is finitely generated, it follows that $\ul{\Xi}(M)$ is as well. The argument for $\Xi$ is identical.
\end{proof}

\subsection{Two variants of completion}

For a smooth $\cI$-module $M$, we define the {\bf residual completion} of $M$, denoted $\Xi^{\res}(M)$ to be the cokernel of the natural map $M \to \Xi(M)$. For a graded $\cI$-module $M$, we define the {\bf coresidual completion} of $M$, denoted $\ul{\Xi}^{\cor}(M)$, to be the kernel of the natural map $\ul{\Xi}(\Phi(M)) \to M$. We thus have functors
\begin{displaymath}
\Xi^{\res} \colon \Rep(\cI) \to \Rep(\cI), \qquad
\ul{\Xi}^{\cor} \colon \uRep(\cI) \to \uRep(\cI).
\end{displaymath}
The main properties of these functors are summarized in the following proposition.

\begin{proposition} \label{prop:xibar}
We have the following:
\begin{enumerate}
\item The functors $\Xi^{\res}$ and $\ul{\Xi}^{\cor}$ are exact and co-continuous.
\item Let $\lambda=\mu\cdot a^n$ be a constraint word where either $\mu$ is the empty word or ends in $b$. Then
\begin{displaymath}
\Xi^{\res}(\bE^{\lambda}) \cong \bigoplus_{i=0}^{n-1} \bE^{\mu \cdot a^i}, \qquad
\ul{\Xi}^{\cor}(\ul{\bE}^{\lambda}) \cong \bigoplus_{i=0}^{n-1} \ul{\bE}^{\mu \cdot a^i}.
\end{displaymath}
\item We have
\begin{displaymath}
\Xi^{\res}(\bA^n) \cong \bigoplus_{i=0}^{n-1} \bA^i, \qquad
\ul{\Xi}^{\cor}(\ul{\bA}^n) \cong \bigoplus_{i=0}^{n-1} \ul{\bA}^i.
\end{displaymath}
\end{enumerate}
\end{proposition}

\begin{proof}
(a) This follows from Proposition~\ref{prop:funseq}.

(b) This follows form Proposition~\ref{prop:xistd}.

(c) This is a special case of part (b).
\end{proof}

\subsection{Some applications of completion}

We now give some applications of the completion functor. Perhaps the most striking is:

\begin{theorem}
The principal modules $\bA^r$ and $\ul{\bA}^r$ are injective, for all $r \in \bN$.
\end{theorem}

\begin{proof}
If $\bA^r$ is injective, then so is $\ul{\bA}^{r+1}=\Gamma(\bA^r)$ by Corollary~\ref{cor:gaminj}; similarly, if $\ul{\bA}^r$ is injective then so is $\bA^r=\Phi(\ul{\bA}^r)$ by Corollary~\ref{cor:phiinj}. Since $\bA^0$ is injective (Corollary~\ref{cor:trivinj}), the theorem follows.
\end{proof}

We next use the completion functor to establish some fundamental properties of the functor $\Gamma$. For a smooth $\cI$-module $M$, we define the {\bf generation degree} of $M$, denoted $g(M)$, to be the infimum of integers $g$ such that $M$ is generated by $M^{\cI_{\ge g}}$. Note that if $M=0$ then $g(M)=-\infty$, while if $M$ is not generated by $M^{\cI_{\ge g}}$ for any $g$ then $g(M)=\infty$. Also note that $g(M) \le g$ if and only if $M$ is a quotient of a sum of principal modules $\bA^r$ with $r \le g$.

\begin{lemma} \label{lem:gxi}
Let $M$ be a smooth $\cI$-module with generation degree $g<\infty$. Then $\Xi^{\res}(M)$ has generation degree $<g$.
\end{lemma}

\begin{proof}
Choose a surjection $F \to M$, where $F$ is a sum of modules of the form $\bA^r$ with $r \le g$. Then $\Xi^{\res}(F) \to \Xi^{\res}(M)$ is surjective, and the source is a sum of modules of the form $\bA^r$ with $r<g$ (Proposition~\ref{prop:xibar}). The result follows.
\end{proof}

\begin{proposition} \label{prop:gammafinite}
If $M$ is a finitely generated smooth $\cI$-module then $\rR^n \Gamma(M)$ is finitely generated for all $n$, and vanishes for $n>g(M)$.
\end{proposition}

\begin{proof}
We proceed by induction on $g(M)$. If $g(M)=0$ then $\cI$ acts trivially on $M$, and so $\Gamma(M)$ is just $\dim(M)$ copies of $\ul{\bA}^1$ and $\rR^n \Gamma(M)=0$ by Corollary~\ref{cor:trivinj}. Thus the result holds. Suppose now $g(M)>0$. Consider the exact sequence
\begin{displaymath}
0 \to M \to \Xi(M) \to \Xi^{\res}(M) \to 0.
\end{displaymath}
Note that $g(\Xi^{\res}(M))<g(M)$ by Lemma~\ref{lem:gxi}. Applying $\Gamma$, and appealing to Proposition~\ref{prop:Rgamma}, we obtain an exact sequence
\begin{displaymath}
0 \to \Gamma(M) \to \Gamma(\Xi(M)) \to \Gamma(\Xi^{\res}(M)) \to \rR^1 \Gamma(M) \to 0
\end{displaymath}
and isomorphisms
\begin{displaymath}
\rR^{n+1} \Gamma(M) \cong \rR^n \Gamma(\Xi^{\res}(M))
\end{displaymath}
for $n \ge 1$. Since $\Gamma(\Xi^{\res}(M))$ is finitely generated by the inductive hypothesis, so is $\rR^1 \Gamma(M)$. Similarly, since $\rR^n \Gamma(\Xi^{\res}(M))$ is finitely generated for $n \ge 1$ and vanishes for $n>g(\Xi^{\res}(M))$, we see that $\rR^n \Gamma(M)$ is finitely generated for $n \ge 2$ and vanishes for $n>g(M)$.

It remains to show that $\Gamma(M)$ itself is finitely generated. For this, choose a surjection $F \to M$, where $F$ is a finite sum of principal modules. Since $\Gamma \circ \Xi$ is exact (as $\rR^n \Gamma \circ \Xi=0$ for $n \ge 1$ by Proposition~\ref{prop:Rgamma}), the map $\Gamma(\Xi(F)) \to \Gamma(\Xi(M))$ is surjective. By explicit computations with principal modules (Corollaries~ \ref{prop:gammaprin} and~\ref{cor:xiprin}), $\Gamma(\Xi(F))$ is finitely generated. Thus $\Gamma(\Xi(M))$ is finitely generated as well. As $\Gamma(M)$ is a submodule of this, it too is finitely generated.
\end{proof}

\begin{proposition} \label{prop:rgammacolim}
The functor $\rR^n \Gamma$ commutes with filtered colimits for all $n \ge 0$.
\end{proposition}

\begin{proof}
The functor $(-)^{\cI_{\ge r}}$ obviously commutes with filtered colimits, for any $r$, and so $\Gamma$ commutes with filtered colimits. To prove the proposition, it thus suffices to show that that a filtered colimit of $\Gamma$-acyclic objects is $\Gamma$-acyclic (Proposition~\ref{prop:dercolimit}). Thus let $\{M_i\}_{i \in I}$ be a filtered system of $\Gamma$-acyclic objects. Let $N_i=\Xi(M_i)$, and let $C_i=\Xi^{\res}(M_i)$. Since $N_i=\Phi(\ul{\Xi}(M_i))$ is $\Gamma$-acyclic (Proposition~\ref{prop:Rgamma}), it follows that $C_i$ is as well. Let $M$, $N$, and $C$ be the colimits of the systems $\{M_i\}$, $\{N_i\}$, and $\{C_i\}$. Since filtered colimits are exact, we have an exact sequence
\begin{displaymath}
0 \to M \to N \to C \to 0
\end{displaymath}
Since $\Xi$ is cocontinuous, we have $N=\Xi(M)$; in particular, $N$ is $\Gamma$-acyclic. We thus have an exact sequence
\begin{displaymath}
0 \to \Gamma(M) \to \Gamma(N) \to \Gamma(C) \to \rR^1 \Gamma(M) \to 0
\end{displaymath}
and isomorphisms
\begin{displaymath}
\rR^n \Gamma(C)=\rR^{n+1} \Gamma(M)
\end{displaymath}
for $n \ge 1$. Now, the sequence
\begin{displaymath}
0 \to \Gamma(M_i) \to \Gamma(N_i) \to \Gamma(C_i) \to 0
\end{displaymath}
is exact for all $i$, since each $M_i$ is $\Gamma$-acyclic. Taking the direct limit, and use the fact that $\Gamma$ commutes with this operation, we find that the sequence
\begin{displaymath}
0 \to \Gamma(M) \to \Gamma(N) \to \Gamma(C) \to 0
\end{displaymath}
is exact. Thus $\rR^1 \Gamma(M)=0$. Since our choice of $\{M_i\}$ was arbitrary, it follows that $\rR^1 \Gamma$ vanishes on all filtered colimits of acyclic objects. In particular, $0=\rR \Gamma^1(C) \cong \rR \Gamma^2(M)$. Continuing in this way, we find $\rR^n \Gamma(M)=0$ for all $n$, as required.
\end{proof}

We now prove some finiteness properties of $\Ext$. In the graded case, we deduce the results using projective resolutions. In the smooth case, these are not available; instead, we deduce the results from the graded case by making use of properties of $\Gamma$.

\begin{proposition} \label{prop:uextcolim}
Let $M$ be a finitely generated graded $\cI$-module and let $i \in \bN$.
\begin{enumerate}
\item If $N$ is a finitely generated graded $\cI$-module then $\uExt^i_{\cI}(M, N)$ is finite dimensional.
\item The functor $\uExt^i_{\cI}(M, -)$ commutes with filtered colimits.
\end{enumerate}
\end{proposition}

\begin{proof}
Let $P_{\bullet} \to M$ be a projective resolution of $M$ with each $P_i$ finitely generated. Such a resolution exists since $\uRep(\cI)$ is locally noetherian and has enough finitely generated projectives. To prove (a), simply note that $\uExt^{\bullet}_{\cI}(M, N)$ is computed by the complex $\uHom_{\cI}(P_{\bullet}, N)$, each term of which is finite dimensional. To prove (b), suppose that $\{N_i\}_{i \in \cU}$ is a filtered direct system in $\uRep(\cI)$. Then the natural map
\begin{displaymath}
\varinjlim \uHom_{\cI}(P_{\bullet}, N_i) \to \uHom_{\cI}(P_{\bullet}, \varinjlim N_i)
\end{displaymath}
is an isomorphism of complexes, since each term of $P_{\bullet}$ is a finitely generated projective. As the left complex computes $\varinjlim \uExt^{\bullet}_{\cI}(M, N_i)$ and the right computes $\uExt^{\bullet}_{\cI}(M, \varinjlim N_i)$, the result follows.
\end{proof}

\begin{proposition} \label{prop:extcolim}
Let $M$ be a finitely generated smooth $\cI$-module and let $i \in \bN$.
\begin{enumerate}
\item If $N$ is a finitely generated smooth $\cI$-module then $\Ext^i_{\cI}(M, N)$ is finite dimensional.
\item The functor $\Ext^i_{\cI}(M, -)$ commutes with filtered colimits.
\end{enumerate}
\end{proposition}

\begin{proof}
Let $M'=\Gamma(M)$, which is a finitely generated graded $\cI$-module with $\Psi(M') \cong M$. We have
\begin{displaymath}
\rR \Hom_{\cI}(M, -)=\rR \uHom_{\cI}(M', \rR \Gamma(-))
\end{displaymath}
by adjunction. Since both $\rR \uHom_{\cI}(M', -)$ and $\rR \Gamma$ preserve finiteness and commute with filtered colimits (by Propositions~\ref{prop:gammafinite}, \ref{prop:rgammacolim} and~\ref{prop:uextcolim}), so does $\rR \Hom_{\cI}(M, -)$.
\end{proof}

\subsection{The right adjoint of $\Phi$} \label{ss:phiright}

The forgetful functor $\Phi$ is both continuous and cocontinuous (Proposition~\ref{prop:Phi-cont}), and therefore has both a left and a right adjoint. As we have seen (Proposition~\ref{prop:xi}), the left adjoint is the graded completion functor $\ul{\Xi}$. The right adjoint, it turns out, can be described in terms of functors we have already studied. Denote the right adjoint by $F$. By definition, we have an isomorphism
\begin{displaymath}
\uHom_{\cI}(M, F(N)) \cong \Hom_{\cI}(\Phi(M), N).
\end{displaymath}
Taking $M=\ul{\bA}^r$, we see that $F(N)_r = N^{\cI_{>r}}$. It is not difficult to see that the action of $\alpha_i$ on $F(N)_r$ is simply induced by the action of $\alpha_i$ on $N$. It is now straightforward to construct an isomorphism $F(N) \cong \ul{\Sigma}(\Gamma(N)^{\dag})^{\dag}$ of graded $\cI$-modules: indeed, it follows immediately from the definitions that the two are identified as graded vector spaces, and then one simply has to work through the definitions to see that this identification is compatible with the $\cI$-actions.

\section{The theorem on level} \label{s:level}

\subsection{Level} \label{ss:level}

Let $\Rep(\cI)_{\le r}$ be the localizing subcategory of $\Rep(\cI)$ generated by the standard modules $\bE^{\lambda}$ of rank $\le r$. (See \S \ref{ss:loc} for background on localizing subcategories.) We define the {\bf level} of a smooth $\cI$-module $M$, denoted $\lev(M)$, to be the infimum of the integers $r$ for which $M$ belongs to $\Rep(\cI)_{\le r}$. By convention, the zero module has level $-\infty$. We put $\Rep(\cI)_r=\Rep(\cI)_{\le r}/\Rep(\cI)_{\le r-1}$, and $\Rep(\cI)_{>r}=\Rep(\cI)/\Rep(\cI)_{\le r}$. We let $T_{>r} \colon \Rep(\cI) \to \Rep(\cI)_{>r}$ be the localization functor, and let $S_{>r} \colon \Rep(\cI)_{>r} \to \Rep(\cI)$ be its right adjoint (the section functor).

We make similar definitions in the graded case. We let $\uRep(\cI)_{\le r}$ be the localizing subcategory of $\uRep(\cI)$ generated by the $\ul{\bE}^{\lambda}$ of rank $\le r$, and define the {\bf homogeneous level} of a graded $\cI$-module $M$, denoted $\ulev(M)$, to be the infimum of the integers $r$ for which $M$ belongs to $\uRep(\cI)_{\le r}$. We put $\uRep(\cI)_r=\uRep(\cI)_{\le r}/\uRep(\cI)_{\le r-1}$ and $\uRep(\cI)_{>r}=\uRep(\cI)/\uRep(\cI)_{\le r}$ and define $\ul{T}_{>r}$ and $\ul{S}_{>r}$ as before.

The following proposition offers a concrete characterization of modules of level $\le r$:

\begin{proposition} \label{prop:levfilt}
A finitely generated smooth $\cI$-module $M$ has level $\le r$ if and only if there is a finite length filtration $0=F_0 \subset \cdots \subset F^n=M$ and constraint words $\lambda_1, \ldots, \lambda_n$ of rank $\le r$ such that $F^i/F^{i-1}$ is isomorphic to a subquotient of $\bE^{\lambda_i}$ for all $1 \le i \le n$. The analogous statement holds in the graded case.
\end{proposition}

\begin{proof}
This follows immediately from Proposition~\ref{prop:locgen}.
\end{proof}

We now examine how various functors interact with level.

\begin{proposition} \label{prop:levelfunc}
We have the following:
\begin{enumerate}
\item $\Phi(\uRep(\cI)_{\le r}) \subset \Rep(\cI)_{\le r}$.
\item $\Psi(\uRep(\cI)_{\le r}) \subset \Rep(\cI)_{\le r-1}$.
\item $\ul{\Xi}(\Rep(\cI)_{\le r}) \subset \uRep(\cI)_{\le r}$.
\item $\Xi^{\res}(\Rep(\cI)_{\le r}) \subset \Rep(\cI)_{\le r-1}$.
\item $\ul{\Xi}^{\cor}(\uRep(\cI)_{\le r}) \subset \uRep(\cI)_{\le r-1}$.
\end{enumerate}
\end{proposition}

\begin{proof}
Since the functors in question are exact and cocontinuous, it suffices by Proposition~\ref{prop:locmap} to check the statement on the appropriate standard modules, and this follows from our explicit computations.
\end{proof}

\begin{proposition} \label{prop:gammalev}
Suppose that $M$ is a smooth $\cI$-module of level $\le r$. Then $\rR^i \Gamma(M)$ has homogeneous level $\le r+1-i$.
\end{proposition}

\begin{proof}
We first prove the result for $i=0$. We claim that $\Gamma \circ \Xi$ maps $\Rep(\cI)_{\le r}$ into $\uRep(\cI)_{\le r+1}$. Since $\Gamma \circ \Xi$ is exact and commutes with filtered colimits, it suffices by Proposition~\ref{prop:locmap} to verify this on standard modules of level $\le r$, as these generate $\Rep(\cI)_{\le r}$ as a localizing subcategory. This follows from explicit computations (Proposition~\ref{prop:xistd}). Now suppose that $M$ has level $\le r$. Since natural map $\Gamma(M) \to \Gamma(\Xi(M))$ is injective and the target has level $\le r+1$, we see that $\Gamma(M)$ has level $\le r+1$ as well. This establishes the $i=0$ case.

We now prove the result in general by induction on level. The result is tautologically true in level $<0$. Suppose now that it has been proved in level $<r$ and let us verify it in level $\le r$. Thus let $M$ of level $\le r$ be given. Consider the exact sequence
\begin{displaymath}
0 \to M \to \Xi(M) \to \Xi^{\res}(M) \to 0.
\end{displaymath}
We have $\lev(\Xi^{\res}(M))<r$ by Proposition~\ref{prop:levelfunc}. Applying $\Gamma$, and using the $\Gamma$-acyclicty of $\Xi(M)$ (Proposition~\ref{prop:Rgamma}), we obtain an exact sequence
\begin{displaymath}
0 \to \Gamma(M) \to \Gamma(\Xi(M)) \to \Gamma(\Xi^{\res}(M)) \to \rR^1 \Gamma(M) \to 0
\end{displaymath}
and isomorphisms
\begin{displaymath}
\rR^{i+1} \Gamma(M) \cong \rR^i \Gamma(\Xi^{\res}(M))
\end{displaymath}
for $i \ge 1$. By the inductive hypothesis, $\Gamma(\Xi^{\res}(M)$ has level $\le r$, and so $\rR^1 \Gamma(M)$ has level $\le r$ as well. Similarly, the inductive hypothesis shows that $\rR^i \Gamma(\Xi^{\res}(M))$ has level $\le r+i$, and so the same is true for $\rR^{i+1} \Gamma(M)$ for $i \ge 1$. As we have already proved the result for $i=0$, we are done.
\end{proof}

An $\cI$-module (either smooth or graded) has level~0 if and only if it is locally finite. We have seen that the forgetful functor $\Phi \colon \uRep(\cI)^{\lf} \to \Rep(\cI)^{\lf}$ is an equivalence. The following result generalizes this to higher level:

\begin{proposition} \label{prop:levequiv}
For any $r \in \bN$, the forgetful functor $\Phi$ induces an equivalence $\uRep(\cI)_r \to \Rep(\cI)_r$. The quasi-inverse functor is induced by $\ul{\Xi}$.
\end{proposition}

\begin{proof}
Let $M \in \Rep(\cI)_{\le r}$. Then the unit $M \to \Phi(\ul{\Xi}(M))=\Xi(M)$ is injective (Proposition~\ref{prop:xi}(c)) with cokernel $\Xi^{\res}(M)$ of level $<r$ (Proposition~\ref{prop:levelfunc}(d)). Thus $M \to \Phi(\ul{\Xi}(M))$ is an isomorphism in $\Rep(\cI)_r$.

Similarly, let $N \in \uRep(\cI)_{\le r}$. Then the co-unit $\ul{\Xi}(\Phi(N)) \to N$ is surjective (Proposition~\ref{prop:xi}(d)) with kernel $\ul{\Xi}^{\cor}(N)$ of level $<r$ (Proposition~\ref{prop:levelfunc}(e)). Thus $\ul{\Xi}(\Phi(N)) \to N$ is an isomorphism in $\uRep(\cI)_r$.
\end{proof}

\subsection{Saturation}
Fix a constraint word $\lambda=\lambda_1 \cdots \lambda_t$ of rank $r$ with $\lambda_t=\whitetri$, and put $\mu=\lambda_1\cdots \lambda_{t-1}$. We have $\Psi(\ul{\bE}^{\lambda}) \cong \bE^{\mu}$ (Proposition~\ref{prop:psistd}) and $\Gamma(\bE^{\mu}) \cong \ul{\bE}^{\lambda}$ (Proposition~\ref{prop:gammastd}). For a submodule $K$ of $\ul{\bE}^{\lambda}$, we define the {\bf saturation} of $K$, denoted $\Sat(K)$, to be $\Gamma(\Psi(K))$, which is naturally a submodule of $\ul{\bE}^{\lambda}$.

We now make this construction somewhat more explicit. Put $\bE^{\mu}_{\le n}=\sum_{k=0}^n \ul{\bE}^{\mu}_k$, regarded as a subspace of $\bE^{\mu}$. We define the {\bf degree} of a non-zero element $x \in \bE^{\mu}$, denoted $\deg(x)$, to be the minimal $n$ such that $x$ belongs to $\bE^{\mu}_{\le n}$. Define a linear map $\bE^{\mu}_{<n} \to \ul{\bE}^{\lambda}_n$, denoted $x \mapsto x \wedge e_n$, as follows:
\begin{displaymath}
e_{i_1,\ldots,i_{t-1}} \wedge e_n = e_{i_1,\ldots,i_{t-1},n}
\end{displaymath}
This map is clearly an isomorphism of vector spaces.

\begin{proposition}
Let $K$ be a homogeneous submodule of $\ul{\bE}^{\lambda}$.
\begin{enumerate}
\item $\Psi(K)$ consists of those elements $x \in \bE^{\mu}$ for which $x \wedge e_n \in K$ for some $n>\deg(x)$.
\item $\Sat(K)$ consists of all elements of the form $x \wedge e_n$ with $x \in \Psi(K)$ and $n>\deg(x)$.
\end{enumerate}
\end{proposition}

\begin{proof}
(a) Recall (Remark~\ref{rmk:psistd}) the explicit form of the isomorphism $\Psi(\ul{\bE}^{\lambda}) \to \bE^{\mu}$: the element of $\Psi(\ul{\bE}^{\lambda})$ represented by $e_{i_1,\ldots,i_t} \in \ul{\bE}^{\lambda}$ is mapped to $e_{i_1,\ldots,i_{t-1}} \in \bE^{\mu}$. A homogeneous element of $K$ has the form $x \wedge e_n$ for some $x \in \bE^{\mu}_{<n}$ and some $n$. The element of $\Psi(K)$ represented by $x \wedge e_n$ maps to $x$ under the isomorphism $\Psi(\ul{\bE}^{\lambda}) \to \bE^{\mu}$. This proves the claim.

(b) Again, recall (Remark~\ref{rmk:gammastd}) the isomorphism $\Gamma(\bE^{\mu}) = \ul{\bE}^{\lambda}$: the element $e_{i_1,\ldots,i_{t-1}}$ of $\Gamma(\bE^{\mu})_n=\bE^{\mu}_{<n}$ corresponds to the element $e_{i_1,\ldots,i_{t-1},n}=e_{i_1,\ldots,i_{t-1}} \wedge e_n$ of $\ul{\bE}^{\lambda}$. We thus see that $\Gamma(\Psi(K))_n$ consists of elements of those elements of the form $x \wedge e_n$ with $x \in \Psi(K) \cap \bE^{\mu}_{<n}$, which proves the claim.
\end{proof}

Let $N_k \subset \ul{\bE}^{\lambda}$ be the homogeneous submodule spanned by those basis vectors $e_{i_1,\ldots,i_t}$ with $i_t-i_{t-1} \ge k$.

\begin{proposition}
Let $\nu(j)=\lambda_1 \cdots \lambda_{t-1} \cdot \blacktri^j$. Then $\ul{\bE}^{\lambda}/N_k$ has a finite length filtration with graded pieces $\ul{\bE}^{\nu(j)}$ for $1 \le j \le k-1$. In particular, $\ul{\bE}^{\lambda}/N_k$ has homogeneous level $\le r-1$.
\end{proposition}

\begin{proof}
Note that $N_1=\ul{\bE}^{\lambda}$. The quotient $N_k/N_{k+1}$ has for a basis the set $\cB_+$ of those basis vectors $e_{i_1,\ldots,i_t}$ of $\ul{\bE}^{\lambda}$ with $i_t-i_{t-1}=k$. Moreover, we see that the set $\cB=\cB_+ \cup \{0\}$ is stable under the action of $\cI$. Thus $\cB$ is a pointed $\cI$-set and $N_k/N_{k+1} \cong \bk[\cB]$ is the associated monomial module. The map $\cE^{\nu(k)}_+ \to \cB_+$ given by $e_{i_1,\ldots,i_{t-1},i_t,\ldots,i_{t-1+k}} \to e_{i_1,\ldots,i_{t-1},i_{t-1+k}}$ is clearly a bijection. One easily sees that the induced bijection $\cE^{\nu(k)} \to \cB$ is an isomorphism of pointed $\cI$-sets, and so $N_k/N_{k+1}$ is isomorphic to $\ul{\bE}^{\nu(k)}$. The result now follows.
\end{proof}

\begin{proposition} \label{prop:bdsat}
Let $K \subset \ul{\bE}^{\lambda}$ be a homogeneous submodule. Then there exists $k$ such that $N_k \cap \Sat(K) \subset K$.
\end{proposition}

\begin{proof}
Let $x_1, \ldots, x_s$ be a Gr\"obner basis for $\Psi(K)$, as defined in \S \ref{ss:grgrob}. For each $1 \le i \le s$, let $d_i=\deg(x_i)$ and let $n_i>d_i$ be an integer such that $x_i \wedge e_{n_i}$ belongs to $K$. Let $k$ be the maximum value of $n_i-d_i$ for $1 \le i \le s$. We claim that $N_k \cap \Sat(K) \subset K$, which will prove the proposition.

To verify the claim, let $z$ be a given element of $N_k \cap \Sat(K)$. Write $z=y \wedge e_m$ where $y \in \Psi(K)$ and $m>\deg(y)$. Put $\delta=\deg(y)$. Since $z \in N_k$, we have $m-\delta \ge k$. Write $y = \sum_{i,j} c_{i,j} \sigma_{i,j}(x_i)$ where the $c_{i,j}$ are scalars and the $\sigma_{i,j}$ are elements of $\cI$ such that $\deg(\sigma_{i,j} x_i) \le \deg(y)$ for all $i$, $j$, which is possible by Proposition~\ref{prop:grobexp}. Now, if $\sigma'_{i,j}$ is an element of $\cI$ such that $\sigma'_{i,j}(n)=\sigma_{i,j}(n)$ for $1 \le n \le d_i$ then $\sigma'_{i,j} x_i=\sigma_{i,j} x$. Since $\sigma_{i,j}(d_i) \le \delta$ and $n_i-d_i \le k \le m-\delta$, we can find such $\sigma'_{i,j}$ with $\sigma'_{i,j}(n_i)=m$. Relabeling, we simply assume $\sigma_{i,j}(n_i)=m$. We thus have $\sigma_{i,j}(x_i \wedge e_{n_i})=\sigma_{i,j}(x_i) \wedge e_m$ for all $i$ and $j$, and so $\sum_{i,j} c_{i,j} \sigma_{i,j}(x_i \wedge e_{n_i})=y \wedge e_m=z$. Thus $z \in K$, as required.
\end{proof}

\begin{corollary} \label{cor:satlev}
Let $K \subset \ul{\bE}^{\lambda}$ be a homogeneous submodule. Then $\Sat(K)/K$ has homogeneous level $\le r-1$.
\end{corollary}

\begin{proof}
Let $k$ be such that $N_k \cap \Sat(K) \subset K$. Then the map $\Sat(K)/K \to \ul{\bE}^{\lambda}/N_k$ is injective. Since the target has homogeneous level $\le r-1$, so does the source, and the result follows.
\end{proof}

\subsection{Concatenation with $\ul{\bB}^1$} \label{ss:pi}

In this section, we study the operation of concatenating with the simple module $\ul{\bB}^1$. For brevity, for a smooth $\cI$-module $M$, we put $\Pi(M)=\ul{\bB}^1 \odot M$. For a graded $\cI$-module $M$, we put $\ul{\Pi}(M)=\ul{\bB}^1 \odot M$. We also let $\ul{\Pi}^{\dag}$ be the conjugation of $\ul{\Pi}$ by the transpose functor, i.e., $\ul{\Pi}^{\dag}=\ul{\Pi}(M^{\dag})^{\dag}$. Explicitly, we have $\ul{\Pi}^{\dag}(M)=M \odot \ul{\bB}^1$.

\begin{proposition} \label{prop:pisubmod}
Let $M$ be a smooth $\cI$-module. Then the map
\begin{displaymath}
\{ \text{$\cI$-submodules of $M$} \} \to \{ \text{$\cI$-submodules of $\Pi(M)$} \}, \qquad
N \mapsto \Pi(N)
\end{displaymath}
is a bijection. The analogous statements hold in the graded case for both $\ul{\Pi}$ and $\ul{\Pi}^{\dag}$.
\end{proposition}

\begin{proof}
Let $\xi$ be a generator for $\ul{\bB}^1$, so that $\Pi(M)=\xi \otimes M$. An $\cI$-submodule of $\Pi(M)$ has the form $\xi \otimes N$ for some subspace $N$ of $M$. To prove the proposition (in the smooth case), it suffices to show that $N$ is an $\cI$-submodule of $M$. But this is clear: for $x \in N$ and $i \ge 1$, we have $\alpha_{i+1}(\xi \otimes x)=\xi \otimes \alpha_i x$, which belongs to $\xi \otimes N$ by assumption; thus $\alpha_i x$ belongs to $N$, as required.

It is clear that if $M$ is a graded module then the above correspondence takes homogeneous submodules to homogeneous submodules, in each direction. Thus the analogous statement for $\ul{\Pi}$ holds. The one for $\ul{\Pi}^{\dag}$ follows simply by taking transposes.
\end{proof}

\begin{proposition} \label{prop:pilev}
For a smooth $\cI$-module $M$ we have $\lev{\Pi(M)} \le \lev{M}$. The analogous results hold in the graded case for both $\ul{\Pi}$ and $\ul{\Pi}^{\dag}$.
\end{proposition}

\begin{proof}
We must show that $\Pi$ maps the subcategory $\Rep(\cI)_{\le r}$ of $\Rep(\cI)$ into itself. Since $\Pi$ is exact and cocontinuous, it suffices by Proposition~\ref{prop:locmap} to check that $\Pi(\bE^{\lambda})$ has level $\le r$ for all standard modules $\bE^{\lambda}$ of level $\le r$. This follows from the explicit computation of the concatenation product on standard modules (Proposition~\ref{prop:stdcat}). The proof for $\ul{\Pi}$ and $\ul{\Pi}^{\dag}$ is similar.
\end{proof}

\subsection{The main theorem}

We now come to the main theorem on level:

\begin{theorem} \label{thm:level}
Let $\lambda$ be a constraint word of rank $r$. Then any proper quotient of $\bE^{\lambda}$ has level $<r$. Similarly, any proper homogeneous quotient of $\ul{\bE}^{\lambda}$ has homogeneous level $<r$.
\end{theorem}

Consider the following two statements:
\begin{itemize}
\item[$(\sA_r)$] If $\lambda$ is a constraint word of rank $\le r$ then any proper quotient of $\bE^{\lambda}$ has level $<r$.
\item[$(\sB_r)$] If $\lambda$ is a constraint word of rank $\le r$ then any proper homogeneous quotient of $\ul{\bE}^{\lambda}$ has homogeneous level $<r$.
\end{itemize}
Clearly, to prove the theorem it is enough to prove these statements for all $r$. We do this by induction on $r$. The statements $(\sA_0)$ and $(\sB_0)$ are clear, since standard modules of rank~0 are simple, in both the graded and ungraded case.

\begin{lemma}
We have $(\sB_r) \implies (\sA_r)$.
\end{lemma}

\begin{proof}
Let $\lambda$ be a constraint word of rank $r$ and let $K$ be a non-zero submodule of $\bE^{\lambda}$. Let $\epsilon \colon \bE^{\lambda} \to \Xi(\bE^{\lambda})$ and $\eta \colon \ul{\Xi}(\bE^{\lambda}) \to \ul{\bE}^{\lambda}$ be the unit and counit of the adjunction between $\ul{\Xi}$ and $\Phi$. As $\Phi(\eta) \circ \epsilon$ is the identity, we see that $\eta(\ul{\Xi}(K))$ contains $K$: indeed, for $x \in K$, the element $\eta(\epsilon(x))=x$ belongs to $\eta(\ul{\Xi}(K))$. In particular, we see that $\eta(\ul{\Xi}(K))$ is non-zero. We have an exact sequence
\begin{displaymath}
\ul{\Xi}^{\cor}(\bE^{\lambda}) \to \ul{\Xi}(\bE^{\lambda}/K) \to \ul{\bE}^{\lambda}/\eta(\ul{\Xi}(K)) \to 0
\end{displaymath}
By $(\sB_r)$, the right module has level $<r$; by Proposition~\ref{prop:levelfunc}(e), the left module has level $<r$. Thus $\ul{\Xi}(\bE^{\lambda}/K)$ has level $<r$. By Proposition~\ref{prop:levelfunc}(a), we see that $\Xi(\bE^{\lambda}/K)$ has level $<r$. Since $\bE^{\lambda}/K$ injects into this module, it too has level $<r$, which proves $(\sA_r)$.
\end{proof}

\begin{lemma}
We have $(\sA_r) \implies (\sB_{r+1})$.
\end{lemma}

\begin{proof}
Let $\lambda=\lambda_1 \cdots \lambda_t$ be a constraint word of rank $r+1$, and let $\mu=\lambda_1 \cdots \lambda_{t-1}$. We must show that any proper homogeneous quotient of $\ul{\bE}^{\lambda}$ has homogeneous level $\le r$.

First suppose that $\lambda_t=\whitetri$. Let $K$ be a non-zero homogeneous submodule of $\ul{\bE}^{\lambda}$. We show that $\ul{\bE}^{\lambda}/K$ has homogeneous level $\le r$. Let $K'=\Sat(K)$ be the saturation of $K$. We have an exact sequence
\begin{displaymath}
0 \to K'/K \to \ul{\bE}^{\lambda}/K \to \ul{\bE}^{\lambda}/K' \to 0.
\end{displaymath}
Since $K'/K$ has level $\le r$ (Corollary~\ref{cor:satlev}), it thus suffices to prove that $\ul{\bE}^{\lambda}/K'$ has level $\le r$. Consider the exact sequence
\begin{displaymath}
0 \to \Psi(K') \to \bE^{\mu} \to C \to 0
\end{displaymath}
By $(\sA_r)$, the level of $C$ is $<r$; note that $\Psi(K')$ is non-zero since $K'=\Gamma(\Psi(K'))$. Applying $\Gamma$ to the above sequence, we find
\begin{displaymath}
0 \to K' \to \ul{\bE}^{\lambda} \to \Gamma(C).
\end{displaymath}
Since $C$ has level $<r$, it follows that $\Gamma(C)$ has homogeneous level $\le r$ (Proposition~\ref{prop:gammalev}). Since $\ul{\bE}^{\lambda}/K'$ is a homogeneous submodule of $\Gamma(C)$, it too has homogeneous level $\le r$, and so the result follows.

Now suppose $\lambda_t=\blacktri$. We have $\ul{\bE}^{\lambda}=\ul{\Pi}^{\dag}(\ul{\bE}^{\mu})$. Suppose that $K$ is a non-zero submodule of $\ul{\bE}^{\lambda}$. Then by Proposition~\ref{prop:pisubmod}, we have $K=\ul{\Pi}^{\dag}(K')$ for some submodule $K'$ of $\ul{\bE}^{\mu}$, necessarily non-zero. Since the functor $\ul{\Pi}^{\dag}$ is exact, we see that $\ul{\bE}^{\lambda}/K=\ul{\Pi}^{\dag}(\ul{\bE}^{\mu}/K')$. Since $\mu$ is shorter than $\lambda$, we can assume by induction that $\ul{\bE}^{\mu}/K'$ has homogeneous level $\le r$. Since $\ul{\Pi}^{\dag}$ does not increase level (Proposition~\ref{prop:pilev}), we see that $\ul{\bE}^{\lambda}/K$ has homogeneous level $\le r$ as well, and the result follows.
\end{proof}

The above two lemmas establish $(\sA_r)$ and $(\sB_r)$ for all $r$, and so the theorem follows.

\subsection{Some consequences of the main theorem}

Theorem~\ref{thm:level} yields the following improvement of Proposition~\ref{prop:levfilt}, in which the word ``subquotient'' is changed to ``submodule.''

\begin{proposition} \label{prop:levfilt2}
A finitely generated smooth $\cI$-module $M$ has level $\le r$ if and only if there is a finite length filtration $0=F_0 \subset \cdots \subset F^n=M$ and constraint words $\lambda_1, \ldots, \lambda_n$ of rank $\le r$ such that $F^i/F^{i-1}$ is isomorphic to a \emph{submodule} of $\bE^{\lambda_i}$ for all $1 \le i \le n$. A similar statement holds in the graded case.
\end{proposition}

\begin{proof}
Obviously if $M$ admits such a filtration then it has level $\le r$, so let us prove the converse. Say that $M$ is {\bf $r$-good} if it admits a filtration $0=F_0 \subset \cdots \subset F^n=M$ with each $F^i/F^{i-1}$ isomorphic to a submodule of a rank $\le r$ standard module. We note that an extension of $r$-good modules is $r$-good. We must show that if $M$ has level $\le r$ then it is $r$-good. We proceed by induction on $r$. The case $r=0$ is clear. Suppose now that $r \ge 1$ and the satement has been proved for $r-1$.

First, suppose that $\lambda$ has rank $\le r$ and $M$ is a subquotient of $\bE^{\lambda}$, say $M=A/B$ with $B \subset A \subset \bE^{\lambda}$. If $B=0$ then $M=A$ is a submodule of $\bE^{\lambda}$, and obviously $r$-good. If $B \ne 0$ then $\bE^{\lambda}/B$ has level $<r$ by Theorem~\ref{thm:level}, and thus the submodule $M$ has level $<r$ as well. Thus, by the inductive hypothesis, $M$ is $(r-1)$-good, and therefore $r$-good as well.

Now let $M$ be an arbitrary finitely generated module of level $\le r$. By Proposition~\ref{prop:levfilt}, we have a filtration $0=F_0 \subset \cdots \subset F^n=M$ and constraint words $\lambda_1, \ldots, \lambda_n$ of rank $\le r$ such that $F^i/F^{i-1}$ is isomorphic to a subquotient of $\bE^{\lambda_i}$. By the previous paragraph, each $F^i/F^{i-1}$ is $r$-good. Therefore $M$ is $r$-good as well. This completes the proof.
\end{proof}

Recall that if $\cT$ is a triangulated category and $S$ is a collection of objects in $\cT$ then the triangulated subcategory of $\cT$ {\bf generated} by $S$ is the smallest triangulated subcategory of $\cT$ containing $S$. The following is the most important immediate consequence of Theorem~\ref{thm:level}:

\begin{theorem} \label{thm:dergen}
The standard modules generate $\rD^b_{\fgen}(\Rep(\cI))$ as a traingulated category. More precisely, the standard modules of rank $\le r$ generate $\rD^b_{\fgen}(\Rep(\cI)_{\le r})$ as a triangulated category, for all $r$. The analogous statements hold in the graded case.
\end{theorem}

\begin{proof}
Let $\cT_r$ be the triangulated subcategory of $\rD^b_{\fgen}(\Rep(\cI)_{\le r})$ generated by the standard modules. We must show $\cT_r=\rD^b_{\fgen}(\Rep(\cI)_{\le r})$. We proceed by induction on $r$. For $r<0$ the statement is vacuously true. Suppose now it has been proved for $r-1$, and let us prove it for $r$.

It suffices to show that $M \in \cT_r$ for all $M \in \Rep(\cI)_{\le r}^{\fgen}$. Thus let such $M$ be given. By Proposition~\ref{prop:levfilt2}, we have a filtration $0=F^0 \subset \cdots \subset F^n=M$ and constraint words $\lambda_1, \ldots, \lambda_n$ of rank $\le r$ such that $F^i/F^{i-1}$ is isomorphic to a subobject, say $A_i$, of $\bE^{\lambda_i}$. Of course, we may assume $A_i$ is non-zero, as otherwise we could shorten the filtration. Consider the exact sequence
\begin{displaymath}
0 \to A_i \to \bE^{\lambda_i} \to \bE^{\lambda_i}/A_i \to 0.
\end{displaymath}
By Theorem~\ref{thm:level}, the quotient $\bE^{\lambda_i}/A_i$ has rank $<r$. Thus, by the inductive hypothesis, it belongs to $\cT_{r-1}$, and therefore to $\cT_r$. Of course, $\bE^{\lambda_i}$ belongs to $\cT_r$ by definition. Thus $A_i$ belongs to $\cT_r$ as well. Since each $F^i/F^{i-1} \cong A_i$ belongs to $\cT_r$, we conclude that $M$ belongs to $\cT_r$.
\end{proof}

\begin{proposition} \label{prop:levcat1}
Every object of $\Rep(\cI)_r$ is locally of finite length. Moreover, if $\lambda$ has rank $r$ then $T_{\ge r}(\bE^{\lambda})$ is simple in this category (assuming it is non-zero), and all simple objects are isomorphic to one of this form. The analogous statements hold in the graded case.
\end{proposition}

\begin{proof}
By Theorem~\ref{thm:level}, any non-trivial quotient of $T_{\ge r}(\bE^{\lambda})$ is zero, and so $T_{\ge r}(\bE^{\lambda})$ is either zero or simple. By Propostion~\ref{prop:levfilt2}, we see that every finitely generated object of $\Rep(\cI)_r$ admits a finite filtration with graded pieces of the form $T_{\ge r}(\bE^{\lambda})$. This shows that every finitely generated object is of finite length, and that every simple object is isomorphic to one of the form $T_{\ge r}(\bE^{\lambda})$. The proposition follows.
\end{proof}

\begin{remark}
Proposition~\ref{prop:levcat2} offers some improvements on Proposition~\ref{prop:levcat1}.
\end{remark}

\section{Grothendieck groups, Hilbert series, and Krull dimension} \label{s:groth}

\subsection{Grothendieck groups} \label{ss:groth}

Let $\rK(\cI)$ denote the Grothendieck group of $\Rep(\cI)^{\fgen}$. Similarly, let $\ul{\rK}(\cI)$ denote the Grothendieck group of $\uRep(\cI)^{\fgen}$. We let $[M]$ denote the class of the module $M$ in the relevant Grothendieck group. We have several homomorphisms between these groups:
\begin{align*}
\phi, \psi &\colon \ul{\rK}(\cI) \to \rK(\cI) &
\sigma &\colon \rK(\cI) \to \rK(\cI) \\
\ul{\xi}, \ul{\gamma} &\colon \rK(\cI) \to \ul{\rK}(\cI) &
(-)^{\dag}, \ul{\sigma} &\colon \ul{\rK}(\cI) \to \ul{\rK}(\cI).
\end{align*}
With the exception of $\ul{\gamma}$, these maps are simply induced by the corresponding functors: for example, $\phi([M])=[\Phi(M)]$. For $\ul{\gamma}$ we also have to use the derived functors:
\begin{displaymath}
\ul{\gamma}([M])=\sum_{i \ge 0} (-1)^i [\rR^i \Gamma(M)].
\end{displaymath}
This is well-defined by Proposition~\ref{prop:gammafinite}. The following is our main theorem on Grothendieck groups:

\begin{theorem} \label{thm:groth}
The classes of the standard modules $\bE^{\lambda}$ form an integral basis for $\rK(\cI)$. Similarly, the classes of the standard modules $\ul{\bE}^{\lambda}$ form an integral basis for $\ul{\rK}(\cI)$.
\end{theorem}

\begin{proof}
The classes of the modules $\bE^{\lambda}$ span $\rK(\cI)$ by Theorem~\ref{thm:dergen}, and similarly in the graded case. We now prove linear independence. We begin by treating the graded case. We establish linear independence by constructing sufficiently many functionals on $\ul{\rK}(\cI)$ to distinguish the classes in question.

Let $\pi \colon \ul{\rK}(\cI) \to \ul{\rK}(\cI)$ be the endomorphism given by $\pi(x) = \ul{\gamma}(\psi(x^{\dag}))^{\dag}$. Appealing to previous computations (Propositions~\ref{prop:transstd}, \ref{prop:psistd}, and~\ref{prop:gammastd}, and Corollary~\ref{prop:Rgammastd}), we have
\begin{displaymath}
\pi([\ul{\bE}^{\lambda}]) = \begin{cases}
[\ul{\bE}^{\lambda}] & \text{if $\lambda_1=\whitetri$} \\
0 & \text{if $\lambda=\emptyset$ or $\lambda_1=\blacktri$} \end{cases}
\end{displaymath}
Define endomorphisms $\beta^{\whitetri}, \beta^{\blacktri} \colon \ul{\rK}(\cI) \to \ul{\rK}(\cI)$ by
\begin{displaymath}
\beta^{\whitetri} = (\ul{\sigma}-1) \circ \pi, \qquad
\beta^{\blacktri} = \ul{\sigma} \circ (1-\pi).
\end{displaymath}
For a non-empty constraint word $\lambda=\lambda_1 \cdots \lambda_t$, let $\tau(\lambda)=\lambda_2 \cdots \lambda_t$. Appealing to Proposition~\ref{prop:shiftstd}, we find
\begin{displaymath}
\beta^{\whitetri}([\ul{\bE}^{\lambda}]) = \begin{cases}
[\ul{\bE}^{\tau(\lambda)}] & \text{if $\lambda_1=\whitetri$} \\
0 & \text{if $\lambda=\emptyset$ or $\lambda_1=\blacktri$} \end{cases}
\qquad
\beta^{\blacktri}([\ul{\bE}^{\lambda}]) = \begin{cases}
[\ul{\bE}^{\tau(\lambda)}] & \text{if $\lambda_1=\blacktri$} \\
0 & \text{if $\lambda=\emptyset$ or $\lambda_1=\whitetri$} \end{cases}
\end{displaymath}
Define $\beta^{\lambda} \colon \ul{\rK}(\cI) \to \ul{\rK}(\cI)$ to be the endomorphism $\beta^{\lambda_t} \circ \cdots \circ \beta^{\lambda_1}$; if $\lambda=\emptyset$ then $\beta^{\lambda}=\id$. For a constraint word $\mu=\mu_1 \cdots \mu_s$, we have
\begin{displaymath}
\beta^{\lambda}([\ul{\bE}^{\mu}]) = \begin{cases}
[\ul{\bE}^{\tau^t(\mu)}] & \text{if $s \ge t$ and $\mu_1 \cdots \mu_t = \lambda$} \\
0 & \text{otherwise} \end{cases}
\end{displaymath}
Let $\rho_{\emptyset} \colon \ul{\rK}(\cI) \to \bZ$ be the functional given by $\rho_{\emptyset}([M])=\dim M_0$. Note that $\rho_{\emptyset}([\ul{\bE}^{\lambda}])=\delta_{\lambda,\emptyset}$. For a constraint word $\lambda$, define a functional $\rho_{\lambda} \colon \ul{\rK}(\cI) \to \bZ$ by $\rho_{\lambda}=\rho_{\emptyset} \circ \beta^{\lambda}$. Appealing to the above computation of $\beta^{\lambda}$, we find
\begin{displaymath}
\rho_{\lambda}([\ul{\bE}^{\mu}]) = \delta_{\lambda,\mu}.
\end{displaymath}
From the existence of these functionals, we see that the classes $[\ul{\bE}^{\lambda}]$ are linearly independent. Indeed, suppose that $\sum_{\lambda} a_{\lambda} [\ul{\bE}^{\lambda}]=0$ is a linear relation then. Applying $\rho_{\lambda}$, we find $a_{\lambda}=0$. As this holds for all $\lambda$, the relation is trivial.

Linear independence of the classes $[\bE^{\lambda}]$ in $\rK(\cI)$ now follows from properties of the completion functor. Precisely, by Proposition~\ref{prop:xistd}, the map $\ul{\xi} \colon \rK(\cI) \to \ul{\rK}(\cI)$ is upper-triangular, in the sense that $\ul{\xi}([\bE^{\lambda}])$ is equal to $[\ul{\bE}^{\lambda}]$ plus lower order terms (i.e., classes of the form $[\ul{\bE}^{\mu}]$ where $\mu$ has smaller rank than $\lambda$). Since the $[\ul{\bE}^{\lambda}]$ are linearly independent, it follows that the $[\bE^{\lambda}]$ are as well.
\end{proof}

\begin{corollary} \label{cor:groth}
The forgetful homomorphism $\phi \colon \ul{\rK}(\cI) \to \rK(\cI)$ is an isomorphism.
\end{corollary}

\begin{proof}
As $\phi([\ul{\bE}^{\lambda}])=[\bE^{\lambda}]$, we see that $\phi$ maps a basis to a basis.
\end{proof}

Let $\rK(\cI)_{\le r}$ be the Grothendieck group of the category $\Rep(\cI)^{\fgen}_{\le r}$, and similarly define $\ul{\rK}(\cI)_{\le r}$.

\begin{proposition} \label{prop:groth-bd}
We have the following:
\begin{enumerate}
\item The classes $[\bE^{\lambda}]$ with $\lambda$ of rank $\le r$ form a basis of $\rK(\cI)_{\le r}$.
\item The natural map $\rK(\cI)_{\le r} \to \rK(\cI)$ is injective.
\end{enumerate}
The analogous results hold in the graded case.
\end{proposition}

\begin{proof}
Theorem~\ref{thm:dergen} shows that the classes $[\bE^{\lambda}]$ with $\lambda$ of rank $\le r$ span $\rK(\cI)_{\le r}$. Since these classes map to linear independent classes in $\rK(\cI)$ by Theorem~\ref{thm:groth}, we see that they must be linearly independent in $\rK(\cI)_{\le r}$ and that the map is an injection.
\end{proof}

\begin{proposition} \label{prop:groth-lev}
Let $M$ be a non-zero smooth $\cI$-module of level $r$, and let $[M]=\sum_{\lambda} c_{\lambda} [\bE^{\lambda}]$ be the expression for $[M]$ in terms of the basis $[\bE^{\lambda}]$. Then:
\begin{enumerate}
\item If $\lambda$ has rank $>r$ then $c_{\lambda}=0$.
\item If $\lambda$ has rank $r$ then $c_{\lambda} \ge 0$.
\item There exists some $\lambda$ of rank $r$ such that $c_{\lambda}>0$.
\end{enumerate}
The analogous results hold in the graded case.
\end{proposition}

\begin{proof}
By Proposition~\ref{prop:levfilt2}, we can find a filtration $0=F^0 \subset \cdots \subset F^n=M$ and constraint words $\lambda_1, \ldots, \lambda_n$ of rank $\le r$ such that $F^i/F^{i-1}$ is isomorphic to a submodule of $\bE^{\lambda_i}$ for each $i$. Obviously, we can assume that each $F_i/F_{i-1}$ is non-zero. Suppose $\lambda_i$ has rank $r$. Choose an injection $F^i/F^{i-1} \to \bE^{\lambda_i}$, and let $C$ be the cokernel. By Theorem~\ref{thm:level}, $C$ has level $<r$, and so $[C]$ belongs to $\rK(\cI)_{<r}$. We see that $[F^i/F^{i-1}]=[\ul{\bE}^{\lambda}]$ modulo $\rK(\cI)_{<r}$. Thus
\begin{displaymath}
[M]=\sum_{\rank(\lambda_i)=r} [\bE^{\lambda_i}] \mod{\rK(\cI)_{<r}}
\end{displaymath}
There must be some $\lambda_i$ of rank $r$, as otherwise $M$ would have level $<r$, and so the sum above is non-empty. The statements of the proposition thus follow.
\end{proof}

\begin{corollary} \label{cor:groth-lev}
Let $M$ be a finitely generated smooth $\cI$-module. Then $M$ has level $\le r$ if and only if $[M]$ belongs to $\rK(\cI)_{\le r}$. The analogous statement holds in the graded case.
\end{corollary}

\begin{corollary}
Let $M$ be a graded $\cI$-module. Then $M$ has homogeneous level $\le r$ if and only if $\Phi(M)$ has level $\le r$.
\end{corollary}

\begin{proof}
This follows from Corollary~\ref{cor:groth-lev} and the fact that the isomorphism $\phi \colon \ul{\rK}(\cI) \to \rK(\cI)$ maps $\ul{\rK}(\cI)_{\le r}$ onto $\rK(\cI)_{\le r}$.
\end{proof}

\begin{proposition} \label{prop:levcat2}
We have the following:
\begin{enumerate}
\item If $\lambda$ is a constraint word of rank $r$ then $T_{\ge r}(\bE^{\lambda})$ is a simple object of $\Rep(\cI)_r$.
\item If $\mu$ is a second constraint word of rank $r$ such that $T_{\ge r}(\bE^{\lambda})$ is isomorphic to $T_{\ge r}(\bE^{\mu})$ then $\lambda=\mu$.
\item The simple objects of $\Rep(\cI)_r$ are exactly the objects $T_{\ge r}(\bE^{\lambda})$ with $\lambda$ of rank $r$.
\end{enumerate}
Moreover, the analogous results hold in the graded case.
\end{proposition}

\begin{proof}
(a) We have already shown (Proposition~\ref{prop:levcat1}) that $T_{\ge r}(\bE^{\lambda})$ is either zero or simple, so it suffices now to show that it is non-zero. Suppose that it is zero. It follows that $\bE^{\lambda}$ belongs to the kernel of the functor $T_{\ge r}$, i.e., the category $\Rep(\cI)_{<r}$. Thus $[\bE^{\lambda}]$ belongs to $\rK(\cI)_{<r} \subset \rK(\cI)$. However, the former is spanned by classes of the form $[\bE^{\mu}]$ with $\mu$ of rank $<r$, and so this would contradict Theorem~\ref{thm:groth}.

(b) Suppose that $T_{\ge r}(\bE^{\lambda})$ and $T_{\ge r}(\bE^{\mu})$ are isomorphic. We can thus find a subobjects $M \subset \bE^{\lambda}$ and $N \subset \bE^{\mu}$ and a map $f \colon M \to \bE^{\mu}/N$ such that $\bE^{\lambda}/M$, $N$, $\ker(f)$, and $\coker(f)$ belong to $\Rep(\cI)_{<r}$. We have the relation
\begin{displaymath}
[\bE^{\lambda}]=[\bE^{\mu}]+[\ker{f}]-[\coker{f}]+[\bE^{\lambda}/M]-[N]
\end{displaymath}
in $\rK(\cI)$. As the terms on the right, other than the first, belong to $\rK(\cI)_{<r}$, they can be expressed in terms of the classes $[\bE{\nu}]$ with $\nu$ of rank $<r$. By Theorem~\ref{thm:groth}, we conclude that $\lambda=\mu$.

(c) This follows from what we have already proved in Proposition~\ref{prop:levcat1}.
\end{proof}

The concatenation product $\odot$ on $\uRep(\cI)$ is a monoidal structure that is exact in each argument, and preserves finite generation. It thus endows $\ul{\rK}(\cI)$ with the structure of an associative and unital (but not necessarily commutative) ring. Moreover, the action of $\uRep(\cI)$ on $\Rep(\cI)$ via concatenation product endows $\rK(\cI)$ with the structure of a left module over the ring $\ul{\rK}(\cI)$. We now determine this additional structure on the Grothendieck groups:

\begin{proposition} \label{prop:grothring}
We have the following:
\begin{enumerate}
\item Letting $\bZ\{a,b\}$ denote the non-commutative polynomial ring in the variables $a$ and $b$, the unique ring homomorphism $f \colon \bZ\{a,b\} \to \ul{\rK}(\cI)$ satisfying $f(a)=[\ul{\bA}^1]$ and $f(b)=[\ul{\bB}^1]$ is an isomorphism.
\item As a $\ul{\rK}(\cI)$-module, $\rK(\cI)$ is free of rank one with basis $[\bB^0]$.
\end{enumerate}
\end{proposition}

\begin{proof}
(a) By Proposition~\ref{prop:stdcat}, we see that $f$ takes a word $\lambda$ in $a$ and $b$ to $[\ul{\bE}^{\lambda}]$. Thus, by Theorem~\ref{thm:groth}, $f$ maps a basis of $\bZ\{a,b\}$ to a basis of $\ul{\rK}(\cI)$, and is therefore an isomorphism.

(b) It suffices to show that the map $g \colon \ul{\rK}(\cI) \to \rK(\cI)$ defined by $g(x)=x \cdot [\bB^0]$ is an isomorphism. We have $g([M])=[M] \cdot [\bB^0]=[M \odot \bB^0]=[\Phi(M)]$. Thus $g=\phi$, which we have already shown to be an isomorphism (Corollary~\ref{cor:groth}).
\end{proof}

\subsection{Hilbert series} \label{ss:hilb}

Let $M$ be a finitely generated graded $\cI$-module. We define the {\bf Hilbert series} of $M$ by
\begin{displaymath}
\ul{\rH}_M(t) = \sum_{n \ge 0} \umu_n(M) t^n = \sum_{n \ge 0} \dim(M_n) t^n.
\end{displaymath}
Now suppose that $M$ is a finitely generated smooth $\cI$-module. We define the {\bf Hilbert series} of $M$ by
\begin{displaymath}
\rH_M(t) = \sum_{n \ge 0} \mu_n(M) t^n.
\end{displaymath}

\begin{proposition} \label{prop:Xihilb}
Let $M$ be a finitely generated smooth $\cI$-module. Then $\rH_M(t)=\ul{\rH}_{\ul{\Xi}(M)}(t)$.
\end{proposition}

\begin{proof}
This follows immediately from Proposition~\ref{prop:Ximult}.
\end{proof}

\begin{proposition} \label{prop:hilbcat}
Let $M$ and $N$ be finitely generated graded $\cI$-modules. Then
\begin{displaymath}
\ul{\rH}_{M \odot N}(t)=\ul{\rH}_M(t) \cdot \ul{\rH}_N(t).
\end{displaymath}
\end{proposition}

\begin{proof}
This follows immediately from the definition of the concatenation product.
\end{proof}

\begin{proposition} \label{prop:hilbstd}
Let $\lambda$ be a constraint word with $n$ (resp.\ $m$) $a$ (resp.\ $b$). Then $\ul{\rH}_{\ul{\bE}^{\lambda}}(t)=t^{n+m} (1-t)^{-n}$.
\end{proposition}

\begin{proof}
The result is clear if $\lambda=\blacktri$ or $\lambda=\whitetri$. The general case follows from these two cases and Proposition~\ref{prop:hilbcat} (as well as Corollary~\ref{cor:gapprod}).
\end{proof}

\begin{theorem} \label{thm:hilb-gr}
Let $M$ be a finitely generated graded $\cI$-module. Then $\ul{\rH}_M(t)$ is a rational function of $t$ whose denominator is a power of $1-t$. Moreover, the level of $M$ is exactly the order of the pole of $\ul{\rH}_M(t)$ at $t=1$. The analogous result holds in the smooth case.
\end{theorem}

\begin{proof}
In the graded case, this follows immediately from Propositions~\ref{prop:hilbstd} and~\ref{prop:groth-lev}. The smooth case follows from the graded case and Proposition~\ref{prop:Xihilb}, and the observation that $\lev(M)=\ulev(\ul{\Xi}(M))$. (To see this, observe that $\ulev(\ul{\Xi}(M)) \le \lev(M)$ by Proposition~\ref{prop:levelfunc}(c). On the other hand,
\begin{displaymath}
\lev(M) \le \lev(\Xi(M)) = \lev(\Phi(\ul{\Xi}(M))) \le \lev(\ul{\Xi}(M))
\end{displaymath}
where the first inequality is a consequence of the injection $M \to \Xi(M)$, and the second of Proposition~\ref{prop:levelfunc}(a).)
\end{proof}

\begin{corollary}
Let $M$ be a non-zero finitely generated graded $\cI$-module. Then there exists a polynomial $p \in \bQ[t]$ such that $\umu_n(M)=p(n)$ for all $n \gg 0$. Moreover, the degree of $p$ is $\ulev(M)-1$ (using the convention that the zero polynomial has degree $-1$). The analogous result holds in the smooth case.
\end{corollary}

\subsection{Krull dimension}

Let $\cA$ be a Grothendieck abelian category. Recall that there is a notion of Krull dimension for objects in $\cA$, defined as follows. Define $\cA^{\le -1}=0$. Having defined $\cA^{\le r-1}$, let $\cA^{\le r}$ be the full subcategory of $\cA$ spanned by objects that are locally of finite length in the quotient $\cA/\cA^{\le r-1}$. We say that an object $M$ of $\cA$ has {\bf Krull dimension} $r$ if $M$ belongs to $\cA^{\le r}$, and $r$ is minimal with this property.

The following proposition characterizes Krull dimension for $\cI$-modules:

\begin{proposition} \label{prop:krull}
Let $M$ be a smooth $\cI$-module. Then $M$ has Krull dimension $\le r$ if and only if it has level $\le r$. The analogous statement in the graded case holds as well.
\end{proposition}

\begin{proof}
Let $\cA=\Rep(\cI)$, let $\cA^{\le r}$ be defined as in the previous paragraph, and let $\cA_{\le r}=\Rep(\cI)_{\le r}$ be the category of objects of level $\le r$. We must show $\cA^{\le r}=\cA_{\le r}$ for all $r$. We proceed by induction on $r$. For $r=-1$ the statement is vacuously true. Thus, assuming the statement for $r-1$, let us prove it for $r$. We must show that a smooth $\cI$-module has level $\le r$ if and only if the image of $M$ in $\cA/\cA_{\le r-1}$ is locally of finite length. If $M$ has level $\le r$ then it is locally finite in $\cA/\cA_{\le r-1}$ by Proposition~\ref{prop:levcat1}. Suppose now that $M$ is locally finite in $\cA/\cA_{\le r-1}$. Let $N$ be a finitely generated submodule of $M$. Thus $N$ has finite length in $\cA/\cA_{\le r-1}$. Since the simple objects of $\cA/\cA_{\le r-1}$ are exactly the localizations of the standard modules of rank $r$ (Proposition~\ref{prop:levcat2}), it follows that $[N]$ is a sum of classes of such standard modules, modulo $\rK(\cI)_{\le r-1}$. Thus $N$ has level $\le r$ by Corollary~\ref{cor:groth-lev}. Since this holds for all finitely generated submodules $N$ of $M$, we see that $M$ has level $\le r$ as well.
\end{proof}

\section{Induction and coinduction} \label{s:coinduction}

\subsection{Coinduction} \label{ss:coinduction}

Let $M$ be a $\cI$-module. We define a new $\cI$-module $\sC(M)$, called the {\bf coinduction} of $M$, as follows. As a vector space, $\sC(M)$ is $[M]_0 \oplus [M]_1$, where $[M]_i$ is just a copy of $M$. For $x \in M$, we denote by $[x]_i$ the corresponding element of $[M]_i$. The $\cI$-action is defined as follows:
\begin{align*}
\alpha_1 [x]_0 &= [\alpha_1 x]_0 + [x]_1 &
\alpha_i [x]_0 &= [\alpha_i x]_0 \\
\alpha_1 [x]_1 &= 0 &
\alpha_i [x]_1 &= [\alpha_{i-1} x]_1
\end{align*}
where $i \ge 2$. We leave to the reader the simple verification that $\sC(M)$ is a well-defined $\cI$-module. It is clearly smooth if $M$ is. If $M$ is a graded $\cI$-module then we give $\sC(M)$ a grading as follows: for $x \in M_n$, we declare $[x]_0$ to have degree $n$ and $[x]_1$ to have degree $n+1$. We leave to the reader the simple verification that this indeed defines a graded $\cI$-module. We write $\ul{\sC}(M)$ for this graded $\cI$-module. We have thus defined functors
\begin{displaymath}
\sC \colon \Rep(\cI) \to \Rep(\cI), \qquad
\ul{\sC} \colon \uRep(\cI) \to \uRep(\cI).
\end{displaymath}
It is clear that both are exact and cocontinuous. The rationale behind the definition of $\sC$ is explained in \S \ref{ss:indexp}.

Recall from \S \ref{ss:pi} that, for a smooth $\cI$-module $M$ we have defined $\Pi(M)=\ul{\bB}^1 \odot M$, and for a graded $\cI$-module $M$ we have defined $\ul{\Pi}(M)=\ul{\bB}^1 \odot M$. We also write $\xi$ for a basis element of $\bB^1$.

\begin{proposition} \label{prop:coindseq}
Let $M$ be a smooth $\cI$-module. Then $[M]_1$ is an $\cI$-submodule of $\sC(M)$, and the map $\Pi(M) \to [M]_1$ given by $\xi \odot x \mapsto [x]_1$ is an isomorphism of $\cI$-modules. Similarly, $[M]_0$ is a quotient $\cI$-module of $\sC(M)$, and the map $M \to [M]_0$ given by $x \mapsto [x]_0$ is an isomorphism of $\cI$-modules. We thus have a canonical exact sequence
\begin{displaymath}
0 \to \Pi(M) \to \sC(M) \to M \to 0.
\end{displaymath}
The analogous statement holds in the graded case.
\end{proposition}

\begin{proof}
This is clear.
\end{proof}

\begin{corollary} \label{cor:coindfin}
If $M$ is a finitely generated (resp.\ finite length) smooth $\cI$-module then $\sC(M)$ is also finitely generated (resp.\ finite length). The analogous statement holds in the graded case.
\end{corollary}

The main interest in coinduction stems from the following proposition:

\begin{proposition} \label{prop:coind}
The functor $\sC$ is naturally the right adjoint to the shift functor $\Sigma$. The analogous statement holds in the graded case.
\end{proposition}

\begin{proof}
Define $\epsilon \colon M \to \sC(\Sigma(M))$ by $\epsilon(x)=[(\alpha_1 x)^{\flat}]_0+[x^{\flat}]_1$. We verify that this is $\cI$-linear. We have
\begin{displaymath}
\epsilon(\alpha_1 x)
=[(\alpha_1^2 x)^{\flat}]_0+[(\alpha_1 x)^{\flat}]_1.
\end{displaymath}
and
\begin{displaymath}
\alpha_1 \epsilon(x) = [\alpha_1 (\alpha_1 x)^{\flat}]_0 + [(\alpha_1 x)^{\flat}]_1.
\end{displaymath}
These are equal, as $\alpha_1 (\alpha_1 x)^{\flat}=(\alpha_2 \alpha_1 x)^{\flat}=(\alpha_1^2 x)^{\flat}$. For $i \ge 2$, we have
\begin{displaymath}
\epsilon(\alpha_i x)
=[(\alpha_1 \alpha_i x)^{\flat}]_0+[(\alpha_i x)^{\flat}]_1
=[\alpha_i (\alpha_1 x)^{\flat}]_0+[\alpha_{i-1} x^{\flat}]_1
=\alpha_i \epsilon(x)
\end{displaymath}
Thus $\epsilon$ is $\cI$-linear.

Define $\eta \colon \Sigma(\sC(M)) \to M$ by $\eta([x]_0^{\flat})=0$ and $\eta([x]_1^{\flat})=x$. We have
\begin{displaymath}
\eta(\alpha_i [x]_0^{\flat})=\eta( (\alpha_{i+1} [x]_0)^{\flat}) = \eta([\alpha_{i+1} x]_0^{\flat})=0=\alpha_i \eta([x]_0^{\flat}).
\end{displaymath}
and
\begin{displaymath}
\eta(\alpha_i [x]_1^{\flat})=\eta((\alpha_{i+1} [x]_1)^{\flat})=\eta([\alpha_i x]_1^{\flat})=\alpha_i x = \alpha_i \eta([x]_1^{\flat}).
\end{displaymath}
Thus $\eta$ is $\cI$-linear.

We leave to the reader the simple verification that $\epsilon$ and $\eta$ satisfy the necessary conditions to be the unit and counit of an adjunction. We note that if $M$ is a graded $\cI$-module then $\epsilon$ and $\eta$ are homogeneous, and thus give an adjunction between the graded functors.
\end{proof}

\begin{corollary} \label{cor:coind}
The functor $\sC \colon \Rep(\cI) \to \Rep(\cI)$ takes injective objects to injective objects. The analogous statement holds in the graded case.
\end{corollary}

\begin{proposition} \label{prop:coindcat}
Let $M$ be a graded $\cI$-module and let $N$ be a smooth $\cI$-module. Then we have a natural isomorphism
\begin{displaymath}
\sC(M \odot N) \cong \left[ \ul{\sC}(M) \odot N \right] \oplus \left[ M_0 \otimes \sC(N) \right].
\end{displaymath}
If $N$ is graded then this isomorphism is homogeneous.
\end{proposition}

\begin{proof}
If $M$ is concentrated in degree~0, the isomorphism is clear. It thus suffices to treat the case where $M$ is pure, which we assume from now on. We have
\begin{displaymath}
\sC(M \odot N) = [M \odot N]_0 \oplus [M \odot N]_1, \qquad
\ul{\sC}(M) \odot N = ([M]_0 \oplus [M]_1) \odot N.
\end{displaymath}
Define
\begin{displaymath}
f \colon \sC(M \odot N) \to \ul{\sC}(M) \odot N, \qquad
f([x \odot y]_{\epsilon})=[x]_{\epsilon} \odot y
\end{displaymath}
for $\epsilon \in \{0,1\}$. This is clearly an isomorphism of graded vector spaces, and if $N$ is graded then it is homogeneous. We show that it is $\cI$-linear.

Suppose $x \in M_i$ and $y \in M_j$ so that $x \odot y$ has degree $n=i+j$. Let $k \ge 1$ and $\epsilon \in \{0,1\}$. We must show
\begin{equation} \label{eq:indcat}
f(\alpha_k \cdot [x \odot y]_{\epsilon}) = \alpha_k \cdot f([x \odot y]_{\epsilon}).
\end{equation}
We proceed in four cases.

\textit{\textbf{Case 1:} $\epsilon=0$ and $k=1$.} We have
\begin{displaymath}
\alpha_1 \cdot [x \odot y]_0
= [\alpha_1 \cdot (x \odot y)]_0 + [x \odot y]_1
= [(\alpha_1 x) \odot y]_0 + [x \odot y]_1.
\end{displaymath}
In the second step, we used that $i \ge 1$ because $M$ is pure. On the other hand,
\begin{displaymath}
\alpha_1 \cdot ([x]_0 \odot y)
=(\alpha_1 \cdot [x]_0) \odot y
=([\alpha_1 x]_0+[x]_1) \odot y,
\end{displaymath}
where, again, we used $i \ge 1$ in the first step. Thus \eqref{eq:indcat} follows.

\textit{\textbf{Case 2:} $\epsilon=0$ and $k \ge 2$.} We have
\begin{displaymath}
\alpha_k \cdot [x \odot y]_0
=[\alpha_k \cdot (x \odot y)]_0
=\begin{cases}
[(\alpha_k x) \odot y]_0 & \text{if $k \le i$} \\
[x \odot (\alpha_{k-i} y)]_0 & \text{if $k \ge i+1$}
\end{cases}
\end{displaymath}
On the other hand,
\begin{displaymath}
\alpha_k \cdot ([x]_0 \odot y)
= \begin{cases}
(\alpha_k [x]_0) \odot y & \text{if $k \le i$} \\
[x]_0 \odot (\alpha_{k-i} y) & \text{if $k \ge i+1$}
\end{cases}
\end{displaymath}
As $\alpha_k [x]_0=[\alpha_k x]_0$, \eqref{eq:indcat} follows.

\textit{\textbf{Case 3:} $\epsilon=1$ and $k=1$.} We have
\begin{displaymath}
\alpha_k \cdot [x \odot y]_1 = 0,
\end{displaymath}
and
\begin{displaymath}
\alpha_k \cdot ([x]_1 \odot y) = (\alpha_k [x]_1) \odot y = 0.
\end{displaymath}
Thus \eqref{eq:indcat} holds.

\textit{\textbf{Case 4:} $\epsilon=1$ and $k \ge 2$.} We have
\begin{displaymath}
\alpha_k \cdot [x \odot y]_1
= [\alpha_{k-1} \cdot (x \odot y)]_1
= \begin{cases}
[(\alpha_{k-1} x) \odot y]_1 & \text{if $k-1 \le i$} \\
[x \odot (\alpha_{k-1-i} y)]_1 & \text{if $k-1 \ge i+1$}
\end{cases}
\end{displaymath}
On the other hand,
\begin{displaymath}
\alpha_k \cdot ([x]_1 \odot y)
= \begin{cases}
(\alpha_k [x]_1) \odot y & \text{if $k \le i+1$} \\
[x]_1 \odot (\alpha_{k-i-1} y) & \text{if $k \ge i+2$.}
\end{cases}
\end{displaymath}
Here we have used that $[x]_1$ has degree $i+1$. Since $\alpha_k [x]_1=[\alpha_{k-1} x]_1$, we see that \eqref{eq:indcat} holds.
\end{proof}

\subsection{Splitting coinduction}

Let $M$ be a smooth $\cI$-module. We say that $\sC(M)$ {\bf splits} if the exact sequence of Proposition~\ref{prop:coindseq} splits, and make the analogous definition in the graded case. We note that $\ul{\sC}(M)$ splits if and only if $\sC(\Phi(M))$ splits, by Proposition~\ref{prop:phi-split}. Thus we will confine our investigation of splitting to the smooth case.

Let $\wt{\cI}$ be the monoid generated by $\cI$ and an additional element $\beta_1$ subject to the relations $\beta_1 \alpha_1=1$ and $\beta_1 \alpha_k = \alpha_{k-1} \beta_1$ for $k \ge 2$.

\begin{proposition}
Let $M$ be a smooth $\cI$-module. Then $\sC(M)$ splits if and only if the action of $\cI$ on $M$ can be extended to an action of $\wt{\cI}$.
\end{proposition}

\begin{proof}
Suppose that $\phi \colon M \to \sC(M)$ is a splitting of the canonical surjection $\sC(M) \to M$. We thus have $\phi(x)=[x]_0+[\beta_1 x]_1$ for some unique $\beta_1 x \in M$. Suppose $k \ge 2$. Then
\begin{displaymath}
\alpha_k \phi(x)=[\alpha_k x]_0+[\alpha_{k-1} \beta_1 x]_1
\end{displaymath}
while
\begin{displaymath}
\phi(\alpha_k x)=[\alpha_k x]_0+[\beta_1 \alpha_k x]_1.
\end{displaymath}
Since $\phi$ is $\cI$-linear, the two right sides above agree, and so $\beta_1 \alpha_k x=\alpha_{k-1} \beta_1 x$. Similarly,
\begin{displaymath}
\alpha_1 \phi(x) = [\alpha_1 x]_0+[x]_1,
\end{displaymath}
while
\begin{displaymath}
\phi(\alpha_1 x)=[\alpha_1 x]_0+[\beta_1 \alpha_1 x]_1,
\end{displaymath}
and so $\beta_1 \alpha_1 x=x$. We have thus extended the action of $\cI$ to one of $\wt{\cI}$. The reasoning is entirely reversible: given an action of $\wt{\cI}$, we define a splitting $\phi$ by $\phi(x)=[x]_0+[\beta_1 x]_1$.
\end{proof}

\begin{corollary} \label{cor:coind-split}
If $\alpha_1 \colon M \to M$ is not injective then $\sC(M)$ is not split.
\end{corollary}

\begin{proof}
If the action of $\cI$ extends to $\wt{\cI}$ then $\alpha_1$ must be injective since $\beta_1 \alpha_1 x = x$ for all $x$.
\end{proof}

\begin{corollary} \label{cor:A1coind}
Let $N$ be a smooth $\cI$-module and let $M=\ul{\bA}^1 \odot N$. Then $\sC(M)$ is split.
\end{corollary}

\begin{proof}
We extend the action of $\cI$ on $M$ to one of $\wt{\cI}$ by
\begin{displaymath}
\beta_1 \cdot (e_i \odot x) = \begin{cases}
0 & \text{if $i=1$} \\
e_{i-1} \odot x & \text{if $i \ge 2$}
\end{cases}
\end{displaymath}
We now verify the requisite identities. We have
\begin{displaymath}
\beta_1 \alpha_1 \cdot (e_i \odot x) = \beta_1 \cdot (e_{i+1} \odot x) = e_i \odot x,
\end{displaymath}
and so $\beta_1 \alpha_1$ acts as the identity, as required. Now suppose $k \ge 2$. Then
\begin{displaymath}
\alpha_k \cdot (e_i \odot x) = \begin{cases}
e_{i+1} \odot x & \text{if $k \le i$} \\
e_i \odot \alpha_{k-i} x & \text{if $k \ge i+1$,} \end{cases}
\end{displaymath}
and so
\begin{displaymath}
\beta_1 \alpha_k \cdot (e_i \odot x) = \begin{cases}
e_i \odot x & \text{if $k \le i$} \\
0 & \text{if $i=1$} \\
e_{i-1} \odot \alpha_{k-i} x & \text{if $i \ge 2$ and $k \ge i+1$.} \end{cases}
\end{displaymath}
(Note that if $i=1$ then $k \ge i+1$, as we have assumed $k \ge 2$.) On the other hand,
\begin{displaymath}
\alpha_{k-1} \beta_1 \cdot (e_i \odot x) =
\begin{cases}
0 & \text{if $i=1$} \\
e_i \odot x & \text{if $i \ge 2$ and $k \le i$} \\
e_{i-1} \odot \alpha_{k-i} x & \text{if $i \ge 2$ and $k \ge i+1$.} \end{cases}
\end{displaymath}
We thus find that $\beta_1 \alpha_k=\alpha_{k-1} \beta_1$ holds on $M$, as required.
\end{proof}


\subsection{Indecomposability of coinduction}

We now investigate the indecomposability of coinduction. For an $\cI$-module $M$, let $M[\alpha_1]=\{x \in M \mid \alpha_1 x =0\}$. This is an $\cI$-submodule of $M$: indeed, if $x \in M[\alpha_1]$ then $\alpha_1 \alpha_k x = \alpha_{k+1} \alpha_1 x = 0$, and so $\alpha_k x \in M[\alpha_1]$. 

\begin{lemma}
For any $\cI$-module $M$, we have $\sC(M)[\alpha_1]=[M]_1$.
\end{lemma}

\begin{proof}
We have
\begin{displaymath}
\alpha_1([x]_0+[y]_1) = [\alpha_1 x]_0+[x]_1,
\end{displaymath}
and this vanishes if and only if $x=0$. The result follows.
\end{proof}

\begin{proposition}
Suppose that $M$ is an indecomposable smooth $\cI$-module and $\sC(M)$ is non-split. Then $\sC(M)$ is also indecomposable.
\end{proposition}

\begin{proof}
Suppose $\sC(M)=A \oplus B$ for non-zero $\cI$-submodules $A$ and $B$. Then
\begin{displaymath}
M \cong \sC(M)/\sC(M)[\alpha_1] = A/A[\alpha_1] \oplus B/B[\alpha_1],
\end{displaymath}
where the first isomorphism is a consequence of the previous lemma. Since $M$ is indecomposable, one of the terms on the right vanishes, say $B/B[\alpha_1]$. Thus $\alpha_1$ vanishes on $B$, and so $B \subset [M]_1$. It follows that $[M]_1=B \oplus (A \cap [M]_1)$. However, $[M]_1 \cong \Pi(M)$ is indecomposable by Proposition~\ref{prop:pisubmod}, and so we see that $A \cap [M]_1=0$ and $B=[M_1]$. Thus the surjection $\sC(M) \to M$ maps $A$ isomorphically onto $M$, and is hence split, a contradiction. Thus $\sC(M)$ is indecomposable.
\end{proof}

\begin{corollary} \label{cor:coind-indecomp}
Suppose that $M$ is an indecomposable smooth $\cI$-module with $M[\alpha_1] \ne 0$. Then $\sC(M)$ is also indecomposable.
\end{corollary}

\begin{proof}
Indeed, $\sC(M)$ is non-split by Corollary~\ref{cor:coind-split}.
\end{proof}

\subsection{Induction} \label{ss:induction}

We have just defined and studied the coinduction functor. We now define the induction functor. We will not need to use induction in what follows, so we keep this discussion brief; it is included just for the sake of completeness.

Let $M$ be an $\cI$-module. Define a new $\cI$-module, called the {\bf induction} of $M$, and denoted $\sI(M)$, as follows. As a vector space, $\sI(M)=\bigoplus_{n \ge 0} \{M\}_n$ where each $\{M\}_n$ is simply a copy of $M$. For $x \in M$, we denote by $\{x\}_n$ the corresponding element of $\{M\}_n$. The action of $\cI$ on the induction is given by
\begin{displaymath}
\alpha_i \cdot \{x\}_n = \begin{cases}
\{ x \}_{n+1} & \text{if $i \le n+1$} \\
\{ \alpha_{i-n-1} x \}_n & \text{if $i \ge n+2$} \end{cases}
\end{displaymath}
We leave to the reader the simple verification that this is well-defined. It is clear that if $M$ is smooth then so is $\sI(M)$. If $M$ is graded then $\sI(M)$ admits a grading as follows: for $x \in M_n$, we let $\{x\}_k$ have degree $n+k$. We denote this graded $\cI$-module by $\ul{\sI}(M)$. We thus have functors
\begin{displaymath}
\sI \colon \Rep(\cI) \to \Rep(\cI), \qquad
\ul{\sI} \colon \uRep(\cI) \to \uRep(\cI).
\end{displaymath}
It is clear that both are exact and cocontinuous.

\begin{proposition}
The functor $\sI$ is naturally the left adjoint to the shift $\Sigma$. The analogous statement holds in the grade case.
\end{proposition}

\begin{proof}
The unit $\epsilon \colon M \to \Sigma(\sI(M))$ is defined by $\epsilon(x)=\{x\}_0^{\flat}$. The counit $\eta \colon \sI(\Sigma(M)) \to M$ is defined by $\eta(\{x^{\flat}\}_n)=\alpha_1^n x$. We leave the remainder of the proof to the reader.
\end{proof}

\begin{remark}
In fact, $\sI(M)$ is simply isomorphic to $\ul{\bA}^1 \odot M$, via $\{x\}_n \mapsto e_{n+1} \odot x$.
\end{remark}

\subsection{Explanation of definitions} \label{ss:indexp}

We now explain where the formulas for induction and coinduction come from. The shift functor $\Sigma$ can be described as follows: first, restrict from $\cI$ to $\cI_{\ge 2}$; then, use the isomorphism $\cI \cong \cI_{\ge 2}$ given by $\alpha_i \mapsto \alpha_{i+1}$ to turn the $\cI_{\ge 2}$-module back into an $\cI$-module. Thus the left and right adjoint can be described as follows: first transfer the $\cI$-module to an $\cI_{\ge 2}$-module via the isomorphism; then apply induction or coinduction from $\cI_{\ge 2}$ to $\cI$.

Thus to understand the adjoint operations, we simply need to understand induction or coinduction from $\cI_{\ge 2}$ to $\cI$. For this, we need to understand the structure of $\cI$ as a left or right $\cI_{\ge 2}$-set. From Proposition~\ref{prop:alphagen}, we find
\begin{displaymath}
\cI=\coprod_{n \ge 0} \alpha_1^n \cI_{\ge 2}, \qquad
\cI=\cI_{\ge 2} \amalg \cI_{\ge 2} \alpha_1.
\end{displaymath}
We thus see that if $M$ is an $\cI_{\ge 2}$-module then the natural map
\begin{displaymath}
\bigoplus_{n \ge 0} \alpha_1^n \otimes M  \to \bk[\cI] \otimes_{\bk[\cI_{\ge 2}]} M
\end{displaymath}
is an isomorphism; the $n$th term in this decomposition corresponds to the $\{M\}_n$ summand in the definition of $\sI(M)$. Similarly, the map
\begin{displaymath}
\Hom_{\cI_{\ge 2}}(\cI, M) \to M \oplus M, \qquad f \mapsto (f(1), f(\alpha_1)).
\end{displaymath}
is an isomorphism; the two terms on the right correspond to $[M]_0$ and $[M]_1$ in $\sC(M)$.

\section{Injective modules} \label{s:injectives}

\subsection{Finite length injectives} \label{ss:fininj}

For $n \in \bN_+$, let $\ul{\bJ}^n=\ul{\sC}^{n-1}(\ul{\bB}^1)$; also let $\ul{\bJ}^0=\ul{\bB}^0$. Put $\bJ^n=\Phi(\ul{\bJ}^n)$. We remind the reader that the binomial coefficient $\binom{n}{m}$ vanishes for $m>n$, and is~1 if $n=m$, even for negative values of $n$ and $m$.

\begin{proposition} \label{prop:fininj}
We have the following:
\begin{enumerate}
\item $\ul{\bJ}^n$ is an indecomposable injective in $\uRep(\cI)$ of finite length.
\item $\ul{\bJ}^n$ is the injective envelope of $\ul{\bB}^n$.
\item Every indecomposable injective in $\uRep(\cI)$ that is locally of finite length is isomorphic to some $\ul{\bJ}^n$.
\item The multiplicity of $\ul{\bB}^m$ in $\ul{\bJ}^n$ is $\binom{n-1}{m-1}$.
\end{enumerate}
The analogous results hold in the smooth case.
\end{proposition}

\begin{proof}
(a) It is clear that $\ul{\bJ}^1=\ul{\bB}^1$ is injective (e.g., $\ul{\bB}^0$ is injective, and so $\ul{\sC}(\ul{\bB}^0)=\ul{\bB}^0 \oplus \ul{\bB}^1$ is injective), and obviously indecomposable. It follows from Corollary~\ref{cor:coindfin} that $\ul{\bJ}^n$ is of finite length, from Corollary~\ref{cor:coind} that it is injective, and from Corollary~\ref{cor:coind-indecomp} that it is indecomposable.

(b) By adjunction, the isomorphism $\ul{\Sigma}^{n-1}(\ul{\bB}^n) \to \ul{\bB}^1$ gives a non-zero map $\ul{\bB}^n \to \ul{\bJ}^n$, which is necessarily injective (since $\ul{\bB}^n$ is simple). Since $\ul{\bJ}^n$ is an indecomposable injective, it is the injective envelope of any of its non-zero submodules \stacks{08Y7}, and so in particular of $\ul{\bB}^n$. (The cited result is stated for modules, but holds in the present situation as well.)

(c) Let $J$ be an indecomposable injective in $\uRep(\cI)$ that is locally of finite length. Then $J$ contains a non-zero finite length module, and thus contains $\ul{\bB}^n$ for some $n$. As $J$ is the injective envelope of $\ul{\bB}^n$, it is thus isomorphic to $\ul{\bJ}^n$ \stacks{08Y4}. (Same comment as before.)

(d) Working in $\ul{\rK}(\cI)$ and using Proposition~\ref{prop:coindseq}, we have
\begin{displaymath}
\sC([\ul{\bB}^n])=[\ul{\bB}^n]+[\ul{\bB}^{n+1}]
\end{displaymath}
since $\ul{\Pi}(\ul{\bB}^n)=\ul{\bB}^{n+1}$. Thus we can write
\begin{displaymath}
\ul{\sC}([M]) = (1+[\ul{\bB}^1]) \odot [M]
\end{displaymath}
for finite length $M$. Therefore,
\begin{displaymath}
[\ul{\bJ}^n] = \ul{\sC}^{n-1}([\ul{\bB}^1]) = (1+[\ul{\bB}^1])^{\odot (n-1)} \odot [\ul{\bB}^1] = \sum_{i=0}^{n-1} \binom{n-1}{i} [\ul{\bB}^{i+1}],
\end{displaymath}
as claimed.
\end{proof}

\subsection{Classification of injectives} \label{ss:inj}

Let $\lambda$ be a constraint word of rank $r$. Write
\begin{displaymath}
\lambda=\blacktri^{a_0(\lambda)} \, \whitetri \, \blacktri^{a_1(\lambda)} \cdots \blacktri^{a_{r-1}} \, \whitetri \, \blacktri^{a_r(\lambda)}
\end{displaymath}
with $a_0(\lambda), \ldots, a_r(\lambda) \in \bN$, and put $a(\lambda)=(a_0(\lambda), \ldots, a_r(\lambda)) \in \bN^{r+1}$. In what follows, we partially order $\bN^{n+1}$ by $(a_0, \ldots, a_r) \le (b_0, \ldots, b_r)$ if $a_i \le b_i$ for all $i$. Furthermore, for $a=(a_0, \ldots, a_r) \in \bN^{n+1}$, we let $\vert a \vert=a_0+\cdots+a_r$.

Now, define
\begin{displaymath}
\ul{\bI}^{\lambda}=\ul{\bJ}^{a_0(\lambda)} \odot \ul{\bA}^1 \odot \cdots \odot \ul{\bA}^1 \odot \ul{\bJ}^{a_r(\lambda)}
\end{displaymath}
and put $\bI^{\lambda}=\Phi(\ul{\bI}^{\lambda})$.

\begin{theorem} \label{thm:inj}
We have the following:
\begin{enumerate}
\item $\ul{\bI}^{\lambda}$ is an indecomposable injective object of $\uRep(\cI)$.
\item $\ul{\bI}^{\lambda}$ and $\ul{\bI}^{\mu}$ are isomorphic if and only if $\lambda=\mu$.
\item $\ul{\bI}^{\lambda}$ is the injective envelope of $\ul{\bE}^{\lambda}$.
\item Every indecomposable injective object of $\uRep(\cI)$ is isomorphic to $\ul{\bI}^{\lambda}$, for a unique $\lambda$.
\end{enumerate}
The analogous statements hold in the smooth case.
\end{theorem}

\begin{lemma} \label{lem:inj}
Let $I$ be an indecomposable injective graded $\cI$-module. Then $J_n=\ul{\bJ}^n \odot \ul{\bA}^1 \odot I$ is indecomposable injective for all $n \in \bN$.
\end{lemma}

\begin{proof}
We first treat the $n=0$ case. Since $I$ is injective, so is $I^{\dag}$, and so is $\Gamma(\Phi(I^{\dag}))=I^{\dag} \odot \ul{\bA}^1$, and so is $(I^{\dag} \odot \ul{\bA}^1)^{\dag}=J_0$. Suppose $J_0=A \oplus B$. Then $J_0^{\dag}=I^{\dag} \odot \ul{\bA}^1 = A^{\dag} \oplus B^{\dag}$, and so $\Phi(I^{\dag})=\Psi(A^{\dag}) \oplus \Psi(B^{\dag})$. But $\Phi(I^{\dag})$ is indecomposable by Proposition~\ref{prop:phi-indecomp}, and so $\Psi(A^{\dag})=0$ or $\Psi(B^{\dag})=0$, say the latter. Thus $B^{\dag}$ is torsion. However, $I^{\dag} \odot \ul{\bA}^1$ is torsion-free, and so $B^{\dag}=0$, and thus $B=0$. Thus $J_0$ is indecomposable.

We now treat the $n=1$ case. Since $\ul{\bB}^1=\ul{\bJ}^1$, we have $J_1=\ul{\Pi}(J_0)$. Thus $J_1$ is indecomposable by Proposition~\ref{prop:pisubmod}. By Corollary~\ref{cor:A1coind}, $\ul{\sC}(J_0)$ is split, and so we have a decomposition $\ul{\sC}(J_0) \cong J_0 \oplus J_1$. Since $\ul{\sC}(J_0)$ is injective (Corollary~\ref{cor:coind}), so is the summand $J_1$.

We now treat the $n \ge 2$ case inductively. Since $\ul{\bJ}^{n-1}$ is pure, we have
\begin{displaymath}
\ul{\sC}(J_{n-1})
= \ul{\sC}(\ul{\bJ}^{n-1} \odot \ul{\bA}^1 \odot I)
=\ul{\sC}(\ul{\bJ}^{n-1}) \odot \ul{\bA}^1 \odot I
=J_n,
\end{displaymath}
where the second isomorphism comes from Proposition~\ref{prop:coindcat}. Thus $J_n$ is injective (Corollary~\ref{cor:coind}) and indecomposable (Corollary~\ref{cor:coind-indecomp}).
\end{proof}

\begin{proof}[Proof of Theorem~\ref{thm:inj}]
(a) This follows from inductively applying the previous lemma.

(b) From Proposition~\ref{prop:fininj}(d), we see that $[\ul{\bI}^{\lambda}]=[\ul{\bE}^{\lambda}]+\cdots$, where the remaining terms have the form $[\ul{\bE}^{\nu}]$ with $\nu$ of the same rank as $\lambda$ and $a(\nu)<a(\lambda)$. Thus if $\ul{\bI}^{\lambda}$ is isomorphic to $\ul{\bI}^{\mu}$ then $[\ul{\bI}^{\lambda}]=[\ul{\bI}^{\mu}]$ and so $\lambda=\mu$.

(c) Let $\lambda$ be given. The injections $\ul{\bB}^{a_i(\lambda)} \to \ul{\bJ}^{a_i(\lambda)}$ induces an injection $\ul{\bE}^{\lambda} \to \ul{\bI}^{\lambda}$. Since $\ul{\bI}^{\lambda}$ is indecomposable and injective, it must be the injective envelope of $\ul{\bE}^{\lambda}$.

(d) Let $I$ be an indecomposable injective of $\uRep(\cI)$. Let $M_1 \subset I$ be a non-zero finitely generated submodule, and let $M_2 \subset M_1$ be a non-zero submodule with an injection $M_2 \to \ul{\bE}^{\lambda}$ for some $\lambda$, which exists by Proposition~\ref{prop:levfilt2}. Composing with the injection $\ul{\bE}^{\lambda} \to \ul{\bI}^{\lambda}$, we see that $M_2$ is a submodule of $\ul{\bI}^{\lambda}$. Thus $M_2$ is a subobject of both $I$ and $\ul{\bI}^{\lambda}$. Since both are indecomposable, they are both injective envelopes of $M_2$ \stacks{08Y7}, and thus isomorphic \stacks{08Y4}.
\end{proof}

\begin{corollary} \label{cor:injcond}
Every injective object of $\uRep(\cI)_{\le r}$ remains injective in $\uRep(\cI)$. Similarly in the smooth case.
\end{corollary}

\begin{proof}
The argument used to prove Theorem~\ref{thm:inj}(d) implies that every indecomposable injective of $\uRep(\cI)_{\le r}$ is isomorphic to $\ul{\bI}^{\lambda}$ for some $\lambda$ of rank $\le r$. Since $\uRep(\cI)_{\le r}$ is locally noetherian, every injective is a sum of $\ul{\bI}^{\lambda}$'s, and thus remains injective in $\uRep(\cI)$.
\end{proof}

\subsection{Injective resolutions}

Having understood the injective objects, we now turn our attention to injective resolutions.

\begin{proposition} \label{prop:Lres}
We have an injective resolution $\ul{\bB}^n \to J^{\bullet}$, where $J^i=(\ul{\bJ}^{n-i})^{\oplus \binom{n-1}{n-1-i}}$, for any $n \in \bN$.
\end{proposition}

\begin{proof}
We proceed by induction on $n$. The result is clear for $n=0$ and $n=1$. Suppose now that we have proved the result for $n \ge 1$, and let us prove it for $n+1$. Let $\ul{\bB}^n \to J^{\bullet}$ be the given resolution. Consider the diagram
\begin{displaymath}
\xymatrix{
&& \ul{\sC}(J^{\bullet}) \ar[r] & J^{\bullet} \\
0 \ar[r] & \ul{\bB}^{n+1} \ar[r] & \ul{\sC}(\ul{\bB}^n) \ar[r] \ar[u] & \ul{\bB}^n \ar[u] \ar[r] & 0 }
\end{displaymath}
Here the bottom exact sequence is the canonical exact sequence for $\ul{\sC}$ (Proposition~\ref{prop:coindseq}). The right vertical map is the given injective resolution of $\ul{\bB}^n$. The middle vertical map is also an injective resolution, as $\ul{\sC}$ is exact and takes $\ul{\bJ}^i$ to $\ul{\bJ}^{i+1}$. We thus obtain an injective resolution $K^{\bullet}$ of $\ul{\bB}^{n+1}$ by taking the cone of the map $\ul{\sC}(J^{\bullet}) \to J^{\bullet}$ (and shifting by one). We have
\begin{displaymath}
K^i = \ul{\sC}(J^i) \oplus J^{i-1} = (\bJ^{n-i+1})^{\oplus \binom{n-1}{n-1-i}} \oplus (\bJ^{n-i+1})^{\oplus \binom{n-1}{n-i}} = (\bJ^{n-i+1})^{\oplus \binom{n}{n-i}},
\end{displaymath}
as required.
\end{proof}

\begin{corollary} \label{cor:Lext}
Let $n,m \in \bN$. Then
\begin{displaymath}
\dim \Ext^i_{\cI}(\bB^n, \bB^m) = \dim \uExt^i_{\cI}(\ul{\bB}^n, \ul{\bB}^m)=\begin{cases}
\binom{m-1}{n-1} & \text{if $i=m-n$} \\
0 & \text{otherwise} \end{cases}
\end{displaymath}
\end{corollary}

\begin{proof}
The first equality follows from Theorem~\ref{thm:canongr}, so it suffices to prove the second. Note that $\uHom_{\cI}(\ul{\bB}^n, \ul{\bJ}^i)$ is one-dimensional if $n=i$ and vanishes otherwise, as $\ul{\bB}^n$ must map into the socle of $\ul{\bJ}^i$, which is $\ul{\bB}^i$. Let $\ul{\bB}^m \to J^{\bullet}$ be the resolution constructed in Proposition~\ref{prop:Lres}. Then $\uExt^{\bullet}_{\cI}(\ul{\bB}^n, \ul{\bB}^m)$ is computed by the complex $\uHom_{\cI}(\ul{\bB}^n, J^{\bullet})$. If $n>m$ then all terms in the complex vanish. Otherwise, the complex has a unique non-zero term, in degree $m-n$, and it has dimension $\binom{m-1}{n-1}$. The result follows.
\end{proof}

\begin{proposition} \label{prop:stdres}
The standard module $\ul{\bE}^{\lambda}$ admits an injective resolution $\ul{\bE}^{\lambda} \to I^{\bullet}$ with the following properties:
\begin{enumerate}
\item $I^0=\ul{\bI}^{\lambda}$.
\item $I^n=0$ for $n \gg 0$.
\item $I^n$ is a finite sum of modules of the form $\ul{\bI}^{\mu}$ where $\mu$ has the same rank as $\lambda$ and satisfies $a(\mu) \le a(\lambda)$ and $\vert a(\lambda)-a(\mu)\vert = n$.
\end{enumerate}
In particular, the injective dimension of $\ul{\bE}^{\lambda}$ is at most $\vert a(\lambda) \vert$. The analogous statements hold in the smooth case.
\end{proposition}

\begin{proof}
Let $\lambda=\lambda_1 \cdots \lambda_t$. For each $1 \le i \le t$, let $\ul{\bB}^{a_i(\lambda)} \to J^{i,\bullet}$ be the injective resolution provided by Proposition~\ref{prop:Lres}. Then
\begin{displaymath}
\ul{\bE}^{\lambda} \to J^{0,\bullet} \odot \ul{\bA}^1 \odot \cdots \odot \ul{\bA}^1 \odot J^{r,\bullet}
\end{displaymath}
is the sought resolution of $\ul{\bE}^{\lambda}$.
\end{proof}

\begin{remark}
In fact, the bound on injective dimension given above is slightly suboptimal. The resolution of $\ul{\bB}^{a_i(\lambda)}$ constructed in Proposition~\ref{prop:Lres} has length~0 if $a_i(\lambda)=0$, and length $a_i(\lambda)-1$ otherwise. We thus see that the injective dimension of $\ul{\bE}^{\lambda}$ is at most $\vert a(\lambda) \vert - k$ where $k$ is the number of non-zero entries in $a(\lambda)$.
\end{remark}

\begin{theorem} \label{thm:injres}
Every finitely generated smooth or graded $\cI$-module has finite injective dimension, and, in fact, admits a finite length resolution by finitely generated injective modules.
\end{theorem}

\begin{proof}
Since the $\ul{\bE}^{\lambda}$'s generate $\rD^b_{\fgen}(\uRep(\cI))$ (Theorem~\ref{thm:dergen}) and have finite injective dimension (Proposition~\ref{prop:stdres}), it follows that every finitely generated graded $\cI$-module has finite injective dimension.

To prove the more precise claim, it suffices to prove that every finitely generated graded $\cI$-module injects into a finitely generated injective graded $\cI$-module. We know this is true for standard modules (Theorem~\ref{thm:inj}), and thus submodules of standard modules, and so it is true for all finitely generated modules by Proposition~\ref{prop:levfilt2}.

The same argument applies in the smooth case.
\end{proof}

\begin{corollary} \label{cor:extfin}
Let $M$ and $N$ be graded $\cI$-modules with $N$ finitely generated. Then $\uExt^i_{\cI}(M, N)=0$ for $i \gg 0$. Similarly in the smooth case.
\end{corollary}

\section{Local cohomology and saturation} \label{s:loccoh}

\subsection{Generalities} \label{ss:satgen}

We now review local cohomology and saturation in general. See \cite[\S 4]{symu1} for more details (though note that our notation differs from that of loc.\ cit.).

Let $\cA$ be a Grothendieck abelian category and let $\cB$ be a localizing subcategory. We assume the following condition holds:
\[
\label{cond:inj}
\makebox{\parbox{\dimexpr\linewidth-13\fboxsep-2\fboxrule}{Injective objects of $\cB$ remain injective in $\cA$.}} \tag{Inj}
\]
We let $T \colon \cA \to \cA/\cB$ be the localization functor and $S \colon \cA/\cB \to \cA$ its right adjoint. We define the {\bf saturation functor} $\cS \colon \cA \to \cA$ to be the composition $S \circ T$. The saturation functor is left exact, and its derived functor $\rR \cS$ will play a prominent role. We say that $M \in \cA$ is {\bf ($\cS$-)saturated} if the natural map $M \to \cS(M)$ is an isomorphism, and {\bf derived ($\cS$-)saturated} if the natural map $M \to \rR \cS(M)$ is an isomorphism. For $M \in \cA$, we let $\cH(M)$ denote the maximal subobject of $M$ that belongs to $\cB$. This defines a left exact functor $\cH \colon \cA \to \cA$, and we again consider the derived functors $\rR \cH$, which we refer to as {\bf local cohomology}.

\begin{proposition} \label{prop:sattri}
For any $M \in \rD^+(\cA)$, we have a canonical exact triangle
\begin{displaymath}
\rR \cH(M) \to M \to \rR \cS(M) \to
\end{displaymath}
where the first two maps are the canonical ones.
\end{proposition}

\begin{proof}
See \cite[Proposition~4.6]{symu1}.
\end{proof}

\begin{proposition} \label{prop:satinj}
Let $I \in \cA$ be injective. Then we have a short exact sequence
\begin{displaymath}
0 \to \cH(I) \to I \to \cS(I) \to 0
\end{displaymath}
where $\cH(I)$ and $\cS(I)$ are both injective. Moreover, $T(I) \in \cA/\cB$ is injective.
\end{proposition}

\begin{proof}
See \cite[Proposition~4.3]{symu1}.
\end{proof}

\begin{proposition} \label{prop:sattors}
For $M \in \cB$, the natural map $\rR \cH(M) \to M$ is an isomorphism (that is, $\cH(M)=M$ and $\rR^i \cH(M)=0$ for $i>0$) and $\rR \cS(M)=0$.
\end{proposition}

\begin{proof}
Let $M \to I^{\bullet}$ be an injective resolution in $\cB$, which is an injective resolution in $\cA$ by \eqref{cond:inj}. Thus $\rR \cH(M)$ is computed by $\cH(I^{\bullet})$. But this is just $I^{\bullet}$, since each term belongs to $\cB$, and is thus quasi-isomorphic to $M$ (via the canonical map). Similarly, $\rR \cS(M)$ is compted by $\cS(I^{\bullet})$, and each term vanishes.
\end{proof}

\begin{proposition} \label{prop:satcrit}
Let $M \in \cA$. The following are equivalent:
\begin{enumerate}
\item $M$ is saturated.
\item $\Hom_{\cA}(N, M)=\Hom_{\cA/\cB}(T(N), T(M))$ for all $N \in \cA$.
\item $\Ext^i_{\cA}(N, M)=0$ for $i=0,1$ and all $N \in \cB$.
\end{enumerate}
Moreover, if $\cB$ is locally noetherian then the above are also equivalent to:
\begin{enumerate}
\setcounter{enumi}{3}
\item $\Ext^i_{\cA}(N, M)=0$ for $i=0,1$ and all $N \in \cB^{\fgen}$.
\end{enumerate}
\end{proposition}

\begin{proof}
The equivalence of (a) and (c) is is \cite[p.~371]{gabriel}. The equivalence of (b) and (c) is \cite[p.~370, Lemma]{gabriel}. Suppose now that $\cB$ is locally noetherian. Obviously, (c) implies (d). Conversely, suppose that (d) holds. Let $N \in \cB$ and let $\{N_i\}_{i \in I}$ be the set of finitely generated subobjects of $N$. Then
\begin{displaymath}
\rR \Hom_{\cA}(N, M)=\rR \Hom_{\cA}(\varinjlim N_i, M) = \rR \varprojlim \rR \Hom_{\cA}(N_i, M).
\end{displaymath}
Since $\rR^i \Hom_{\cA}(N_i, M)=0$ for $i=0,1$ by our assumption, it follows that $\rR^i \Hom_{\cA}(N, M)=0$ for $i=0,1$, and so (c) holds.
\end{proof}

\begin{proposition} \label{prop:dersatcrit}
Let $M \in \cA$. The following are equivalent:
\begin{enumerate}
\item $M$ is derived saturated.
\item $\Ext^i_{\cA}(N, M)=0$ for all $i \ge 0$ and all $N \in \cB$.
\end{enumerate}
Moreover, if $\cB$ is locally noetherian then the above are also equivalent to:
\begin{enumerate}
\setcounter{enumi}{2}
\item $\Ext^i_{\cA}(N, M)=0$ for all $i \ge 0$ and all $N \in \cB^{\fgen}$.
\end{enumerate}
\end{proposition}

\begin{proof}
The equivalence of (a) and (b) is proved in \cite[Proposition~4.7]{symu1}. The equivalence of (b) and (c) is proved similar to the equivalence of (c) and (d) in Proposition~\ref{prop:satcrit}.
\end{proof}

\begin{proposition} \label{prop:satfin}
Suppose that $\cB$ is locally noetherian and that for all $N \in \cB^{\fgen}$ the functors $\Ext^i_{\cA}(N, -)$ commute with filtered colimits. Then:
\begin{enumerate}
\item A filtered colimit of saturated objects is saturated.
\item A filtered colimit of derived saturated objects is derived saturated.
\item The functor $\rR^i \cS$ and $\rR^i \cH$ commute with filtered colimits for all $i$.
\end{enumerate}
\end{proposition}

\begin{proof}
(a) Suppose that $\{M_i\}_{i \in \cU}$ is a filtered system of saturated objects. Let $N \in \cB^{\fgen}$. Then for $j=0,1$, we have
\begin{displaymath}
\Ext^j_{\cA}(N, \varinjlim M_i) = \varinjlim \Ext^j_{\cA}(N, M_i) = 0,
\end{displaymath}
where the first isomorphism follows from the hypothesis of the proposition, and the second from Proposition~\ref{prop:satcrit}. Thus $\varinjlim M_i$ is saturated by Proposition~\ref{prop:satcrit}.

(b) The proof is exactly the same as that of (a).

(c) We first show that $\cS$ commutes with filtered colimits. Thus let $\{M_i\}_{i \in \cU}$ be a filtered system, and consider the canonical map $\phi \colon \varinjlim \cS(M_i) \to \cS(\varinjlim M_i)$. The domain of $\phi$ is saturated by (a), while the target of $\phi$ is obviously saturated. Thus to show that $\phi$ is an isomorphism, it suffices to that $T(\phi)$ is an isomorphism. This is clear: $T(\phi)$ is simply the identity morphism of $\varinjlim T(M_i)$. (Note that $T$ commutes with colimits and satisfies $T \circ \cS=T$.)

We next show that $\rR \cS$ commutes with filtered colimits. It suffices to show that a filtered colimit of injective objects in $\cA$ is $\cS$-acyclic (Proposition~\ref{prop:dercolimit}). Thus let $\{I_i\}_{i \in \cU}$ be a filtered system of injective objects with direct limit $I$. For each $i \in I$, we have the exact sequence
\begin{displaymath}
0 \to \cH(I_i) \to I_i \to \cS(I_i) \to 0,
\end{displaymath}
and so, taking the direct limit, we have an exact sequence
\begin{displaymath}
0 \to \varinjlim \cH(I_i) \to I \to \varinjlim \cS(I_i) \to 0.
\end{displaymath}
Now, $\varinjlim \cH(I_i)$ belongs to $\cB$ and is thus $\cS$-acyclic by Proposition~\ref{prop:sattors}. Since $\cS(I_i)$ is saturated and injective (by Proposition~\ref{prop:satinj}), it is thus derived saturated. Hence $\varinjlim \cS(I_i)$ is derived saturated by (b), and, in particular, $\cS$-acyclic. We thus see that $I$ is $\cS$-acyclic as well.

Finally, we show that $\rR \cH$ commutes with filtered colimits. Let $\{M_i\}_{i \in \cU}$ be a filtered system. Consider the diagram
\begin{displaymath}
\xymatrix{
\varinjlim \rR \cH(M_i) \ar[r] \ar[d] & \varinjlim M_i \ar[r] \ar@{=}[d] & \varinjlim \rR \cS(M_i) \ar[r] \ar[d] & \\
\rR \cH(\varinjlim M_i) \ar[r] & \varinjlim M_i \ar[r] & \rR \cS(\varinjlim M_i) \ar[r] & }
\end{displaymath}
where the vertical maps are the natural ones and the rows come from Proposition~\ref{prop:sattri}. Since the middle and right vertical maps are isomorphisms, so is the middle one, which completes the proof.
\end{proof}

Suppose that $\cB'$ is a second localizing subcategory of $\cA$ that contains $\cB$ and for which $\cB' \subset \cA$ satisfies \eqref{cond:inj}. Let $\cS'$ and $\cH'$ be the saturation and zeroth local cohomology functors for $\cB'$.

\begin{proposition}
We have the following:
\begin{enumerate}
\item Any two of the functors $\{ \cS, \cH, \cS', \cH' \}$ commute.
\item We have $\cH \circ \cH'=\cH$.
\item We have $\cS \circ \cS'=\cS'$.
\item We have $\cS \circ \cH = \cS' \circ \cH' = \cS' \circ \cH=0$.
\end{enumerate}
\end{proposition}

\begin{proof}
See \cite[Proposition~4.11]{symu1}.
\end{proof}

\subsection{The case of $\cI$-modules} \label{ss:loccoh}

We now apply the theory of the previous section to the categories $\Rep(\cI)$ and $\uRep(\cI)$. For simplicity, we will only treat the smooth case, but everything proceeds in the same manner in the graded case. The inclusion $\Rep(\cI)_{\le r} \subset \Rep(\cI)$ satisfies \eqref{cond:inj} by Corollary~\ref{cor:injcond}. Recall that
\begin{displaymath}
T_{>r} \colon \Rep(\cI) \to \Rep(\cI)_{>r} = \Rep(\cI)/\Rep(\cI)_{\le r}
\end{displaymath}
is the localization functor, and $S_{>r}$ is its right adjoint. We let $\cS_{>r}=S_{>r} \circ T_{>r}$ be the saturation functor with respect to $\Rep(\cI)_{\le r}$, and we let $\cH_{\le r}$ be the zeroth local cohomology functor with respect to $\Rep(\cI)_{\le r}$. All the results of the previous section apply in the present case, though we will not restate them. (We note in particular that the hypothesis of Proposition~\ref{prop:satfin} is met by Propositions~\ref{prop:uextcolim} and~\ref{prop:extcolim}.)

We now compute the derived saturation and local cohomology of standard objects:

\begin{theorem} \label{thm:satstd}
Let $\lambda$ be a constraint word of rank~$r$. Then
\begin{displaymath}
\rR \cH_{\le s}(\bE^{\lambda}) = \begin{cases}
0 & \text{if $s<r$} \\
\bE^{\lambda} & \text{if $s \ge r$} \end{cases}
\qquad
\rR \cS_{>s}(\bE^{\lambda}) = \begin{cases}
\bE^{\lambda} & \text{if $s<r$} \\
0 & \text{if $s \ge r$} \end{cases}
\end{displaymath}
\end{theorem}

\begin{proof}
By Proposition~\ref{prop:stdres}, we have an injective resolution $\bE^{\lambda} \to I^{\bullet}$ where each $I^n$ is a sum of injective objects of the form $\bI^{\mu}$ with $\mu$ of rank $r$. For $s<r$ we have $\cH(\bI^n)=0$ for each $n$, and so $\rR \cH(\bE^{\lambda})=0$. For $s \ge r$ we have $\cH(\bI^n)=\bI^n$ for each $n$, and so $\rR \cH(\bE^{\lambda})=\bE^{\lambda}$. The result for $\rR \cS_{>s}$ now follows from Proposition~\ref{prop:sattri}.
\end{proof}

\begin{corollary}
Suppose that $\lambda$ is a constraint word of rank $r$ and $M$ is a smooth $\cI$-module of level $<r$. Then $\Ext^i_{\cI}(M, \bE^{\lambda})=0$ for all $i \ge 0$.
\end{corollary}

\begin{proof}
Theorem~\ref{thm:satstd} shows that $\bE^{\lambda}$ is derived saturated with respect to $\Rep(\cI)_{<r}$, and so the result follows from Proposition~\ref{prop:dersatcrit}.
\end{proof}

\begin{corollary} \label{cor:exactsat}
The restriction of $\cS_{\ge r}$ to $\Rep(\cI)_{\le r}$ is exact.
\end{corollary}

\begin{proof}
If $M \in \Rep(\cI)_{<r}$ then $\rR \cS_{\ge r}(M)=0$ by Proposition~\ref{prop:sattors}. Suppose now that $M$ is a non-zero submodule of $\bE^{\lambda}$ with $\lambda$ of rank $r$. Consider the exact sequence
\begin{displaymath}
0 \to M \to \bE^{\lambda} \to N \to 0.
\end{displaymath}
Then $N$ belongs to $\Rep(\cI)_{<r}$ by Theorem~\ref{thm:level}, and so $\rR \cS_{\ge r}(N)=0$. By Theorem~\ref{thm:satstd}, we have $\rR^i \cS_{\ge r}(\bE^{\lambda})=0$ for $i>0$. We thus see that $\rR^i \cS_{\ge r}(M)=0$ for $i>0$, that is $M$ is $\cS_{\ge r}$-acyclic. It now follows from Proposition~\ref{prop:levfilt2} that every finitely generated smooth $\cI$-module of level $\le r$ is $\cS_{\ge r}$-acyclic. Since $\rR \cS_{\ge r}$ commutes with filtered colimits by Proposition~\ref{prop:satfin}(c), we see that all smooth $\cI$-modules of level $
\le r$ are $\cS_{\ge r}$-acyclic, and so the corollary follows.
\end{proof}


\begin{example}
Consider the map $f \colon \bA^2 \to \bA^1$ given by $f(e_{1,2})=e_2$. Let $K$ be the kernel and $C$ the cokernel. Then $C \cong \bB^1$, and thus has level~0. We thus see that the sequence
\begin{displaymath}
0 \to T_{>0}(K) \to T_{>0}(\bA^2) \to T_{>0}(\bA^1) \to 0
\end{displaymath}
is exact in $\Rep(\cI)_{>0}$. Applying $S_{>0}$, and using the fact that $\bA^2$ and $\bA^1$ are derived $\cS_{>0}$-saturated (Theorem~\ref{thm:satstd}), we obtain an exact sequence
\begin{displaymath}
0 \to K \to \bA^2 \to \bA^1 \to \rR^1 \cS_{>0}(K) \to 0.
\end{displaymath}
We thus see that $\rR^1 \cS_{>0}(K) \cong \bB^1$. In particular, this shows that the section and saturation functors are not exact in general.
\end{example}

\subsection{Semistandard modules}

Fix $r \in \bN$. We say that a finitely generated smooth $\cI$-module $M$ is {\bf $r$-semistandard} if it admits a filtration $0=F^0 \subset \cdots \subset F^n=M$ by $\cI$-submodules such that $F^i/F^{i-1}$ is a standard module of rank $r$ for each $1 \le i \le n$. We say that a general smooth $\cI$-module is {\bf $r$-semistandard} if it is the direct union of finitely generated $r$-semistandard submodules. We let $\Rep(\cI)^{\rss}$ be the full subcategory of $\Rep(\cI)$ spanned by the $r$-semistandard modules. We make analogous definitions in the graded case. Throughout this section we work in the smooth case, but all results have analogs in the graded case.

\begin{example}
Let $\lambda$ be a constraint word of rank $r$. Then both $\bE^{\lambda}$ and $\bI^{\lambda}$ are $r$-semistandard.
\end{example}

\begin{proposition} \label{prop:sschar}
Let $M$ be a smooth $\cI$-module. The following are equivalent:
\begin{enumerate}
\item $M$ is $r$-semistandard.
\item $M$ has level $\le r$ and is $\cS_{\ge r}$-saturated.
\item $M$ has level $\le r$ and is derived $\cS_{\ge r}$-saturated.
\item $M$ admits an injection resolution $M \to I^{\bullet}$ where each $I^n$ is a sum of modules of the form $\bI^{\lambda}$ with $\lambda$ of rank $r$.
\end{enumerate}
\end{proposition}

\begin{proof}
Suppose that $M$ is finitely generated and $r$-semistandard. Then $M$ is a successive extension of standard modules of rank $r$. As such standard modules are derived $\cS_{\ge r}$-saturated by Theorem~\ref{thm:satstd}, we see that $M$ is too. Now suppose that $M$ is an arbitrary $r$-semistandard module. Write $M=\varinjlim M_i$ where each $M_i$ is a finitely generated $r$-semistandard submodule of $M$. Since each $M_i$ is derived $\cS_{\ge r}$-saturated, we see that $M$ is as well by Proposition~\ref{prop:satfin}(b). Thus (a) implies (c).

We now claim that for any $N \in \Rep(\cI)_r$, the module $S_{\ge r}(N)$ is $r$-semistandard. Recall that the category $\Rep(\cI)_r$ is locally of finite length, and the simple objects have the form $T_{\ge r}(\bE^{\lambda})$ with $\lambda$ of rank $r$ (Proposition~\ref{prop:levcat1}). First suppose that $N$ has finite length. Thus it is a successive extension of objects of the form $T_{\ge r}(\bE^{\lambda})$ with $\lambda$ of rank $r$. We have $\rR^i S_{\ge r}(T_{\ge r}(\bE^{\lambda}))=\rR^i \cS_{\ge r}(\bE^{\lambda})$, and this is $\bE^{\lambda}$ for $i=0$ and vanishes for $i>0$. We thus see that $S_{\ge r}(N)$ is a successive extension of modules of the form $\bE^{\lambda}$ with $\lambda$ of rank $r$, and is therefore $r$-semistandard. We now treat the general case. Write $N=\varinjlim N_i$ (filtered colimit of subobjects) where the $N_i$ are finite length. Since $S_{\ge r}$ commutes with filtered colimits (an easy consequence of Proposition~\ref{prop:satfin}(c)), we see that $S_{\ge r}(N)=\varinjlim S_{\ge r}(N_i)$ (filtered colimit of subobjects). As each $S_{\ge r}(N_i)$ is $r$-semistandard, we see that $S_{\ge r}(N)$ is $r$-semistandard.

Now suppose that $M$ has level $\le r$ and is $\cS_{\ge r}$-saturated. Then $M=\cS_{\ge r}(M)=S_{\ge r}(T_{\ge r}(M))$. As $T_{\ge r}(M) \in \Rep(\cI)_r$, we see that $M$ is $r$-semistandard by the previous paragraph. Thus (b) implies (a). It is obvious that (c) implies (b).

Again, suppose that $M$ has level $\le r$ and is derived $\cS_{\ge r}$-saturated. Let $T_{\ge r}(M) \to J^{\bullet}$ be an injective resolution in $\Rep(\cI)$. Then $M \to S_{\ge r}(J^{\bullet})$ is exact, since $M$ is derived saturated, and each $S_{\ge r}(J^n)$ is a sum of modules of the form $\bI^{\lambda}$ with $\lambda$ of rank $r$. Thus (c) implies (d).

Finally, suppose that $M \to I^{\bullet}$ is an injective resolution as in (d). Thus $\rR \cS_{\ge r}(M)$ is computed by $\cS_{\ge r}(I^{\bullet})$, but this equals $I^{\bullet}$ since each $I^n$ is $\cS_{\ge r}$-saturated. We thus see that $M \to \rR \cS_{\ge r}(M)$ is a quasi-isomorphism, and so (c) holds.
\end{proof}

\begin{theorem} \label{thm:ss}
The category $\Rep(\cI)^{\rss}$ of $r$-semistandard objects is an abelian subcategory of $\Rep(\cI)$ that is closed under extensions and filtered colimits. Moreover, the functors $T_{\ge r}$ and $S_{\ge r}$ induce mutually quasi-inverse equivalences between $\Rep(\cI)^{\rss}$ and $\Rep(\cI)_r$.
\end{theorem}

\begin{proof}
Let $f \colon M \to N$ be a map of $r$-semistandard objects. Let $K$ be the kernel and $C$ the cokernel. Then $K$ and $C$ are both of level $\le r$. Since $\cS_{\ge r}$ is exact on $\Rep(\cI)_{\le r}$ (Corollary~\ref{cor:exactsat}), we see that the sequence
\begin{displaymath}
0 \to \cS_{\ge r}(K) \to \cS_{\ge r}(M) \to \cS_{\ge r}(N) \to \cS_{\ge r}(C) \to 0
\end{displaymath}
is exact. Consider the diagram
\begin{displaymath}
\xymatrix{
0 \ar[r] & K \ar[r] \ar[d] & M \ar[r]^f \ar[d] & N \ar[r] \ar[d] & C \ar[r] \ar[d] & 0 \\
0 \ar[r] & \cS_{\ge r}(K) \ar[r] & \cS_{\ge r}(M) \ar[r] & \cS_{\ge r}(N) \ar[r] & \cS_{\ge r}(C) \ar[r] & 0 }
\end{displaymath}
The middle two vertical maps are isomorphisms by Propsoition~\ref{prop:sschar}. Thus the outer two vertical maps are isomorphisms as well, and so $K$ and $C$ are $\cS_{\ge r}$-saturated. Since $K$ and $C$ are obviously of level $\le r$, we see that they are $r$-semistandard by Proposition~\ref{prop:sschar}. Thus $\Rep(\cI)^{\rss}$ is an abelian subcategory of $\Rep(\cI)$. It is closed under extensions and filtered colimits by Proposition~\ref{prop:sschar} and~\ref{prop:satfin}(a).

It follows from Proposition~\ref{prop:sschar} that $T_{\ge r}$ carries $\Rep(\cI)^{\rss}$ into $\Rep(\cI)_r$, and that $S_{\ge r}$ carries $\Rep(\cI)_r$ into $\Rep(\cI)^{\rss}$. For $M \in \Rep(\cI)^{\rss}$, the natural map $M \to S_{\ge r}(T_{\ge r}(M))=\cS_{\ge r}(M)$ is an isomorphism by Proposition~\ref{prop:sschar}. For $N \in \Rep(\cI)_r$, the natural map $T_{\ge r}(S_{\ge r}(N)) \to N$ is an isomorphism by general properties of Serre quotients. We thus see that $T_{\ge r}$ and $S_{\ge r}$ are mutually quasi-inverse equivalences.
\end{proof}

\begin{remark}
For $r>0$, a subquotient of an $r$-semistandard module is typically not $r$-semistandard: e.g., any proper quotient of a standard object of rank $r$ is not $r$-semistandard, since it has rank $<r$ by Theorem~\ref{thm:level}. Thus $\Rep(\cI)^{\rss}$ is not a localizing subcategory of $\Rep(\cI)$.
\end{remark}

\begin{corollary}
The forgetful functor $\Phi$ induces an equivalence $\uRep(\cI)^{\rss} \to \Rep(\cI)^{\rss}$.
\end{corollary}

\begin{proof}
It is clear that $\Phi$ maps $\uRep(\cI)^{\rss}$ into $\Rep(\cI)^{\rss}$. Moreover, the diagram
\begin{displaymath}
\xymatrix{
\uRep(\cI)^{\rss} \ar[r]^{\Phi} \ar[d]_{\ul{T}_{\ge r}} &
\Rep(\cI)^{\rss} \ar[d]^{T_{\ge r}} \\
\uRep(\cI)_r \ar[r]^{\Phi} &
\Rep(\cI)_r }
\end{displaymath}
commutes, up to isomorphism. The vertical maps are equivalences by Theorem~\ref{thm:ss}, while the bottom map is an equivalence by Proposition~\ref{prop:levequiv}. Thus the top map is an equivalence, which yields the result.
\end{proof}

\subsection{Semi-orthogonal decompositions} \label{ss:semiorth}

Given a triangulated category $\cD$ and triangulated subcategories $\cT_0, \cT_1, \ldots$, we say that the $\cT$'s give a {\bf semi-orthogonal decomposition} of $\cD$ if the following two conditions hold:
\begin{enumerate}
\item For $X \in \cT_i$ and $Y \in T_j$ with $i<j$, we have $\Hom_{\cD}(X,Y)=0$.
\item The $\cT_i$ generate $\cD$ as a triangulated category.
\end{enumerate}
In this case, we write $\cD=\langle \cT_0, \cT_1, \ldots \rangle$.

\begin{theorem} \label{thm:semiorth}
Let $\cD=\rD^b_{\fgen}(\Rep(\cI))$. For $r \ge 0$, let $\cT_r$ be the full subcategory of $\cD$ spanned by objects $M$ such that $\rH^i(M) \in \Rep(\cI)^{\rss}$ for all $i$. Then:
\begin{enumerate}
\item $\cT_r$ is a triangulated subcategory of $\cD$.
\item We have an equivalence of triangulated categories $\cT_r \cong \rD^b_{\fgen}(\Rep(\cI)^{\rss})$.
\item We have a semi-orthogonal decomposition $\cD=\langle \cT_0, \cT_1, \ldots \rangle$.
\end{enumerate}
The analogous result holds in the graded case.
\end{theorem}

\begin{proof}
(a) This follows from the fact that $\Rep(\cI)^{\rss}$ is an abelian subcategory of $\Rep(\cI)$ that is closed under extensions (Theorem~\ref{thm:ss}).

(b) The inclusion $\Rep(\cI)^{\rss} \to \Rep(\cI)$ induces a functor $F \colon \rD^b_{\fgen}(\Rep(\cI)^{\rss}) \to \cT_r$. We first claim that $F$ is fully faithful. To see this, suppose that $M$ and $N$ are objects in $\rD^b_{\fgen}(\Rep(\cI)^{\rss})$. We can then find quasi-isomorphisms $M \to I$ and $N \to J$ where $I^n$ and $J^n$ are sums of modules of the form $\bI^{\lambda}$ with $\lambda$ of rank $r$. We thus see that $\Hom(M, N)$, in either $\rD^b_{\fgen}(\Rep(\cI)^{\rss})$ or in $\rD^b_{\fgen}(\Rep(\cI))$, is computed by $\Hom_{\cK}(I, J)$, where $\cK$ is the homotopy category of complexes, and so $F$ is fully faithful.

We now show that $F$ is essentially surjective. Let $M \in \cT_r$ be given. By Proposition~\ref{prop:sschar}(d), and basic homological algebra, we can find a quasi-isomorphism $M \to I$ where each $I^n$ is a sum of modules of the form $\bI^{\lambda}$ with $\lambda$ of rank $r$. Thus $I$ defines an object of $\rD^b_{\fgen}(\Rep(\cI)^{\rss})$ and we have an isomorphism $F(I) \cong M$. Thus $F$ is essentially surjective.

(c) Suppose that $M$ and $N$ are $\cI$-modules that are $r$- and $s$-semistandard, respectively, with $r<s$. By Proposition~\ref{prop:sschar}(c), $N$ is derived $\cS_{\ge s}$-saturated, and so the natural map
\begin{displaymath}
\rR \Hom_{\cI}(M, N) \to \rR \Hom_{\Rep(\cI)_{\ge s}}(T_{\ge s}(M), T_{\ge s}(N))
\end{displaymath}
is an isomorphism. Since $T_{\ge s}(M)=0$, we thus find that $\rR \Hom_{\cI}(M, N)=0$. It follows that $\Hom_{\cD}(X, Y)=0$ for $X \in \cT_r$ and $Y \in \cT_s$. Since the $\cT$'s generate $\cD$ (by Theorem~\ref{thm:dergen}, for example), the result follows.
\end{proof}

\section{Structure of level categories} \label{s:levcat}

\subsection{Multigraded $\cI^r$-modules}

Let $r \in \bN$, and let $\cI^r$ denote the $r$-fold direct product $\cI \times \cdots \times \cI$. An {\bf $\bN^r$-grading} on a vector space $M$ is a decomposition $M=\bigoplus_{\ul{n} \in \bN^r} M_{\ul{n}}$. A {\bf graded $\cI^r$-module} is a vector space $M$ equipped with a $\bN^r$-grading such that: (1) for $x \in M_{\ul{n}}$ and $\sigma \in \cI^r$ we have $\sigma(x) \in M_{\sigma \ul{n}}$; and (2) if $x \in M_{\ul{n}}$ and $\sigma \in \cI^r$ fixes $\ul{n}$ then $\sigma(x)=x$. We let $\uRep(\cI^r)$ denote the category of graded $\cI^r$-modules.

If $M_1, \ldots, M_r$ are graded $\cI$-modules then their tensor product is naturally a graded $\cI^r$-module, which we denote by $M_1 \boxtimes \cdots \boxtimes M_r$. This defines a functor
\begin{displaymath}
\boxtimes \colon \uRep(\cI)^r \to \uRep(\cI^r).
\end{displaymath}
We say that an object of $\uRep(\cI^r)$ is {\bf factorizable} if it belongs to the essential image of $\boxtimes$.

\begin{proposition} \label{prop:tenhom}
Let $M_1, \ldots, M_r$ and $N_1, \ldots, N_r$ be finitely generated graded $\cI$-modules, and put $M=M_1 \boxtimes \cdots \boxtimes M_r$ and $N=N_1 \boxtimes \cdots \boxtimes N_r$. Then the natural map
\begin{displaymath}
\bigotimes_{i=1}^r \uHom_{\cI}(M_i, N_i) \to \uHom_{\cI^r}(M, N)
\end{displaymath}
is an isomorphism.
\end{proposition}

\begin{proof}
Suppose that we have an exact sequence
\begin{displaymath}
M_1'' \to M_1' \to M_1 \to 0
\end{displaymath}
of finitely generated graded $\cI$-modules. Let $M'$ and $M''$ be defined like $M$, but using $M_1'$ and $M_1''$, and let $H=\bigotimes_{i=2}^r \uHom_{\cI}(M_i, N_i)$. We then get a commutative diagram
\begin{displaymath}
\xymatrix{
0 \ar[r] & \uHom_{\cI}(M_1, N_1) \otimes H \ar[r] \ar[d] & \uHom_{\cI}(M_1', N_1) \otimes H \ar[r] \ar[d] & \uHom_{\cI}(M_1'', N_1) \otimes H \ar[d] \\
0 \ar[r] & \uHom_{\cI^r}(M, N) \ar[r] & \uHom_{\cI^r}(M', N) \ar[r] & \uHom_{\cI^r}(M'', N) }
\end{displaymath}
We thus see that if the right two vertical maps are isomorphisms then so is the leftmost vertical maps. Thus to prove the proposition for $(M_1, \ldots, M_r)$ (and all $N$'s) it suffices to prove it for $(M_1', M_2, \ldots, M_r)$ and $(M_1'', M_2, \ldots, M_r)$. Obviously, the same holds for the other $M_i$'s.

We can thus reduce to the case where each $M_i$ is a finitely generated projective module. Moreover, since everything is additive, we can assume that each $M_i$ is an indecomposable projective, say $M_i=\ul{\bA}^{m_i}$. Thus $\uHom_{\cI}(M_i, N_i)$ is $(N_i)_{m_i}$, the degree $m_i$ piece of $N_i$. One easily sees that $M$ is a projective object in $\uRep(\cI^r)$, and that for any graded $\cI^r$-module $K$ we have $\uHom_{\cI^r}(M, K)=K_{m_1,\ldots,m_r}$. We thus see that
\begin{displaymath}
\uHom_{\cI^r}(M, N)=N_{m_1,\dots,m_r}=(N_1)_{m_1} \otimes \cdots \otimes (N_r)_{m_r}.
\end{displaymath}
Thus the source and target of the map in question are each identified with the same space, and one easily sees that the map in question is identified with the identity map, which proves the proposition.
\end{proof}

The functor $\boxtimes$ realizes $\uRep(\cI^r)$ as the $r$-fold tensor power of the category $\uRep(\cI)$. The precise meaning of this statement is exactly the following proposition:

\begin{proposition} \label{prop:tencat}
Let $\cT$ be a Grothendieck abelian category, and let $\sF_0 \colon \uRep(\cI)^r \to \cT$ be a functor that is cocontinuous in each variable. Then there exists a unique (up to isomorphism) functor $\sF \colon \uRep(\cI^r) \to \cT$ that is cocontinuous and for which there is a functorial isomorhpism $\sF(M_1 \boxtimes \cdots \boxtimes M_r) \cong \sF_0(M_1, \ldots, M_r)$.
\end{proposition}

\begin{proof}
For $(n_1, \ldots, n_r) \in \bN^r$, let $\ul{\bA}^{n_1,\ldots,n_r}$ be the graded $\cI^r$-module $\ul{\bA}^{n_1} \boxtimes \cdots \boxtimes \ul{\bA}^{n_r}$. Let $\cQ \subset \uRep(\cI^r)$ be the full subcategory spanned by these modules. Similarly, let $\cQ_0 \subset \uRep(\cI)^r$ be the full subcategory on objects of the form $(\ul{\bA}^{n_1}, \ldots, \ul{\bA}^{n_r})$. Now, consider the following functor categories:
\begin{itemize}
\item $\cX_1$ is the category of all cocontinuous functors $\uRep(\cI^r) \to \cT$.
\item $\cX_2$ is the category of all additive functors $\cQ \to \cT$.
\item $\cX_3$ is the category of all functor $\uRep(\cI)^r \to \cT$ that are cocontinuous in each variable.
\item $\cX_4$ is the category of all functors $\cQ_0 \to \cT$ that are additive in each variable.
\end{itemize}
Consider the diagram of restriction functors:
\begin{displaymath}
\xymatrix{
\cX_1 \ar[r] \ar[d] & \cX_2 \ar[d] \\
\cX_3 \ar[r] & \cX_4 }
\end{displaymath}
The top functor is an equivalence by Proposition~\ref{prop:funonproj2}. A similar argument shows that the bottom functor is also an equivalence.

We now claim that the right functor is an equivalence. Indeed, suppose that $\sG_0 \in \cX_4$. We define a functor $\sG \colon \cQ \to \cT$ as follows. First, we put $\sG(\ul{\bA}^{n_1,\ldots,n_r})=\sG_0(\ul{\bA}^{n_1},\ldots,\ul{\bA}^{n_r})$. Now, consider the diagram
\begin{displaymath}
\xymatrix{
\prod_{i=1}^n \uHom_{\cI}(\ul{\bA}^{n_i}, \ul{\bA}^{m_i}) \ar[r] \ar[d] &
\bigotimes_{i=1}^n \uHom_{\cI}(\ul{\bA}^{n_i}, \ul{\bA}^{m_i}) \ar@{..>}[d] \\
\Hom_{\cT}(\sG(\ul{\bA}^{n_1}, \ldots, \ul{\bA}^{n_r}), \sG(\ul{\bA}^{m_1}, \ldots, \ul{\bA}^{m_r})) \ar@{=}[r] &
\Hom_{\cT}(\sG'(\ul{\bA}^{n_1,\ldots,n_r}), \sG'(\ul{\bA}^{m_1,\ldots,m_r})) }
\end{displaymath}
The left vertical map here is the one induced by $\sG_0$. Since $\sG_0$ is additive in each variable, this map is additive in each variable. The dotted arrow therefore fills in uniquely as an additive map. Since the top right group is $\uHom_{\cI^r}(\ul{\bA}^{n_1,\ldots,n_r},\ul{\bA}^{m_1,\ldots,m_r})$ by Proposition~\ref{prop:tenhom}, the right map defines $\sG$ on morphisms. Thus $\sG$ is a well-defined additive functor. One easily sees that $\sG_0 \mapsto \sG$ defines a functor $\cX_4 \to \cX_2$ quasi-inverse to the restriction functor, which establishes the claim.

It now follows that the restriction functor $\cX_1 \to \cX_3$ is an equivalence, which completes the proof.
\end{proof}

\subsection{The functor $\Omega$}

Let $\Omega_0 \colon \uRep(\cI)^{r+1} \to \uRep(\cI)$ be the functor given by
\begin{displaymath}
\Omega_0(M_0, \ldots, M_r) = M_0 \odot \ul{\bA}^1 \odot \cdots \odot \ul{\bA}^1 \odot M_r.
\end{displaymath}
This is cocontinuous in each argument. Thus, by Proposition~\ref{prop:tencat}, it uniquely extends to a cocontinuous functor $\Omega \colon \uRep(\cI^{r+1}) \to \uRep(\cI)$.

\begin{proposition}
The functor $\Omega$ is exact.
\end{proposition}

\begin{proof}
The vector space underlying $\Omega_0(M_0, \ldots, M_r)$ is canonically isomorphic to $M_0 \otimes \cdots \otimes M_r \otimes \bA^r$, where here we are simply regarding all the objects as vector spaces. Thus, if $F \colon \uRep(\cI) \to \Vec$ denotes the forgetful functor, then $F \circ \Omega$ and $M \mapsto M \otimes \bA^r$ are two cocontinuous functors $\uRep(\cI^{r+1}) \to \Vec$ that have isomorphic restrictions to $\uRep(\cI)^{r+1}$ via $\boxtimes$. Thus, by the uniqueness of Proposition~\ref{prop:tencat}, we see that they are isomorphic. In other words, the vector space underlying $\Omega(M)$ is canonically isomorphic to $M \otimes \bA^r$. As this is exact in $M$, we see that $\Omega$ is exact.
\end{proof}

\begin{proposition} \label{prop:omega-ext}
Let $M_0, \ldots, M_r$ and $N_0, \ldots, N_r$ be graded $\cI$-modules of finite length. Put $M = M_0 \boxtimes \ldots \boxtimes M_r$ and similarly define $N$. Then the canonical map
\begin{displaymath}
\bigotimes_{i=0}^r \rR \uHom_{\cI}(M_i, N_i) \to \rR \uHom_{\cI}(\Omega(M), \Omega(N))
\end{displaymath}
is an isomorphism.
\end{proposition}

\begin{proof}
Suppose that
\begin{displaymath}
0 \to N_0 \to N_0' \to N_0'' \to 0
\end{displaymath}
is an exact sequence in $\uRep(\cI)^{\rf}$, and put
\begin{displaymath}
N'=N_0' \boxtimes N_1 \boxtimes \cdots \boxtimes N_r, \qquad
N''= N_0'' \boxtimes N_1 \boxtimes \cdots \boxtimes N_r.
\end{displaymath}
Since $\boxtimes$ is exact in each variable and $\Omega$ is exact, we have an exact sequence
\begin{displaymath}
0 \to \Omega(N) \to \Omega(N') \to \Omega(N'') \to 0.
\end{displaymath}
Letting $H=\bigotimes_{i=1}^r \rR \Hom_{\cI}(M_i, N_i)$, we thus have a morphism of triangles
\begin{displaymath}
\xymatrix{
\rR \Hom_{\cI}(M_0, N_0) \otimes H \ar[r] \ar[d] &
\rR \Hom_{\cI}(M_0, N'_0) \otimes H \ar[r] \ar[d] &
\rR \Hom_{\cI}(M_0, N'_0) \otimes H \ar[r] \ar[d] &
\\
\rR \Hom_{\cI}(\Omega(M), \Omega(N)) \ar[r] &
\rR \Hom_{\cI}(\Omega(M), \Omega(N')) \ar[r] &
\rR \Hom_{\cI}(\Omega(M), \Omega(N'')) \ar[r] & }
\end{displaymath}
If two of the three vertical maps are isomorphisms, then so is the third. We thus see that if the lemma holds for two of $N_0$, $N_0'$, and $N_0''$ (with the other modules fixed) then it holds for the third as well. Of course, the same holds for the other $N_i$ and the $M$'s.

Using the above observation, we can reduce to the case where each $M_i$ is simple and each $N_i$ is injective; say $M_i=\ul{\bB}^{a_i}$ and $N_i=\ul{\bJ}^{b_i}$. Thus if $\lambda$ and $\mu$ are the constraint words with $a(\lambda)=(a_0,\ldots,a_r)$ and $a(\mu)=(b_0,\ldots,b_r)$ then $\Omega(M)=\ul{\bE}^{\lambda}$ and $\Omega(N)=\ul{\bI}^{\mu}$. Since $N$ and the $N_i$ are injective, all the $\rR \Hom$'s are concentrated in degree~0, and so it suffices to show that the map
\begin{displaymath}
\bigotimes_{i=0}^r \uHom_{\cI}(M_i, N_i) \to \uHom_{\cI}(\Omega(M), \Omega(N))
\end{displaymath}
is an isomorphism. If $\lambda=\mu$ then all these $\Hom$ spaces are one-dimensional, and the map is easily seen to be non-zero, and thus an isomorphism. Now suppose $\lambda \ne \mu$. We claim that both sides vanish. Indeed, we have $a_i \ne b_i$ for some $i$, and so $\uHom_{\cI}(M_i, N_i)=0$, as the socle of $N_i$ is $\ul{\bB}^{b_i}$, which is a simple object not isomorphic to $M_i$. Similarly, we have an injection $\ul{\bE}^{\mu} \to \ul{\bI}^{\mu}$, and so $T_{\ge r}(\ul{\bE}^{\mu})$ is the socle of $T_{\ge r}(\ul{\bI}^{\mu})$ in $\uRep(\cI)_r$. Thus $\uHom(T_{\ge r}(\ul{\bE}^{\lambda}), T_{\ge r}(\ul{\bI}^{\mu}))=0$, which implies that $\uHom_{\cI}(\ul{\bE}^{\lambda}, \ul{\bI}^{\mu})=0$ since $\ul{\bI}^{\mu}$ is $\cS_{\ge r}$-saturated.
\end{proof}

\begin{proposition} \label{prop:omega-hom}
Notation as in Proposition~\ref{prop:omega-ext}, the map
\begin{displaymath}
\uHom_{\cI^{r+1}}(M, N) \to \uHom_{\cI}(\Omega(M), \Omega(N))
\end{displaymath}
induced by $\Omega$ is an isomorphism. In other words, $\Omega$ is fully faithful on objects of $\uRep(\cI^{r+1})$ that are factorizable and of finite length.
\end{proposition}

\begin{proof}
Applying $\rH^0$ to the isomorphism in Proposition~\ref{prop:omega-ext}, we see that the natural map
\begin{displaymath}
\bigotimes_{i=0}^r \uHom_{\cI}(M_i, N_i) \to \uHom_{\cI}(\Omega(M), \Omega(N))
\end{displaymath}
is an isomorphism. Since the natural map
\begin{displaymath}
\bigotimes_{i=0}^r \uHom_{\cI}(M_i, N_i) \to \uHom_{\cI^{r+1}}(M, N)
\end{displaymath}
is also an isomorphism (Proposition~\ref{prop:tenhom}), the result follows.
\end{proof}

\subsection{The main theorem}

The categories $\uRep(\cI)^{\rss}$ and $\uRep(\cI)_r$ are somewhat mysterious; at least, it can be difficult to study them directly. The category $\uRep(\cI^{r+1})^{\lf}$, on ther other hand, is very concrete. The following theorem, and its corollary, can therefore be regarded as a solution to the problem of describing the former two categories.

\begin{theorem} \label{thm:levstruc}
The functor $\Omega$ induces an equivalence $\uRep(\cI^{r+1})^{\lf} \to \uRep(\cI)^{\rss}$.
\end{theorem}

\begin{proof}
We first observe that $\Omega$ does indeed map $\uRep(\cI^{r+1})^{\lf}$ into $\uRep(\cI)^{\rss}$. Indeed, since $\Omega$ is cocontinuous and the target category is closed under direct limits, it suffices to check on finite length objects. Since $\Omega$ is exact and $\uRep(\cI)^{\rss}$ is closed under extensions, it suffices to check on simple objects. A simple object of $\uRep(\cI^{r+1})^{\lf}$ has the form $\ul{\bB}^{a_0} \boxtimes \cdots \boxtimes \ul{\bB}^{a_r}$, and this is sent to $\ul{\bE}^{\lambda}$ under $\Omega$, where $\lambda$ is the constraint word of rank $r$ with $a(\lambda)=(a_0, \ldots, a_r)$. As $\ul{\bE}^{\lambda}$ belongs to $\uRep(\cI)^{\rss}$, the statement follows.

We next claim that $\Omega$ induces an equivalence $\IndInj(\uRep(\cI^{r+1})^{\lf}) \to \IndInj(\uRep(\cI)^{\rss})$, which will complete the proof by Propsosition~\ref{prop:indinjequiv}. For $a=(a_0,\ldots,a_r) \in \bN^{r+1}$, let $\ul{\bJ}^a=\ul{\bJ}^{a_0} \boxtimes \cdots \boxtimes \ul{\bJ}^{a_r}$. Using arguments similar to those in \S \ref{ss:fininj}, one can show that the $\ul{\bJ}^a$ are exactly the indecomposable injectives of $\uRep(\cI^{r+1})^{\lf}$. We have $\Omega(\ul{\bJ}^a)=\ul{\bI}^{\lambda}$, where $\lambda$ is the constraint word with $a(\lambda)=a$. We thus see that $\Omega$ induces a well-defined functor
\begin{displaymath}
\Omega' \colon \IndInj(\uRep(\cI^{r+1})^{\lf}) \to \IndInj(\uRep(\cI)^{\rss}).
\end{displaymath}
Since every indecomposable injective of $\uRep(\cI^{r+1})^{\lf}$ is factorizable, Proposition~\ref{prop:omega-hom} shows that $\Omega'$ is fully faithful. Finally, since every indecomposable injective of $\uRep(\cI)^{\rss}$ has the form $\ul{\bI}^{\lambda}$ for $\lambda$ of rank $r$, we see that $\Omega'$ is essentially surjective. Thus $\Omega'$ is an equivalence, and the theorem is proved.
\end{proof}

\begin{corollary} \label{cor:levstruc}
We have an equivalence $\uRep(\cI)_r \cong \uRep(\cI^{r+1})^{\lf}$.
\end{corollary}

\begin{proof}
Combine the previous equivalence with the equivalence $\uRep(\cI)^{\rss} \cong \uRep(\cI)_r$ from Theorem~\ref{thm:ss}.
\end{proof}

\begin{remark} \label{rmk:urepbest}
Consider the four nearly equivalent categories: $\uRep(\cI)$, $\uRep(\cI)_+$, $\Mod_{\OI}$, and the category of co-semi-simplicial vector spaces. If $\cA$ is any one of these categories, one can define the level $r$ subquotient $\cA_r$. For $r>0$, all four choices yield the same subquotient category. But only for $\cA=\uRep(\cI)$ is it true that $\cA_r$ is equivalent to $\cA_0^{\otimes (r+1)}$. Thus, at least from the point of view of this paper, $\uRep(\cI)$ is the best of these categories. In particular, it is most natural to have the grading for a graded $\cI$-module to start at~0.
\end{remark}

\section{Koszul duality} \label{s:koszul}

\subsection{Notation and conventions}

We now set some notation and conventions related to homological algebra. We work exclusively with cochain complexes. Thus, if $M$ is a complex then the differential $d$ maps $M^n$ to $M^{n+1}$. If $M$ is a complex, we let $M[n]$ be the shifted complex: it has $(M[n])^k=M^{n+k}$, and differentials those of $M$ scaled by $(-1)^n$. We have $\rH^k(M[n])=\rH^{n+k}(M)$. If $M$ is a module, regarded as a complex in degree~0, then $M[n]$ is in degree $-n$.

We say that a bigraded vector space (such as a complex of graded vector spaces or graded $\cI$-modules, or the homology of such a complex) is {\bf degreewise finite} if each graded piece is finite dimensional.

\subsection{Tor and minimal resolutions} \label{ss:tor}

Let $M$ be a graded $\cI$-module. Define a graded vector space $\cT(M)$ by
\begin{displaymath}
\cT(M)_n = M_n/\left(\sum_{k=1}^{n-1} \alpha_k M_{n-1} \right).
\end{displaymath}
Thus $\cT(M)_n$ records those elements of $M_n$ which cannot be generated by elements of lower degree; one can regard these as minimal generators of $M$. The functor $\cT$ is right exact, and its left derived functors exist, as $\uRep(\cI)$ has enough projectives. We call $\rL_i \cT$ the $i$th {\bf Tor functor}.

We say that a projective resolution $P_{\bullet} \to M$ is {\bf minimal} if the differentials in the complex $\cT(P_{\bullet})$ all vanish. It is not difficult to prove that $M$ admits a unique minimal projective resolution, up to isomorphism; the proof is just like that for the analogous fact in commutative algebra, see \cite[Theorem~20.2]{eisenbud}. If $P_{\bullet} \to M$ is a minimal resolution then $\rL_i \cT(M)=\cT(P_i)$. Note $\cT(\ul{\bA}^r)$ is one-dimensional and concentrated in degree $r$. We thus see that $\dim \rL_i \cT(M)_r$ is the number of $\ul{\bA}^r$ summands in the $i$th term of the minimal projective resolution of $M$.

\subsection{The Koszul complex} \label{ss:koszul}

We now aim to construct an explicit complex that computes Tor. To motivate our construction, suppose that $M$ is a graded $\cI$-module. We have a short exact sequence
\begin{displaymath}
\bk^{n-1} \otimes M_{n-1} \stackrel{d_1}{\longrightarrow} M_n \longrightarrow \cT(M)_n \to 0
\end{displaymath}
where $d_1$ is given by $e_i \otimes x \mapsto \alpha_i x$. There are some obvious elements of the kernel of $d_1$, coming from the fundamental relations. To be precise, define
\begin{displaymath}
d_2 \colon \lw^2(\bk^{n-1}) \otimes M_{n-2} \to \bk^{n-1} \otimes M_{n-1}, \qquad
e_{j,i} \otimes x \mapsto e_i \otimes \alpha_{j-1} x - e_j \otimes \alpha_i x,
\end{displaymath}
where $e_{j,i}=e_j \wedge e_i$ for $j>i$. Then
\begin{displaymath}
d_1d_2(e_{j,i} \otimes x) = \alpha_i \alpha_{j-1} x- \alpha_j \alpha_i x = 0.
\end{displaymath}
Thus $d_1 \circ d_2=0$.

Continuing in the obvious way, we are lead to the following construction. We define a complex $\cK(M)$ of graded vector spaces. The terms are given by
\begin{displaymath}
\cK^{-m}(M) = \bigoplus_{n \ge m} \lw^m(\bk^{n-1}) \otimes M_{n-m},
\end{displaymath}
where the $n$th summand has degree $n$. The differential $\cK^{-m}(M) \to \cK^{-m+1}(M)$ is given by
\begin{displaymath}
d(e_{a_m,\ldots,a_1} \otimes x) = \sum_{j=1}^m (-1)^{j+1} e_{a_m,\ldots,\hat{a}_j,\ldots,a_1} \otimes \alpha_{a_j-j+1} x
\end{displaymath}
where $1 \le a_1<a_2<\cdots<a_m \le n-1$. Note that $1 \le a_j-j+1 \le n-m$, so $\alpha_{a_j-j+1} x$ belongs to $M_{n-m+1}$; this verifies that the differential is homogeneous. We call $\cK(M)$ the {\bf Koszul complex} for the module $M$. We verify that it is indeed a complex:

\begin{proposition}
Notation as above, we have $d^2=0$.
\end{proposition}

\begin{proof}
We compute:
\begin{align*}
d^2(e_{a_n,\ldots,a_1} \otimes x)
=& d \left( \sum_{i=1}^n (-1)^{i+1} e_{a_n,\ldots,\hat{a}_i,\ldots,a_1} \otimes \alpha_{a_i-i+1} x \right) \\
=& \sum_{i=1}^n \sum_{j=1}^{i-1} (-1)^{i+j} e_{a_n,\ldots,\hat{a}_i,\ldots,\hat{a}_j,\ldots,a_1} \otimes \alpha_{a_j-j+1} \alpha_{a_i-i+1} x \\
&+ \sum_{i=1}^n \sum_{j=i+1}^n (-1)^{i+j+1} e_{a_n,\ldots,\hat{a}_j,\ldots,\hat{a}_i,\ldots, a_1} \otimes \alpha_{a_j-j+2} \alpha_{a_i-i+1} x
\end{align*}
The second sum is the negative of the first: to see this, first change the order of the $\alpha$'s (using the fundamental relation), then switch the roles of $i$ and $j$, and finally change the order of summation. \end{proof}

We next study the Koszul complex on principal modules. We begin with a lemma.

\begin{lemma}
Let $1 \le c_1<\cdots<c_m \le n$ be integers, and let $a \in [n] \setminus \{c_1,\ldots,c_m\}$. Then there exists a unique $a' \in [n-m]$ such that $a=\alpha_{c_m} \cdots \alpha_{c_1} a'$.
\end{lemma}

\begin{proof}
Suppose the result has been proved for $m-1$, and let us prove it for $m$. We consider two cases.

First suppose that $a<c_m$. We have $1 \le c_1 < \cdots \le c_{m-1} \le n-1$ and $a \in [n-1] \setminus \{c_1,\ldots,c_{m-1}\}$. By the inductive hypothesis, we can find $a' \in [(n-1)-(m-1)]=[n-m]$ such that $a=\alpha_{c_{m-1}} \cdots \alpha_{c_1} a'$. As $a=\alpha_{c_m}a$, we have $a=\alpha_{c_m} \cdots \alpha_{c_1} a'$.

Now suppose that $a>c_m$. We have $1 \le c_1 < \cdots \le c_{m-1} \le n-1$ and $a-1 \in [n-1] \setminus \{c_1,\ldots,c_{m-1}\}$. By the inductive hypothesis, we can find $a' \in [n-m]$ such that $a-1=\alpha_{c_{m-1}} \cdots \alpha_{c_1} a'$. As $a=\alpha_{c_m}(a-1)$, we have $a=\alpha_{c_m} \cdots \alpha_{c_1} a'$, as required.
\end{proof}

\begin{proposition} \label{prop:Kprin}
We have $\rH^i(\cK(\ul{\bA}^r))=0$ for all $r \in \bN$ and $i \ne 0$.
\end{proposition}

\begin{proof}
For $r=1$, the complex $\cK(\ul{\bA}^1)_n$ is the complex computing the simplicial homology of the $n-2$ simplex, and is therefore acyclic away from cohomological degree~0.

We now treat the general case. We claim that the complex $\cK(\ul{\bA}^r)_n$ is isomorphic to a direct sum of $\binom{n-1}{r-1}$ copies of the complex $\cK(\ul{\bA}^1)_{n-r+1}$, which will prove the proposition. To see this, fix a tuple $a=(a_1, \ldots, a_{r-1})$ with $1 \le a_1 < \cdots < a_{r-1} \le n-1$. Let $\sigma \colon [n-r] \to [n-1]$ be the unique order-preserving injection whose image is disjoint from $\{a_1,\ldots,a_{r-1}\}$. For $1 \le b_1<\cdots<b_m \le n-r$, let $1 \le a_i' \le n-m-1$ be the unique integer such that $a_i=\alpha_{\sigma(b_m)} \cdots \alpha_{\sigma(b_1)} a_i'$, which exists by the previous lemma. We define a map of complexes
\begin{displaymath}
\phi_a \colon \cK(\ul{\bA}^1)_{n-r+1} \to \cK(\ul{\bA}^r)_n,
\end{displaymath}
as follows. In cohomological degree $-m$, it is the map
\begin{displaymath}
\lw^m(\bk^{n-r}) \to \lw^m(\bk^{n-1}) \otimes \bA^r_{n-m}, \qquad
e_{b_m,\ldots,b_1} \mapsto e_{\sigma(b_m),\ldots,\sigma(b_1)} \otimes e_{a_1',\ldots,a_{r-1}',n-m}.
\end{displaymath}
(Here we have identified each homogeneous piece of $\ul{\bA}^1$ with $\bk$.) We now verify that $\phi_a$ is a map of complexes. We have
\begin{displaymath}
d(e_{b_m,\ldots,b_1}) = \sum_{j=1}^m (-1)^{j+1} e_{b_m,\ldots,\hat{b}_j,\ldots,b_1}
\end{displaymath}
and
\begin{displaymath}
d(e_{\sigma(b_m),\ldots,\sigma(b_1)} \otimes e_{a_1',\ldots,a_{r-1}',n-m}) =
\sum_{j=1}^m (-1)^{j+1} e_{\sigma(b_m),\ldots,\widehat{\sigma(b_j)},\ldots,\sigma(b_1)} \otimes \alpha_{\sigma(b_j)-j+1} e_{a_1',\ldots,a_{r-1}',n-m}.
\end{displaymath}
It thus suffices to show that
\begin{displaymath}
\phi_a(e_{b_m,\ldots,\hat{b}_j,\ldots,b_1}) = e_{\sigma(b_m),\ldots,\widehat{\sigma(b_j)},\ldots,\sigma_a(b_1)} \otimes \alpha_{\sigma(b_j)-j+1} e_{a_1',\ldots,a_{r-1}',n-m}.
\end{displaymath}
Let $\tau=\alpha_{\sigma(b_m)} \cdots \widehat{\alpha_{\sigma(b_j)}} \cdots \alpha_{\sigma(b_1)}$ and let $a_i''$ be the unique integer such that $a_i=\tau a_i''$. We have
\begin{displaymath}
\phi_a(e_{b_m,\ldots,\hat{b}_j,\ldots,b_1}) = e_{\sigma(b_m),\ldots,\widehat{\sigma(b_j)},\ldots,\sigma(b_1)} \otimes e_{a_1'',\ldots,a_{r-1}'',n-m-1}.
\end{displaymath}
It thus suffices to show that $a_i''=\alpha_{\sigma(b_j)-j+1} a_j'$. We can verify this after appying $\tau$ to each side. Thus we must show that $a_i=\tau \alpha_{\sigma(b_j)-j+1} a_i'$. But
\begin{displaymath}
\tau \alpha_{\sigma(b_j)-j+1} = \alpha_{\sigma(b_m)} \cdots \alpha_{\sigma(b_1)},
\end{displaymath}
and so the result follows from the definition of $a_i'$.

We thus have a map of complexes
\begin{displaymath}
\phi \colon \lw^{r-1}(\bk^{n-1}) \otimes \cK(\ul{\bA}^1)_{n-r+1} \to \cK(\ul{\bA}^r)_n, \qquad
e_{a_1,\dots,a_r} \otimes x \mapsto \phi_a(x).
\end{displaymath}
In each cohomological degree, $\phi$ carries basis elements to basis elements. One easily verifies that it is a bijeciton on basis elements, and so $\phi$ is an isomorphism of complexes. This completes the proof.
\end{proof}

We have defined the Koszul complex $\cK(M)$ of a graded $\cI$-module $M$. More generally, suppose that $M$ is complex of graded $\cI$-modules. We then define
\begin{displaymath}
\cK^{-m}(M) = \bigoplus_{k \in \bZ} \bigoplus_{n \ge m} \lw^k(\bk^{n-1}) \otimes M^{k-m}_{n-k},
\end{displaymath}
where the $(k,n)$ summand is placed in degree $n$. The differential is given by
\begin{displaymath}
d(e_{a_k,\ldots,a_1} \otimes x) = (-1)^{n+k} e_{a_k,\ldots,a_1} \otimes dx + \sum_{j=1}^k (-1)^{j+1} e_{a_k,\ldots,a_{j+1},a_{j-1},\ldots,a_1} \otimes \alpha_{a_j-j+1} x,
\end{displaymath}
where $x \in M^{k-m}_{n-k}$. This is easily seen to square to~0. We note that we have a natural morphism $\cK(M) \to \cT(M)$ in $\Ch(\GrVec)$: this map is zero on $\lw^k(\bk^{n-1}) \otimes M^{k-m}_{n-k}$ if $k>0$, and the quotient map $M^{-m}_n \to \cT(M^{-m}_n)$ if $k=0$.

\begin{proposition} \label{prop:kosztor}
The functor $\cK$ induces a functor $\rD(\uRep(\cI)) \to \rD(\GrVec)$, which is the left derived functor of $\cT$.
\end{proposition}

\begin{proof}
The functor $\cK$ satisfies the hypotheses of Proposition~\ref{prop:acyclic}, and so, by that proposition, we see that $\cK$ takes quasi-isomorphisms to quasi-isomorphisms. It thus induces a functor on the derived categories. Suppose now that $M$ is a bounded above complex in $\uRep(\cI)$, and choose a quasi-isomorphism $P \to M$ where $P$ is a bounded above complex of projectives. Then the left derived functor of $\cT$ is defined at $M$ and equal to $\cT(P)$. Since the maps $\cK(M) \leftarrow \cK(P) \to \cT(P)$ are quasi-isomorphisms (for the second map, this follows from Proposition~\ref{prop:Kprin}), we see that $\cK(M)$ is the value of the derived functor as well. This shows that $\cK$ is the left derived functor of $\cT$ on the bounded above category. The general case can be deduced from a limiting argument (every complex is a filtered colimit of bounded above complexes, and both $\cK$ and $\rL \cT$ commute with filtered colimits); we will not actually need this, so we omit the details.
\end{proof}

\begin{corollary} \label{cor:kosztor}
Let $M$ be a graded $\cI$-module. Then $\rL_i \cT(M) \cong \rH^{-i}(\cK(M))$.
\end{corollary}

\subsection{The duality functor} \label{ss:duality}

For $1 \le k \le n$, let $T_k \colon \ul{\bA}^{n+1} \to \ul{\bA}^n$ be the map given by
\begin{displaymath}
T_k(e_{i_1,\ldots,i_{n+1}}) = e_{i_1,\ldots,i_{k-1},i_{k+1},\ldots,i_{n+1}}.
\end{displaymath}
The $T_k$ give a basis for $\uHom_{\cI}(\ul{\bA}^{n+1}, \ul{\bA}^n)$. Given a complex $M \in \Ch(\uRep(\cI))$ we now define another complex $\cD(M) \in \Ch(\uRep(\cI))$. The terms are given as follows:
\begin{displaymath}
\cD^{-m}(M) = \bigoplus_{k \in \bN} (M_k^{m-k})^* \otimes \ul{\bA}^k.
\end{displaymath}
Here, $(M_k^{m-k})^*$ is simply regarded as an ungraded vector space; the $\uRep(\cI)$ structure on $\cD^{-m}(M)$ comes entirely from $\ul{\bA}^k$. For $\lambda \in (M_k^{m-k})^*$ and $x \in \ul{\bA}^k$, the differential is given by
\begin{displaymath}
d(\lambda \otimes x) =(-1)^k d^*(\lambda) \otimes x + \sum_{i=1}^{k-1} (-1)^{i+1} \alpha_i^*(\lambda) \otimes T_i(x)
\end{displaymath}
where $d^*$ is ``naive'' dual to the differential on $M$ (i.e., we do not introduce any signs). We note that the differential is a morphism in the category $\uRep(\cI)$. We verify below (Lemma~\ref{lem:DK}) that this is actually a complex. Granted this, we have thus defined a functor
\begin{displaymath}
\cD \colon \Ch(\uRep(\cI))^{\op} \to \Ch(\uRep(\cI)).
\end{displaymath}
If this construction seems capricious, see \S \ref{ss:conceptual} for an explanation of its origin. We note that when $M$ is a graded $\cI$-module, regarded as a complex in degree~0, the complex $\cD(M)$ simplifies to 
\begin{displaymath}
\cdots \to M_2^* \otimes \ul{\bA}^2 \to M_1^* \otimes \ul{\bA}^1 \to M_0^* \otimes \ul{\bA}^0 \to 0 \to \cdots
\end{displaymath}
with differential
\begin{displaymath}
d(\lambda \otimes x) = \sum_{i=1}^{k-1} (-1)^{i+1} \alpha_i^*(\lambda) \otimes T_i(x), \qquad
\lambda \otimes x \in M_k^* \otimes \ul{\bA}^k.
\end{displaymath}
The main result of this section is the following theorem:

\begin{theorem} \label{thm:duality}
The functor $\cD$ induces a contravariant functor $\rD(\uRep(\cI)) \to \rD(\uRep(\cI))$. Moreover, there is a canonical natural transformation $M \to \cD(\cD(M))$ in $\rD(\uRep(\cI))$ that is an isomorphism if $\rH^{\bullet}(M)$ is degreewise finite.
\end{theorem}

The following lemma shows that $\cD(M)$ is just the Kozul complex $\cK(M)$ with the terms shuffled around and dualized.

\begin{lemma} \label{lem:DK}
Let $M \in \Ch(\uRep(\cI))$.
\begin{enumerate}
\item We have a canonical isomorphism of vector spaces $\cD^{-m}(M)_n \cong (\cK^{m-n}(M)_n)^*$.
\item The isomorphism from (a) is compatible with the differentials, in the sense that the diagram
\begin{displaymath}
\xymatrix{
\cD^{-m}(M)_n \ar[r]^d \ar@{=}[d] & \cD^{-m+1}(M)_n \ar@{=}[d] \\
(\cK^{m-n}(M)_n)^* \ar[r]^{d^*} & (\cK^{m-n-1}(M)_n)^* }
\end{displaymath}
commutes.
\item $\cD(M)$ is a complex, that is, $d^2=0$.
\item We have an isomorphism $\cD(M)_n \cong \cK(M)_n^*[n]$ in $\Ch(\Vec)$.
\item We have an isomorphism of vector spaces $\rH^{-m}(\cD(M))_n \cong \rH^{m-n}(\cK(M))_n^*$.
\end{enumerate}
\end{lemma}

\begin{proof}
(a) We have
\begin{displaymath}
\cD^{-m}(M)_n = \bigoplus_{k \in \bZ} (M_k^{m-k})^* \otimes \ul{\bA}^k_n
\end{displaymath}
and
\begin{displaymath}
\cK^{m-n}(M)_n = \bigoplus_{k \in \bZ} \lw^{n-k}(\bk^{n-1}) \otimes M^{m-k}_k.
\end{displaymath}
In the second sum, we have changed $k$ to $n-k$. For notational ease, put
\begin{displaymath}
D^{m,n,k}=(M_k^{m-k})^* \otimes \ul{\bA}^k_n, \qquad
K^{m,n,k}=\lw^{n-k}(\bk^{n-1}) \otimes M^{m-k}_k
\end{displaymath}
so that $\cD^{-m}(M)_n$ is the sum of the $D^{m,n,k}$ and $\cK^{m-n}(M)_n$ is the sum of the $K^{m,n,k}$. We note that $D^{m,n,k}=K^{m,n,k}=0$ if $k<0$ or $k>n$. In particular, the above direct sums are finite.

Now, $\ul{\bA}^k_n$ has a basis consisting of elements $e_{i_1,\ldots,i_k}$ with $1 \le i_1 < \cdots < i_k=n$. We identify this with $\lw^{k-1}(\bk^{n-1})$ in the obvious manner (discard $i_k$). Finally, $\lw^{k-1}(\bk^{n-1})$ and $\lw^{n-k}(\bk^{n-1})$ admit a canonical perfect pairing to the one-dimensional space $\lw^{n-1}(\bk^{n-1})$. Via the isomorphism $\mu \colon \lw^{n-1}(\bk^{n-1}) \to \bk$ mapping $e_{1,2,\ldots,n-1}$ to~1, this identifies $\lw^{k-1}(\bk^{n-1})$ with the dual of $\lw^{n-k}(\bk^{n-1})$ and vice versa. We thus see that the terms in the sum for $\cD^{-m}(M)_n$ are identified with the terms in the sum for $\cK^{m-n}(M)_n$. Since both sums are finite, this identifies $\cD^{-m}(M)_n$ with the dual of $\cK^{m-n}(M)_n$.

(b) Let
\begin{displaymath}
\langle, \rangle \colon \cD^{-m}(M)_n \times \cK^{m-n}(M)_n \to \bk.
\end{displaymath}
be the perfect pairing from part (a). Explicitly, given $\lambda \otimes v \in D^{m,n,k}$, with $\lambda \in (M_k^{m-k})^*$ and $v \in \ul{\bA}^k_n \cong \lw^{k-1}(\bk^{n-1})$, and $w \otimes x \in K^{m,n,\ell}$, with $w \in \lw^{n-\ell}(\bk^{n-1})$ and $x \in M_{\ell}^{m-\ell}$, we have
\begin{displaymath}
\langle \lambda \otimes v, w \otimes x \rangle = \delta_{k,\ell} \mu(v \wedge w) (\lambda, x),
\end{displaymath}
where $(,)$ is the canonical pairing between $M^{m-k}_k$ and its dual.

To prove that the isomorphism of $\cD$ with the dual of $\cK$ is compatible with the differentials, it is equivalent to show that the differentials are adjoint with respect to the pairing. Thus let $\lambda \otimes v \in D^{m,n,k}$ and $w \otimes x \in K^{m-1,n,\ell}$ be given. We must show
\begin{displaymath}
\langle d(\lambda \otimes v), w \otimes x \rangle = \langle \lambda \otimes v, d(x \otimes w) \rangle.
\end{displaymath}
It suffices to treat the case where $v$ and $w$ are pure tensors, say $v=e_{a_1,\ldots,a_{k-1}}$ with $1 \le a_1 < \cdots < a_{k-1} \le n-1$ and $w=e_{b_1,\ldots,b_{n-\ell}}$ with $1 \le b_1 < \cdots < b_{n-\ell} \le n-1$. We have
\begin{displaymath}
d(\lambda \otimes v) = (-1)^k d^*(\lambda) \otimes v + \sum_{i=1}^k (-1)^{i+1} \alpha_i^*(\lambda) \otimes e_{a_1,\ldots,\hat{a}_i,\ldots,a_{k-1}}.
\end{displaymath}
Now, the first term above belongs to $D^{m-1,n,k}$ and the second belongs to $D^{m-1,n,k-1}$. We thus have
\begin{displaymath}
\langle d(\lambda \otimes v), w \otimes x \rangle = \begin{cases}
(-1)^k \mu(v \wedge w) ( \lambda, dx) & \text{if $k=\ell$} \\
\sum_{j=1}^k (-1)^{j+1} \mu(e_{a_1,\ldots,\hat{a}_j,\ldots,a_{k-1},b_1,\ldots,b_{n-\ell}}) (\lambda, \alpha_j x) & \text{if $k=\ell+1$} \\
0 & \text{otherwise}
\end{cases}
\end{displaymath}
On the other hand, we have
\begin{displaymath}
d(w \otimes x) = (-1)^{\ell} w \otimes dx + \sum_{i=1}^{n-\ell} (-1)^{i+1} e_{b_1,\ldots,\hat{b}_i,\ldots,b_{n-\ell}} \otimes \alpha_{b_i-i+1} x.
\end{displaymath}
The first term belongs to $K^{m,n,\ell}$ and the second belongs to $K^{m,n,\ell+1}$. We thus see that
\begin{displaymath}
\langle \lambda \otimes v, d(w \otimes x) \rangle = \begin{cases}
(-1)^k \mu(v \wedge w) (\lambda, dx) & \text{if $k=\ell$} \\
\sum_{i=1}^{n-\ell} (-1)^{i+1} \mu(e_{a_1,\ldots,a_k,b_1,\ldots,\hat{b}_i,\ldots,b_{n-l}}) (\lambda, \alpha_{b_i-i+1} x) & \text{if $k=\ell+1$} \\
0 & \text{otherwise}
\end{cases}
\end{displaymath}
If $k \ne \ell+1$ then we clearly have our desired equality. Thus let us now assume $k=\ell+1$. We must prove
\begin{displaymath}
\sum_{j=1}^{\ell+1} (-1)^{j+1} \mu(e_{a_1,\ldots,\hat{a}_j,\ldots,a_{\ell},b_1,\ldots,b_{n-\ell}}) (\lambda, \alpha_j x)
= \sum_{i=1}^{n-\ell} (-1)^{i+1} \mu(e_{a_1,\ldots,a_{\ell+1},b_1,\ldots,\hat{b}_i,\ldots,b_{n-\ell}}) (\lambda, \alpha_{b_i-i+1}x).
\end{displaymath}
The only way for either side to be non-empty is if the sets $\{a_1,\ldots,a_{\ell+1}\}$ and $\{b_1,\ldots,b_{n-\ell}\}$ have exactly one element in common. Thus suppose this is the case, and that $a_j=b_i$. Then it is the $j$th term on the left and $i$th term on the right that are possibly non-zero, so we must show that they coincide. Now, we have
\begin{displaymath}
\{1,\ldots,b_i\} = \{1,\ldots,a_j\} \amalg \{1,\ldots,b_{i-1}\},
\end{displaymath}
and so, counting, we find $b_i=j+i-1$. Thus $\alpha_{b_i-i+1}=\alpha_j$. Furthermore,
\begin{displaymath}
\mu(e_{b_{n-\ell},\ldots,\hat{b}_i,\ldots,b_1,a_1,\ldots,a_{\ell+1}})=(-1)^{i+j} \mu(e_{b_{n-\ell},\ldots,b_1,a_1,\ldots,\hat{a}_j,\ldots,a_{\ell}}).
\end{displaymath}
as it takes $i+j-2$ transpositions to transform the subscript sequence on left to the one on the right: it clearly takes zero transpositions if $i=j=1$, and one more every time $i$ or $j$ is incremented. This proves our desired equality.

(c--e) These statements follow easily from parts (a) and (b).
\end{proof}

\begin{remark}
The lemma implies that $\bigoplus_{n \ge 0} \rH^{m-n}(\cK(M))^*_n$ canonically carries the structure of a graded $\cI$-module, as it is naturally isomorphic to $\rH^{-m}(\cD(M))$. When $M$ is a module, this is the dual of the $m$th linear strand of the minimal resolution of $M$.
\end{remark}

\begin{lemma}
The functor $\cD$ takes quasi-isomorphisms to quasi-isomorphisms. In particular, it induces a functor on the derived category.
\end{lemma}

\begin{proof}
Let $f \colon M \to N$ be a quasi-isomorphism in $\Ch(\uRep(\cI))$. For any $m,n \in \bZ$, the diagram
\begin{displaymath}
\xymatrix{
\rH^{-m}(\cD(N))_n \ar[r] \ar@{=}[d] & \rH^{-m}(\cD(M))_n \ar@{=}[d] \\
\rH^{m-n}(\cK(N))_n^* \ar[r] & \rH^{m-n}(\cK(M))_n^* }
\end{displaymath}
commutes, where the vertical isomorphisms are those from Lemma~\ref{lem:DK} and the horizontal maps are induced by $f$. Since $\cK$ takes quasi-isomorphisms to quasi-isomorphisms, it follows that the bottom arrow is an isomorphism. Thus the top arrow is as well, which implies that $\cD(f) \colon \cD(N) \to \cD(M)$ is a quasi-isomorphism.
\end{proof}

\begin{proposition} \label{prop:Dprin}
We have $\cD(\ul{\bB}^n) \cong \ul{\bA}^n[n]$ and $\cD(\ul{\bA}^n)=\ul{\bB}^n[n]$.
\end{proposition}

\begin{proof}
It follows directly from the definition that $\cD^{-m}(\ul{\bB}^n)$ is zero unless $m=n$, and it is then $\ul{\bA}^n$. Thus $\cD^{-m}(\ul{\bB}^n)=\ul{\bA}^n[n]$, as claimed.

From Lemma~\ref{lem:DK}, we have $\rH^{-m}(\cD(\ul{\bA}^r))_n \cong \rH^{m-n}(\cK(\ul{\bA}^r))_n^*$. By Proposition~\ref{prop:Kprin}, $\rH^i(\cK(\ul{\bA}^r))_n$ is one-dimensional if $i=0$ and $n=r$, and vanishes otherwise. Thus $\rH^{-m}(\cD(\ul{\bA}^r))_n$ is one-dimensional if $m=n=r$, and vanishes otherwise. It follows that $\cD(\ul{\bA}^r)=\ul{\bB}^r[r]$, as claimed.
\end{proof}

For a complex of graded $\cI$-modules $M$, define a complex $\cE(M)$ as follows. The terms are given by
\begin{equation} \label{eq:E1}
\cE(M)^n = \bigoplus_{k,\ell \in \bN} M^{k+n-\ell}_{\ell} \otimes (\ul{\bA}^{\ell}_k)^* \otimes \ul{\bA}^k.
\end{equation}
For $x \in M_{\ell}^{k+n-\ell}$ and $\lambda \in (\ul{\bA}^{\ell}_k)^*$ and $y \in \ul{\bA}^k$, the differential is given by
\begin{equation} \label{eq:E2}
\begin{aligned}
d(x \otimes \lambda \otimes y) =&
(-1)^{k+\ell} dx \otimes \lambda \otimes y +
(-1)^k \sum_{i=1}^{\ell} \left[ (-1)^{i+1} \alpha_i(x) \otimes T_i^*(\lambda) \otimes y \right] \\
&+ \sum_{j=1}^{k-1} \left[ (-1)^{j+1} x \otimes \alpha_j^*(\lambda) \otimes T_j(y) \right].
\end{aligned}
\end{equation}
We note that $\cE(M)$ is exactly the complex one obtains by computing $\cD(\cD(M))$ and formally moving duals over tensor products and removing double duals.

\begin{lemma} \label{lem:EDD}
For $M \in \Ch(\uRep(\cI))$ we have a canonical injection $\cE(M) \to \cD(\cD(M))$. If $M$ is degreewise finite then this map is an isomorphism of complexes.
\end{lemma}

\begin{proof}
This follows from the description of $\cE(M)$ given above.
\end{proof}

\begin{lemma} \label{lem:Eid}
For $M \in \Ch(\uRep(\cI))$ we have a canonical quasi-isomorphism $\cE(M) \to M$.
\end{lemma}

\begin{proof}
Let $\epsilon \colon (\ul{\bA}^k_k)^* \to \bk$ be the isomorphism taking the standard basis vector of $(\ul{\bA}^k_k)^*$ to $1 \in \bk$. For a graded $\cI$-module $N$, there is a canonical map $\ul{\bA}^k \otimes N_k \to N$ of graded $\cI$-modules stemming from the mapping property of $\ul{\bA}^k$. We let $y \ast x \in N$ denote the image of $y \otimes x \in \ul{\bA}^k \otimes N_k$.

Define a map $f \colon \cE(M) \to M$ as follows. Let $x \otimes \lambda \otimes y$ be an element of $\cE(M)^n$, using notation as in \eqref{eq:E2}. We define:
\begin{displaymath}
f(x \otimes \lambda \otimes y) = \begin{cases}
0 & \text{if $k \ne \ell$} \\
(-1)^{\binom{k}{2}} \cdot \epsilon(\lambda) \cdot (y \ast x) & \text{if $k=\ell$} \end{cases}
\end{displaymath}
We claim that $f$ is a map in the category $\Ch(\uRep(\cI))$. On each term, it is a map in $\uRep(\cI)$ by definition, so it suffices to prove that $f$ commutes with the differentials. Thus let $x \otimes \lambda \otimes y$ as above be given. We proceed in four cases:

\textit{Case 1:} $k<\ell$. There is nothing to prove, as $\ul{\bA}^{\ell}_k=0$.

\textit{Case 2:} $k=\ell$. In \eqref{eq:E2}, we have $T_i^*(\lambda) \in (\ul{\bA}^{k+1}_k)^*=0$ and $\alpha_j^*(\lambda) \in (\ul{\bA}^k_{k-1})^*=0$; furthermore, $(-1)^{k+\ell}=1$ since $k=\ell$. We thus find
\begin{displaymath}
d(x \otimes \lambda \otimes y) = dx \otimes \lambda \otimes y.
\end{displaymath}
Hence,
\begin{displaymath}
df(x \otimes \lambda \otimes y) = (-1)^{\binom{k}{2}} \cdot \epsilon(\lambda) \cdot d(y \ast x), \qquad
f(d(x \otimes \lambda \otimes y))=(-1)^{\binom{k}{2}} \cdot \epsilon(\lambda) \cdot (y \ast dx).
\end{displaymath}
These quantites are equal, since the map $\ast \colon \ul{\bA}^k \otimes N_k \to N$ is functorial in the graded $\cI$-module $N$.

\textit{Case 3:} $k=\ell+1$. We have $f(x \otimes \lambda \otimes y)=0$, and so we must show that $f(d(x \otimes \lambda \otimes y))=0$. Now, in \eqref{eq:E2}, the first time is killed by $f$, since $\lambda \in (\ul{\bA}^{k-1}_k)^*$. We thus have
\begin{align*}
f(d(x \otimes \lambda \otimes y))
=& (-1)^{k+\binom{k}{2}} \sum_{i=1}^{k-1} \left[ (-1)^{i+1} \epsilon(T_i^*(\lambda)) \cdot (y \ast \alpha_i(x)) \right] \\
&+ (-1)^{\binom{k-1}{2}} \sum_{j=1}^{k-1} \left[ (-1)^{j+1} \epsilon(\alpha_j^*(\lambda)) \cdot (T_j(y) \ast x) \right]
\end{align*}
The signs are always opposite of each other. It thus suffices to show
\begin{displaymath}
\epsilon(T_i^*(\lambda))=\epsilon(\alpha_i^*(\lambda)), \qquad
y \ast \alpha_i(x) = T_i(y) \ast x,
\end{displaymath}
which we leave to the reader.

\textit{Case 4:} $k>\ell+1$. Both $x \otimes \lambda \otimes y$ and $d(x \otimes \lambda \otimes y)$ are killed by $f$, as follows directly from the definition.

We now show that $f$ is a quasi-isomorphism. By Proposition~\ref{prop:qi}, it is enough to treat the case where $M=\ul{\bA}^r$, regarded as a complex concentrated in degree~0. We thus assume this is the case. By Lemma~\ref{lem:EDD} and Proposition~\ref{prop:Dprin}, we have $\cE(\ul{\bA}^r) \cong \cD(\cD(\ul{\bA}^r)) \cong \ul{\bA}^r$ in $\rD(\uRep(\cI))$. Thus to verify the claim, it suffices to show that the map $\rH^0(f) \colon \rH^0(\cE(\ul{\bA}^r)) \to \ul{\bA}^r$ is an isomorphism. Since $\rH^0(\cE(\ul{\bA}^r)) \cong \bA^r$, it is enough to show that this map is surjective in degree $r$. Since $\cE(\ul{\bA}^r)^1=0$, it suffices to show that $f \colon \cE(\ul{\bA}^r)^0_r \to \ul{\bA}^r_r$ is surjective. This is clear from the definition of $f$.
\end{proof}

\begin{lemma} \label{lem:vecqi}
Let $W \in \Ch(\Vec)$ be a chain complex of vector spaces such that each cohomology group is finite dimensional, and let $U \subset W$ be a subcomplex such that each term is finite dimensional. Then there exists an intermediate complex $U \subset V \subset W$ such that each term of $V$ is finite dimensional and the inclusion $V \to W$ is a quasi-isomorphism.
\end{lemma}

\begin{proof}
Left to the reader.
\end{proof}

\begin{lemma} \label{lem:degfin}
Let $M \in \Ch(\uRep(\cI))$. Suppose that $\rH^{\bullet}(M)$ is degreewise finite. Then there exists a subcomplex $N \subset M$ that is degreewise finite and for which the inclusion is a quasi-isomorphism.
\end{lemma}

\begin{proof}
We construct a chain $N_0 \subset N_1 \subset \cdots \subset M$ of subcomplexes with the following properties: (a) each term of $N_n$ is finitely generated (and thus degreewise finite); (b) $N_n \to M$ is quasi-isomorphism in degrees $\le n$; (c) $N_n$ and $N_{n+1}$ are equal in degrees $\le n$. Taking $N=\bigcup_{n \ge 0} N_n$ gives the requisite subcomplex.

Suppose $N_{n-1}$ has been constructed; for $n=0$ we take $N_{n-1}=0$. We now construct $N_n$. Appealing to Lemma~\ref{lem:vecqi}, choose a subcomplex $V$ of $M_n$ containing $(N_{n-1})_n$ such that the terms of $V$ are finite dimensional and the inclusion $V \to M_n$ is a quasi-isomorphism. We then take $N_n$ to be the sum (inside of $M$) of $N_{n-1}$ and the complex of submodules generated by $V$. Clearly, $N_n$ coincides with $N_{n-1}$ in degrees $\le n-1$, and $N_n$ coincides with $V$ in degree $n$. Thus it has the required properties.
\end{proof}

\begin{proof}[Proof of Theorem~\ref{thm:duality}]
Let $M \in \Ch(\uRep(\cI))$. By Lemmas~\ref{lem:EDD} and~\ref{lem:Eid}, we have maps
\begin{displaymath}
M \leftarrow \cE(M) \to \cD(\cD(M)).
\end{displaymath}
Since the first map is a quasi-isomorphism, this diagram defines a morphism $M \to \cD(\cD(M))$ in the derived category. To complete the proof, it suffices to show that the second map is a quasi-isomorphism when $\rH^{\bullet}(M)$ is degreewise finite. Thus suppose this is the case. By Lemma~\ref{lem:degfin}, we can find a subcomplex $N \subset M$ that is degreewise finite and for which the inclusion is a quasi-isomorphism. Now consider the following diagram:
\begin{displaymath}
\xymatrix{
\cE(N) \ar[r] \ar[d] & \cD(\cD(N)) \ar[d] \\
\cE(M) \ar[r] & \cD(\cD(M)) }
\end{displaymath}
Since $\cE$ and $\cD$ take quasi-isomorphisms to quasi-isomorphisms, the vertical maps are quasi-isomorphisms. By Lemma~\ref{lem:Eid}, the top map is a quasi-isomorphism. Thus the bottom map is as well.
\end{proof}

\subsection{Duality and the concatenation product}

We now investiage how the duality functor interacts with the concatenation product. Given complexes $M$ and $N$ of graded $\cI$-modules, we let $M \odot N$ be the complex with
\begin{displaymath}
(M \odot N)^n = \bigoplus_{i+j=n} M^i \odot N^j
\end{displaymath}
and differential
\begin{displaymath}
d(x \odot y) = dx \odot y + (-1)^i x \odot dy
\end{displaymath}
for $x \in M^i$ and $y \in N^j$.

\begin{proposition} \label{prop:dualcat}
For any $M,N \in \Ch(\uRep(\cI))$ we have a canonical morphism
\begin{displaymath}
\cD(M) \odot \cD(N) \to \cD(M \odot N)
\end{displaymath}
in $\Ch(\uRep(\cI))$. It is an isomorphism if at least one of $M$ or $N$ is degreewise finite.
\end{proposition}

\begin{proof}
For $K \in \Ch(\uRep(\cI))$, we have
\begin{displaymath}
\cD(K) = \bigoplus_{n,r} (K^n_r)^*[n+r] \otimes \ul{\bA}^r,
\end{displaymath}
as a graded object of $\uRep(\cI)$. In this form, the differential is given by
\begin{displaymath}
d(\lambda \otimes x) = (-1)^r d^*(\lambda) \otimes x + \sum_{i=1}^{r-1} (-1)^{i+1} \alpha_i^*(\lambda) \otimes T_i(x)
\end{displaymath}
for $\lambda \in (K^n_r)^*$ and $x \in \ul{\bA}^i$. We will use this description of $\cD$ in this proof.

We have
\begin{displaymath}
\cD(M) \odot \cD(N) = \bigoplus_{m,n,r,s} ((M^m_r)^*[m+r] \otimes \ul{\bA}^r) \odot ((N^n_s)^*[n+s] \otimes \ul{\bA}^s).
\end{displaymath}
On the other hand,
\begin{displaymath}
\cD(M \odot N) = \bigoplus_{m,n,r,s} (M^m_r \otimes N^n_s)^*[m+n+r+s] \otimes \ul{\bA}^{r+s}.
\end{displaymath}
We thus have a map
\begin{displaymath}
\phi \colon \cD(M) \odot \cD(N) \to \cD(M \odot N)
\end{displaymath}
given by
\begin{displaymath}
\phi(\lambda \otimes x \odot \mu \otimes y) = \epsilon^{m,n}_{r,s} (\lambda \odot \mu) \otimes (x \odot y),
\end{displaymath}
for $\lambda \in (M^m_r)^*$, $\mu \in (N^n_s)^*$, $x \in \ul{\bA}^r$, and $y \in \ul{\bA}^s$. Here $\lambda \odot \mu$ denotes the natural element of $(M^m \odot N^n)^*$ and $x \odot y \in \ul{\bA}^{r+s}=\ul{\bA}^r \odot \ul{\bA}^s$, and $\epsilon^{m,n}_{r,s}$ denotes a sign to be determined. It is clear that, on each term, $\phi$ is a map of graded $\cI$-modules, and that it is an isomorphism if at least one of $M$ or $N$ is degreewise finite. To complete the proof, it suffices to show that $\phi$ is compatible with differentials.

Let $\lambda$, $\mu$, $x$, and $y$ be as above. In what follows, we omit the symbols $\otimes$ and $\odot$. Regarding $\lambda x \mu y \in \cD(M) \odot \cD(N)$, we have
\begin{align*}
d(\lambda x \mu y)
=& d(\lambda x) \mu y + (-1)^{m+r} \lambda x d(\mu y) \\
=& (-1)^r d^*(\lambda)x\mu y+\sum_{i=1}^{r-1} (-1)^{i+1} \alpha_i^*(\lambda) T_i(x) \mu y \\
&+ (-1)^{m+r+s} \lambda x d^*(\mu) y + \sum_{j=1}^{s-1} (-1)^{m+r+j+1} \lambda x \alpha_j^*(\mu) T_j(y)
\end{align*}
and so
\begin{align*}
\phi(d(\lambda x \mu y))
=& (-1)^r \epsilon^{m-1,n}_{r,s} d^*(\lambda)\mu xy+\sum_{i=1}^{r-1} (-1)^{i+1} \epsilon^{m,n}_{r-1,s} \alpha_i^*(\lambda) \mu T_i(x) y \\
&+ (-1)^{m+r+s} \epsilon^{m,n-1}_{r,s} \lambda d^*(\mu) xy + \sum_{j=1}^{s-1} (-1)^{m+r+j+1} \epsilon^{m,n}_{r,s-1} \lambda \alpha_j^*(\mu) x T_j(y).
\end{align*}
On the other hand, regarding $\lambda \mu x y \in \cD(M \odot N)$, we have
\begin{displaymath}
d(\lambda \mu x y) = (-1)^{r+s} d^*(\lambda \mu) xy + \sum_{k=1}^{r+s-1} (-1)^{k+1} \alpha_k^*(\lambda \mu) T_k(xy).
\end{displaymath}
We now come to the key calculations:
\begin{displaymath}
\alpha_k^*(\lambda \mu) = \begin{cases}
\alpha_k^*(\lambda) \mu & \text{if $1 \le k \le r-1$} \\
\lambda \alpha_{k-r}^*(\mu) & \text{if $r+1 \le k \le r+s-1$} \\
0 & \text{if $k=r$} \end{cases}
\end{displaymath}
and
\begin{displaymath}
T_k(xy) = \begin{cases}
T_k(x) y & \text{if $1 \le k \le r-1$} \\
x T_{k-r}(y) & \text{if $r+1 \le k \le r+s-1$} \\
? & \text{if $k=r$.} \end{cases}
\end{displaymath}
The question mark indicates that there is not a nice formula. The proofs are easy and left to the reader. We thus find
\begin{align*}
\epsilon^{m,n}_{r,s} d(\phi(\lambda x \mu y))
=& (-1)^{r+s} d^*(\lambda) \mu xy + (-1)^{r+s+m} \lambda d^*(\mu) xy \\
&+ \sum_{i=1}^{r-1} (-1)^{i+1} \alpha_i^*(\lambda) \mu T_i(x) y + \sum_{j=1}^{s-1} (-1)^{r+j+1} \lambda \alpha_j^*(\mu) x T_j(y).
\end{align*}
We now see that by choosing $\epsilon^{m,n}_{r,s} = (-1)^{ms}$ we obtain $\phi d = d \phi$.
\end{proof}

For a constraint word $\lambda$, define the {\bf conjugate word} $c(\lambda)$ to be the constraint word with $b$ changed to $a$ and vice versa. Combining the proposition with Proposition~\ref{prop:Dprin}, we obtain:

\begin{corollary} \label{cor:dualstd}
We have $\cD(\ul{\bE}^{\lambda}) \cong \ul{\bE}^{c(\lambda)}[\ell]$, where $\ell$ is the length of $\lambda$.
\end{corollary}

\subsection{The finiteness theorem}

Theorem~\ref{thm:duality} is essentially a formality: its proof does not rely on any difficult results about $\cI$-modules. By contrast, the following theorem relies on one of our main theorems:

\begin{theorem} \label{thm:Dfin}
The duality functor $\cD$ carries $\rD^b_{\fgen}(\uRep(\cI))$ into itself.
\end{theorem}

\begin{proof}
By Theorem~\ref{thm:dergen}, $\rD^b_{\fgen}(\uRep(\cI))$ is generated (as a triangulated category) by the standard objects $\ul{\bE}^{\lambda}$. It therefore suffices to show that $\cD(\ul{\bE}^{\lambda})$ belongs to $\rD^b_{\fgen}(\uRep(\cI))$ for all $\lambda$. This follows from Corollary~\ref{cor:dualstd}.
\end{proof}

\begin{corollary} \label{cor:Dfg}
The duality functor induces an equivalence
\begin{displaymath}
\cD \colon \rD^b_{\fgen}(\uRep(\cI))^{\op} \to \rD^b_{\fgen}(\uRep(\cI))
\end{displaymath}
whose square is canonically isomorphic to the identity functor.
\end{corollary}

\begin{remark}
By Corollary~\ref{cor:Dfg}, $\cD$ induces a self-map of $\ul{\rK}(\cI)$, which we still denote by $\cD$. From Corollary~\ref{cor:dualstd}, we see that $\cD [\ul{\bE}^{\lambda}] = (-1)^{\ell(\lambda)} [\ul{\bE}^{c(\lambda)}]$. Identifying $\ul{\rK}(\cI)$ with $\bZ\{a,b\}$, we see that $\cD$ is the ring automorphism taking $a$ to $-b$ and $b$ to $-a$.
\end{remark}

\subsection{Betti tables} \label{ss:betti}

Let $M$ be a graded $\cI$-module. By analogy with commutative algebra, we define the {\bf Betti table} of $M$, denoted $\beta(M)$, to be the two-dimensional array given by
\begin{displaymath}
\beta(M)_{i,j} = \dim \rL_j \cT(M)_{i+j}.
\end{displaymath}
Thus the entry in row $i$ and column $j$ records the number of generators in degree $i+j$ of the $j$th term in the minimal resolution of $M$. We define the {\bf (Castelnuovo--Mumford) regularity} of $M$, denoted $\rho(M)$, to be the minimum $n$ such that $\beta_{i,j}(M)=0$ for all $i>n$ and all $j \in \bN$, or $\infty$ if no such $n$ exists; thus $\rho(M)$ is the index of the final non-zero row in $\beta(M)$. The following theorem is our main result on these invariants:

\begin{theorem} \label{thm:betti}
Let $M$ be a finitely generated graded $\cI$-module. Then $\rho(M)$ is finite, that is, $\beta(M)$ has only finitely many non-zero rows. Furthermore, each row is eventually polynomial, that is, for each $i$ there is a polynomial $p_i(t) \in \bQ[t]$ such that $\beta(M)_{i,j}=p_i(j)$ for $j \gg 0$.
\end{theorem}

\begin{proof}
By Corollary~\ref{cor:kosztor} and Lemma~\ref{lem:DK}(e) we have
\begin{displaymath}
\beta(M)_{i,j}= \dim \rL_j \cT(M)_{i+j} =\dim \rH^{-j}(\cK(M))_{i+j} =\dim \rH^{-i}(\cD(M))_{i+j}
\end{displaymath}
By Theorem~\ref{thm:Dfin}, $\rH^{-i}(\cD(M))$ is a finitely generated graded $\cI$-module for all $i$, and non-zero for only finitely many $i$. We thus see that $\beta(M)$ has only finitely many non-zero rows. Furthermore, for fixed $i$, we see that $\beta(M)_{i,j}$ is the dimension of the $i+j$ graded piece of the finitely generated graded $\cI$-module $\rH^{-i}(\cD(M))$, and thus is eventually a polynomial in $j$ by Theorem~\ref{thm:hilb-gr}.
\end{proof}

\subsection{Conceptual explanation of $\cK$ and $\cD$} \label{ss:conceptual}

We now give a conceptual explanation for some of the constructions appearing in this section. Let $A$ denote the shuffle algebra $\Sym_{\medshuffle}(\bk\langle 1 \rangle)$, as defined in \S \ref{ss:shuffle}. Recall that $\Mod_A$ is essentially equivalent to the category of graded $\cI$-modules; for the present discussion, we ignore the minor difference between the two categories. If $M$ is an $A$-module then the usual construction yields a Koszul complex $\lw^{\bullet}_{\medshuffle}(\bk\langle 1 \rangle) \otimes M$. Translating this back across the equivalence of categories yields our Koszul complex $\cK(M)$. Furthermore, the usual form of Koszul duality yields a derived equivalence between $A$ and $B$, where $B=\lw_{\medshuffle}(\bk\langle 1 \rangle)$ is the exterior algebra on $\bk\langle 1 \rangle$. It so happens that $A$ and $B$, while not isomorphic as algebras, are Morita equivalent, that is, their module categories are equivalent; one constructs an equivalence simply by playing with signs. This yields a second equivalence between the derived categories of $A$ and $B$. Combining the two produces a derived auto-equivalence of $A$. Translating this back to $\uRep(\cI)$ yields our duality functor $\cD$.

Thus, from the perspective of shuffle algebras, the rather opaque constructions we made above become somewhat more clear. We chose to stick with the language of $\cI$-modules, however, to maintain consistency with the rest of the paper.

\section{Grothendieck groups revisited} \label{s:groth2}

\subsection{The pairing on $\rK(\cI)$} \label{ss:pairing}

Let $\cA$ be a $\bk$-linear abelian category with Grothendieck group $\rK(\cA)$. Suppose that the following condition holds:
\begin{itemize}
\item[$(\ast)$] For all objects $M,N \in \cA$ the $\bk$-vector space $\Ext^i_{\cA}(M,N)$ is finite dimensional and vanishes for $i \gg 0$.
\end{itemize}
Then $\rK(\cA)$ admits a canonical pairing, as follows:
\begin{displaymath}
\langle [M], [N] \rangle = \sum_{i \ge 0} (-1)^i \dim \Ext^i_{\cA}(M, N).
\end{displaymath}
We have seen (Propositions~\ref{prop:uextcolim} and~\ref{prop:extcolim} and Corollary~\ref{cor:extfin}) that both $\Rep(\cI)^{\fgen}$ and $\uRep(\cI)^{\fgen}$ satisfy condition $(\ast)$, and so $\rK(\cI)$ and $\ul{\rK}(\cI)$ admit pairings.  We let $\langle , \rangle$ denote the pairing on $\ul{\rK}(\cI)$ and $(,)$ the pairing on $\rK(\cI)$.

\subsection{Higher multiplicities} \label{ss:highermult}

Let $\lambda$ be a constraint word of rank $r$. Then $T_{\ge r}(\ul{\bE}^{\lambda})$ is a simple object of the category $\uRep(\cI)_{\ge r}$ (Proposition~\ref{prop:levcat2}), and so there is an associated multiplicity function $\umu_{\lambda}$ (see \S \ref{ss:catmult}). We define $\umu_{\lambda}$ on $\uRep(\cI)$ by simply composing with $T_{\ge r}$, that is, for $M \in \uRep(\cI)$ we put $\umu_{\lambda}(M)=\umu_{\lambda}(T_{\ge r}(M))$. It follows directly from the definition that $\umu_{\lambda}(M)=0$ if $\ulev(M)<r$. We similarly define $\mu_{\lambda}$ in the smooth case.

\begin{proposition} \label{prop:highermult}
Let $\lambda$ be a constraint word and let $M$ be a graded $\cI$-module. Then
\begin{displaymath}
\umu_{\lambda}(M)=\dim \uHom_{\cI}(M, \ul{\bI}^{\lambda}) = \langle [M], [\ul{\bI}^{\lambda}] \rangle
\end{displaymath}
In particular, if $M$ is finitely generated then $\umu_{\lambda}(M)$ is finite. The analogous statements holds in the smooth case.
\end{proposition}

\begin{proof}
It follows from the theory in \S \ref{ss:satgen} that $T_{\ge r}(\ul{\bI}^{\lambda})$ is the injective envelope of $T_{\ge r}(\ul{\bE}^{\lambda})$, where $r$ is the rank of $\lambda$. We thus have:
\begin{align*}
\umu_{\lambda}(M)
&= \dim \Hom(T_{\ge r}(M), T_{\ge r}(\ul{\bI}^{\lambda})) \\
&= \dim \uHom_{\cI}(M, \cS_{\ge r}(\ul{\bI}^{\lambda})) \\
&= \dim \uHom_{\cI}(M, \ul{\bI}^{\lambda}) \\
&= \langle [M], [\ul{\bI}^{\lambda}] \rangle.
\end{align*}
The first step follows from Proposition~\ref{prop:multinj}; the second uses the adjunction between $T_{\ge r}$ and $S_{\ge r}$: the third uses the fact that $\ul{\bI}^{\lambda}$ is saturated (which follows from the theory in \S \ref{ss:satgen}); and the fourth step simply uses the fact that $\ul{\bI}^{\lambda}$ is injective, so that all higher $\Ext$'s into it vanish.
\end{proof}

\begin{corollary}
The multiplicity function $\umu_{\lambda}$ induces a homomorphism $\umu_{\lambda} \colon \ul{\rK}(\cI) \to \bZ$.
\end{corollary}

\begin{example}
If $\lambda=\blacktri^n$ then $\umu_{\lambda}$ is the multiplicity function $\umu_n$ already studied.
\end{example}

\subsection{The non-commutative Hilbert series} \label{ss:noncommhilb}

Let $M$ be a finitely generated graded $\cI$-module. We define the {\bf non-commutative Hilbert series} of $M$ by
\begin{displaymath}
\ul{\rG}_M = \sum_{\lambda} \umu_{\lambda}(M) \lambda,
\end{displaymath}
where the sum is over all constraint words $\lambda$. This is a formal $\bZ$-linear combination of words, and thus an element of the ring $\bZ\ldb a,b \rdb$ of non-commutative power series. We note that $\ul{\rG}_M$ becomes the usual Hilbert series $\ul{\rH}_M(t)$ under the substitutions $a \to 0$ and $b \to t$. We also note that if $M$ has homogeneous level $\le r$ then $\umu_{\lambda}(M)=0$ if $\lambda$ has rank $>r$, and so every monomial appearing in $\ul{\rG}_M(s,t)$ with non-zero coefficient has $a$ appearing at most $r$ times. The invariant $\ul{\rG}$ obviously factors through the Grothendieck group; thus for $x \in \ul{\rK}(\cI)$, or $x \in \bZ\{a,b\}$, we have an associated quantity $\ul{\rG}_x$.

For a finitely generated smooth $\cI$-module $M$, we define
\begin{displaymath}
\rG_M = \sum_{\lambda} \mu_{\lambda}(M) \lambda.
\end{displaymath}
Similar comments to the above apply to this construction.

\begin{proposition} \label{prop:Gsmooth}
Let $M$ be a finitely generated smooth $\cI$-module. Then $\rG_M=\ul{\rG}_{\ul{\Xi}(M)}$.
\end{proposition}

\begin{proof}
Using the adjunction $(\ul{\Xi}, \Phi)$ (Proposition~\ref{prop:xi}), we have
\begin{displaymath}
\umu_{\lambda}(\ul{\Xi}(M)) = \dim \uHom_{\cI}(\ul{\Xi}(M), \ul{\bI}^{\lambda}) = \dim \Hom_{\cI}(M, \bI^{\lambda}) = \mu_{\lambda}(M),
\end{displaymath}
and so the result follows.
\end{proof}

Every element $f$ of $\bZ\ldb a,b \rdb$ can be written uniquely in the form $f=x+ay+bz$ with $x \in \bZ$ and $y,z \in \bZ\ldb a,b \rdb$. Let
\begin{displaymath}
\ul{\rG}_M = \umu_0(M) + a \ul{\rG}^a_M + b \ul{\rG}^b_M
\end{displaymath}
be this decomposition of $\ul{\rG}_M$. Recall (\S \ref{ss:groth}) that we have constructed several operators on the Grothendieck groups, including $(-)^{\dag}$, $\ul{\xi}$, $\psi$, and $\ul{\sigma}$.

\begin{proposition} \label{prop:Gpart}
Let $\kappa$ be the endomorphism of $\ul{\rK}(\cI)$ given by $\kappa(x)=(\ul{\xi}(\psi(x^{\dag})))^{\dag}$. Then for any $x \in \ul{\rK}(\cI)$, we have
\begin{displaymath}
\ul{\rG}_x^{\whitetri} = \ul{\rG}_{\kappa(x)}
\end{displaymath}
and
\begin{displaymath}
\ul{\rG}_x^{\blacktri} = \ul{\rG}_{\ul{\sigma}(x)} - \umu_0(x) - a \ul{\rG}^a_x.
\end{displaymath}
\end{proposition}

\begin{proof}
By Proposition~\ref{prop:gammagraded}, we have
\begin{displaymath}
\ul{\bI}^{a \lambda}=\ul{\bA}^1 \odot \ul{\bI}^{\lambda}=\Gamma(\Phi((\ul{\bI}^{\lambda})^{\dag}))^{\dag}.
\end{displaymath}
Thus, for a graded $\cI$-module $M$ we have
\begin{displaymath}
\uHom_{\cI}(M, \ul{\bI}^{a\lambda})=\uHom_{\cI}(\ul{\Xi}(\Psi(M^{\dag}))^{\dag}, \ul{\bI}^{\lambda}),
\end{displaymath}
where we have used the adjunctions $(\Psi, \Gamma)$ and $(\ul{\Xi}, \Phi)$ (Propositions~\ref{prop:phigamma} and~\ref{prop:xi}). It follows that $\umu_{a\lambda}(x)=\umu_{\lambda}(\kappa(x))$ for $x \in \ul{\rK}(\cI)$, and so
\begin{displaymath}
\ul{\rG}^a_x = \sum_{\lambda} \umu_{a \lambda}(x) \lambda = \sum_{\lambda} \umu_{\lambda}(\kappa(x)) \lambda = \ul{\rG}_{\kappa(x)},
\end{displaymath}
which proves the first identity.

We now prove the second. Let $\lambda$ be a constraint word. It follows from the proof of Lemma~\ref{lem:inj} that
\begin{displaymath}
\ul{\sC}(\ul{\bI}^{\lambda}) = \begin{cases}
\ul{\bI}^{b\lambda} & \text{if $\lambda$ has first letter $b$} \\
\ul{\bI}^{b\lambda} \oplus \ul{\bI}^{\lambda} & \text{otherwise.} \end{cases}
\end{displaymath}
Thus, using the adjunction $(\ul{\Sigma}, \ul{\sC})$ (Proposition~\ref{prop:coind}), we see that for a graded $\cI$-module $M$ we have
\begin{displaymath}
\uHom_{\cI}(\ul{\Sigma}(M), \ul{\bI}^{\lambda}) = \begin{cases}
\uHom_{\cI}(M, \ul{\bI}^{b\lambda}) & \text{if $\lambda$ has first letter $b$} \\
\uHom_{\cI}(M, \ul{\bI}^{b\lambda} \oplus \ul{\bI}^{\lambda}) & \text{otherwise.} \end{cases}
\end{displaymath}
Thus for $x \in \ul{\rK}(\cI)$, we have
\begin{displaymath}
\umu_{b\lambda}(x) = \umu_{\lambda}(\ul{\sigma}(x)) - \begin{cases}
0 & \text{if $\lambda$ has first letter $b$} \\
\umu_{\lambda}(x) & \text{otherwise.} \end{cases}
\end{displaymath}
We thus have
\begin{displaymath}
\ul{\rG}^b_x = \sum_{\lambda} \umu_{b\lambda}(x) \lambda
=\sum_{\lambda} \umu_{\lambda}(\ul{\sigma}(x)) - \umu_0(x) - \sum_{\lambda=a \nu} \umu_{\lambda}(x) \lambda
\end{displaymath}
As the first term is $\ul{\rG}_{\ul{\sigma}(x)}$ and the third is $a\ul{\rG}^a_x$, the result follows.
\end{proof}

\begin{proposition} \label{prop:Grecur}
Let $x \in \bZ\{a,b\}$. Then
\begin{displaymath}
\ul{\rG}_{bx}=b \ul{\rG}_x.
\end{displaymath}
Furthermore, assuming $x=1$ or $x \in b \bZ\{a,b\}$ and $n \ge 1$, we have
\begin{displaymath}
\ul{\rG}_{a^n x} = a \ul{\rG}_{(1+\cdots+a^{n-1})x}+b (1-b)^{-1} \ul{\rG}_{a^{n-1}x}.
\end{displaymath}
\end{proposition}

\begin{proof}
We have $\kappa(bx)=(\ul{\xi}(\psi(x^{\dag} b)))^{\dag}=0$ by Proposition~\ref{prop:psistd}, and so $\ul{\rG}^a_{bx}=0$ by Proposition~\ref{prop:Gpart}. As $\ul{\sigma}(bx)=x$ (Proposition~\ref{prop:shiftstd}), we find $\ul{\rG}^b_{bx}=\ul{\rG}_x$ by Proposition~\ref{prop:Gpart}. Finally, we have $\umu_0(bx)=0$ (obvious), and so we conclude $\ul{\rG}_{bx}=b \ul{\rG}^b_{bx}=b \ul{\rG}_x$.

Now suppose $x=1$ or $x \in b\bZ\{a,b\}$ and $n \ge 1$. Then
\begin{displaymath}
\kappa(a^n x)=(\ul{\xi}(\psi(x^{\dag} a^n))^{\dag}=(\ul{\xi}(x^{\dag} a^{n-1}))^{\dag}=(1+\cdots+a^{n-1}) x
\end{displaymath}
(Proposition~\ref{prop:xistd}). Thus $\ul{\rG}^a_{a^nx}=\ul{\rG}_{(1+\cdots+a^{n-1}) x}$ by Proposition~\ref{prop:Gpart}. We have $\ul{\sigma}(a^n x)=a^n x+a^{n-1} x$ (Proposition~\ref{prop:shiftstd}), and so
\begin{displaymath}
\ul{\rG}^b_{a^n x}=\ul{\rG}_{a^n x}+\ul{\rG}_{a^{n-1} x}-\umu_0(a^nx)-a \ul{\rG}_{a^n x}^a=b\ul{\rG}^b_{a^n x}+\ul{\rG}_{a^{n-1} x}.
\end{displaymath}
Thus $\ul{\rG}^b_{a^n x}=(1-b)^{-1} \ul{\rG}_{a^{n-1} x}$. As $\umu_0(a^nx)=0$, we have
\begin{displaymath}
\ul{\rG}_{a^nx}=a\ul{\rG}^a_{a^nx}+b\ul{\rG}^b_{a^nx}=a \ul{\rG}_{(1+\cdots+a^{n-1})x}+b(1-b)^{-1} \ul{\rG}_{a^{n-1}x},
\end{displaymath}
which proves the proposition.
\end{proof}

Proposition~\ref{prop:Grecur} gives recursive formulas that allow one to calculation $\ul{\rG}_x$ for any $x \in \bZ\{a,b\}$ in finitely many steps. These formulas also yield the following theorem:

\begin{theorem} \label{thm:noncommhilb}
Let $M$ be a finitely generated graded $\cI$-module. Then $\ul{\rG}_M$ is a non-commutative rational function. More precisely, it is a finite sum of terms of the form $f_0(b) a f_1(b) a \cdots a f_k(b)$ where $f_0, \ldots, f_k$ are rational functions of $b$ with denominators a power of $1-b$. The analogous statements hold in the smooth case.
\end{theorem}

\begin{proof}
The statement in the graded case follows immediately from Proposition~\ref{prop:Grecur}, while the smooth case follows from the graded case and Proposition~\ref{prop:Gsmooth}.
\end{proof}

\begin{example}
Using Proposition~\ref{prop:Grecur}, we find
\begin{displaymath}
\ul{\rG}_{aba}=aba+ab^2(1-b)^{-1}+b^2(1-b)^{-1}a+b^3(1-b)^{-2}.
\end{displaymath}
The coefficient of $ab^n$ is~0 if $n=0,1$ and~1 if $n \ge 2$. We thus see that $\ul{\bE}^a$ and $\ul{\bE}^{ab}$ have multiplicity~0 in $\ul{\bE}^{aba}$, while $\ul{\bE}^{ab^n}$ has multiplicity~1 in $\ul{\bE}^{aba}$ for $n \ge 2$. (Technically, these multiplicities should be taken in the quotient category $\uRep(\cI)_{\ge 1}$.)
\end{example}

\subsection{Another non-commutative series} \label{ss:Fseries}

For $x \in \bZ\{a,b\}$, define
\begin{displaymath}
\ul{\rF}_x = \sum_{\lambda} \langle x, \lambda \rangle \lambda, \qquad
\rF_x = \sum_{\lambda} (x, \lambda) \lambda.
\end{displaymath}
Arguing as in the previous section, one can show
\begin{displaymath}
\ul{\rF}_{a^nx}=a\ul{\rF}_{(1+\cdots+a^{n-1})x}+b\ul{\rF}_{a^{n-1}x}, \qquad
\ul{\rF}_{bx} = b(1+b)^{-1} \ul{\rF}_x, \qquad
\rF_x=\ul{\rF}_{\ul{\xi}(x)},
\end{displaymath}
where in the first formula we assume $x=1$ or $x \in b\bZ\{a,b\}$ and $n \ge 1$. These formulas allow one to compute $\ul{\rF}_x$ and $F_x$ in finitely many steps, and thus the pairings $\langle x, \lambda \rangle$ and $(x, \lambda)$ as well.

\begin{example}
We have
\begin{displaymath}
\ul{\rF}_{aba}=(a+b) b(1+b)^{-1} (a+b).
\end{displaymath}
The coefficient of $ab^n$ here is~0 if $n=0,1$ and $(-1)^n$ if $n \ge 2$. It follows that
\begin{displaymath}
\sum_{i \ge 0} (-1)^i \dim \uExt^i_{\cI}(\ul{\bE}^{aba}, \ul{\bE}^{ab^n}) = \begin{cases}
0 & \text{if $n=0,1$} \\
(-1)^n & \text{if $n \ge 2$} \end{cases}
\qedhere
\end{displaymath}
\end{example}

\begin{remark}
One can also consider the following variant series:
\begin{displaymath}
\ul{\rF}'_x = \sum_{\lambda} \langle \lambda, x \rangle \lambda, \qquad
\rF'_x = \sum_{\lambda} (\lambda, x ) \lambda
\end{displaymath}
The duality functor $\cD$ induces an involution $\delta$ on $\ul{\rK}(\cI)$ that satisfies $\langle \delta(x), \delta(y) \rangle = \langle y, x \rangle$. Thus $\ul{\rF}'_x$ is be obtained from $\ul{\rF}_{\delta(x)}$ under the substitutions $a \to -b$ and $b \to -a$. The adjunction $(\Phi, F)$ discussed in \S \ref{ss:phiright} gives $\rF'_x=\ul{\rF}'_y$ where $y=\ul{\sigma}(\ul{\gamma}(x)^{\dag})^{\dag}$. Thus one can effectively compute $\ul{\rF}'_x$ and $\rF'_x$.
\end{remark}

\subsection{Effectivity} \label{ss:effective}

Let $\cA$ be an abelian category. Every element of the Grothendieck group $\rK(\cA)$ can be written in the form $[M]-[N]$ where $M$ and $N$ are objects of $\cA$. We say that an element $\rK(\cA)$ is {\bf effective} if it has the form $[M]$ for some object $M$ of $\cA$. The collection of effective elements forms a submonoid of $\rK(\cA)$. Characterizing this submonoid can be a difficult problem in general. We now solve this problem for $\rK(\cI)$ and $\ul{\rK}(\cI)$.

\begin{theorem} \label{thm:effective}
An element $x$ of $\ul{\rK}(\cI)$ is effective if and only if $\umu_{\lambda}(x) \ge 0$ for all constraint words $\lambda$. The analogous result holds in the smooth case.
\end{theorem}

\begin{proof}
In this proof, we say that $x \in \ul{\rK}(\cI)$ is {\bf $\umu$-positive} if $\umu_{\lambda}(x) \ge 0$ for all constraint words $\lambda$. For a finitely generated graded $\cI$-module $M$, the quantity $\umu_{\lambda}(M)$ is non-negative, and so any effective element of $\ul{\rK}(\cI)$ is $\umu$-positive. We now prove the converse. For $r \in \bZ$, consider the following statement:
\begin{itemize}
\item[$S(r)$:] Let $x \in \ul{\rK}(\cI)$ be effective and let $y \in \ul{\rK}(\cI)_{\le r}$ be such that $x-y$ is $\umu$-positive. Then $x-y$ is effective.
\end{itemize}
Clearly, it suffices to prove $S(r)$ for all $r$. We proceed by induction on $r$; note that $S(r)$ is trivially true for $r<0$.

Let $r \ge 0$ be given and suppose $S(r-1)$ holds. To prove $S(r)$, it suffices to treat the case where $y=[\ul{\bE}^{\lambda}]$ for a constraint word $\lambda$ of rank $r$, as $\ul{\rK}(\cI)_{\le r}$ is spanned by these classes and $\ul{\rK}(\cI)_{\le r-1}$. Thus suppose that $M$ is a finitely generated graded $\cI$-module such that $[M]-[\ul{\bE}^{\lambda}]$ is $\umu$-positive. We will show that it is effective.

We have
\begin{displaymath}
\umu_{\lambda}([M]-[\ul{\bE}^{\lambda}])=\umu_{\lambda}(M)-1 \ge 0,
\end{displaymath}
and so $\umu_{\lambda}(M) \ge 1$. We thus see (Proposition~\ref{prop:highermult}) that there is a non-zero map $f \colon M \to \ul{\bI}^{\lambda}$. Since $T_{\ge r}(\ul{\bI}^{\lambda})$ is an essential extension of $T_{\ge r}(\ul{\bE}^{\lambda})$, it follows that $K=f(M) \cap \ul{\bE}^{\lambda}$ is non-zero. Now, we have
\begin{displaymath}
[M]=[\ker(f)]+[\im(f)/K]+[K],
\end{displaymath}
and
\begin{displaymath}
[\ul{\bE}^{\lambda}]=[K]+[\ul{\bE}^{\lambda}/K].
\end{displaymath}
We thus see that
\begin{displaymath}
[M]-[\ul{\bE}^{\lambda}] = [\ker(f) \oplus \im(f)/K] - [\ul{\bE}^{\lambda}/K].
\end{displaymath}
This element is $\umu$-positive by assumption. Since the first term on the right is effective and the second belongs to $\ul{\rK}(\cI)_{\le r-1}$ (by Theorem~\ref{thm:level}), this element is thus effective by $S(r-1)$. This completes the proof in the graded case, and the same argument applies in the smooth case.
\end{proof}

\begin{corollary}
The effective submonoid of $\ul{\rK}(\cI)$ is saturated: that is, if $x \in \ul{\rK}(\cI)$ and $nx$ is effective for some positive integer $n$ then $x$ is effective. Similarly in the smooth case.
\end{corollary}

\begin{proof}
Suppose $nx$ is effective. Then $\umu_{\lambda}(nx)=n \umu_{\lambda}(x) \ge 0$, and so $\umu_{\lambda}(x) \ge 0$. Thus $x$ is effective.
\end{proof}

\begin{remark}
Theorem~\ref{thm:effective} on its own is not quite an effective test for effectivity, as it requires verifying infinitely many conditions. However, combining Theorem~\ref{thm:effective} with Proposition~\ref{prop:Grecur} gives an effective test, as follows. Let $x \in \ul{\rK}(\cI) \cong \bZ\{a,b\}$, and suppose we are given $x$ in the form $\sum_{i=1}^n n_i \lambda_i$ where the $n_i$ are integers and the $\lambda_i$ are constraint words. Theorem~\ref{thm:effective} states that $x$ is effective if and only if every coefficient of $\ul{\rG}_x$ is non-negative. Using Proposition~\ref{prop:Grecur}, we can compute $\ul{\rG}_x$ in finitely many steps. In fact, it yields an expression for $\ul{\rG}_x$ as a finite sum of terms of the form $f_0(b)a\cdots a f_k(b)$, where each $f_k(b)$ is given as a polynomial of $b$ divided by some power of $1-b$. We can then test, in finitely many steps, if all coefficients in the series expansion of this expression are non-negative.
\end{remark}

\begin{remark}
The isomorphism $\phi \colon \ul{\rK}(\cI) \to \rK(\cI)$ carries effective elements to effective elements, however, it is not a bijection between the sets of effective elements. Indeed, $[\bA^1]-[\bA^0]$ is an effective element of $\rK(\cI)$ as it is the class of the kernel of the augmentation map $\epsilon \colon \bA^1 \to \bA^0$. However, $[\ul{\bA}^1]-[\ul{\bA}^0]$ is not an effective element of $\ul{\rK}(\cI)$: indeed, it has $\umu_0=-1$.
\end{remark}

\appendix
\section{Categorical background} \label{s:catbg}

\subsection{Multiplicities} \label{ss:catmult}

Let $\cA$ be an abelian category and let $L$ be a simple object of $\cA$. Let $\sL_n$ be the class of all objects $M$ of $\cA$ for which there exists a chain
\begin{displaymath}
F_1 \subset G_1 \subset \cdots \subset F_n \subset G_n \subset M
\end{displaymath}
with $G_i/F_i \cong L$ for all $i$. Clearly, we have $\sL_{n+1} \subset \sL_n$. We define the {\bf multiplicity} of $L$ in $M$, denoted $\mu_L(M)$, to be the maximum $n$ for which $M \in \sL_n$, or $\infty$ if $M$ belongs to $\sL_n$ for all $n$.

\begin{proposition}
We have the following:
\begin{enumerate}
\item $\mu_L$ is additive in short exact sequences, that is, given a short exact sequence
\begin{displaymath}
0 \to M_1 \to M_2 \to M_3 \to 0
\end{displaymath}
we have $\mu_L(M_2)=\mu_L(M_1)+\mu_L(M_3)$.
\item If $N$ is a subquotient of $M$ then $\mu_L(N) \le \mu_L(M)$.
\item Suppose $\cA$ satisfies (AB5). If $M=\bigcup_{i \in I} M_i$ (directed union) then we have $\mu_L(M)=\sup_{i \in I} \mu_L(M_i)$.
\end{enumerate}
\end{proposition}

\begin{proof}
(a) It is clear that $\mu_L(M_3) \ge \mu_L(M_1)+\mu_L(M_3)$: indeed, if $M_1 \in \sL_a$ and $M_3 \in \sL_b$ then we can splice together the given chains in $M_1$ and $M_3$ to get one in $M_2$, which shows that $M_2 \in \sL_{a+b}$. We now prove the reverse inequality. First suppose that $F \subset G \subset M$ satisfy $G/F \cong L$, and let $F'=F \cap M_1$ and $F''=(F+M_1)/M_1$, and similarly define $G'$ and $G''$. Then either $F'/G' \cong L$ or $F''/G'' \cong L$. Now suppose $M \in \sL_n$, and let $F_1 \subset G_1 \subset \cdots \subset F_n \subset G_n \subset M$ be the given chain. Define $F_i'$, $F_i''$, $G'_i$, and $G''_i$ as before. Thus, for each $i$, we have $F'_i/G'_i \cong L$ or $F''_i/G''_i \cong L$. This shows that $M_1 \in \sL_a$ and $M_3 \in \sL_b$ for some $a$ and $b$ summing to $n$, which completes the proof.

(b) This follows immediately from (a).

(c) From (b), we have $\mu_L(M) \ge \mu_L(M_i)$ for each $i$, and so $\mu_L(M) \ge \sup_{i \in I} \mu_L(M_i)$. We now prove the reverse inequality. Thus suppose that $\mu_L(M) \ge n$, and let us show that $\mu_L(M_i) \ge n$ for some $i$.

First suppose $F \subset G \subset M$ satisfy $G/F \cong L$. Let $F^{(i)} = F \cap M_i$ and $G^{(i)} = G \cap M_i$. Since (AB5) holds, we have $\bigcup_{i \in I} F^{(i)} = F$, and similarly for $G$. For each $i$, we have an injection $F^{(i)}/G^{(i)} \to F/G \cong L$, and so $F^{(i)}/G^{(i)}$ is either~0 or isomorphic to $L$, since $L$ is simple. Since (AB5) holds, the direct limit of the $F^{(i)}/G^{(i)}$ is $F/G \cong L$, and so $F^{(i)}/G^{(i)}$ cannot vanish for all $i$, that is, there exists some $i$ for which it is $L$. Note that for $i \le j$ we have an injection $F^{(i)}/G^{(i)} \to F^{(j)}/G^{(j)}$, so if the former is non-zero so is the latter.

Now, returning to the main point, let $F_1 \subset G_1 \subset \cdots \subset F_n \subset G_n \subset M$ be the given chain in $M$. Let $i_1, \ldots, i_n \in I$ be such that $F^{(i_k)}_k/G^{(i_k)}_k \cong L$ for each $k$, and pick $i \in I$ such that $i_1, \ldots, i_n \le i$. Then $F^{(i)}_k/G^{(i)}_k \cong L$ for each $k$. Thus the chian $F^{(i)}_1 \subset G^{(i)}_1 \subset \cdots \subset M_i$ shows that $M_i \in \sL_n$, which completes the proof.
\end{proof}

\begin{proposition}
Suppose $\cA$ is a locally noetherian Grothendieck abelian category, and that $\mu_L(M)=0$ for all simple objects $L$. Then $M=0$.
\end{proposition}

\begin{proof}
Let $M$ be a non-zero object of $\cA$. Since $\cA$ is locally noetherian, $M$ has a non-zero noetherian submodule $M'$. Since $M'$ is noetherian, it has a maximal proper submodule $M''$. Thus $L=M'/M''$ is a simple subquotient of $M$, and so $\mu_L(M)>0$.
\end{proof}

%

\begin{proposition} \label{prop:multinj}
Let $\cA$ be a $\bk$-linear abelian category, with $\bk$ a field. Suppose that $L$ is a simple object of $\cA$ with $\End(L)=\bk$ and $I$ is its injective envelope. Then for any $M \in \cA$ we have $\mu_L(M)=\dim \Hom(M, I)$.
\end{proposition}

\begin{proof}
Put $\mu=\mu_L$ and let $\nu$ be defined by $\nu(M)=\dim \Hom(M, I)$. We must show $\mu=\nu$. To begin with, we show that they are equal when $\mu$ is finite, by induction on $\mu$.

First, suppose that $\mu(M)=0$ and let us show that $\nu(M)=0$. Suppose $f \colon M \to I$ is a non-zero map. Then $f(M)$ is a non-zero subobject of $I$, and thus meets $L$, since $I$ is an essential extension of $L$, and thus contains $L$, since $L$ is simple. This shows that $L$ is a subquotient of $M$, which contradicts $\mu(M)=0$. Thus there are no non-zero maps $M \to I$ and so $\nu(M)=0$.

Next, suppose that $\mu(M)=\nu(M)$ whenever $\mu(M)<n$, and let us show that the equality continues to hold when $\mu(M)=n$. Thus let $M$ be given with $\mu(M)=n$. We have already handled the $n=0$ case, so we can assume $n>0$. Thus $L$ occurs as  a subquotient of $M$, say $L=G/F$ with $F \subset G \subset M$. We have $\mu(M)=\mu(F)+\mu(G/F)+\mu(M/G)$. Since $\mu(G/F)=1$, it follows that $\mu(F)$ and $\mu(M/G)$ are both strictly less than $n$. Thus, by the inductive hypothesis, we have $\mu(F)=\nu(F)$ and $\mu(M/G)=\nu(M/G)$. Furthermore, the hypothesis $\End(L)=\bk$ ensures that $\Hom(L, I)$ is one-dimensional, and so $\nu(G/F)=\nu(L)=1$. As $\nu(M)=\nu(F)+\nu(G/F)+\nu(M/G)$, we thus find $\nu(M)=\mu(M)$.

We have thus shown that $\mu(M)=\nu(M)$ whenever $\mu(M)$ is finite. To conclude, let us show that the equality continues to hold when $\mu(M)$ is infinite. Let $n>0$ be an integer. Then, by the definition of $\mu$, we can find $F_1 \subset G_1 \subset \cdots \subset F_n \subset G_n \subset M$ with $G_i/F_i \cong L$ for each $i$. Let $f_i \colon M \to I$ be a map extending the map $G_i \to G_i/F_i \cong L \subset I$; such an extension exists since $I$ is injective. The $f_i$ are linearly independent: indeed, we cannot express $f_i$ as a linear combination of $f_{i+1}, \ldots, f_n$ since $f_i(G_i)$ is non-zero but $f_j(G_i)=0$ for $j>i$. We thus see that $\nu(M)>n$, which completes the proof.
\end{proof}

%
%

\subsection{Derived functors and limits and colimits}

\begin{proposition} \label{prop:dercolimit}
Let $F \colon \cA \to \cB$ be a left-exact functor of abelian categories. Suppose that:
\begin{enumerate}
\item $\cA$ is a Grothendieck category.
\item $F$ commutes with filtered colimits.
\item A filtered colimit of injective objects in $\cA$ is $F$-acyclic.
\end{enumerate}
Then $\rR^i F$ commutes with filtered colimits.
\end{proposition}

\begin{proof}
Let $\{M_i\}_{i \in I}$ be a filtered system in $\cA$. Let $\{M_i \to J_i^{\bullet}\}_{i \in I}$ be a filtered system of injective resolutions; such a system exists since $\cA$ has functorial injective resolutions \stacks{079H}. We thus see that $\varinjlim M_i \to \varinjlim J_i^{\bullet}$ is an $F$-acyclic resolution. The result now follows easily.
\end{proof}

\begin{proposition} \label{prop:dercont}
Let $F \colon \cA \to \cB$ be an additive functor of abelian categories. Assume
\begin{enumerate}
\item $\cA$ is a Grothendieck category.
\item $\cB$ has exact countable products.
\item $F$ commutes with countable products.
\end{enumerate}
Then $\rR F \colon \rD(\cA) \to \rD(\cB)$ commutes with derived limits.
\end{proposition}

\begin{proof}
See \stacks{08U1}.
\end{proof}

\subsection{Localizing subcategories} \label{ss:loc}

Let $\cA$ be a Grothendieck abelian category and let $S$ be some collection of objects in $\cA$. Recall that a {\bf localizing subcategory} of $\cA$ is a Serre subcategory closed under arbitrary direct sums (or, equivalently, direct limits). Let $\cB$ be the smallest localizing subcategory of $\cA$ containing $S$; we call this the localizing subcategory {\bf generated} by $S$. The following proposition gives a concrete description of the finitely generated objects in this category under a noetherian hypothesis:

\begin{proposition} \label{prop:locgen}
Suppose that $\cA$ is locally noetherian. Then a finitely generated object $M$ of $\cA$ belongs to $\cB$ if and only if there is a finite length filtration $0=F_0 \subset \cdots \subset F^n=M$ and objects $E_1, \ldots, E_n$ of $S$ such that $F^i/F^{i-1}$ is isomorphic to a subquotient of $E_i$ for all $1 \le i \le n$.
\end{proposition}

\begin{proof}
Let $\cC_0$ be the subcategory of $\cA$ consisting of finitely generated objects $M$ admitting a filtration as in the statement of the proposition. Then $\cC_0$ is closed under subquotients. Indeed, suppose $M$ belongs to $\cC_0$ and $N$ is a subobject of $M$. Let $F^0 \subset \cdots \subset F^n$ and $E_1, \ldots, E_n$ be the given data on $M$, and let $G^{\bullet} = F^{\bullet} \cap N$. Then $G^i/G^{i-1}$ is naturally a submodule of $F^i/F^{i-1}$, and thus isomorphic to a subquotient of $E_i$. This shows that $N$ belongs to $\cC_0$. The proof for quotients is similar. It is clear that $\cC_0$ is closed under extensions. Thus $\cC_0$ is a Serre subcategory of $\cA$.

Let $\cC$ be the subcategory of $\cA$ consisting of all objects for which every finitely generated subobject belongs to $\cC_0$. We show that $\cC$ is a localizing subcategory:
\begin{itemize}
\item Suppose that $M$ belongs to $\cC$ and $N$ is a subobject of $M$. Since a finitely generated subobject of $N$ is also one of $M$, and thus belongs to $\cC_0$, we see that $N$ belongs to $\cC$.
\item Next, say that $M$ belongs to $\cC$ and $N$ is a quotient of $M$. Let $N_0$ be a finitely generated subobject of $N$, and let $\{M_i\}_{i \in I}$ be the collection of all finitely generated subobjects of $M$. Let $\ol{M}_i$ be the image of $M$ in $N$. Then the image of $M$ in $N$ is contained in $\sum_{i \in I} \ol{M}_i$, and so $\sum_{i \in I} \ol{M}_i=N$. Thus, since $\cA$ is a Grothendieck category, we have $\sum_{i \in I} (N_0 \cap \ol{M}_i)=N_0$. Since $N_0$ is finitely generated, we have $N_0=\ol{M}_i \cap N_0$ for some $i$, and so $N_0 \subset \ol{M}_i$. Since $M_i$ belongs to $\cC_0$ (since $M$ belongs to $\cC$) and $N_0$ is a subquotient of $M_i$, it too belongs to $\cC_0$ (since $\cC_0$ is a Serre subcategory). Thus $N$ belongs to $\cC$.
\item Suppose now that $M$ belongs to $\cA$ and there is a subobject $N$ of $M$ such that $N$ and $M/N$ belong to $\cC$. Let $M_0$ be a finitely generated subobject of $M$. Let $N_0=N \cap M_0$. This is finitely generated, since $M_0$ is finitely generated and $\cA$ is locally noetherian, and thus belongs to $\cC_0$, since $N$ belongs to $\cC$. Similarly, $M_0/N_0$ is finitely generated, being a quotient of $M_0$, and is a subobject of $M/N$, and thus belongs to $\cC_0$, since $M/N$ belongs to $\cC$. Thus $M_0$ is an extension of objects in $\cC_0$ and thus belongs to $\cC_0$ since $\cC_0$ is a Serre subcategory.
\item We have thus shown that $\cC$ is Serre subcategory. It remains to show that it is closed under direct sums. Thus suppose that $\{M_i\}_{i \in I}$ is a collection of objects in $\cC$, and consider $M=\bigoplus_{i \in I} M_i$. Let $N$ be a finitely generated subobject of $M$. For a finite subset $J$ of $I$, let $M_J=\bigoplus_{i \in J} M_i$, so that $M=\sum_J M_J$. Thus $N=\sum_J (M_J \cap N)$. Since $N$ is finitely generated, we have $N=M_J \cap N$ for some $J$, i.e., $N \subset M_J$. Since $M_J$ is a finite sum of objects in $\cC$, it belongs to $\cC$ (since $\cC$ is a Serre subcategory). Thus $N$ belongs to $\cC_0$. This shows that $M$ belongs to $\cC$, verifying the claim.
\end{itemize}
Since $\cC$ is a localizing subcategory and contains $S$, we have $\cB \subset \cC$ by definition of $\cB$. On the other hand, it is clear that $\cC_0 \subset \cB$, and so $\cC \subset \cB$ since $\cB$ is closed under direct limits. Thus $\cC=\cB$. If $M$ is a finitely generated object in this category then it necessarily belongs to $\cC_0$, and thus admits the requisite filtration.
\end{proof}

The following proposition gives a simple criterion for a functor to map one localizing category into another.

\begin{proposition} \label{prop:locmap}
Let $\sF \colon \cA \to \cA'$ be a cocontinuous functor of Grothendieck abelian categories, with $\cA$ locally noetherian. Let $S$ be a collection of objects in $\cA$ generating a localizing subcategory $\cB$, and let $\cB'$ be a localizing subcategory of $\cA'$. Suppose $\sF$ is exact on $\cB$ and $\sF(S) \subset \cB'$. Then $\sF(\cB) \subset \cB'$.
\end{proposition}

\begin{proof}
Let $M$ be a finitely generated object of $\cB$. By Proposition~\ref{prop:locgen}, we have a filtration $0=F^0 \subset \cdots \subset F^n=M$ and objects $E_1, \ldots, E_n$ of $S$ such that $F^i/F^{i-1}$ is a subquotient of $E_i$. By hypothesis, $\sF(E_i)$ belongs to $\cB'$. Since $\sF$ is exact on $\cB$, we see that $\sF(F^i/F^{i-1})$ is a subquotient of $\sF(E_i)$, and thus belongs to $\cB'$. The exactness of $\sF$ also shows that $\sF(M)$ is a succsesive extension of the objects $\sF(F^i/F^{i-1})$, and therefore belongs to $\cB$.

Now suppose that $M$ is an arbitrary object of $\cB$. Then $M=\varinjlim_{i \in I} M_i$ for some direct set $I$, where the $M_i$ are finitely generated subobjects of $M$. Since $\sF$ is cocontinuous, we have $\sF(M) = \varinjlim_{i \in I} \sF(M_i)$. By the previous paragraph, $\sF(M_i)$ belongs to $\cB'$. Since $\cB'$ is closed under direct limits, we see that $\sF(M)$ also belongs to $\cB'$.
\end{proof}

\subsection{Defining functors on projectives}

\begin{proposition} \label{prop:funonproj}
Let $\cA$ and $\cT$ be abelian categories. Suppose that $\cA$ has enough projectives, and let $\cP$ be the full subcategory spanned by the projective objects. Then the restriction functor
\begin{displaymath}
\Phi \colon \{ \text{right exact functors $\cA \to \cT$} \} \to \{ \text{additive functors $\cP \to \cT$} \}
\end{displaymath}
is an equivalence. Moreover, if $\cA$ and $\cT$ are cocomplete and direct sums are exact then $\Phi$ induces an equivalence
\begin{displaymath}
\{ \text{cocontinuous functors $\cA \to \cT$} \} \to \{ \text{functors $\cP \to \cT$ commuting with all direct sums} \}.
\end{displaymath}
\end{proposition}

\begin{proof}
Suppose that $G \colon \cP \to \cT$ is an additive functor. Let $M \in \cA$, and choose a presentation $P_1 \to P_0 \to M \to 0$ with $P_1$ and $P_0$ projective. Define $F(M)$ to be the cokernel of the map $G(P_1) \to G(P_0)$. One easily sees that this yields a well-defined right exact functor $F \colon \cA \to \cT$. Define $\Psi(G)=F$. One easily sees that $\Psi$ defines a quasi-inverse functor to $\Phi$. This yields the first equivalence. For the second, simply note that if $F$ and $G$ correspond under the first equivalence then $F$ is cocontinuous if and only if $G$ commutes with all direct sums.
\end{proof}

\begin{proposition} \label{prop:funonproj2}
Let $\cA$ and $\cT$ be Grothendieck abelian categories. Let $\cQ$ be a full subcategory of $\cA$ whose objects are finitely generated projectives. Suppose that every object of $\cA$ is a quotient of a sum of objects of $\cQ$. Then the restriction functor
\begin{displaymath}
\{ \text{cocontinuous functors $\cA \to \cT$} \} \to \{ \text{additive functors $\cQ \to \cT$} \}
\end{displaymath}
is an equivalence.
\end{proposition}

\begin{proof}
Let $\cP \subset \cA$ be the full subcategory spanned by the projective objects. Appealing to the previous proposition, it suffices to show that the restriction functor
\begin{displaymath}
\{ \text{functors $\cP \to \cT$ commuting with all direct sums} \} \to \{ \text{additive functors $\cQ \to \cT$} \}
\end{displaymath}
is an equivalence.

Let $\cQ'$ be the following category. An object is a family $\{Q_i\}_{i \in I}$ with $Q_i \in \cQ$. A morphism $\{Q_i\}_{i \in I} \to \{Q'_j\}_{j \in J}$ consists of giving, for each $i$, a morphism $Q_i \to \bigoplus_{j\in J_i} Q'_j$, where $J_i$ is a finite subset of $J$. Let $\cQ''$ be the Karoubian envelope of $\cQ'$. There is a natural functor $\cQ' \to \cP$ taking $\{Q_i\}_{i \in I}$ to $\bigoplus_{i \in I} Q_i$, which extends to a functor $\cQ'' \to \cP$ since $\cP$ is Karoubian. This functor is an equivalence: the key point is that if $Q \in \cQ$ then any morphism $Q \to \bigoplus_{i \in I} M_i$ with $M_i \in \cA$ factors through some finite direct sum, since $Q$ is finitely generated.

Now, any functor $\cQ \to \cT$ induces a functor $\cQ'' \to \cT$, since $\cQ''$ is obtained naturally from $\cQ$ and $\cT$ has all direct sums, and thus a functor $\cP \to \cT$. One verifies that this construction provides the requisite quasi-inverse to the restriction functor.
\end{proof}

\subsection{An acyclicity criterion}

\begin{proposition} \label{prop:acyclic}
Let $\cA$ and $\cB$ be Grothendieck abelian categories, and let $\sF \colon \Ch(\cA) \to \Ch(\cB)$ be a functor. Suppose that:
\begin{enumerate}
\item $\sF$ is exact and cocontinuous.
\item $\sF$ takes cones to cones, that is, if $f \colon M \to N$ is a morphism in $\Ch(\cA)$ then we have a canonical isomorphism $\sF(\cone(f)) \cong \cone(\sF(f))$ in $\Ch(\cB)$.
\item $\sF$ ``only adds terms to the left,'' that is, $\sF(\Ch^{\le n}(\cA)) \subset \Ch^{\le n}(\cB)$, where $\Ch^{\le n}$ is the category of chain complexes supported in degrees $\le n$.
\end{enumerate}
Then $\sF$ takes acyclic complexes to acyclic complexes and quasi-isomorphisms to quasi-isomorphisms.
\end{proposition}

\begin{proof}
In light of condition (b), the two conclusions of the proposition are equivalent. We prove the statement about acyclic complexes. Thus let $M$ be an ayclic complex in $\Ch(\cA)$. To show that $\sF(M)$ is acyclic, we proceed in four steps:

\textit{\textbf{Step 1:} $M$ is a 2-term complex.} In this case, $M=\cone(f)$ for some isomorphism $f$ of 1-term complexes. Thus $\sF(M) \cong \cone(\sF(f))$ by condition~(b), and therefore acyclic, since $\sF(f)$ is an isomorphism.

\textit{\textbf{Step 2:} $M$ is a bounded complex.} For a complex $N$, let $t_{\le n}(N)$ denote its canonical truncation: the terms are given by
\begin{displaymath}
t_{\le n}(N)^k = \begin{cases}
N^k & \text{if $k<n$} \\
\ker(d \colon N^n \to N^{n+1}) & \text{if $k=n$} \\
0 & \text{if $k>n$} \end{cases}
\end{displaymath}
Then $t_{\le n}(N)$ is a subcomplex of $N$, and is acyclic if $N$ is. Suppose now that $M^k=0$ for $k>n$. Consider the exact sequence of complexes
\begin{displaymath}
0 \to t_{\le n-1}(M) \to M \to M' \to 0.
\end{displaymath}
Then $M'$ is a 2-term acyclic complex, while $t_{\le n-1}(M)$ is an acyclic complex that is shorter than $M$. By Step~1, $\sF(M')$ is acyclic, while by induction on the length of the complex, $\sF(t_{\le n-1}(M))$ is acyclic. By condition~(a), we see that
\begin{displaymath}
0 \to \sF(t_{\le n-1}(M)) \to \sF(M) \to \sF(M') \to 0
\end{displaymath}
is exact, and so $\sF(M)$ is acyclic.

\textit{\textbf{Step 3:} $M$ is bounded above.} Let $n$ be given, and let us show that $\rH^n(\sF(M))=0$. Consider the exact sequence
\begin{displaymath}
0 \to t_{\le n-1}(M) \to M \to M' \to 0.
\end{displaymath}
All the complexes above are acyclic, and $M'$ is bounded. By Step~2, $\sF(M')$ is acyclic. We thus see that the natural map
\begin{displaymath}
\rH^n(\sF(t_{\le n-1}(M))) \to \rH^n(\sF(M))
\end{displaymath}
is an isomorphism. But $t_{\le n-1}(M)^k=0$ for $k \ge n$, so $\sF(t_{\le n-1}(M))^k=0$ for $k \ge n$ as well by condition~(c). Thus the above homology groups vanish.

\textit{\textbf{Step 4:} the general case.} The complex $M$ is the direct limit of the acyclic subcomplexes $t_{\le n}(M)$, each of which is bounded above. By condition~(a), we thus see that $\sF(M)$ is the direct limit of the complexes $\sF(t_{\le n}(M))$, each of which is acyclic by Step~3. Thus $\sF(M)$ is acyclic.
\end{proof}

\begin{proposition} \label{prop:qi}
Let $\cA$ and $\cB$ be Grothendieck abelian categories, let $\sF,\sG \colon \Ch(\cA) \to \Ch(\cB)$ be functors, and let $\phi \colon \sF \to \sG$ be a natural transformation. Suppose that:
\begin{enumerate}
\item $\sF$ and $\sG$ satisfy the three conditions of Proposition~\ref{prop:acyclic}.
\item $\cA$ has enough projectives.
\item If $P \in \cA$ is projective then $\phi_P \colon \sF(P) \to \sG(P)$ is a quasi-isomorphism.
\end{enumerate}
Then $\phi_M \colon \sF(M) \to \sG(M)$ is a quasi-isomorphism for all $M \in \Ch(\cA)$.
\end{proposition}

\begin{proof}
Suppose now that $M$ is a bounded complex of projective modules. We show that $\sF(M) \to \sG(M)$ is a quasi-isomorphism by induction on the length of the complex. Assuming $M$ is non-zero, let $n$ be maximal such that $M^n$ is non-zero, and let $N$ be the ``stupid truncation'' of $M$ to degrees $<n$; that is, $N^k=M^k$ for $k<n$ and $N^k=0$ for $k \ge n$. We then have an exact sequence of complexes
\begin{displaymath}
0 \to N \to M \to K \to 0,
\end{displaymath}
where $K=M^n[-n]$ is concentrated in degree $n$. This gives rise to a diagram
\begin{displaymath}
\xymatrix{
0 \ar[r] & \sF(N) \ar[r] \ar[d] & \sF(M) \ar[r] \ar[d] & \sF(K) \ar[r] \ar[d] & 0 \\
0 \ar[r] & \sG(N) \ar[r] & \sG(M) \ar[r] & \sG(K) \ar[r] & 0 }
\end{displaymath}
The rows are exact since $\sF$ and $\sG$ are exact. The right vertical map is a quasi-isomorphism by assumption, while the left one is by the inductive hypothesis. Thus the middle one is too.

Now suppose that $M$ is a bounded above complex of projective modules. We show that $\rH^i(\sF(M)) \to \rH^i(\sG(M))$ is an isomorphism for all $i$. Thus let $i$ be given. Let $N$ be the stupid truncation of $M$ to degrees $\le i-1$, and let $K$ be the quotient $M/N$. We again get a diagram like the above. Since $K$ is bounded, we know that $\rH^i(\sF(K)) \to \rH^i(\sG(K))$ is an isomorphism. Since $N$ is supported in degrees $\le i-1$, so is the complex $\sF(N)$, and so $\rH^i(\sF(K))=\rH^{i+1}(\sF(K))=0$; similarly for the $\sG$ versions. It follows that $\rH^i(\sF(M)) \to \rH^i(\sG(M))$ is an isomorphism.

Suppose now that $M$ is an arbitrary bounded above complex. Let $P \to M$ be a quasi-isomorphism, with $P$ a bounded above complex of projectives. By Proposition~\ref{prop:acyclic}, $\sF$ and $\sG$ takes quasi-isomorphisms to quasi-isomorphisms. Thus $\sF(P) \to \sF(M)$ is a quasi-isomorphism, and similarly for $\sG$. We have shown that $\sF(P) \to \sG(P)$ is a quasi-isomorphis. Considering the obvious commutative square, it follows that $\sF(M) \to \sG(M)$ is a quasi-isomorphism.

Finally, suppose that $M$ is an arbitrary complex. Then $M$ is the direct limit of bounded above subcomplexes. Since $\sF$ and $\sG$ commute with direct limits, the result follows.
\end{proof}

\subsection{An equivalence criterion}

For an abelian category $\cA$, let $\Inj(\cA)$, resp.\ $\IndInj(\cA)$, denote the full subcategory spanned by the injective, resp.\ indecomposable injective, objects of $\cA$.

\begin{proposition} \label{prop:indinjequiv}
Let $\sF \colon \cA \to \cB$ be a functor of locally noetherian Grothendieck abelian categories. Suppose that:
\begin{enumerate}
\item $\sF$ is left-exact and commutes with filtered colimits.
\item $\sF$ takes finitely generated objects to finitely generated objects.
\item $\sF$ induces an equivalence $\IndInj(\cA) \to \IndInj(\cB)$.
\end{enumerate}
Then $\sF$ is an equivalence.
\end{proposition}

\begin{proof}
We first show that $\sF$ is fully faithful. To this end, we say that a pair $(M,N)$ of objects of $\cA$ is {\bf good} if the  natural map
\begin{displaymath}
\Hom_{\cA}(M, N) \to \Hom_{\cB}(\sF(M), \sF(N))
\end{displaymath}
is an isomorphism. From (c), we see that $(I, J)$ is good for all $I,J \in \IndInj(\cA)$. We show that all pairs of objects are good in several steps.

\textit{\textbf{Step 1.}}
Suppose that $I \in \Inj(\cA)$ and $J \in \IndInj(\cA)$. Write $I=\bigoplus_{\alpha \in \cU} I_{\alpha}$ where $\cU$ is an index set and each $I_{\alpha}$ is an indecomposable injective; this is possible because $\cA$ is locally noetherian. We have a commutative square
\begin{displaymath}
\xymatrix{
\Hom_{\cA}(I, J) \ar[r] \ar@{=}[d] &
\Hom_{\cB}(\sF(I), \sF(J)) \ar@{=}[d] \\
\prod_{\alpha \in \cU} \Hom_{\cA}(I_{\alpha}, J) \ar[r] &
\prod_{\alpha \in \cU} \Hom_{\cB}(\sF(I_{\alpha}), \sF(J)) }
\end{displaymath}
To get the right vertical equiality, we used that $\sF(I)=\bigoplus_{\alpha \in \cU} \sF(I_{\alpha})$, which follows from (a). Since $(I_{\alpha}, J)$ is good for all $\alpha$, the bottom map is an isomorphism. Thus the top map is an isomorphism as well. We conclude that $(I, J)$ is good for all $I \in \Inj(\cA)$ and $J \in \IndInj(\cA)$.

\textit{\textbf{Step 2.}}
Now suppose that $J \in \IndInj(\cA)$ and $M \in \cA$ is arbitrary. Choose an exact sequence
\begin{displaymath}
0 \to M \to I_0 \to I_1
\end{displaymath}
with $I_0, I_1 \in \Inj(\cA)$. This is possible because $\cA$ has enough injectives. We then get a diagram
\begin{displaymath}
\xymatrix{
\Hom_{\cA}(I_1, J) \ar[r] \ar[d] & \Hom_{\cA}(I_0, J) \ar[r] \ar[d] & \Hom_{\cA}(M, J) \ar[r] \ar[d] & 0 \\
\Hom_{\cB}(\sF(I_1), \sF(J)) \ar[r] & \Hom_{\cB}(\sF(I_0), \sF(J)) \ar[r] & \Hom_{\cB}(\sF(M), \sF(J)) \ar[r] & 0 }
\end{displaymath}
The rows are exact: this follows from the fact that $\sF$ is left-exact and that $J$ and $\sF(J)$ are injective objects. Since $(I_0, J)$ and $(I_1, J)$ are good (Step~1), the two left vertical maps are isomorphisms, and so the right vertical map is as well. We conclude that $(M, J)$ is good for all $M \in \cA$ and all $J \in \IndInj(\cA)$.

\textit{\textbf{Step 3.}}
Now suppose that $J \in \Inj(\cA)$ and $M \in \cA^{\fgen}$. Write $J=\bigoplus_{\alpha \in \cU} J_{\alpha}$ with $J_{\alpha}$ indecomposable. We have a commutative square
\begin{displaymath}
\xymatrix{
\Hom_{\cA}(M, J) \ar[r] \ar@{=}[d] &
\Hom_{\cB}(\sF(M), \sF(J)) \ar@{=}[d] \\
\bigoplus_{\alpha \in \cU} \Hom_{\cA}(M, J_{\alpha}) \ar[r] &
\bigoplus_{\alpha \in \cU} \Hom_{\cB}(\sF(M), \sF(J_{\alpha})) }
\end{displaymath}
The left vertical equality follows from $M$ being finitely generated (and thus noetherian). The right vertical equality follows from $\sF(M)$ being finitely generated, which comes from (b), and the identification $\sF(J)=\bigoplus_{\alpha \in \cU} \sF(J_{\alpha})$, which comes from (a). Since $(M, J_{\alpha})$ is good (Step~2), the bottom map is an isomorphism. Thus the top one is as well. We conclude that $(M, J)$ is good for all $J \in \Inj(\cA)$ and all $M \in \cA^{\fgen}$.

\textit{\textbf{Step 4.}}
Let $J \in \Inj(\cA)$ and $M \in \cA$. Write $M=\varinjlim_{\alpha \in \cU} M_{\alpha}$ where $\cU$ is a directed set and each $M_{\alpha}$ is a finitely generated subobject of $M$. This is possible since $\cA$ is locally noetherian. We have a commutative square
\begin{displaymath}
\xymatrix{
\Hom_{\cA}(M, J) \ar[r] \ar@{=}[d] &
\Hom_{\cB}(\sF(M), \sF(J)) \ar@{=}[d] \\
\varprojlim \Hom_{\cA}(M_{\alpha}, J) \ar[r] &
\varprojlim \Hom_{\cB}(\sF(M_{\alpha}), \sF(J)) }
\end{displaymath}
Here we are simply using the universal property of direct limits, and the fact that $\sF(M)=\varinjlim \sF(M_{\alpha})$, which comes from (a). Since each $(M_{\alpha}, J)$ is good (Step~3), the lower map is an isomorphism. Thus the upper map is an isomorphism. We conclude that $(M, J)$ is good for all $M \in \cA$ and all $J \in \Inj(\cA)$.

\textit{\textbf{Step 5.}}
Finally, let $M,N \in \cA$ be arbitrary. Pick an exact sequence
\begin{displaymath}
0 \to N \to J_0 \to J_1
\end{displaymath}
with $J_0, J_1 \in \Inj(\cA)$. We obtain a commutative diagram
\begin{displaymath}
\xymatrix{
0 \ar[r] & \Hom_{\cA}(M, N) \ar[r] \ar[d] & \Hom_{\cA}(M, J_0) \ar[r] \ar[d] & \Hom_{\cA}(M, J_1) \ar[d] \\
0 \ar[r] & \Hom_{\cB}(\sF(M), \sF(N)) \ar[r] & \Hom_{\cB}(\sF(M), \sF(J_0)) \ar[r] & \Hom_{\cB}(\sF(M), \sF(J_1)) }
\end{displaymath}
The rows are exact: this follows from the left exactness of $\Hom$ and $\sF$. Since $(M, J_0)$ and $(M, J_1)$ are good (Step~4), the right two vertical maps are isomorphisms. Thus the left vertical map is an isomorphism. We conclude that $(M, N)$ is good for all $M,N \in \cA$.

\textit{\textbf{Completion of proof.}}
We have shown that $\sF$ is fully faithful. To complete the proof, we must show that $\sF$ is essentially surjective. We first claim that given $J \in \Inj(\cB)$ there exists $I \in \Inj(\cA)$ with $\sF(I) \cong J$. Indeed, write $J=\bigoplus_{\alpha \in \cU} J_{\alpha}$ with $J_{\alpha} \in \IndInj(\cB)$. By (c), for each $\alpha$ we can find some $I_{\alpha} \in \IndInj(\cA)$ and an isomorphism $J_{\alpha} \cong \sF(I_{\alpha})$. Let $I=\bigoplus_{\alpha \in \cU} I_{\alpha}$, which is an injective object of $\cA$. Then $\sF(I)=\bigoplus_{\alpha \in \cU} \sF(I_{\alpha}) \cong J$, where in the first identification we have used (a). This establishes the claim.

Now let $N \in \cB$ be an arbitrary object. Choose an exact sequence
\begin{displaymath}
0 \to N \to J_0 \stackrel{g}{\to} J_1
\end{displaymath}
with $J_0, J_1 \in \Inj(\cB)$. By the previous paragraph, we can find $I_0, I_1 \in \Inj(\cA)$ with $\sF(I_0) \cong J_0$ and $\sF(I_1) \cong J_1$. Since $\sF$ is fully faithful, the morphism $g \in \Hom_{\cB}(J_0, J_1)$ comes from a morphism $f \in \Hom_{\cA}(I_0, I_1)$. Let $M=\ker(f)$. Then
\begin{displaymath}
\sF(M) \cong \sF(\ker(f)) \cong \ker(\sF(f)) \cong \ker(g) \cong N,
\end{displaymath}
where in the second identification we used the left-exactness of $\sF$. Thus $\sF$ is essentially surjective, which completes the proof.
\end{proof}

\subsection{Exact sequences of functors}

\begin{proposition} \label{prop:funseq}
Let $\cA$ and $\cB$ be Grothendieck abelian categories. Let
\begin{displaymath}
0 \to F_1 \to F_2 \to F_3 \to 0
\end{displaymath}
be an exact sequence in the functor category $\Fun(\cA, \cB)$.
\begin{enumerate}
\item If two of the functors are exact then so is the third.
\item If two of the functors commute with filtered colimits then so does the third.
\end{enumerate}
\end{proposition}

\begin{proof}
(a) This is a simple consequence of the 9-lemma (also called the $3 \times 3$ lemma).

(b) Suppose that $\{M_i\}_{i \in \cU}$ is a filtered system. Consider the diagram
\begin{displaymath}
\xymatrix{
0 \ar[r] & \varinjlim F_1(M_i) \ar[r] \ar[d] & \varinjlim F_2(M_i) \ar[r] \ar[d] & \varinjlim F_3(M_i) \ar[r] \ar[d] & 0 \\
0 \ar[r] & F_1(\varinjlim M_i) \ar[r] & F_2(\varinjlim M_i) \ar[r] & F_3(\varinjlim M_i) \ar[r] & 0 }
\end{displaymath}
The vertical maps are the canonical maps. The rows are exact since $\cA$ and $\cB$ are Grothendieck categories. Thus if two of the vertical maps are isomorphisms then so is the third.
\end{proof}

\subsection{Injectives}

The following is a version of Baer's criterion:

\begin{proposition} \label{prop:baer}
Let $\cA$ be a Grothendieck abelian category. Suppose that every object is the union of its finitely generated subobjects; for example, $\cA$ could be locally noetherian. Suppose that $I$ is an object of $\cA$ satisfying the following condition: given any diagram
\begin{displaymath}
\xymatrix{
& I \\
M \ar[r]^i \ar[ru]^f & N \ar@{..>}[u]_g }
\end{displaymath}
in which $M$ and $N$ are finitely generated and $i$ is injective, one can find a morphism $g$ such that the diagram commutes. Then $I$ is injective.
\end{proposition}

\begin{proof}
The proof is similar to the usual proof of Baer's criterion; see \stacks{079G} for a similar result.
\end{proof}

\end{document}